\documentclass[oneside]{book}

\usepackage{lipsum}

\usepackage{booktabs}

\usepackage{graphicx}
\setkeys{Gin}{width=\linewidth,totalheight=\textheight,keepaspectratio}
\graphicspath{{graphics/}}
\usepackage{pgfplots}
\usepgfplotslibrary{colormaps}
 
\pgfplotsset{
    compat=newest,
    colormap={mycolormap}{color=(lightgray) color=(white) color=(lightgray)}}

\usepackage{fancyvrb}
\fvset{fontsize=\normalsize}

\usepackage{xspace}

\usepackage{tikz}
\usepackage{tikz-cd}
\usetikzlibrary{calc}

\usepackage{hyperref} 

\newcommand{\monthyear}{%
  \ifcase\month\or January\or February\or March\or April\or May\or June\or
  July\or August\or September\or October\or November\or
  December\fi\space\number\year
}

\usepackage{units}

\newcommand{\hlred}[1]{\textcolor{Maroon}{#1}}
\newcommand{\hangleft}[1]{\makebox[0pt][r]{#1}}

\providecommand{\XeLaTeX}{X\lower.5ex\hbox{\kern-0.15em\reflectbox{E}}\kern-0.1em\LaTeX}

\newcommand{\tuftebs}{\symbol{'134}}
\newcommand{\doccmddef}[2][]{%
  \hlred{\texttt{\tuftebs#2}}\label{cmd:#2}%
  \ifthenelse{\isempty{#1}}%
    {
      \index{#2 command@\protect\hangleft{\texttt{\tuftebs}}\texttt{#2}}
    }%
    {
      \index{#2 command@\protect\hangleft{\texttt{\tuftebs}}\texttt{#2} (\texttt{#1} package)}
      \index{#1 package@\texttt{#1} package}\index{packages!#1@\texttt{#1}}
    }%
}
\newcommand{\doccmd}[2][]{%
  \texttt{\tuftebs#2}%
  \ifthenelse{\isempty{#1}}%
    {
      \index{#2 command@\protect\hangleft{\texttt{\tuftebs}}\texttt{#2}}
    }%
    {
      \index{#2 command@\protect\hangleft{\texttt{\tuftebs}}\texttt{#2} (\texttt{#1} package)}
      \index{#1 package@\texttt{#1} package}\index{packages!#1@\texttt{#1}}
    }%
}

\usepackage{makeidx}
\makeindex

\usepackage{amssymb, amsmath, amscd, amsthm,mathrsfs,amsfonts,bm,mathtools}

\newcommand{\K}{\mathbb{K}}
\newcommand{\R}{\mathbb{R}}
\newcommand{\Z}{\mathbb{Z}}
\newcommand{\N}{\mathbb{N}}
\newcommand{\p}{\varphi}

\newcommand{\eps}{\varepsilon}
\newcommand{\dgm}{\mathrm{Dgm}}

\newcommand{\myemph}[1]{\index{#1}\emph{#1}}%

\newcommand{\Imm}{\mathrm{Im}\ }
\newcommand{\id}{\mathrm{id}}

\newcommand{\GENEO}{\mathrm{GENEO}}

\newcommand\mycom[2]{\genfrac{}{}{0pt}{}{#1}{#2}}

\newtheorem{theorem}{Theorem}
\newtheorem{lemma}{Lemma}
\newtheorem{definition}{Definition}
\newtheorem{ex}{Example}
\newtheorem{exercise}{Exercise}
\newtheorem{remark}{Remark}
\newtheorem{proposition}{Proposition}
\newtheorem{corollary}{Corollary}
\newtheorem{assumption}{Assumption}

\newcommand{\match}{\mathrm{match}}
\newcommand{\Homeo}{{\mathrm{Homeo}}}

\begin{document}

\nocite{Sp95}
\author{Patrizio Frosini, Ulderico Fugacci, Nicola Quercioli, Francesca Tombari}
\title{Persistent Homology and Equivariance \\  in Data Analysis \\[1ex] \large A Topological Introduction}
\date{}
\maketitle

\frontmatter

\clearpage
\thispagestyle{empty}
\null\vfill
\begin{center}
  This draft manuscript is currently under review for publication as a book.\\
  We would be grateful to readers who report any errors, inaccuracies, or misprints they may encounter to the following email address:\\
\href{patrizio.frosini@unipi.it}{patrizio.frosini@unipi.it}.
\end{center}

\chapter*{Abstract}
This new book is intended as a first elementary introduction to Topological Data Analysis for mathematics students seeking a rigorous account of the foundations of persistent homology, as well as for computer scientists interested in its theoretical underpinnings. The exposition is as self-contained as possible: all the required background is recalled when needed, and only a few standard results are cited without proof. One section of the book, devoted to monodromy in biparameter persistence (Section 4.4), requires more advanced knowledge of algebraic topology.

Persistent homology can be introduced from different perspectives, reflecting the variety of mathematical languages that have shaped its development over the years. Some approaches emphasize the algebraic foundations of the theory, while others highlight its topological essence. In this book, we adopt the latter viewpoint - the one that historically marked the birth of the subject - because we believe it offers both conceptual clarity and pedagogical effectiveness, making it particularly suitable for undergraduate and early graduate students.

This book differs from existing introductory texts in several respects. First, it adopts a functional viewpoint: rather than representing data as finite (pseudo-)metric spaces, it treats them as functions encoding the information to be analyzed. This interpretative framework allows data to be viewed as measurable objects and highlights the role of observers and their equivariances in the analysis process. Second, this perspective provides a natural bridge between Topological Data Analysis and machine learning through the theory of Group Equivariant Non-Expansive Operators (GENEOs), which offers a mathematically grounded framework for incorporating symmetries and invariances into learning systems.

\chapter{Preface}
Topological Data Analysis (TDA) is a rapidly evolving mathematical discipline that explores the topological structure of data — that is, the shape of data — in a metric, multilevel, stable, and reparametrization-invariant way.
Its remarkable growth stems from a simple but powerful observation: data often encode meaningful geometric information that can be organized, simplified, computed, and ultimately used for analysis and comparison.

Among the core techniques of TDA, persistent homology stands out as one of the most influential.
Its distinctive feature, compared to classical homology theories, lies in its ability to detect and quantify both local and global features of a topological space.

Persistent homology can be introduced through different perspectives, reflecting the variety of mathematical languages that have shaped its development over the years.
Some approaches emphasize the algebraic foundations of the theory, while others highlight its topological essence.
In this book, we adopt the latter viewpoint — the one that historically marked the birth of the subject — because we believe it offers both conceptual clarity and pedagogical effectiveness, making it particularly suitable for undergraduate and early graduate students.

This book also differs from existing introductory texts in several ways.
First, it adopts a functional viewpoint: instead of representing data as finite (pseudo-) metric spaces, it regards them as functions encoding the information to be analyzed.
This interpretative key allows data to be treated as measurable objects and underscores the role of observers and their equivariances in the analysis process.
This perspective also provides the conceptual foundation for many modern applications of TDA in machine learning.

The book opens with two preliminary chapters (Chapters \ref{Chapter0} and \ref{CC}), which recall some basic notions about pseudo-metrics and Euclidean submanifolds and introduce simplicial and singular homology.
The core of the text begins in Chapter \ref{ChapterHC}, where persistent homology groups, the rank invariant, and persistence diagrams of a function on a topological space are defined.
After presenting these fundamental objects, the Representation Theorem establishes the correspondence between rank invariants and persistence diagrams, while the Stability Theorem relates variations in persistence diagrams to changes in the associated filtering functions.

Chapter \ref{MPH} extends these ideas to the study of rank invariants for functions valued in $\mathbb{R}^2$, introducing the concept of Extended Pareto Grid.
Its main result, the Position Theorem, provides a geometric tool for localizing points in persistence diagrams associated with planar lines of positive slope intersecting the grid.
This chapter also presents a generalization of the Stability Theorem to the biparametric setting and analyzes the phenomenon of monodromy, a distinctive feature of the two-parameter model.

Chapter \ref{NPD&GENEOs} explores the interaction between persistence theory and equivariant machine learning.
Here, we show how transformation groups acting on data can be incorporated into the framework through the notions of natural pseudo-distance and Group Equivariant Non-Expansive Operators (GENEOs).
The chapter concludes with proofs of compactness and convexity for the space of GENEOs.

A list of further recommended readings and bibliographic references on TDA is provided at the end of the text.

This book is intended both as a reference and a learning companion for mathematics students seeking a rigorous introduction to the foundations of persistent homology, and for computer scientists interested in its theoretical underpinnings.
To this end, numerous examples and exercises are included to support understanding and encourage active exploration.
The exposition is as self-contained as possible: all the required background is recalled when needed, and only a few standard results are cited without proof.
As such, the text is accessible to students who has completed the first year of a degree program in mathematics, physics, computer science, or engineering, with elementary notions of 
linear algebra and general topology.
One section of this book, devoted to monodromy in biparameter persistence (Section~\ref{Monodromy}), requires more advanced knowledge in algebraic topology. 
This may be of interest to a more expert audience. 
However, this section is designed to possibly be omitted for a first exposure to the topic.

Finally, the material is flexible enough to accommodate different teaching needs.
The text can serve as the basis for a semester-long introductory course on TDA, but it can also be tailored for shorter modules by selectively focusing on specific chapters.
For example, a concise course may omit the preliminary material on pseudo-metric spaces and smooth manifolds (Chapter 0) and focus on Chapters 1 and 2, while a course devoted to the interaction between TDA and the theory of group actions may centre on Chapter 4, where persistent homology naturally integrates with the concept of equivariant operators.

\vspace{\baselineskip}
\begin{flushright}\noindent
March 2026
\hfill {\it Patrizio Frosini}\\
\hfill {\it Ulderico Fugacci}\\
\hfill {\it Nicola Quercioli}\\
\hfill {\it Francesca Tombari}\\
\end{flushright}

\tableofcontents

\mainmatter


\chapter{Preliminary notions}\label{Chapter0}
In this preliminary chapter, we illustrate the definitions of (extended pseudo-) metric space, and the basics about Euclidean manifolds. 
These are some of the main objects of study throughout the rest of the chapters.

\section{Pseudo-metric spaces}
In this section we recall some definitions on (extended pseudo-)metric spaces, that will be relevant for the rest of the book.  

\begin{definition}\label{pseudo}
A \myemph{metric space} is a pair $(X, d)$ where $X$ is a set and $d\colon X\times X\to [0,\infty)$ is a function such that, for every $x, x', x''$ in $X$,
\begin{enumerate}
\item $d(x, x)=0$, 
\item $d(x, x')=0$ implies $x=x'$,
\item $d(x, x')=d(x', x)$, 
\item $d(x,x'')\le d(x,x')+d(x',x'')$. 
\end{enumerate}
The function $d$ is referred to as a \myemph{metric} on $X$. 
If $d\colon X\times X\to [0,\infty]$, then we say that $d$ is an \myemph{extended metric}, and $(X, d)$ is an \myemph{extended metric space}.
If (2) is not required to hold, then we say that $d$ is an (extended) \myemph{pseudo-metric}, and $(X, d)$ is an (extended) \myemph{pseudo-metric space}.  
\end{definition}

From now on, we will avoid specifying when a metric is extended unless it is relevant to the context.

\begin{remark} If $x,x'$ and $x''$ are elements of a pseudo-metric space $(X,d)$, then the reverse triangle inequality holds: 
$\lvert d(x,x'')-d(x',x'')\rvert\le d(x,x')$.
\end{remark}

An example of a metric space that will be used throughout the book is the following. 
\begin{ex}
Let $X$ be a set and $(\R^k, \lVert \cdot\rVert_\infty)$ be the normed vector space, where we set $\lVert(v_1, \dots, v_k)\rVert_\infty=\max _{i=1}^k\lvert v_i \rvert$. 
The set of functions $\p=(\p_1, \dots, \p_k)\colon X\to \R^k$ is an extended metric space with the uniform metric 
\[
\lVert\p- \psi\rVert_\infty=\sup _{x\in X}\lVert \p(x)-\psi(x)\rVert_\infty=\sup _{x\in X}\max _{i=1}^k\lvert \p_i(x)-\psi_i(x)\rvert. 
\] 
Observe that if $X$ is compact and the functions are continuous, 
then the uniform metric assumes finite values, and, hence, is a metric. 
\end{ex}

\begin{ex}\label{ex_pseudo}
Consider the function $d\colon \R^2\times \R^2\to \R$ defined by setting $$d((v_1, v_2), (v'_1,v'_2))=\lvert v_1-v'_1\rvert.$$
This satisfies properties 1,3, and 4 of Definition~\ref{pseudo}, so it is a pseudo-metric. 
However, it does not satisfy 2 because 
$d((v_1, v_2), (v_1,v'_2))=0$
even if $v_2\neq v'_2$.
Thus, it is not a metric.
\end{ex}

\begin{definition}
Consider two (pseudo-)metric spaces $(X,d_X)$ and $(Y,d_Y)$. 
A function $f\colon X\to Y$ is \myemph{non-expansive} if, for every $x,x'\in X$,
\[
d_Y(f(x), f(x'))\le d_X(x, x').
\]
\end{definition}




\begin{definition}
Let $(X, d_X)$ and $(Y, d_Y)$ be pseudo-metric spaces. A bijective map
\(
f\colon X \to Y
\)
is called an \emph{isometry} if for all $x, x' \in X$,
\(
d_Y\big(f(x), f(x')\big) = d_X(x, x').
\)
\end{definition}

We observe that a subspace $K$ of a pseudo-metric space $(X, d_X)$ is compact 
if and only if it is sequentially compact, i.e., every sequence in $K$ admits a convergent subsequence.
Equivalently, $K$ is compact if and only if it is complete and totally bounded.

In this book, the following definition will also be useful:

\begin{definition}
The \myemph{uniform extended (pseudo-)metric} between two functions
$f, gc$ from $X$ to $Y$, where $X$ is a set and $(Y, d)$ is a (pseudo-)metric space,
is defined by
\[
D(f,g) = \sup_{x \in X} d\bigl(f(x), g(x)\bigr).
\]
If $f$ and $g$ are bounded, then the above formula defines a (pseudo-)metric in the strict sense.
\end{definition}
In this context, the associated notion of convergence is
the \myemph{uniform convergence}.
In particular, a sequence
$(f_n)_n$ converges uniformly to a function $f$
if and only if
\[
\lim_{n \to \infty} D(f_n,f) = 0.
\]
If \( (Y,\lVert\cdot\rVert) \) is a normed vector space, then the norm $\lVert\cdot\rVert$ naturally induces a metric.
In this setting, 
the value of the extended metric between two functions
$f, g \colon X \to Y$ remains well defined and is given by
\[
D(f,g) = \sup_{x \in X} \lVert f(x) - g(x)\rVert.
\]

\section{Smooth manifolds}
This section provides a brief introduction to selected notions from differential geometry, focusing on submanifolds embedded in Euclidean spaces. Rather than aiming for completeness, the objective is to establish a minimal theoretical framework that will support the development of subsequent chapters.
For more details, refer to~\cite{milnor1997topology,munkres2018analysis, tu2011manifolds}.

Recall that if $U\subseteq \R^n$ and $V\subseteq \R^m$ are open subsets, a function $F\colon U\to V$ is smooth if its partial derivatives of all orders exist.
Let $X\subseteq \R^n$ and $Y\subseteq \R^m$ be arbitrary subsets. 
A function $F\colon X\to Y$ is called smooth if for every $x\in X$ there exist an open set $U\subseteq \R^n$, with $x\in U$, and a smooth function $\bar F\colon U\to \R^m$ such that $\bar F_{\mid U\cap X}=F_{\mid U\cap X}$. 
A function $F\colon X\to Y$ is a \myemph{homeomorphism} if it is invertible and both $F$ and $F^{-1}$ are continuous. 
A function $F\colon X\to Y$ is called a \myemph{diffeomorphism} if it is 
invertible
and both $F$ and $F^{-1}$ are smooth. 


\begin{definition}\label{def_manifold}
A subset $M \subseteq \mathbb{R}^n$ is called a \myemph{smooth manifold} of dimension $k$
if for each $x \in M$ there exist two open sets $U \subseteq \mathbb{R}^n$ and
$V \subseteq \mathbb{R}^k$, with $x \in U$, and a diffeomorphism $\alpha \colon U\cap M \to V$, called a \myemph{coordinate chart}. The map $\alpha^{-1}$ is called a local \myemph{parametrisation}.
\end{definition}

\begin{ex}
The Euclidean space $\mathbb{R}^n$ can be viewed as a smooth manifold of dimension $n$ where a coordinate chart, for every $x$, is given by $\id\colon \R^n\to \R^n$.
\end{ex} 

\begin{ex}
Let $V \subseteq \mathbb{R}^n$ be an open set, and let $f \colon V \to \mathbb{R}^m$ be a smooth function. 
The graph of $f$, $G(f) = \{(x, f(x)) \in V \times \mathbb{R}^m\}$, is a smooth manifold of 
dimension $n$, 
where a coordinate chart on $G(f)$ is given by $\alpha \colon G(f)\to U$ with $\alpha(x, f(x))=x$, and a parametrisation of $G(f)$ is its inverse $\alpha^{-1}\colon U\to G(f)$, $\alpha^{-1}(x)=(x, f(x))$.
\end{ex}

For a smooth function $F\colon U\to V$, with $U\subseteq \R^n$ and $V\subseteq \R^m$ open, the \myemph{directional derivative} at $x\in U$ with respect to $v\in \R^n$ is the vector $dF_x(v)$ defined by 
\[
\mathrm{d}F_x(v)=\lim_{t\to 0} \frac{F(x+tv)-F(x)}{t}. 
\]
The directional derivative $\mathrm{d}F_x\colon \R
^n\to \R^m$ is a linear map. 
Choosing the basis associated with the standard coordinates of $\R^n$, denoted as $r^1, \dots, r^n$, this linear map corresponds to the Jacobian matrix of $F=(F^1, \dots, F^m)$ at $x$:
\[
\mathbf{J}_{F} (x)=
\begin{bmatrix}
\frac{\partial F^1}{\partial r^1}(x) & \cdots & \frac{\partial F^1}{\partial r^n}(x) \\
\vdots & \ddots & \vdots \\
\frac{\partial F^m}{\partial r^1}(x) & \cdots & \frac{\partial F^m(x)}{\partial r^n}(x)
\end{bmatrix}.
\]

\begin{definition}\label{def_tangent}
Let $M\subseteq \R^n$ be a smooth manifold of dimension $h$ and $\alpha^{-1}\colon U\to M$ a parametrisation of $\alpha^{-1}(U)$ containing a point $x\in M$, with $\alpha(x)=u$. 
The \myemph{tangent space} of $M$ at $x$ is $TM_x=\Imm \mathrm{d}\alpha^{-1}_u(\R^h)$.
\end{definition}

It is important to notice that Definition~\ref{def_tangent} does not depend on the chosen parametrisation. 
Furthermore, for every $x\in M$, the tangent space $TM_x$ is a well-defined vector space of dimension $h$.

Now we consider two smooth manifolds $M\subseteq \R^n$ and $N\subseteq \R^m$, and a smooth function $F\colon M\to N$ with $F(x)=y$. 
The \myemph{directional derivative} 
\[
\mathrm{d}F_x\colon TM_x\to TN_y
\]
is defined in the following way. 
Since $F$ is smooth there exists $U\subseteq \R^n$, with $x\in U$, and a smooth function $\bar F\colon U\to \R^m$ such that $\bar F_{\mid U\cap M}=F$. 
For every $v\in TM_x$, we define the directional derivative at $x$ with respect to $v$, again denoted $\mathrm{d}F_x(v)$, to be the vector $\mathrm{d}\bar F_x(v)$. 
One can see that the image of $\mathrm{d}F_x$ is, in fact, contained in $TN_y$, and that it does not depend on the choice of $\bar F$.

\begin{definition}
Let $F\colon M\to \R$ be a smooth function. 
The \myemph{gradient} of $F$ at $x\in M$ is the unique vector $\nabla F(x)$ such that $\langle \nabla F(x), v\rangle=\mathrm{d}F_x(v)$, where $\langle \cdot, \cdot\rangle$ denotes the standard dot product in $\R^h$.
\end{definition}

\begin{definition}
Consider a smooth function $F\colon M\to \R^m$.  
A point $x\in M$ is \myemph{regular} if $dF_x$ has maximal rank, otherwise it is \myemph{critical}. 
An element $c\in \R^m$ is called a \myemph{regular value} if 
$F^{-1}(c)$ is empty or all its points are regular.
The image of a critical point is called a \myemph{critical value}.
\end{definition}

A smooth function $F\colon M\to \R^m$ from a smooth manifold of dimension $h$ is called an \myemph{immersion} if the rank of $\mathrm{d}F_x$ is $h$, for every $x\in M$. 

The following theorem, which we report without proof, provides a method to produce examples of smooth manifolds, without the need of checking the diffeomorphism condition for every point of the manifold. 
\begin{theorem}
Let $F \colon N \to \mathbb{R}$ be a smooth function on a smooth manifold of dimension $k$.
If $c\in F(N)$ is a regular value, then the level set $M = F^{-1}(c)$ is a smooth manifold of dimension $k-1$.
\end{theorem}

\begin{ex}
Let $r$ be a positive real number. 
Consider the smooth function $F\colon \mathbb{R}^3 \to \mathbb{R}$ defined by $F(x, y, z) = x^2 + y^2 + z^2$.
Then the 2-sphere of radius $r>0$ centred at the origin is the level set $\mathbb{S}^2 := F^{-1}(r^2)$.
This level set is a smooth manifold of dimension 2 in $\mathbb{R}^3$ because $r^2$ is a regular value of the smooth function $F$. 
Indeed, choosing the basis associated with the standard coordinates for $\R^3$, $\nabla F{(x,y,z)}$ is the vector $(2x,2y,2z)$, which is non-zero, for every $(x,y,z)\in \mathbb{S}^2$.
\end{ex}

\begin{ex}
Let $R > r$ be two positive real numbers.
Consider the smooth function $F \colon \mathbb{R}^3 \to \mathbb{R}$ defined by $F(x, y, z) = \left(x^2 + y^2 + z^2 + R^2 - r^2\right)^2 - 4R^2(x^2 + y^2)$.
Then the level set $\mathbb{T}^2 := F^{-1}(0)$ is a 2-torus in $\mathbb{R}^3$.
Intuitively, the constant $R$ represents the distance from the centre of the tube to the centre of the torus, while $r$ is the radius of the tube.
This surface is a smooth manifold of dimension 2 in $\mathbb{R}^3$ because $0$ is a regular value of the function $F$. In particular, for every point $(x, y, z) \in \mathbb{T}^2$, the vector $\nabla F{(x,y,z)}$ is non-zero.
\end{ex}

\chapter{Homology}\label{CC}

In this chapter, we introduce the notion of singular homology of a topological space which intuitively captures its number of ``holes''. 
Specifically, homology provides a collection of vector spaces indexed by integers: the 0-th homology corresponds 
to the connected components of the topological space, 
while for positive indices it captures higher-dimensional holes such as loops, tunnels, cavities, and so on.
Furthermore, we introduce the notion of simplicial homology which, while retaining the same information as singular homology, is more computationally tractable thanks to its combinatorial nature.


\section{Chain complexes and homology}


In this section, we define chain complexes and homology in an algebraic and general setting. 
Their use will allow us to define simplicial and singular homology groups and to formalize the intuitive notion of the number of holes in a geometric object represented by a simplicial complex or a topological space, as will be explained in Sections \ref{SimpH} and \ref{SingH}.

For more details on homology theories, we refer to~\cite{Munkres,Ha02}.
For simplicity, we will restrict ourselves to considering homology with coefficients in $\Z_2$.

\begin{definition}\label{def_chain_complex}
A \myemph{chain complex} $\mathcal{C}:=(C_k,\partial_k)_k$ is a family of $\Z_2$-vector spaces $C_k$ over $\mathbb{Z}_2$
and linear maps $\partial_k\colon C_k\to C_{k-1}$, indexed by integer numbers, such that $\partial_{k}  \partial_{k+1}$ is the null map, for every $k\in\Z$. 
The map $\partial_k$ is called \myemph{$k$-boundary map}.
The elements of $C_k$, $\ker \partial_k$, and $\Imm \partial_{k+1}$ are respectively called \myemph{$k$-chains}, \myemph{$k$-cycles}, and \myemph{$k$-boundaries}.
\end{definition}

Assuming $\partial_{k}  \partial_{k+1}$ is the null map is equivalent to assuming $\Imm \partial_{k+1}\subseteq \ker \partial_{k}$,
and hence the quotient vector space $\ker \partial_{k}/\Imm \partial_{k+1}$ is well-defined.
Sometimes, we will use the symbols $Z_k(\mathcal{C})$ and $B_k(\mathcal{C})$ to denote the vector spaces $\ker \partial_k$ and $\Imm \partial_{k+1}$, respectively.


\begin{exercise}\label{ex_chain}
    Check if the indexed families $\mathcal{C}=(C_k,\partial_k)_k$ defined as follows are chain complexes:
    \begin{enumerate}
        \item $C_k:=\mathbb{Z}_2$ and $\partial_k$ is the null map, for any $k$;
        \item $C_k:=\mathbb{Z}_2$ and $\partial_k$ is the identity map, for any $k$;
        \item $C_k:=\mathbb{Z}_2$, for any $k$, and $\partial_k$ is the identity map for $k$ even and the null map for $k$ odd;
        \item $C_k:=\begin{cases}
            \mathbb{Z}_2 & \text{ for } k=i\\
            0 & \text{ otherwise}
        \end{cases}$, for a chosen integer $i$, and $\partial_k$ is the null map, for any $k$;
        \item $C_k:=\begin{cases}
            \mathbb{Z}_2 & \text{ for } k=i, i+1\\
            0 & \text{ otherwise}
        \end{cases}$, for a chosen integer $i$, and $\partial_k$ is the identity map for $k=i+1$ and the null map otherwise.
    \end{enumerate}
\end{exercise}



\begin{definition}
Given a chain complex $\mathcal{C}=(C_k,\partial_k)_k$, we define the \myemph{$k$-th homology group} (or homology group in degree $k$) of $\mathcal{C}$ as the vector space $\ker \partial_{k}/\Imm \partial_{k+1}$ over $\mathbb{Z}_2$. We use the symbol $H_k(\mathcal{C})$ to denote such a vector space.
\end{definition}

In this case, $H_k(\mathcal{C})$ is not just a group, but a vector space. 
However, the homology of a chain complex, and the notion of chain complex itself, can be defined in a wider setting suitably replacing $\mathbb{Z}_2$ with a field or a ring. In this latter case, the homology is a group and not a vector space and this is why the homology is usually referred as a group.

Throughout this book, the expression ``to compute the homology'' will often be used to refer to the computation of the homology groups in all their degrees.

\begin{remark}
Note that $H_k(\mathcal{C})=0$ if and only if $\ker \partial_{k}=\Imm \partial_{k+1}$.
\end{remark}

\begin{exercise}\label{ex_hom}
Compute the homology of the indexed families $\mathcal{C}=(C_k,\partial_k)_k$ in Exercise \ref{ex_chain} that are chain complexes.
\end{exercise}

\begin{exercise}
Prove that for every indexed family $(V_k)_k$ of vector spaces, there exists a chain complex $\mathcal{C}=(C_k,\partial_k)_k$ such that $H_k(\mathcal{C})=V_k$, for every $k$.
\end{exercise}



\begin{definition}
Let $\mathcal{C}=(C_k,\partial_k)_{k}$ and $\mathcal{C}'=(C'_k,\partial'_k)_{k}$ be two chain complexes.
An indexed family of linear maps $F=(F_k\colon C_k\to C'_k)_{k}$ is called
a \myemph{chain map} from $\mathcal{C}$ to $\mathcal{C}'$ if $\partial'_k  F_k=F_{k-1}  \partial_k$ for every integer $k$.
In other terms, $F$ is a chain map if, for any $k$, the following diagram commutes:
\[
\begin{tikzcd}
C_k \arrow[r, "\partial_k"] \arrow[d, "F_k"'] & C_{k-1} \arrow[d, "F_{k-1}"] \\
C'_k \arrow[r, "\partial'_k"'] & C'_{k-1}
\end{tikzcd}
\]
If each map $F_k$ is an isomorphism, we say that the chain map $F = (F_k \colon C_k \to C'_k)_k$ is a \myemph{chain isomorphism}.
\end{definition}


A chain map between chain complexes uniquely defines a linear map between their homology groups, as shown in the next result. 
In what follows, an element of the quotient $H_k(\mathcal{C})$ is represented as $[z]$, where $z$ is a $k$-cycle, i.e., $\partial_k(z)=0$.

\begin{proposition}\label{prop_map_induced_by_a_chain_map}
Let $F=(F_k\colon C_k\to C'_k)_{k}$ be a chain map from $\mathcal{C}=(C_k,\partial_k)_{k}$ to $\mathcal{C}'=(C'_k,\partial'_k)_{k}$.
Then the \myemph{induced map} 
$F_{k*}\colon H_k(\mathcal{C})\to H_k(\mathcal{C}')$ defined by setting $F_{k*}([z]):=[F_k(z)]$ is a well-defined linear map for every $k\in\Z$. 
If $F$ is an isomorphism between chain complexes, then each linear map $F_{k*}\colon H_k(\mathcal{C})\to H_k(\mathcal{C}')$ is an isomorphism.
\end{proposition}

\begin{proof}
First of all, 
for any $z\in C_k$,
$\partial'_k(F_k(z))=F_{k-1}(\partial_k(z))=F_{k-1}(0)=0$, and hence $F_k(z)\in Z_k(\mathcal{C}')$. Therefore,
$F_k(z)$ identifies an element $[F_k(z)]$ of $H_k(\mathcal{C}')$.
Furthermore, if $[z']=[z]$ we have that $z'-z$ is a $k$-boundary, i.e., there exists a $(k+1)$-chain $c$ in $C_{k+1}$ such that $z'-z=\partial_{k+1}(c)$.
It follows that $F_k(z')-F_k(z)=F_k(z'-z)=F_{k}(\partial_{k+1}(c))=\partial'_{k+1}(F_{k+1}(c))$,
and hence $F_k(z')-F_k(z)$ is a $k$-boundary.
This implies that $[F_k(z')]=[F_k(z)]$ in $H_k(\mathcal{C}')$.
Thus, the map $F_{k*}$ is well-defined.
The linearity of $F_{k*}$ follows immediately from the linearity of $F_{k}$.
If $F$ is an isomorphism between chain complexes and $[z]\in H_k(\mathcal{C})$,
then $(F^{-1})_{k*} F_{k*}([z])=[F_k^{-1} F_{k}(z)]=[z]$ and $F_{k*}(F^{-1})_{k*}([z])=[F_{k}F_k^{-1} (z)]=[z]$, hence $F_{k*}$ is an isomorphism, for every $k\in \Z$.
\end{proof}

\section{Simplicial homology}
\label{SimpH}



In this section, we introduce the homology of a simplicial complex. 
A simplicial complex, in this context, is a suitable collection of subsets of $\R^n$, called simplices, which can be seen as the elementary bricks of the simplicial complex. 
We show how to associate a chain complex with a given simplicial complex, and how to define and compute its homology groups. 

\subsection{Simplicial complexes}


Given a finite set of points $V:= \{v_0,\dots, v_k\} \subseteq \mathbb{R}^n$, the \myemph{convex hull} of $V$, denoted as $\langle v_0, \ldots,v_{k}\rangle$ or $\mathrm{conv}(V)$, is the (unique) minimal convex set containing $V$. 
Equivalently, $\mathrm{conv}(V)$ is the set of all linear combinations $\sum_{i=0}^k a_iv_i$, with $\sum_{i=0}^k a_i=1$ and $a_i\ge 0$, of the points $v_i$ in $V$. 


The set $V$ is called \myemph{affinely independent} if the vectors $v_1 - v_0, \dots, v_k - v_0$ are linearly independent in $\R^n$.
We note that if $k=0$ the set $V=\{v_0\}$ is affinely independent because the empty set is linearly independent. 
Examples of affinely independent points are given by any set consisting of two distinct points, of three non-collinear points, of four non-coplanar points, and so on.

\begin{definition}
The \myemph{$k$-simplex} $\sigma = \langle v_0, \ldots, v_k\rangle$ spanned by a non-empty, finite and affinely independent subset $V=\{ v_0, \dots, v_k \}$ of $\R^n$ is the convex hull of $V$.

\end{definition}


\begin{figure}[!htbp]
\begin{center}
\includegraphics[width=8cm]{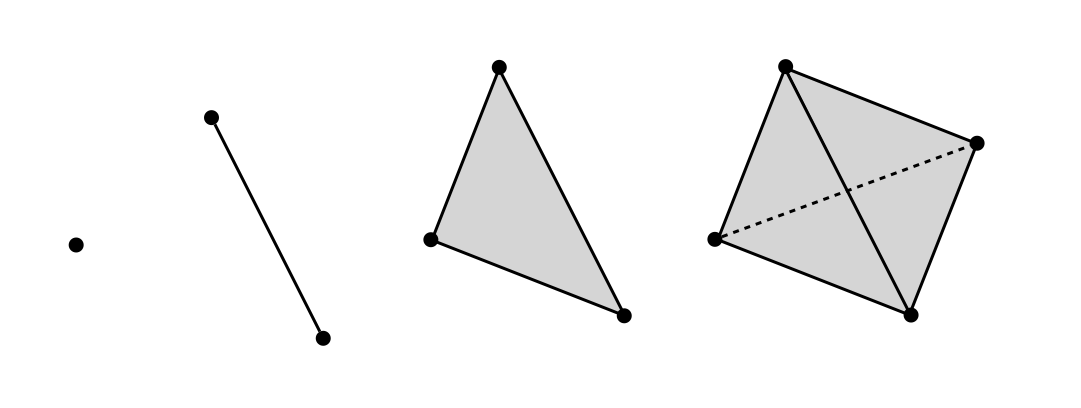}
\caption{Examples of simplices of dimension 0, 1, 2, 3.}
\label{fig:simplices}
\end{center}
\end{figure}

As depicted in Figure \ref{fig:simplices}, a 0-simplex is a vertex, a 1-simplex is an edge, a 2-simplex is a triangle, a
3-simplex is a tetrahedron, and so on.

The points $v_0, \ldots, v_k$ spanning a $k$-simplex $\sigma$ are called the \myemph{vertices} of $\sigma$.
The value $k$ is called the \myemph{dimension} of $\sigma$ and denoted as $\dim(\sigma)$. 
Any simplex $\tau$ spanned by a non-empty subset of $V$ is called a \myemph{face} of $\sigma$. 
Conversely, $\sigma$ is called a \myemph{coface} of $\tau$.
If $\tau$ is a face of $\sigma$, we write $\tau\preceq \sigma$.

\begin{definition}
A \myemph{simplicial complex} $\Sigma$ in $\R^n$ is a non-empty finite set of simplices in $\R^n$ such that:
\begin{itemize}
    \item if $\sigma\in \Sigma$ and 
    $\tau\preceq \sigma$, then $\tau\in \Sigma$;
    \item if $\sigma_1, \sigma_2\in \Sigma$ 
    with $\sigma_1\cap \sigma_2\neq \varnothing$, then $\sigma_1\cap \sigma_2\preceq \sigma_1$ and $\sigma_1\cap \sigma_2\preceq \sigma_2$.
\end{itemize}
\end{definition}





We define the \myemph{dimension} of a simplicial complex $\Sigma$ in $\R^n$, denoted by $\dim(\Sigma)$, to be the maximum dimension of the simplices of $\Sigma$.
A subset of $\Sigma$
that is itself a simplicial complex is called a \myemph{subcomplex} of $\Sigma$.

\begin{exercise}\label{ex:simpl_complexes}
    Check that the following sets of simplices in $\R^2$ are not simplicial complexes and, for each of them, build the smallest simplicial complex containing it:
    
    \begin{itemize}
        \item $\Sigma:=\{\{(x,0)\mid -1\le x\le 1 \}\}$;
        \item $\Sigma:=\{\{(-1, 0)\}, \{(1,0)\}, \{(0, -1)\}, \{(0,1)\}, \{(x,0)\mid -1\le x\le 1 \}, \{(0,y)\mid -1\le y\le 1 \}\}.$
    \end{itemize}
\end{exercise}

One way to endow a simplicial complex $\Sigma$ with a topology is to consider it as a subspace of $\R^n$.
Specifically, we define the \myemph{underlying space} of $\Sigma$ the subset $|\Sigma|$ of $\R^n$ defined as the union of the simplices of $\Sigma$ and endowed with the subspace topology with respect to $\R^n$. 
    

\begin{exercise}
    Determine the underlying space of the simplicial complexes obtained in Exercise \ref{ex:simpl_complexes}.
\end{exercise}

\begin{exercise}
    Define a simplicial complex $\Sigma$ such that its underlying space is homeomorphic to the $2$-sphere $\mathbb{S}^2$.
\end{exercise}


\subsection{Homology of a simplicial complex}

Given a simplicial complex $\Sigma$, one defines its simplicial homology by associating it with a chain complex $\mathcal{C}(\Sigma)=(C_k(\Sigma), \partial_k)_k$ and considering its homology.

Refer to~\cite{Greub} for more details on the following definition.

\begin{definition}
Let $S$ be a set, the \myemph{free vector space} with coefficients in $\mathbb{Z}_2$ over $S$ is the vector space over $\mathbb{Z}_2$ whose elements are the finitely supported function $f:S\to \mathbb{Z}_2$ endowed with the usual sum of functions and the usual multiplication by elements of the field $\mathbb{Z}_2$.
\end{definition}

Each element $f$ of the free vector space with coefficients in $\mathbb{Z}_2$ over a set $S$ can be represented by the formal linear combination $\sum_{i=1}^{k}a_i \sigma_i$, where $\sigma_i$ belongs to $S$ and $a_i$ equals $f(\sigma_i)$. 
In simple terms, the linear combination $\sum_{i=1}^{k} a_i \sigma_i$ can be viewed as selecting the set of elements $\sigma_i$ corresponding to the coefficients for which $a_i = 1$, thereby identifying the function $f$ with its support.

For every integer $k$, define $C_k(\Sigma)$ as the free vector space with coefficients in $\mathbb{Z}_2$ over the set of $k$-simplices in $\Sigma$, if $k\geq 0$,
and as the null vector space over $\mathbb{Z}_2$, otherwise.
Furthermore, define the \myemph{boundary map}
$\partial_k\colon  C_k(\Sigma)\to C_{k-1}(\Sigma)$ as the linear map which is null if $k\leq 0$
and which takes the $k$-simplex $\sigma=\langle v_0, \ldots, v_k \rangle$ in $\Sigma$
to the formal sum $\sum_{i=0}^{k} \langle v_0, \ldots, \hat v_i \ldots, v_k \rangle$, where
$\langle v_0,\ldots,\hat v_i,\ldots, v_{k}\rangle$ denotes $\mathrm{conv}\left(\{v_0,\ldots, v_{k}\}\setminus \{v_i\}\right)$, otherwise.
Figure \ref{fig:simplicial_boundary} depicts the boundary of the 2-simplex $\langle v_0, v_1, v_2 \rangle$ which is the sum $\langle v_1, v_2 \rangle + \langle v_0, v_2 \rangle + \langle v_0, v_1 \rangle$.


\begin{figure}[!htbp]
\begin{center}
\includegraphics[width=8cm]{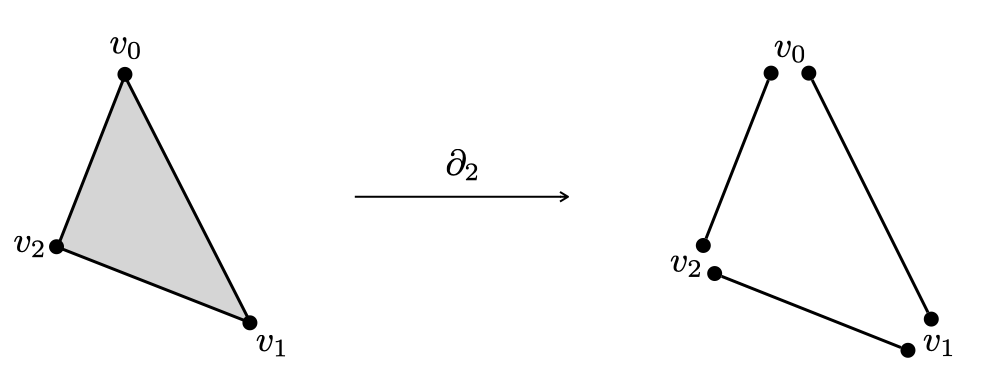}
\caption{The boundary map $\partial_2$ applied to the $2$-simplex $\langle v_0, v_1, v_2 \rangle$.}
\label{fig:simplicial_boundary}
\end{center}
\end{figure}

\begin{proposition}\label{prop_simpl_chain_complex}
Let $\Sigma$ be a simplicial complex, the indexed family $\mathcal{C}(\Sigma):=(C_k(\Sigma),\partial_k)_{k}$ is a chain complex.
\end{proposition}

\begin{proof}
For every integer $k$, $C_k(\Sigma)$ is a vector space over $\mathbb{Z}_2$ by construction.
It remains to prove that, for every $k$, $\partial_{k}  \partial_{k+1}$ is the null map.
This statement is trivial for $k\leq 0$, therefore let us assume that $k>0$.
Given a $(k+1)$-simplex $\sigma=\langle v_0, \ldots, v_{k+1} \rangle$ of $\Sigma$, let us define the symbol $\sigma_{ij}$ denoting the chain
$\langle v_0,\ldots,\hat v_i,\ldots,\hat v_j,\ldots,v_{k+1}\rangle$ in $C_{k-1}(\Sigma)$ , if $i<j$, and the null chain, otherwise.
We have that
\begin{align*}
            \partial_{k}  \partial_{k+1}(\sigma)&= \partial_{k}\left(\sum_{i=0}^{k+1}\langle v_0,  \ldots, \hat v_i, \ldots, v_{k+1} \rangle\right)\\
            &= \sum_{i=0}^{k+1}\partial_{k}\left(\langle v_0,  \ldots, \hat v_i, \ldots, v_{k+1} \rangle\right)\\
            &= \sum_{i=0}^{k+1}\left(\sum_{j=0}^{i-1}\langle v_0,\ldots,\hat v_j,\ldots,\hat v_i,\ldots,v_{k+1}\rangle\right.+\\
            &\left.\sum_{j=i+1}^{k+1}\langle v_0,\ldots,\hat v_i,\ldots,\hat v_j,\ldots,v_{k+1}\rangle\right)\\
            &= \sum_{i=0}^{k+1}\left(\sum_{j=0}^{i-1}\sigma_{ji}+\sum_{j=i+1}^{k+1}\sigma_{ij}\right)\\
            &= \sum_{i=0}^{k+1}\left(\sum_{j=0}^{k+1}\sigma_{ji}+\sum_{j=0}^{k+1}\sigma_{ij}\right)\\
            &= \sum_{i=0}^{k+1}\sum_{j=0}^{k+1}\sigma_{ji}+\sum_{i=0}^{k+1}\sum_{j=0}^{k+1}\sigma_{ij}\\
            &= \sum_{i=0}^{k+1}\sum_{j=0}^{k+1}\sigma_{ij}+\sum_{i=0}^{k+1}\sum_{j=0}^{k+1}\sigma_{ij}\\
            &= 0.
\end{align*}

\end{proof}

As a consequence of Proposition \ref{prop_simpl_chain_complex}, one can consider the $k$-th homology group $H_k(\mathcal{C}(\Sigma))$ defined as $\ker\partial_k/\Imm\partial_{k+1}$, for every $k\in\Z$.
This vector space is denoted by $H_k(\Sigma)$, and called the $k$-th \myemph{simplicial homology group of $\Sigma$}.
In the following, for any given simplicial complex $\Sigma$ and every $k\in \Z$ we will adopt the symbols $Z_k(\Sigma)$ and $B_k(\Sigma)$  to denote $Z_k(\mathcal{C}(\Sigma))=\ker\partial_k$ and $B_k(\mathcal{C}(\Sigma))=\Imm\partial_{k+1}$, respectively.
The elements of $Z_k(\Sigma)$ and $B_k(\Sigma)$ are called, respectively, (simplicial) $k$-\myemph{cycles} and (simplicial) $k$-\myemph{boundary} of $\Sigma$.

Figure~\ref{fig:simplicial_homology} shows a simplicial complex of dimension 2. 
The edges in green and orange represent two 1-cycles that are not 1-boundaries. 
Thus, they belong to non-trivial equivalence classes of $H_1(\Sigma)$.
Moreover, these two 1-cycles, $z_1$ and $z_2$, belong to the same equivalence class because $z_1-z_2\in \Imm \partial_2$.
Furthermore, the edges depicted in blue represent a 1-cycle which is also a 1-boundary. So, it equivalence class in $H_1(\Sigma)$ is the trivial one. 

\begin{figure}[!htbp]
\begin{center}
\includegraphics[width=8cm]{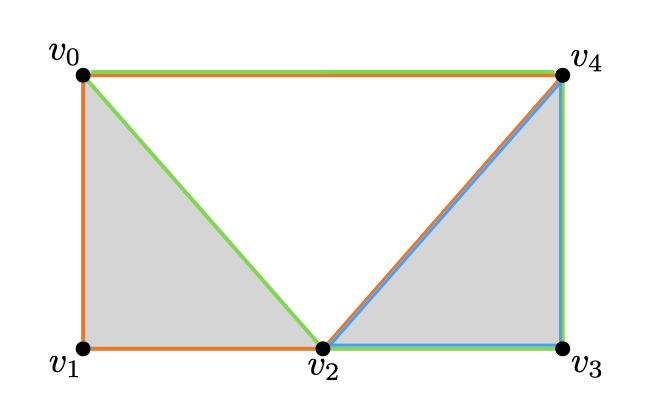}
\caption{A simplicial complex in which three 1-cycles are highlighted by using different colors. The one in blue is also a 1-boundary while the ones in green and orange are not.}
\label{fig:simplicial_homology}
\end{center}
\end{figure}


\begin{exercise}\label{ex:simplicial_homology}
Compute the simplicial homology groups of the following simplicial complexes:
\begin{itemize}
    \item 
    $\Sigma$ contains all the proper faces of a 3-simplex;
    \item $\Sigma$ depicted in Figure \ref{fig:simplicial_homology}.
\end{itemize}
\end{exercise}


The effective computation of the homology of simplicial complexes with a large number of simplices is by no means straightforward and requires algorithmic tools. This topic lies beyond the scope of this book; however, interested readers may find further references in \cite{CT_Edels}.

\section{Singular homology of a topological space}
\label{SingH}

In this section, we introduce the notion of homology of a topological space. 
Similarly to simplicial complexes, for any topological space $X$ a chain complex $\mathcal{S}(X)=(S_k(X), \partial_k)_k$ can be defined.

First, let us focus on the definition of the vector spaces $S_k(X)$.
For $k\geq 0$, let us consider the $k$-simplex $\Delta_k:=\langle e_1\ldots,e_{k+1}\rangle\subseteq\R^{k+1}$, where $e_i$ is the $i$-th vector in the canonical basis of $\R^{k+1}$.
This is called the \myemph{standard} $k$-simplex.

\begin{definition}
Given a topological space $X$ and $k\ge 0$, a continuous map $\sigma\colon\Delta_k \to X$ is called a \myemph{singular $k$-simplex} in $X$.
\end{definition}

We emphasize that a singular simplex is not the image of a map, but the map itself.
Figure \ref{fig_simplex} shows an example of a singular $2$-simplex in 
a torus.
\begin{figure}[!htbp]
\begin{center}
\includegraphics[width=8cm]{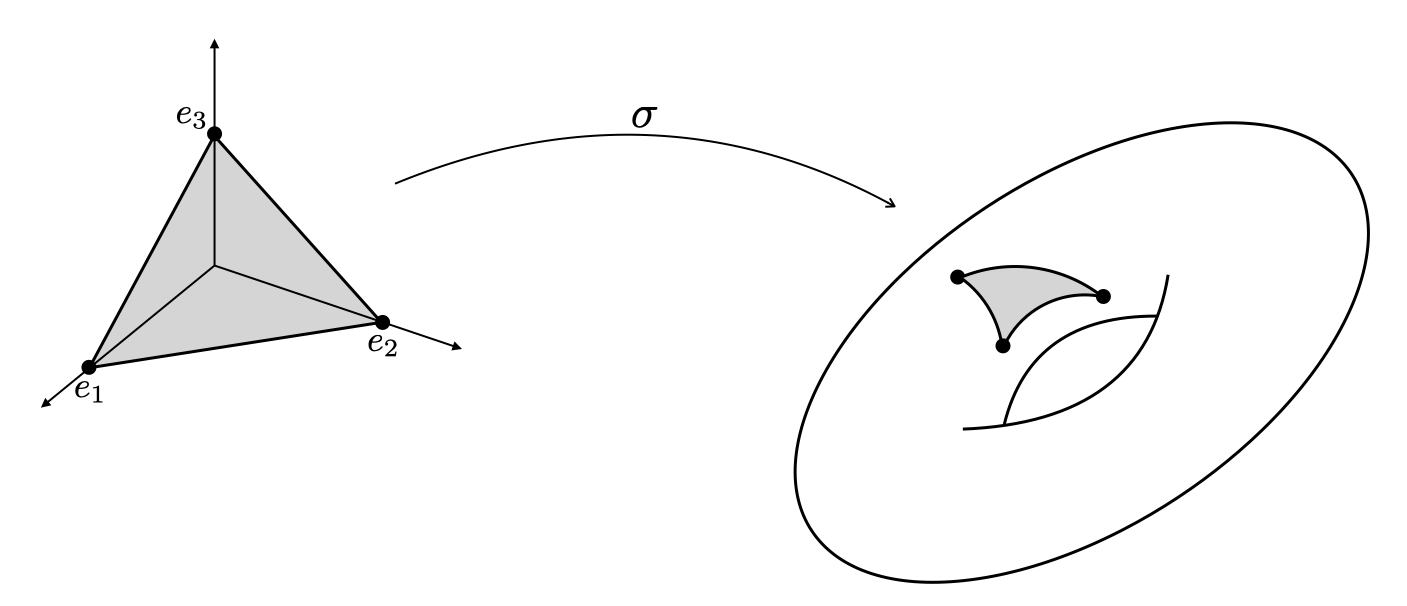}
\caption{A singular 2-simplex $\sigma$ in a torus $X$.}
\label{fig_simplex}
\end{center}
\end{figure}

For every integer $k$, we define $S_k(X)$ as the free vector space with coefficients in $\mathbb{Z}_2$ over the set of singular $k$-simplices in $X$, if $k\geq 0$, and as the null vector space over $\mathbb{Z}_2$, otherwise.

Let us now focus on the definition of the linear maps $\partial_k: S_k(X) \to S_{k-1}(X)$.
Similarly to the simplicial case, we use the notation  $\langle e_1,\ldots,\hat e_i,\ldots, e_{k+1}\rangle$ for $\mathrm{conv}\left(\{e_1,\ldots,e_{k+1}\}\setminus \{e_i\}\right)$, when $k>0$.
With a slight abuse of notation, in the following, we will identify $\Delta_{k-1}$ with the face $\langle e_1,\ldots,\hat e_i,\ldots, e_{k+1}\rangle$ of $\Delta_k$ by the affine map
$\pi^{k-1}_i\colon \Delta_{k-1}\to\langle e_1,\ldots,\hat e_i,\ldots, e_{k+1}\rangle$ such that
$\pi^{k-1}_i(e'_j):=e_j$ if $j<i$ and  $\pi^{k-1}_i(e'_j):= e_{j+1}$ if $j\geq i$, where $e'_j$ is the $j$-th vector in the canonical basis of $\R^{k}$ (see Figure \ref{fig_boundary} for an example).
As a consequence, for each singular $k$-simplex $\sigma:\Delta_k\to X$ with $k>0$ we will write
${\sigma}_{\vert \langle e_1,\ldots,\hat e_i,\ldots,e_{k+1}\rangle}$ to indicate
${\sigma}  \pi^{k-1}_i$.

\begin{figure}[!htbp]
\begin{center}
\includegraphics[width=\textwidth]{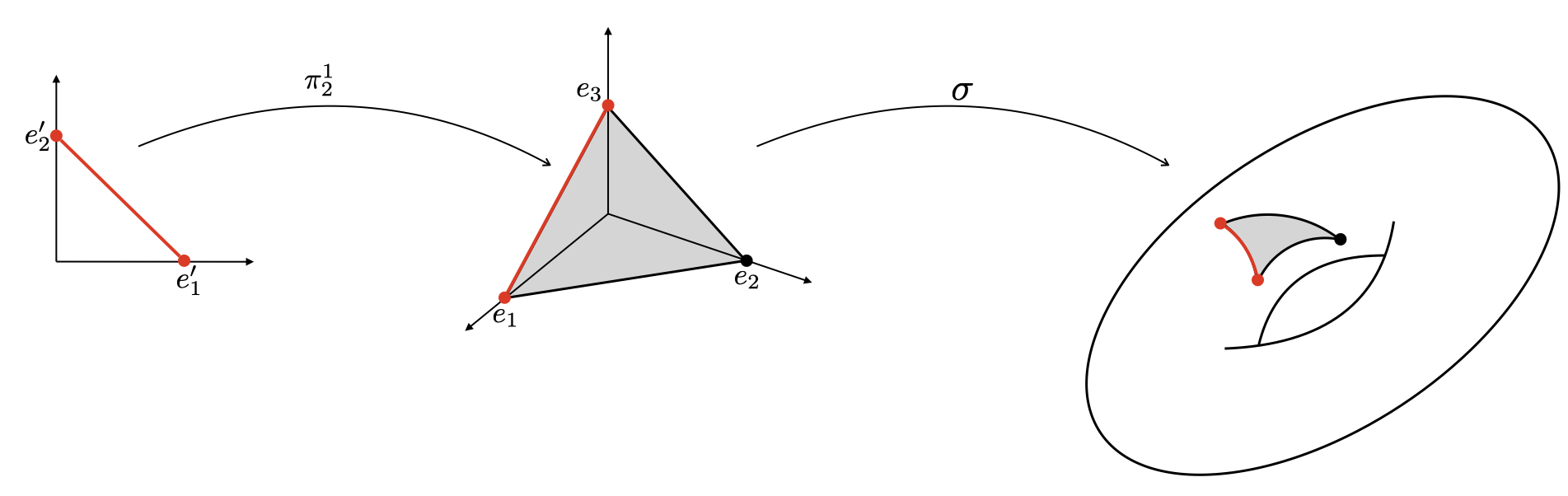}
\caption{The restriction ${\sigma}_{\vert \langle e_1,\hat e_2, e_3 \rangle}$ of the singular 2-simplex $\sigma$ in a torus $X$.}
\label{fig_boundary}
\end{center}
\end{figure}


For every integer $k$, we define 
$\partial_k\colon S_k(X)\to S_{k-1}(X)$ as the linear map which is null if $k\leq 0$ and which takes the $k$-simplex $\sigma\colon\Delta_k\to X$
to the chain $\sum_{i=1}^{k+1}{\sigma}_{\vert \langle e_1,\ldots,\hat e_i,\ldots,e_{k+1}\rangle}$, otherwise.

\begin{proposition}\label{prop_chain_complex}
If $X$ is a topological space, then the indexed family $\mathcal{S}(X):=(S_k(X),\partial_k)_{k}$ is a chain complex.
\end{proposition}

\begin{proof}
For every integer $k$, $S_k(X)$ is a vector space over $\mathbb{Z}_2$ by construction.
It remains to prove that, for every $k$, $\partial_{k}  \partial_{k+1}$ is the null map.
This statement is trivial for $k\leq 0$, therefore let us assume that $k>0$.
Analogously to Proposition \ref{prop_simpl_chain_complex}, given a singular $(k+1)$-simplex $\sigma\colon\Delta_{k+1} \to X$, let us define the symbol $\sigma_{ij}$ denoting the chain 
$\sigma_{\vert \langle e_1,\ldots,\hat e_i,\ldots,\hat e_j,\ldots,e_{k+2}\rangle}$ in $S_{k-1}(X)$, if $i<j$, and the null chain, otherwise.
We have that
\begin{align*}
            \partial_{k}  \partial_{k+1}(\sigma)&= \partial_{k}\left(\sum_{i=1}^{k+2}{\sigma}_{\vert \langle e_1,\ldots,\hat e_i,\ldots,e_{k+2}\rangle}\right)\\
            &= \sum_{i=1}^{k+2}\partial_{k}\left({\sigma}_{\vert \langle e_1,\ldots,\hat e_i,\ldots,e_{k+2}\rangle}\right)\\
            &= \sum_{i=1}^{k+2}\left(\sum_{j=1}^{i-1}{\sigma}_{\vert \langle e_1,\ldots,\hat e_j,\ldots,\hat e_i,\ldots,e_{k+2}\rangle}+\sum_{j=i+1}^{k+2}{\sigma}_{\vert \langle e_1,\ldots,\hat e_i,\ldots,\hat e_j,\ldots,e_{k+2}\rangle}\right)\\
            &= \sum_{i=1}^{k+2}\left(\sum_{j=1}^{i-1}\sigma_{ji}+\sum_{j=i+1}^{k+2}\sigma_{ij}\right)\\
            &= \sum_{i=1}^{k+2}\left(\sum_{j=1}^{k+2}\sigma_{ji}+\sum_{j=1}^{k+2}\sigma_{ij}\right)\\
            &= \sum_{i=1}^{k+2}\sum_{j=1}^{k+2}\sigma_{ji}+\sum_{i=1}^{k+2}\sum_{j=1}^{k+2}\sigma_{ij}\\
            &= \sum_{i=1}^{k+2}\sum_{j=1}^{k+2}\sigma_{ij}+\sum_{i=1}^{k+2}\sum_{j=1}^{k+2}\sigma_{ij}\\
            &= 0.
\end{align*}
\end{proof}

By Proposition~\ref{prop_chain_complex}, $\mathcal{S}(X)$ is a chain complex, so the $k$-th homology group $H_k(\mathcal{S}(X))=\ker\partial_k/\Imm\partial_{k+1}$
is defined, for every $k\in\Z$. This vector space is denoted by the symbol $H_k(X)$, and called the $k$-th \myemph{singular homology group} of $X$.
Similarly to the simplicial case, in the following, for any given topological space $X$ and every $k\in \Z$, we will use the symbols $Z_k(X)$ and $B_k(X)$  to denote $Z_k(\mathcal{S}(X))=\ker\partial_k$ and $B_k(\mathcal{S}(X))=\Imm\partial_{k+1}$, respectively.
The elements of $Z_k(X)$ and $B_k(X)$ are called, respectively, (singular) $k$-\myemph{cycles} and (singular) $k$-\myemph{boundaries} of $X$. 


\begin{exercise}
Find a topological space 
$X$ whose singular homology groups match those given in each of the following cases:
\begin{itemize}
\item $H_0(X)=\Z_2$, $H_i(X)=0$, for every $i\neq 0$;
\item $H_i(X)=\Z_2$, for $i=0, 1$, and $H_i(X)=0$, for every $i\neq 0,1$;
\item $H_i(X)=\Z_2$, for $i=0, 2$, and $H_i(X)=0$, for every $i\neq 0,2$.
\end{itemize}
\end{exercise}

The next result ensures that a continuous function between topological spaces induces a chain map between the associated chain complexes.




\begin{proposition}\label{prop_a_chain_map}
Let $f\colon X\to Y$ be a continuous function between two topological spaces.
For every $k\ge 0$, consider the linear map $f_{k\#}\colon S_k(X)\to S_k(Y)$ defined by setting, for every singular $k$-simplex  $\sigma\colon \Delta_k\to X$, $f_{k\#}(\sigma):=f  \sigma$.
For $k<0$, set $f_{k\#}\colon S_k(X)\to S_k(Y)$ equal to the null map.
The indexed family $f_{\#}:=(f_{k\#}\colon S_k(X)\to S_k(Y))_{k}$ is a chain map from $\mathcal{S}(X)$ to $\mathcal{S}(Y)$.
\end{proposition}

\begin{proof}
For every $k\ge 0$
and every singular $k$-simplex $\sigma$ in $X$,
\begin{align*}
            f_{(k-1)\#}(\partial_k(\sigma))
            &= f_{(k-1)\#}\left(\sum_{i=1}^{k+1}\sigma_{\vert\langle e_1,\ldots,\hat e_i,\ldots,e_{k+1}\rangle}\right)\\
            &= \sum_{i=1}^{k+1}f_{(k-1)\#}\left(\sigma_{\vert\langle e_1,\ldots,\hat e_i,\ldots,e_{k+1}\rangle}\right)\\
            &= \sum_{i=1}^{k+1}f \sigma_{\vert\langle e_1,\ldots,\hat e_i,\ldots,e_{k+1}\rangle}\\
            &= \partial'_k(f \sigma)\\
            &= \partial'_k(f_{\#k}(\sigma)),
\end{align*}
where, for every $k$, $\partial_k$ and $\partial'_k$ denote the boundary maps of $\mathcal{S}(X)$ and of $\mathcal{S}(Y)$, respectively.
Note that the continuity of the function $f$ ensures that $f\sigma$ is a singular $k$-simplex in $Y$.
The same equality trivially holds for $k<0$.
\end{proof}

The following theorem states one of the most important properties of homology.

\begin{theorem}\label{thm_homology_functor}
If $\id\colon X\to X$ is the identity on a topological space $X$, then 
$\id_{k\#*}\colon H_k(X)\to H_k(X)$ 
(induced by the chain map 
$\id_{k\#}$ 
defined in Proposition~\ref{prop_a_chain_map}) is the identity.
If $f\colon X\to Y$, $g\colon Y\to Z$ are two continuous functions  between topological spaces, then $(g  f)_{\#*}=g_{\#*}  f_{\#*}$.
\end{theorem}

\begin{proof}

By construction, $\id_{k\#}\colon S_k(X)\to S_k(X)$ is the identity, for every $k$.
It follows from Proposition~\ref{prop_map_induced_by_a_chain_map} that, for every $[z]\in H_k(X)$, $\id_{k\#*}([z])=[\id_{k\#}(z)]=[z]$.

If $\sigma$ is a singular $k$-simplex in $X$, then
$(g  f)_{k\#}(\sigma)=(g  f)  \sigma=g  (f  \sigma)=g_{k\#}(f_{k\#}(\sigma))$.
By this and Proposition~\ref{prop_map_induced_by_a_chain_map}, if $[z]\in H_k(X)$, then
\begin{align*}
            (g  f)_{k\#*}([z]) &= [(g  f)_{k\#}(z)] \\
            &= [g_{k\#}(f_{k\#}(z))] \\
            &= g_{k\#*}([f_{k\#}(z)]) \\
            &= g_{k\#*}(f_{k\#*}([z])) \\
            &= (g_{k\#*}f_{k\#*})([z]).
\end{align*}

\end{proof}

As a corollary of the above theorem, we have the following crucial result.

\begin{proposition}\label{homeohomologyiso}
Let $f\colon X\to Y$ be a homeomorphism. 
Then, for every $k\in \mathbb{Z}$, $f_{k\#*}\colon H_k(X)\to H_k(Y)$ is an isomorphism. 
\end{proposition}
\begin{proof}
Since $f$ is a homeomorphism, $f^{-1}\colon Y\to X$ exists and it is a continuous function such that $f f^{-1}=\id$ and $f^{-1} f=\id$.
We prove that, for every $k\in \mathbb{Z}$, $f_{k\#*}\colon H_k(X)\to H_k(Y)$ is an isomorphism by showing that $\left(f^{-1}\right)_{k\#*}\colon H_k(Y)\to H_k(X)$ is its inverse.
The first above equation ensures that, for every $k$,
$$(f f^{-1})_{k\#*}=\id_{k\#*}.$$
Thus, by Theorem~\ref{thm_homology_functor}, for every $[z]\in H_k(X)$ and for every $k$,
\begin{align*}
    (f_{k\#*}\left(f^{-1}\right)_{k\#*})([z]) &= (f  f^{-1})_{k\#*}([z]) \\ 
    &= \id_{k\#*}([z]) \\
    &= [z].
\end{align*}

Therefore, $f_{k\#*}\left(f^{-1}\right)_{k\#*}$ is the identity map.
Analogously, by the fact that $f^{-1} f=\id$, one can prove that $\left(f^{-1}\right)_{k\#*}f_{k\#*}$ is the identity map.

\end{proof}

\subsection{Isomorphism between simplicial and singular homology}

With the exception of specific cases (such as when $X$ has finite cardinality), the number of singular $k$-simplices in a topological space $X$ is not finite. 
This represents one aspect, but not the only one, that prevents to ``manually'' compute the homology of
topological spaces.
Such a limitation does not occur for simplicial complexes, making their homology easier to compute.
The following fundamental result ensures that the simplicial homology groups of a simplicial complex are isomorphic to the singular homology groups of its underlying space.

\begin{theorem}\label{thm:isomorphism}
Let $\Sigma$ be a simplicial complex.
For every integer $k$, the $k$-th simplicial homology group of $\Sigma$, $H_k(\Sigma)$, is isomorphic to the $k$-th singular homology group of $|\Sigma|$, $H_k(|\Sigma|)$.
\end{theorem}

The proof of the above result is out of the scope of this book. We refer the interested reader to \cite[Theorem 34.3]{Munkres}.

We are now able to infer the homology of a topological space $X$ from the homology of a suitable simplicial complex. 
The strategy works as follows: 
find a simplicial complex $\Sigma$, whose underlying space is homeomorphic to $X$. 
By Theorem~\ref{thm:isomorphism}, $H_k(\Sigma)$ is isomorphic to $H_k(\lvert \Sigma \rvert)$, and by Proposition~\ref{homeohomologyiso}, $H_k(\lvert \Sigma\vert)$ is isomorphic to $H_k(X)$. 
Thus, $H_k(\Sigma)$ is also isomorphic to $H_k(X)$.


\begin{exercise}
Compute the singular homology groups of the $2$-sphere $\mathbb{S}^2$.
\end{exercise}

%
%
%
\chapter{Persistent homology}\label{ChapterHC}

Every real-valued function induces a filtration of its domain defined by its sublevel sets. 
Each pair of sublevel sets is associated with a persistent homology group which gives a quantitative summary of the topology of the filtered object. 
In this chapter, we introduce persistent homology and describe a convenient way to represent it via persistent Betti numbers functions and persistence diagrams.  
We also define a metric between persistent Betti numbers functions and show that they are stable with respect to the $L^\infty$-distance between filtering functions.

Every real-valued function induces a filtration of its domain defined by its sublevel sets. 
Pairs of sublevel sets are associated with persistent homology groups which give a quantitative summary of the topology of the filtered object. 
In this chapter, we introduce persistent homology and describe a convenient way to represent it via persistent Betti numbers functions and persistence diagrams.  
We also define a metric between persistent Betti numbers functions and show that they are stable with respect to the uniform metric between filtering functions.

\section{Persistent Betti numbers functions}\label{sec:PBNF}

Persistent homology catches the homological variations occurring in the sublevel sets of a real-valued function. This information is encoded by a collection of functions called persistent Betti numbers. 
In this section, we investigate some basic properties of such functions and, after providing their domain with a topology (Section \ref{subsec:metric_delta*}), we analyse their discontinuities.\\

Given a continuous function $\p\colon X\to\R$ on a non-empty compact space $X$, 
a filtration $\{X^\p_v\}_{v\in \R}$ is defined, where, for any $v\in\R$, $X^\p_v$ is the sublevel set $\{x\in X\mid\p(x)\le v\}$ (see Figure \ref{figexFunction}). 
The compactness of $X$ ensures that the sequence of spaces of the filtration $\{X^\p_v\}_{v\in \R}$ stabilizes for a sufficiently high value of $v$. 
Explicitly, for any $v\ge \max\p$, $X^\p_v = X^\p_{\max\p}=X$. In the following, the symbol $X^\p_{\infty}$ denotes $X^\p_{\max\p}$. 

\begin{figure}[htb!]\label{fig_filt}
\begin{center}
\includegraphics[width=0.6\textwidth]{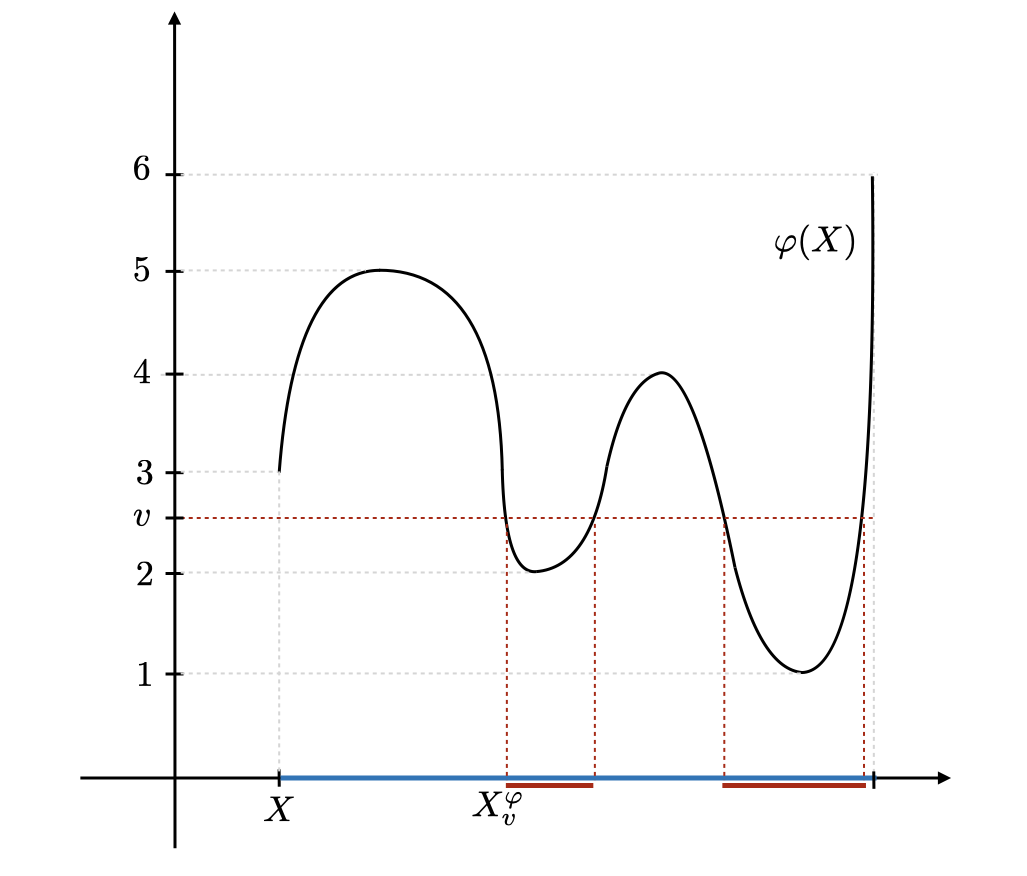}
\caption{
The function $\p\colon X\to\R$ is defined on the closed interval $X$ depicted in blue. The sublevel set $X^\p_v$ is depicted in red.}
\label{figexFunction}
\end{center}
\end{figure}



After fixing a degree $k\in \Z$, for every $(u,v)\in (\R\cup\{\infty\})^2$ with $u\le v$, we can consider the linear map $i^{\p*}_{u,v}\colon H_k(X^\p_u)\to H_k(X^\p_v)$ induced by the inclusion $i^\p_{u,v}\colon X^\p_u\hookrightarrow X^\p_v$.
The map $i^{\p*}_{u,v}$ takes the homology class of a cycle in $X^\p_u$ to the homology class of the same cycle in $X^\p_v$.
Notice that Theorem~\ref{thm_homology_functor} implies $i^{\p*}_{v,w} i^{\p*}_{u,v}= i^{\p*}_{u,w}$ for $u\le v\le w$.

\begin{definition}[Persistent homology group]\label{PHG}
Let $k\in \Z$. The subspace $\mathrm{Im\ }i^{\p*}_{u,v}\subseteq H_k(X^\p_v)$, with $(u,v)\in (\R\cup\{\infty\})^2$ such that $u\le v$, is denoted by the symbol $PH_k^\p{(u,v)}$ and called
the \myemph{$k$-th persistent homology group} at $(u, v)$. 
\end{definition}
Recall that, since we consider homology with coefficients in a field, $H_k(X^\p_u)$ is a vector space for every $u$. 
In particular, $PH_k^\p{(u,v)}$ is also a vector space. 
The name persistent homology group is, however, preferred here for coherence with the literature and for generality purposes. 

As an example, consider the filtration of $X$ given in Figure \ref{fig_filt}. 
The induced linear maps in homology of degree 0 are
\[
  \cdots \xrightarrow[]{}\mathbb{Z}_2 \xrightarrow[]{i^{\p*}_{1,2}}  \mathbb{Z}_2^2 \xrightarrow[]{i^{\p*}_{2,3}}\mathbb{Z}_2 ^3\xrightarrow[]{i^{\p*}_{3,4}}\mathbb{Z}_2^2 \xrightarrow[]{i^{\p*}_{4,5}}\mathbb{Z}_2 \xrightarrow[]{}\dots
\]
where the vector space represented correspond to $H_0(X_v^\varphi)$, with $v=1,2,3,4,5$. 
The persistent homology group at $(u,v)$ tells us if the homological features present at time $u$ are still alive at time $v$. 
In this example, $PH_0^\varphi (2,3)\cong \mathbb{Z}_2^2$ tells us that the two connected components present at $2$ are still there at time $3$, and $PH_0^\varphi (2,4)\cong \mathbb{Z}_2$ corresponds to the fact that the two connected components present at time $2$ merge at a time smaller or equal to $4$.


From now on, the following will hold:
\begin{assumption}\label{ass_fingen}
For every $v\in\R$, the vector space $H_k(X^\p_v)$ is finitely generated. 
\end{assumption}

\begin{exercise}\label{exercize:assumption_1}
Find a continuous function $\p\colon X\to\R$ on a non-empty compact space $X$ which does not satisfy Assumption \ref{ass_fingen}.
\end{exercise}

Let $\Delta^+$ denote the open half-plane $\{(x,y) \in \R^2 \mid x < y \}$ and $\Delta^*$ the union $ \Delta^+\cup\{(x,\infty)\mid x\in\R\}$.

\begin{definition}[Persistent Betti numbers function]\label{PBNF}
Fixed a degree $k\in \Z$, the $k$-th \myemph{persistent Betti numbers function} (PBNF) (or \myemph{rank invariant} \cite{CaZo09}) of $\p$ is defined as the function $\beta_k^\p\colon \Delta^* \to \N$ assigning to each $(u,v)\in \Delta^*$ the dimension $\beta_k^\p{(u,v)}$ of $PH_k^\p{(u,v)}$.
\end{definition}
Because of Assumption~\ref{ass_fingen}, PBNFs assume a finite value, and thus are well-defined for every $(u,v)$ in $\Delta^*$. 

\begin{figure}[htb!]
\begin{center}
\includegraphics[width=7cm]{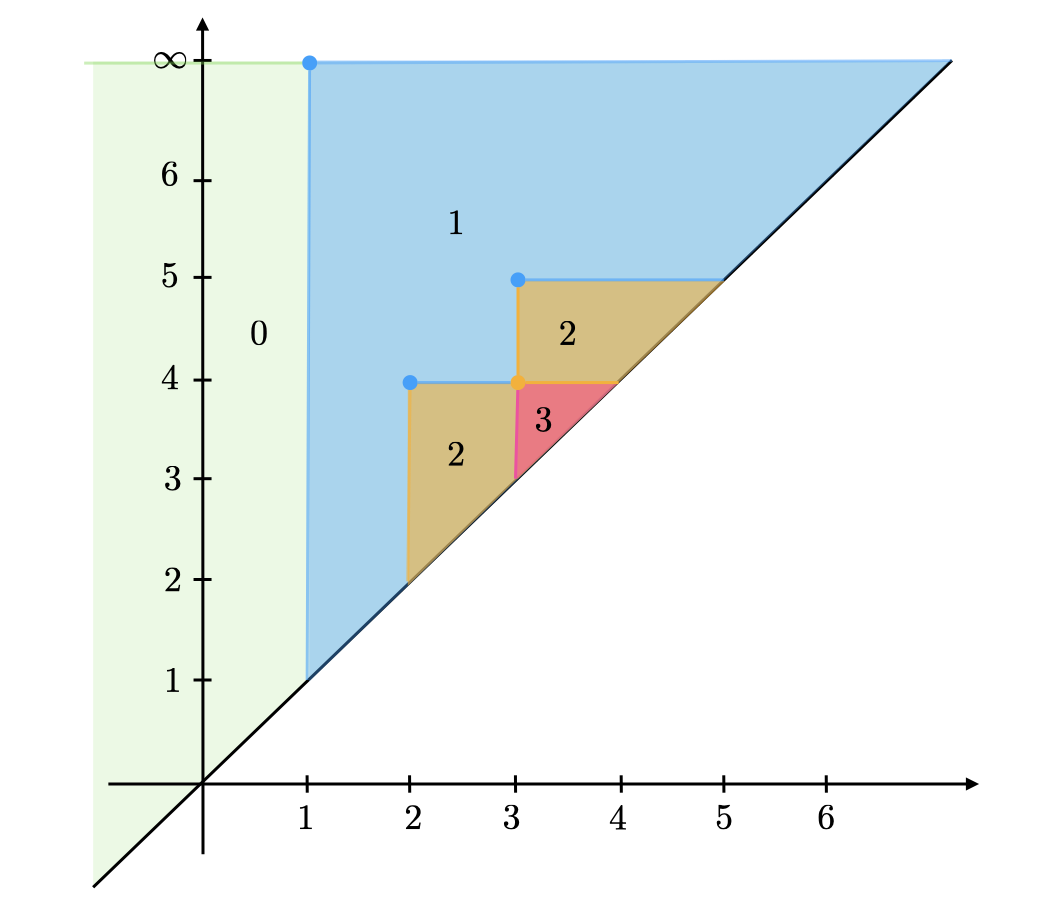}
\caption{The 0-th PBNF of the function $\p$ depicted in Figure \ref{figexFunction}. The numbers associated with the different colors correspond to the values of the PBNF.}
\label{figexPBNF}
\end{center}
\end{figure}

The 0-th PBNF of the function $\p$ considered in Figure \ref{figexFunction} is represented in Figure \ref{figexPBNF}.

\begin{exercise}\label{exercise:1Dcircle}
    Let $X$ be the unit circle $X:=\{(x,y)\in \R^2 \, |\, x^2+y^2=1\}$ and $\p\colon X \to \R$ be the function defined as $\p(x,y)=y$ for any $(x,y)\in X$.
    Determine the PBNFs of $\p$ and of $-\p$.
\end{exercise}

\begin{exercise}\label{exercise:eight}
    Let $X$ be the figure-eight curve in $\R^2$ described as $X:=\{(x,y)\in \R^2 \, |\, x^4-x^2+y^2=0\}$ and $\p\colon X \to \R$ and $\p'\colon X \to \R$ be the functions defined as $\p(x,y)=y$ and $\p'(x,y)=x$ for any $(x,y)\in X$.
    Determine the PBNFs of $\p$ and of $\p'$.
\end{exercise}

\begin{exercise}\label{exercise:1DcircleDistance}
    Let $X$ be the closed square of $\R^2$ defined as $X:=\{(x,y)\in \R^2 \, |\, -2\leq x \leq 2, -2\leq y \leq 2\}$ and $\p\colon X \to \R$ be the Euclidean distance from the circle in $\R^2$ centred at the origin and of radius 1.
    Determine the PBNFs of $\p$ and of $-\p$.
\end{exercise}


\subsection{Properties of PBNFs}\label{subsec:PBNF_properties}

After their definition, we discuss some elementary properties of the PBNFs. Specifically, their increasing/decreasing behavior with respect to the variables and their invariance under precomposition with homeomorphisms.

\begin{remark}\label{remmaxphi}
We recall that $X^\p_{v}=X^\p_{\max\p}=X^\p_{\infty}$, and hence
$$\beta_k^\p{(\cdot,v)}\equiv\beta_k^\p{(\cdot,\max\p)}
\equiv \beta_k^\p{(\cdot,\infty)},$$ for every $v\ge\max\p$.
\end{remark}

The following linear algebra statement holds.

\begin{lemma}\label{lemmaImgf}
Let $U,V,W$ be finite dimensional vector spaces over the field $\K$. If $f:U\to V$ and $g:V\to W$ are linear maps, then
$\dim \mathrm{Im\ } g f = \dim \mathrm{Im\ } f-\dim \ker g_{\vert\mathrm{Im\ } f}$.
\end{lemma}

\begin{proof}
The rank–nullity theorem applied to the linear map $g_{\vert\mathrm{Im\ } f}:\mathrm{Im\ } f\to W$ states that
$\dim \mathrm{Im\ } f=\dim \mathrm{Im\ } g f +\dim \ker g_{\vert\mathrm{Im\ } f}$.
\end{proof}

\begin{proposition}\label{propmonotonicity}
The function $\beta_k^\p$ is non-decreasing in the first variable and non-increasing in the second variable.

\end{proposition}

\begin{proof}
If $u\le u'$, then $i^{\p*}_{u,v}=i^{\p*}_{u',v} i^{\p*}_{u,u'}$.
This implies that $\Imm i^{\p*}_{u,v}\subseteq \Imm i^{\p*}_{u',v}$\ , i.e., $PH_k^\p{(u,v)}\subseteq PH_k^\p{(u',v)}$. Hence
$\dim PH_k^\p{(u,v)}\le \dim PH_k^\p{(u',v)}$.
On the other hand, by applying Lemma~\ref{lemmaImgf} for $f=i^{\p*}_{u,v}$ and $g=i^{\p*}_{v,v'}$ with $v'\ge v$ we get that
$\dim \Imm i^{\p*}_{u,v'}\le \dim \Imm i^{\p*}_{u,v}$\ , i.e.,
$\dim PH_k^\p{(u,v')}\le \dim PH_k^\p{(u,v)}$.
\end{proof}

The following result illustrates an interesting invariance property of the PBNFs with respect to a homeomorphism $g\colon X\to Y$ between two topological spaces. 

\begin{proposition}\label{propinvariancebeta}
For every homeomorphism $g\colon X\to Y$, the functions $\beta_k^\p$ and $\beta_k^{\p g^{-1}}$ coincide.
\end{proposition}


\begin{proof}
Since $g$ is a homeomorphism, its restriction $g_u\colon X_u^{\p}\to Y_u^{\p g^{-1}}$ is also a homeomorphism, and it induces an isomorphism $g_{p*, u}\colon H_k(X_u^{\p})\to H_k(Y_u^{\p g^{-1}})$ by Proposition~\ref{homeohomologyiso}, for every $u$.
The map $g_{p*, u}$ sends the homology class of the singular chain $\sum_{i=1}^ra_i\sigma_i$ to the homology class of the singular chain $\sum_{i=1}^ra_i g\sigma_i$.
The map $\rho\colon PH_k^\p(u,v)\to PH_k^{\p g^{-1}}(u,v)$ is the restriction of $g_{p*, v}$ to $PH_k^\varphi(u,v)$, meaning that it fits into the commutative diagram on the right below.
It is well-defined and it satisfies $i_{u,v}^{\p g^{-1} *}g_{p *, u}=\rho i_{u,v}^{\p *}$, meaning that it
fits into the commutative diagram on the left below,

\[
\begin{tikzcd}
H_k(X_u^\p)\arrow[d, two heads,  "i_{u,v}^{\p *}"]\arrow[r, "g_{p *, u}", "\cong"'] & H_k(Y_u^{\p g^{-1}})\arrow[d, two heads, "i_{u,v}^{\p g^{-1} *}"]\\
PH^\p_k(u,v)\arrow[r,"\rho"] & PH^{\p g^{-1}}_k(u,v)
\end{tikzcd}
\;
\begin{tikzcd}
PH^\p_k(u,v)\arrow[r,"\rho"]\arrow[d, hook] & PH^{\p g^{-1}}_k(u,v)\arrow[d,hook]\\
H_k(X_v^\p)\arrow[r, "g_{p *, v}", "\cong"'] & H_k(Y_v^{\p g^{-1}})
\end{tikzcd}
\]
In particular, the first diagram implies that $\rho$ is surjective and the second that it is injective, hence $\rho$ is an isomorphism.
\end{proof}

\subsection{The metric space $\Delta^*$}\label{subsec:metric_delta*}

To discuss the continuity of PBNFs, $\Delta^*$ has to be a topological space. 
Thus, we provide $\Delta^*$ with a metric $\tilde d$ which will induce a topology on it.\\

Consider the $L^\infty$-distance on $(\mathbb{R}\cup\{\infty\})^2$
\[
d_\infty(A,B):=\max\{\lvert x_A-x_B\rvert, \lvert y_A-y_B\rvert\},
\]
where $\infty-\infty=0$ and, if $v\in \mathbb{R}$, $\infty-v=\infty$.

Define on $\Delta^*$ 
\[
\tilde d(A, B):=\min\left\{\max\left\{|x_A-x_B|,|y_A-y_B|\right\},\max\left\{\frac{|x_A-y_A|}{2},\frac{|x_B-y_B|}{2}\right\}\right\},
\]
and note that it can be written as 
\[
\tilde d(A, B)=\min\left\{d_\infty(A, B), \max\left\{d_\infty(A, A_\Delta), d_\infty(B, B_\Delta)\right\}\right\},
\]
where $A_\Delta$ and $B_\Delta$ are, respectively, the projections of $A$ and $B$ onto the diagonal $\Delta:=\{(x,y) \in \R^2 \mid x = y \}$ and, hence, are the points $A_\Delta=(\frac{x_A+y_A}{2}, \frac{x_A+y_A}{2})$ and $B_\Delta=(\frac{x_B+y_B}{2}, \frac{x_B+y_B}{2})$ (see Figure~\ref{fig_metric}). 

\begin{figure}
    \centering
\includegraphics[width=6cm]{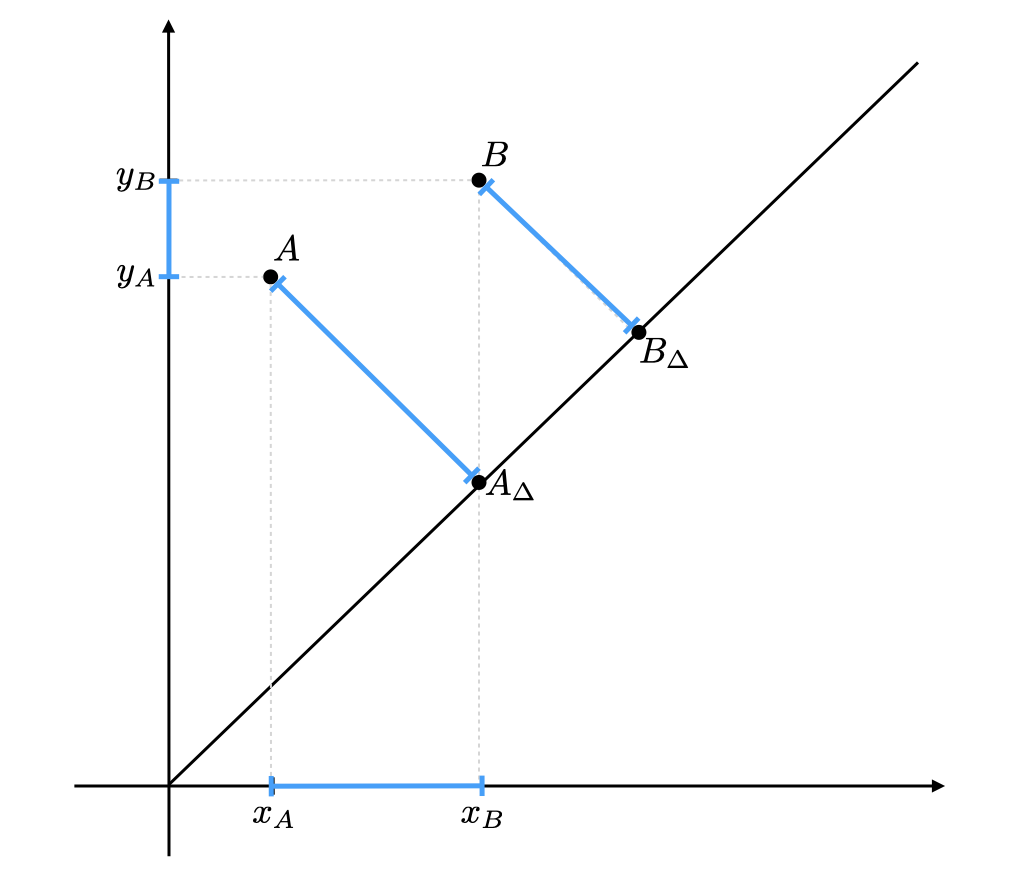}
    \caption{The distance $\tilde d(A,B)$ between two points in $\Delta^*$ is given by the minimum between the distance of the two points from one another according to $d_\infty$, and their maximal distance from the diagonal. }
    \label{fig_metric}
\end{figure}

\begin{proposition}
$(\Delta^*, \tilde d)$ is a metric space. 
\end{proposition}
\begin{proof}
It is clear that $\tilde d(A,B)=0$ is equivalent to $A=B$ and that $\tilde d$ is symmetric. 
So it remains to show the triangle inequality.
We denote by 
$d_\Delta(A,B)=\max\left\{d_\infty (A, A_\Delta),d_\infty(B, B_\Delta)\right\}$, hence, clearly $$d_\Delta(A,B)\ge d_\infty (A, A_\Delta), d_\infty (B, B_\Delta).$$ 
Assume by contradiction that there exist $A, B$ and $C$ such that
\[
\tilde d(A, B)>\tilde d(A,C)+\tilde d(C, B).
\]
This cannot happen if $\tilde d(\ast, \star)= d_\infty(\ast, \star)$ for every $\ast,\star\in\{A,B,C\}$, and if $\tilde d(\ast, \star)= d_\Delta(\ast, \star)$ for every $\ast,\star\in\{A,B,C\}$,
since both $d_\infty$ and $d_\Delta$ verify the triangle inequality.
We, thus, remain with the following cases, up to symmetric choices for $A$ and $B$: 
\begin{itemize}
\item[1.] $\tilde d(A, B)=d_\Delta(A, B)$, $\tilde d(A,C)=d_\infty(A, C)$ and $\tilde d(B,C)=d_\Delta(B,C)$;
\item[2.] $\tilde d(A, B)=d_\Delta(A, B)$, $\tilde d(A,C)=d_\infty(A, C)$ and $\tilde d(B,C)=d_\infty(B,C)$;
\item[3.] $\tilde d(A, B)=d_\infty(A, B)$, $\tilde d(A,C)=d_\Delta(A, C)$ and $\tilde d(B,C)=d_\Delta(B,C)$;
\item[4.] $\tilde d(A, B)=d_\infty(A, B)$, $\tilde d(A,C)=d_\infty(A, C)$ and $\tilde d(B,C)=d_\Delta(B,C)$.
\end{itemize}

\textit{Case} 1. If $d_\Delta(A,B)=d_\infty(A, A_\Delta)$, then we reach a contradictory chain of inequalities
\begin{align*}
d_\infty(A, A_\Delta) & =d_\Delta(A,B)\\
& >d_\infty(A, C)+d_\Delta(C, B)\\
 & \ge d_\infty(A, C)+d_\infty(C, C_\Delta)\\
 & \ge d_\infty (A, C_\Delta) \\
 & \ge d_\infty (A, A_\Delta)
\end{align*}
where the forth step is obtained by applying the triangle inequality for $d_\infty$, and the fifth is given by the fact that $A_\Delta$ is the point on $\Delta$ with minimum distance from $A$. 
If $d_\Delta(A,B)=d_\infty(B, B_\Delta)$, then 
\begin{align*}
d_\infty(B, B_\Delta) & =d_\Delta(A,B)\\
& >d_\infty(A, C)+d_\Delta(C, B)\\
& \ge d_\infty(A, C)+d_\infty (B, B_\Delta)\\
& \ge d_\infty(B, B_\Delta)
\end{align*}
which is a contradiction. 

\textit{Case} 2. Since $\tilde d(A, B)\le d_\infty(A, B)$, by definition,
\[
d_\infty(A,B) \ge d_\Delta(A, B)>d_\infty(A, C)+d_\infty(C, B),
\]
which is a contradiction as $d_\infty$ satisfies the triangle inequality. 

\textit{Case} 3. Since $\tilde d(A, B)\le d_\Delta(A, B)$, by definition,
\[
d_\Delta(A,B) \ge d_\infty(A, B)> d_\Delta(A,C)+d_\Delta(C, B),
\]
which is a contradiction as $d_\Delta$ satisfies the triangle inequality.

\textit{Case} 4. 
Observe that
\begin{align*}
\max\{d_\infty(A, A_\Delta), d_\infty(B, B_\Delta)\} & \ge d_\infty(A, B) \\
& > d_\infty(A, C)+d_\Delta(C, B) \\
& \ge d_\infty (A, C)+d_\infty(C, C_\Delta)\\
& \ge d_\infty (A, C_\Delta)\\
& \ge d_\infty(A, A_\Delta)
\end{align*}
which implies that $d_\infty(B, B_\Delta)>d_\infty(A, A_\Delta)$.
On the other hand, 
\begin{align*}
\max\{d_\infty(A, A_\Delta), d_\infty(B, B_\Delta)\} & \ge d_\infty(A, B)\\
& > d_\infty(A, C)+d_\Delta(C, B) \\
& \ge d_\infty (A, C)+d_\infty(B, B_\Delta)\\
& \ge d_\infty(B, B_\Delta)
\end{align*}
which implies that $d_\infty(A, A_\Delta)>d_\infty(B, B_\Delta)$.  
These two observations together give a contradiction. 
\end{proof}

Another way to prove that $\tilde d$ is a metric— in particular, that it satisfies the triangle inequality— is to show that it coincides with the metric $d'$ defined in the following exercise.
\begin{figure}
\begin{center}
\includegraphics[width=10cm]{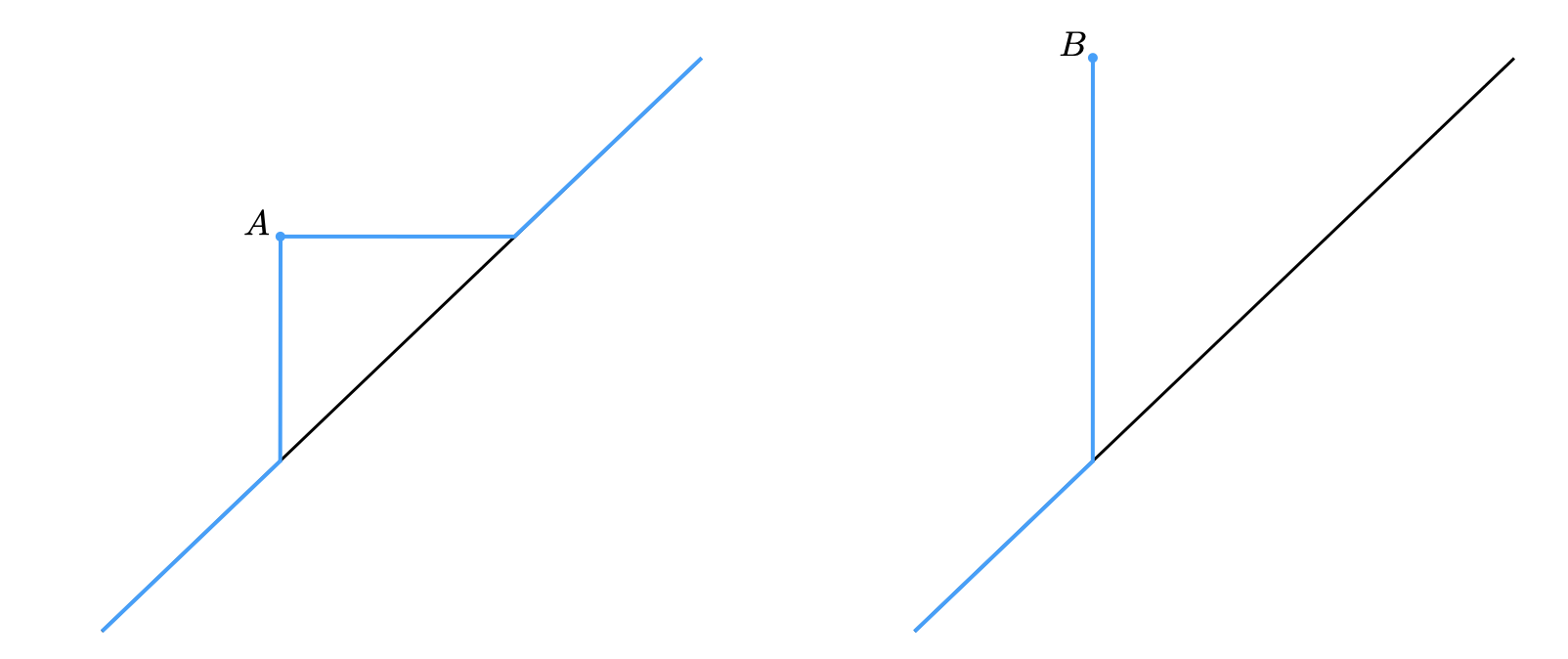}
 \caption{On the left, the profile of $A = (u,v)$ in $\Delta^+$, and on the right, the profile of $B = (u,\infty)$.
}
\label{figura:profilo}
\end{center}
\end{figure}


\begin{exercise}\label{ex:profile}
For every $A\in \Delta^*$ define the \emph{profile} of $A$, $\textup{prof}(A)$, as the boundary of the set 
$\{(x,y)\in\R^2:y\le x \mathrm{\ or\ }
(x\ge x_A\mathrm{\ and\ } y\le y_A)\}$ (see Figure~\ref{figura:profilo}).
Given the set $\mathcal{L}$ of lines of slope $-1$ in $\R^2$ 
and two points $A,B\in \Delta^*$, define $d'(A,B):=\sup_{\ell\in \mathcal{L}}\lVert I_A-I_B\rVert_\infty$, where $I_X$ is the unique point in $\ell\cap \textup{prof}(X)$.
Show that for any $A, B$ in $\Delta^*$, $d'(A, B)=\tilde d(A, B)$. 
Observe that one can reinterpret $d'$ as the 
uniform metric between the functions whose graphs are the two sets obtained by rotating the profiles of $A$ and $B$ by $45^\circ$ clockwise around the point $(0,0)$, divided by $\sqrt{2}$.

[\emph{Hint:} 
The only non-trivial case to examine is the one in which both $A$ and $B$ are proper cornerpoints.
Without loss of generality, by possibly exchanging $A$ and $B$, dilating the plane, and applying a symmetry with respect to the line $x + y = 1$, we may assume that $A = (0,1)$ and that $B = (x_B, y_B)$ satisfies $x_B + y_B \ge 1$ and $y_B - x_B \le 1$.
Thus, if $x_B\le \frac{1}{2}$ then $\tilde d(A, B)=d'(A, B)=x_B$, otherwise $\tilde d(A, B)=d'(A, B)=\frac{1}{2}$.]
\end{exercise}

\subsection{Total multiplicity}\label{subsec:tot_mult}

In this section, we introduce the notion of total multiplicity capturing the variation of PBNFs and being preparatory for the introduction of persistence diagrams.\\


In the following subsections, we will make use of the notion of closure defined below.

\begin{definition}\label{upper_closure}
Consider the product order on $\R^2$ and a set $S \subseteq \Delta^+$. 
We define the \textbf{upper closure} of $S$, denoted by $S^\lnot$, as the collection of points in $\Delta^+$ for which there exists a non-decreasing sequence of points in $S$ converging to them.
\end{definition}

\begin{definition}\label{defbox}
Let $\varepsilon, \eta>0$ 
and $(u,v) \in \Delta^+$, with $u+\varepsilon \le v-\eta$. 
The upper closure of the open rectangle 
$]u-\varepsilon,\,u+\varepsilon[ \times ]v-\eta,\,v+\eta[$
is called an \textit{$(\varepsilon,\eta)$-box centred at $(u,v)$}.
\end{definition}

\begin{figure}
    \centering
    \includegraphics[width=0.45\linewidth]{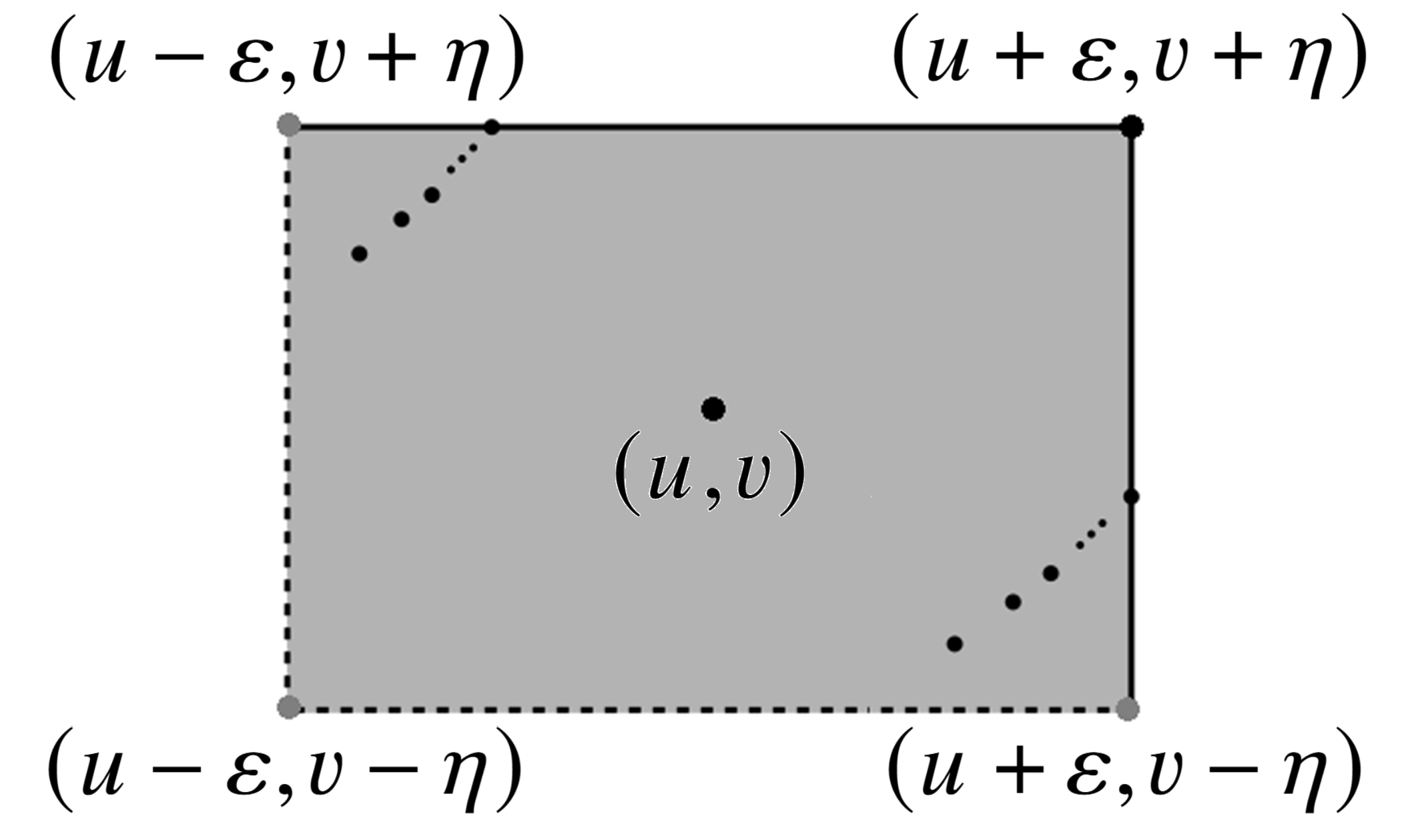}
    \caption{An $(\eps, \eta)$-box centred at $(u,v)$, with $u+\varepsilon\le v-\eta$. Note that the top-right corner is the only corner belonging to the upper closure.}
    \label{fig:placeholder}
\end{figure}

\begin{remark}
    We observe that for any $(u,v) \in \Delta^+$ with $u+\varepsilon \le v-\varepsilon$, $\eps>0$, the open square
$]u-\varepsilon,\,u+\varepsilon[ \times ]v-\varepsilon,\,v+\varepsilon[$
coincides with the open ball centred at $(u,v)$ of radius $\varepsilon$ with respect to the metric $\tilde d$.
\end{remark}


The PBNFs can be summarised visually, as in Figure~\ref{figexPBNF}, as a partition of the half-plane of equation $y\ge x$. 
For every $(\eps, \eta)$-box centred at a point $(u,v)\in\Delta^+$, a notion of multiplicity is defined. 
This multiplicity measures how different the PBNFs are inside the box.
In the following, we will assume that a continuous function $\p\colon X\to\R$ is given.

\begin{definition}\label{defmu}
Let $\varepsilon, \eta>0$
and $(u,v)\in \Delta^+$,
with $u+\varepsilon<v-\eta$. 
We define the \myemph{total multiplicity of the $(\eps, \eta)$-box centred at $(u,v)$} as
$$\mu^\p_{\eps,\eta}(u,v)=
\beta_k^\p{(u+\eps,v-\eta)}
- \beta_k^\p{(u-\eps,v-\eta)} 
+ \beta_k^\p{(u-\eps,v+\eta)}
- \beta_k^\p{(u+\eps,v+\eta)}
.$$


\end{definition}





\begin{ex}\label{exmu}
Consider the PBNF in Figure~\ref{figexPBNF} and its points $(2,4)$ and $(3,4)$. 
If $\eps<1$, the multiplicity of the $(\eps, \eps)$-box centred at $(2,4)$ is equal to 1, and the one of the $(\eps, \eps)$-box centred at $(3,4)$ is equal to 0.
\end{ex}


The next proposition shows that the value $\mu^\p_{\eps,\eta}(u,v)$ is always non-negative.

\begin{proposition}\label{prop_premult}
Given $(u,v)\in\Delta^+$ and $\varepsilon, \eta>0$ 
such that $u +\eps < v - \eta$, it holds that
\[
\beta_k^\p{(u+\eps,v-\eta)} - \beta_k^\p{(u-\eps,v-\eta)}\ge 
\beta_k^\p{(u+\eps,v+\eta)} - \beta_k^\p{(u-\eps,v+\eta)}.
\]
\end{proposition}

\begin{proof}
By applying Lemma~\ref{lemmaImgf} for $f=i^{\p*}_{u+\eps,v-\eta}$ and $g=i^{\p*}_{v-\eta,v+\eta}$ we get
$$\dim \textrm{Im\ } i^{\p*}_{u+\eps,v+\eta} = \dim \textrm{Im\ } i^{\p*}_{u+\eps,v-\eta}-\dim \ker {i^{\p*}_{v-\eta,v+\eta}}_{\vert\textrm{Im\ } i^{\p*}_{u+\eps,v-\eta}}.$$
By applying Lemma~\ref{lemmaImgf} for $f=i^{\p*}_{u-\eps,v-\eta}$ and $g=i^{\p*}_{v-\eta,v+\eta}$ we get
$$\dim \textrm{Im\ } i^{\p*}_{u-\eps,v+\eta} = \dim \textrm{Im\ } i^{\p*}_{u-\eps,v-\eta}-\dim \ker {i^{\p*}_{v-\eta,v+\eta}}_{\vert\textrm{Im\ } i^{\p*}_{u-\eps,v-\eta}}.$$
It follows that
\begin{align*}
              &\beta_k^\p{(u+\eps,v-\eta)}
            -\beta_k^\p{(u-\eps,v-\eta)} \\ =& \dim \textrm{Im\ }i^{\p*}_{u+\eps,v-\eta}
            -\dim \textrm{Im\ } i^{\p*}_{u-\eps,v-\eta}\\
            = & \dim\textrm{Im\ }i^{\p*}_{u+\eps,v+\eta}+\dim \ker {i^{\p*}_{v-\eta,v+\eta}}_{\vert\textrm{Im\ } i^{\p*}_{u+\eps,v-\eta}}\\
            &- \dim\textrm{Im\ }i^{\p*}_{u-\eps,v+\eta}- \dim \ker {i^{\p*}_{v-\eta,v+\eta}}_{\vert\textrm{Im\ } i^{\p*}_{u-\eps,v-\eta}} \\
            \ge &  \beta_k^\p{(u+\eps,v+\eta)} - \beta_k^\p{(u-\eps,v+\eta)}
\end{align*}
where the last inequality follows from the inclusion
$$\textrm{Im\ } i^{\p*}_{u-\eps,v-\eta}=\textrm{Im\ } i^{\p*}_{u+\eps,v-\eta} i^{\p*}_{u-\eps,u+\eps}
\subseteq \textrm{Im\ } i^{\p*}_{u+\eps,v-\eta}$$ which implies
$$\ker {i^{\p*}_{v-\eta,v+\eta}}_{\vert\textrm{Im\ } i^{\p*}_{u-\eps,v-\eta}}\subseteq \ker {i^{\p*}_{v-\eta,v+\eta}}_{\vert\textrm{Im\ } i^{\p*}_{u+\eps,v-\eta}}.$$
\end{proof}

Observe that the inequality in Proposition~\ref{prop_premult} is equivalent to
\[\beta_k^\p{(u+\eps,v-\eta)} - \beta_k^\p{(u+\eps,v+\eta)} \ge \beta_k^\p{(u-\eps,v-\eta)} -
\beta_k^\p{(u-\eps,v+\eta)}.\]

\begin{corollary}\label{propmu>=0}
Given $(u,v)\in\Delta^+$ and $\varepsilon, \eta>0$ 
such that $u + \eps < v- \eta$, the multiplicity $\mu_{\eps,\eta}^\p(u,v)$ is non-negative. 
\end{corollary}


The next proposition shows that, fixed $(u,v) \in \Delta^+$, the function
$\mu^\p_{\eps,\eta}(u,v)$ is non-decreasing in the variables $\eps$ and $\eta$.

\begin{proposition}\label{muincreasing}
Let $(u,v)\in \Delta^+$.
If $0 < \eps\le \eps'$, $0 < \eta\le \eta'$, and $u+\eps' < v-\eta'$, then
$\mu^\p_{\eps,\eta}(u,v)\le \mu^\p_{\eps',\eta'}(u,v)$.
\end{proposition}

\begin{proof}
By directly applying Definition~\ref{defmu}, it is easy to check that
\begin{align*}
             \mu^\p_{\eps',\eta'}\left(u,v\right)&= \mu^\p_{r,p}\left(u-s,v+q\right)
            + \mu^\p_{\eps,p}\left(u,v+q\right)
            + \mu^\p_{r,p}\left(u+s,v+q\right)\\
            & +\mu^\p_{r,\eta}\left(u-s,v\right)
            + \hskip 0.8cm \mu^\p_{\eps,\eta}\left(u,v\right)\hskip 0.7cm
            + \mu^\p_{r,\eta}\left(u+s,v\right)\\
            & +\mu^\p_{r,p}\left(u-s,v-q\right)
            + \mu^\p_{\eps,p}\left(u,v-q\right)
            + \mu^\p_{r,p}\left(u+s,v-q\right)
\end{align*}
where $r=\frac{\eps'-\eps}{2}$, $s=\varepsilon+r=\frac{\eps'+\eps}{2}$,  $p=\frac{\eta'-\eta}{2}$, and $q=\eta+p=\frac{\eta'+\eta}{2}$. 
From the non-negativity of the function $\mu^\p_{\eps, \eta}$ (Corollary~\ref{propmu>=0}), it follows that $\mu^\p_{\eps, \eta}(u,v)\le \mu^\p_{\eps', \eta'}(u,v)$.
\end{proof}

\begin{figure}[htb!]
\begin{center}
\includegraphics[width=10cm]{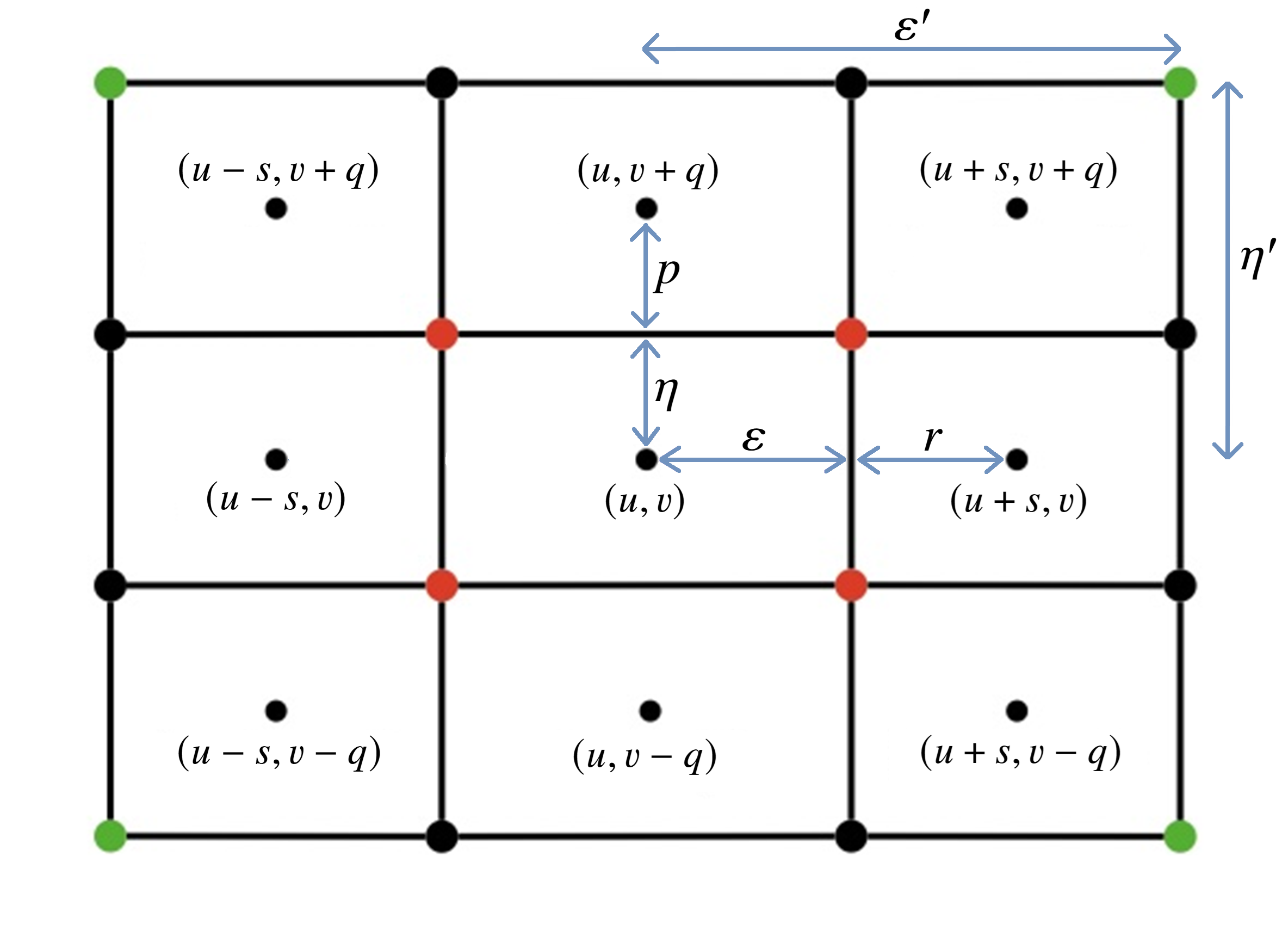}
\caption{The grid used in the proof of Proposition~\ref{muincreasing}.}
\label{fig_in_plain_words}
\end{center}
\end{figure}

Let us explain the proof of Proposition~\ref{muincreasing} in plain words.
The value $\mu^\p_{\eps, \eta}(u,v)$ is equal to the sum of the values of $\beta^\p_k$ at the red points in Figure~\ref{fig_in_plain_words} with alternating signs. 
The value $\mu^\p_{\eps', \eta'}(u,v)$ is equal to the sum of the values of $\beta^\p_k$ at the green points in Figure~\ref{fig_in_plain_words} with alternate signs. 
This last value is equal to the sum of the total multiplicities of the nine boxes in the figure since the contributions of the values of $\beta^\p_k$ at the black and red points cancel each other out.

\subsection{Discontinuities of PBNFs}\label{subsec:PBNF_discontinuities}
We recall that PBNFs are functions with values in the natural numbers. 
A discontinuity point for such a function is, thus, a point of the domain at which the function is not locally constant. 

\begin{lemma}\label{lemmadisco}
Any path-connected open neighborhood of a discontinuity point
for the function $\beta_k^\p$ contains at least one discontinuity point in the first or in the second variable.
\end{lemma}


\begin{proof}
If $(u,\infty)$ is a discontinuity point, then the discontinuity can only be on the first variable. 
Thus, it suffices to show our statement for $(u,v)\in \Delta^+$ discontinuity point for $\beta_k^\p$. 
In this case, in any path-connected open neighborhood $U\subseteq \Delta^+$ of $(u,v)$, a point $(u',v')\in \Delta^+$ exists such that
$\beta_k^\p(u,v)\neq\beta_k^\p(u',v')$.
We can connect $(u,v)$ and $(u',v')$ by a path entirely contained in $U$ and made of horizontal and vertical segments. Since $\beta_k^\p$ cannot be constant
along this path, our claim follows.
\end{proof}

Next, we show that discontinuities in the first variables propagate along the second, and discontinuities in the second variable propagate along the first. 
\begin{proposition}\label{proppropadisc}
The following statements hold:
\begin{itemize}
  \item[1.] If $\bar u$ is a discontinuity point for
  $\beta_k^\p{(\cdot,\bar v)}$ and
  $\bar u< v < \bar v<\infty$, then $\bar u$ is a discontinuity point also for $\beta_k^\p{(\cdot,v)}$;
 \item[2.] If $\bar v$ is a discontinuity point for
  $\beta_k^\p{(\bar u,\cdot)}$ and
  $\bar u< u < \bar v<\infty$, then $\bar v$ is a discontinuity point also for $\beta_k^\p{(u,\cdot)}$.
\end{itemize}
\end{proposition}

\begin{proof}
\begin{itemize}
\item[1.] Since $\beta_k^\p$ is non-decreasing in its first variable,
we have that $$\beta_k^\p{(\bar u+\eps,\bar v)}
            -\beta_k^\p{(\bar u-\eps,\bar v)}\ge 0$$ for any sufficiently small $\eps\ge 0$
            (Proposition~\ref{propmonotonicity}).
            The assumption that $\bar u$ is a discontinuity point for
  $\beta_k^\p{(\cdot,\bar v)}$ implies that
  $\beta_k^\p{(\bar u+\eps,\bar v)}
            -\beta_k^\p{(\bar u-\eps,\bar v)}>0$ for any sufficiently small $\eps>0$. This last inequality and the inequality
 \begin{align*}
            & \beta_k^\p{(\bar u+\eps,v)}
            -\beta_k^\p{(\bar u-\eps,v)}-\beta_k^\p{(\bar u+\eps,\bar v)}
            +\beta_k^\p{(\bar u-\eps,\bar v)}\ge 0
\end{align*}
(Proposition~\ref{prop_premult}) imply that
            $\beta_k^\p{(\bar u+\eps,v)}
            -\beta_k^\p{(\bar u-\eps,v)}>0$ for any sufficiently small $\eps>0$. It follows that $\bar u$ is a discontinuity point also for $\beta_k^\p{(\cdot,v)}$.
 \item[2.] Since $\beta_k^\p$ is non-increasing in its second variable, we have that 
$$\beta_k^\p{(\bar u,\bar v+\eta)}
            -\beta_k^\p{(\bar u,\bar v-\eta)}\le 0$$ for any sufficiently small $\eta\ge 0$
            (Proposition~\ref{propmonotonicity}).
            The assumption that $\bar v$ is a discontinuity point for
  $\beta_k^\p{(\bar u,\cdot)}$ implies that
  $\beta_k^\p{(\bar u,\bar v+\eta)}
            -\beta_k^\p{(\bar u,\bar v-\eta)}< 0$ for any  sufficiently small $\eta> 0$. This last inequality and the inequality
 \begin{align*}
            & \beta_k^\p{(u,\bar v-\eta)}
            -\beta_k^\p{(\bar u,\bar v-\eta)} -\beta_k^\p{(u,\bar v+\eta)}
            +\beta_k^\p{(\bar u,\bar v+\eta)}\ge 0
\end{align*}
(Proposition~\ref{prop_premult}) imply that
            $\beta_k^\p{(u,\bar v-\eta)}
            -\beta_k^\p{(u,\bar v+\eta)}>0$ for any sufficiently small $\eta>0$. It follows that $\bar v$ is a discontinuity point also for $\beta_k^\p{(u,\cdot)}$.
\end{itemize}
\end{proof}

\begin{proposition}\label{propWeps}
For every point $(\bar u,\bar v)\in \Delta^+$, there exists $\eps>0$ such that the open set
$$W_\eps(\bar u,\bar v):=\{(u,v)\in\R^2:|u-\bar u|<\eps,|v-\bar v|<\eps, u\neq \bar u, v\neq \bar v\}$$
is contained in $\Delta^+$ and does not contain any discontinuity point for $\beta_k^\p$.
\end{proposition}


\begin{proof}
Suppose, contrarily to our assertion, that for every positive integer $n$
a discontinuity
point $(u_n,v_n)\in W_\frac{1}{n}(\bar u,\bar v)$ exists. By applying Lemma~\ref{lemmadisco},
possibly by extracting a subsequence from $\{(u_n,v_n)\}_{n \in \N}$, we can assume that either each
$(u_n,v_n)$ is a discontinuity point in the first variable or each
$(u_n,v_n)$ is a discontinuity point in the second variable. In the following, we
shall assume that each $(u_n,v_n)$ is a discontinuity point in the first variable. The case in
which each $(u_n,v_n)$ is a discontinuity point in the second variable has a similar proof.
Let us fix a natural number $N$ that is so large that $\bar u+\frac{1}{N} < \bar v-\frac{1}{N}$,
i.e., the sets $W_\frac{1}{n}(\bar u, \bar v)$ with $n\ge N$ lie entirely above the diagonal $\Delta$. Let
us consider the function
$\beta_k^\p\left(\cdot, \bar v -\frac{1}{N}\right)\colon\ \left]\bar u-\frac{1}{N},\bar u +\frac{1}{N}\right[\  \to\N$.
From Proposition~\ref{proppropadisc} we know that discontinuities in $u$ spread downwards. Thus
the function $\beta_k^\p\left(\cdot, \bar v -\frac{1}{N}\right)$ should have an infinite number of discontinuities.
Now, since $\beta_k^\p(u, v)$ is non-decreasing in the variable $u$ (Proposition~\ref{propmonotonicity}), this fact would
imply that $\beta_k^\p\left(\bar u+\frac{1}{N}, \bar v -\frac{1}{N}\right)=\infty$, contradicting Assumption~\ref{ass_fingen}.
\end{proof}
A consequence of Proposition~\ref{propWeps} is that the discontinuities of PBNFs do not propagate in directions other than the horizontal and vertical ones. 

\section{Persistence diagrams}\label{sec:pers_diag}

In this section we introduce persistence diagrams. 
They provide another summary of the topology of a filtered topological space. 
We analyze the properties of persistence diagrams, in particular, their cardinality.

\subsection{The metric space $\bar\Delta^*$}\label{subsec:metric_delta*_bar}

We introduce a metric on $\bar \Delta^*$, where persistence diagrams are defined. 
This metric is not just the restriction of the $L^\infty$-distance to $\bar \Delta^*$. 
The reason for this choice is that we would like to claim that two points in $\Delta^+$ close to the diagonal $\Delta$ are close to each other, even when they are far away with respect to $d_\infty$. 
This will allow us later to think about persistence diagrams of small perturbations of the same function as close to each other.

First, recall that the completion of a metric space $(X, d)$ is the unique (up to isometries) complete metric space $(\overline X, \bar d)$ whose elements are the equivalence classes of Cauchy sequences in $(X, d)$, with respect to the equivalence relation $\approx$ between Cauchy sequences defined by setting $(A_i)\approx (B_i)$ if $\lim_{i\to\infty} d(A_i,B_i)=0$.
The metric space $(\bar\Delta^*, d)$ is defined as the completion of $(\Delta^*, \tilde d)$. 
This metric space is in particular a one-point completion, where the additional one point is denoted by $\Delta$, abusing the notation for the diagonal. 
Explicitly, 
\begin{proposition}\label{fskjgnfawljfnv}
There exists an isometric bijection between $(\Delta^*\cup \{\Delta\}, d_C)$ and $(\bar\Delta^*, d)$, 
where 
\[
d_C(A, B) :=
\begin{cases}
\tilde d(A, B) & \textup{if } A, B\in \Delta^*,\\
\frac{y_A-x_A}{2} & \textup{if } A\in \Delta^*, B=\Delta,\\
\frac{y_B-x_B}{2} & \textup{if } A= \Delta, B\in\Delta^*,\\
0 & \textup{if } A=B=\Delta.\\
\end{cases}
\]
\end{proposition}

\begin{proof}
Consider the map $\Phi\colon (\Delta^*\cup \{\Delta\}, d_C)\to (\bar\Delta^*, d)$ sending $A$ in $\Delta^*$ to the equivalence class $[(A, A, \dots)]$ of the constant sequence $(A, A, \dots)$, and $\Delta$ to the equivalence class of the Cauchy sequence 
$(P_n)_n$, where $P_n=\left(0,\frac{1}{n+1}\right)$.
The map $\Phi$ is an isometry.
To see this we need to separate the different cases. 
If $A, B\in \Delta^*$, then $d_C(A, B)=\tilde d(A, B)$ by definition, and since $(\bar\Delta^*, d)$ is the metric completion of $(\Delta^*, \tilde d)$, $d([(A,A,\dots)],[(B, B, \dots)])=\tilde d(A, B) $. 
Thus, $d_C(A, B)=d(\Phi(A), \Phi(B))$. 
If $A\in \Delta^*$ and $B=\Delta$, then $d_C(A,\Delta)=d_\infty(A, A_\Delta)$. 
On the other hand, $d([(A, A, \dots)], [(P_n)_n])=\lim_{n\to \infty} \tilde d(A, P_n)$ but this coincides with $d_\infty(A, A_\Delta)$, allowing us to conclude that $d_C(A, \Delta)=d(\Phi(A), \Phi(\Delta))$. 
Moreover, it is trivial to check that $d_C(\Delta, \Delta)=0=d(\Phi(\Delta), \Phi(\Delta))$.
Note that $\Phi$ being an isometry implies that it is also injective. 

It remains to show that $\Phi$ is surjective. 
Consider an element $[(A_n)_n]\in (\bar \Delta^*, d)$. 
There are two cases: either $(A_n)_n$ converges to a point $A$ in $(\Delta^*, d_\infty)$ or not.
In the former case, $(A_n)_n$ converges to $A$ also in $(\Delta^*, \tilde d)$. 
Thus, 
$(A_n)_n\approx (A, A, \dots)$,
and hence 
$[(A_n)_n]=\Phi(A)$.
In the latter case, we have that $$\lim_{n\to \infty}d_\infty(A_n, A_{n\Delta})=0.$$ 
This implies that $[(A_n)_n]=[(P_n)_n]$.
Indeed, 
\begin{align*}
d\left([(A_n)_n],\left[(P_n)_n\right]\right) & = \lim_{n\to \infty}d\left(A_n, P_n\right)\\
& \le \lim_{n\to\infty}\max\left\{d_\infty(A_n, A_{n\Delta}), d_\infty(P_n, P_{n\Delta})\right\}=0, 
\end{align*}
since $d_\infty(P_n, P_{n\Delta})=\frac{1}{2(n+1)}$.
Thus, 
$[(A_n)_n]=\Phi(\Delta)$.
\end{proof}

From now on, the two metric spaces in Proposition~\ref{fskjgnfawljfnv} will be identified and denoted by $\bar\Delta^*$.
\begin{remark}\label{squared_balls}
If $y_A - x_A \ge 2\varepsilon$, the open ball $\bar B$ of radius $\varepsilon >0$ with respect to the metric $d$ and centred at $(x_A, y_A)$ coincides with the open ball $\hat B$ of radius $\varepsilon$ with respect to the metric $d_\infty$, centred at $(x_A, y_A)$.
If $y_A - x_A < 2\varepsilon$, the open ball $\bar B$ coincides with the union of the set $\hat B\cap \Delta^+$ and the set $\{(x,y)\in\mathbb{R}^2 : y < x + 2\varepsilon\} \cup \{\Delta\}$ (see Figure~\ref{box-mult}).
From now on, we will denote the open ball with respect to the metric $d$, centred at $(x_A, y_A)$ and of radius $\eps >0$, by $B_\eps(x_A, y_A)$.
\end{remark}

\begin{figure}
\centering
\includegraphics[width=5cm]{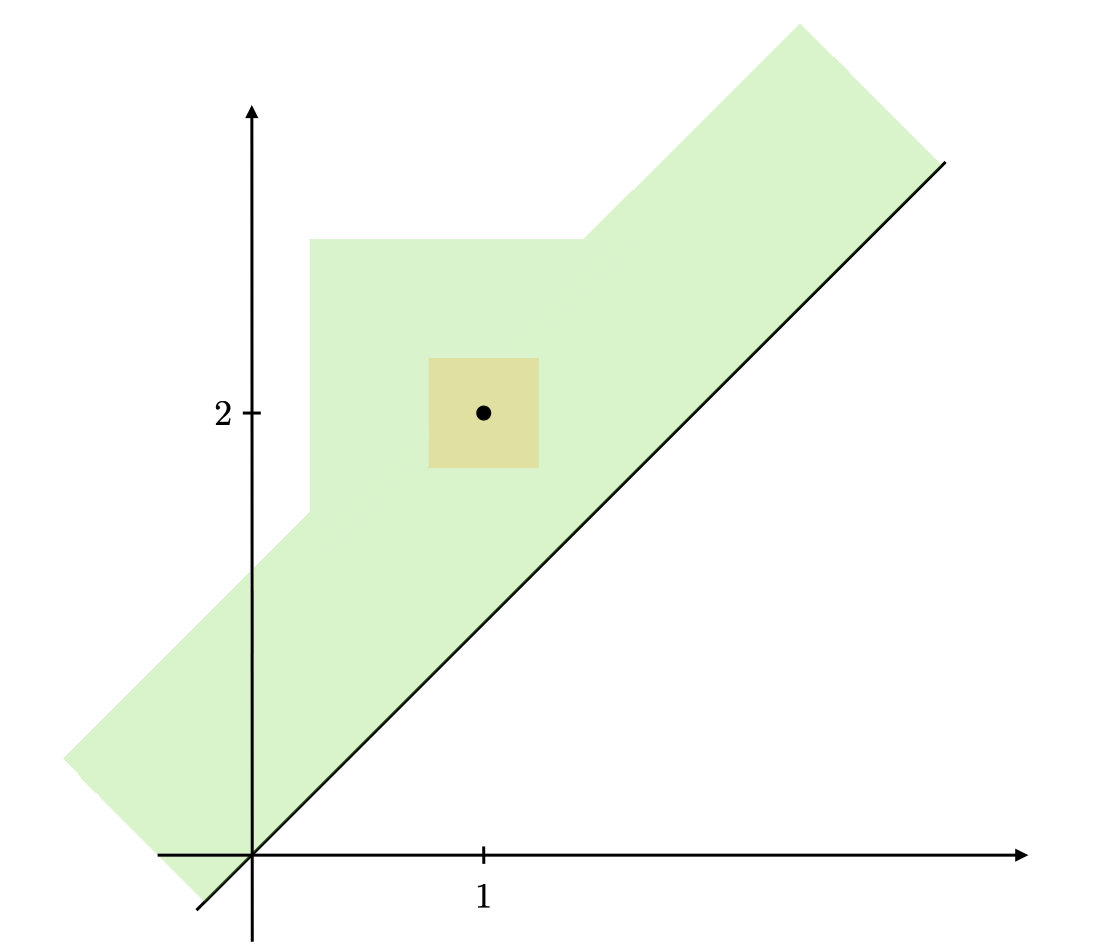}
\caption{
Two open balls of radius $\varepsilon$ for the metric $\tilde d$, centred at $(1,2)$, are represented. For the orange one, $2\varepsilon \le 1$, while for the green one, $2\varepsilon > 1$.
}
\label{box-mult}
\end{figure}

\begin{exercise}
Show that $\bar \Delta^*$ with the topology induced by $d$ is disconnected.
\end{exercise}

\subsection{Multisets and definition of persistence diagrams}\label{subsec:PD}

\begin{definition}[Multiset]\label{Multiset}
A \myemph{multiset} is a pair $(\mathcal{S},f)$ of a set $\mathcal{S}$ and a function $f\colon \mathcal{S}\to\N\cup\{\infty\}$. 
If $S\in \mathcal{S}$ and $f(S)>0$, we say that $S$ is an element of the multiset and its \myemph{multiplicity} is $f(S)$.
We may refer to the multiset $(\mathcal{S},f)$ as the collection $\mathcal{S}_f=\{(S,n)\in \mathcal{S}\times (\N\cup\{\infty\}): 0< n\le f(S)\}$ called its \myemph{realisation}.
\end{definition}

A multiset can be seen as a set in which each element may appear in multiple distinct copies.
We will often say that an element $S$ belongs to the multiset $(\mathcal{S}, f)$ with multiplicity $n=f(S)>0$ if $(S, k)$ belongs to $\mathcal{S}_f$, for every $k\le n$.

\begin{definition}[Map between multisets]\label{mapsmultisets}
Let $(\mathcal{S},f)$ and $(\mathcal{S}',f')$ be two multisets.
Any map from $\mathcal{S}_{f}$ to $\mathcal{S}'_{f'}$ is called a \myemph{multiset map} from the multiset $(\mathcal{S},f)$ to the multiset $(\mathcal{S}',f')$.
\end{definition}

\begin{definition}[Matching between multisets]\label{Bijmultisets}
Let $(\mathcal{S},f)$ and $(\mathcal{S}',f')$ be two multisets.
Any bijection from $\mathcal{S}_{f}$ to $\mathcal{S}'_{f'}$ is called a \myemph{matching} from the multiset $(\mathcal{S},f)$ to the multiset $(\mathcal{S}',f')$.
\end{definition}

\begin{ex}
Consider the set $\Delta^+$ and the function $f\colon \Delta^+\to \mathbb{N}\cup \{\infty\}$ assigning $2$ to the element $(1,2)$ and $0$ to all the other elements. 
The pair $(\Delta^+, f)$ is a multiset that can also be written as $\mathcal{M}=\{((1,2), 1), ((1,2), 2)\}$. 
If $\mathcal{M'}=\{((1,2), 1), ((1,3), 1)\}$ 
is another multiset, we can define a multiset map $\alpha\colon \mathcal{M}\to \mathcal{M}'$ assigning $((1,2), 1)$ to itself and $((1,2),2)$ to $((1,3),1)$. 
The map $\alpha$ is, in particular, a matching between $\mathcal{M}$ and $\mathcal{M}'$.
See Figure~\ref{fig_multiset} for a representation.
\end{ex}
\begin{figure}
\centering
\includegraphics[width=10cm]{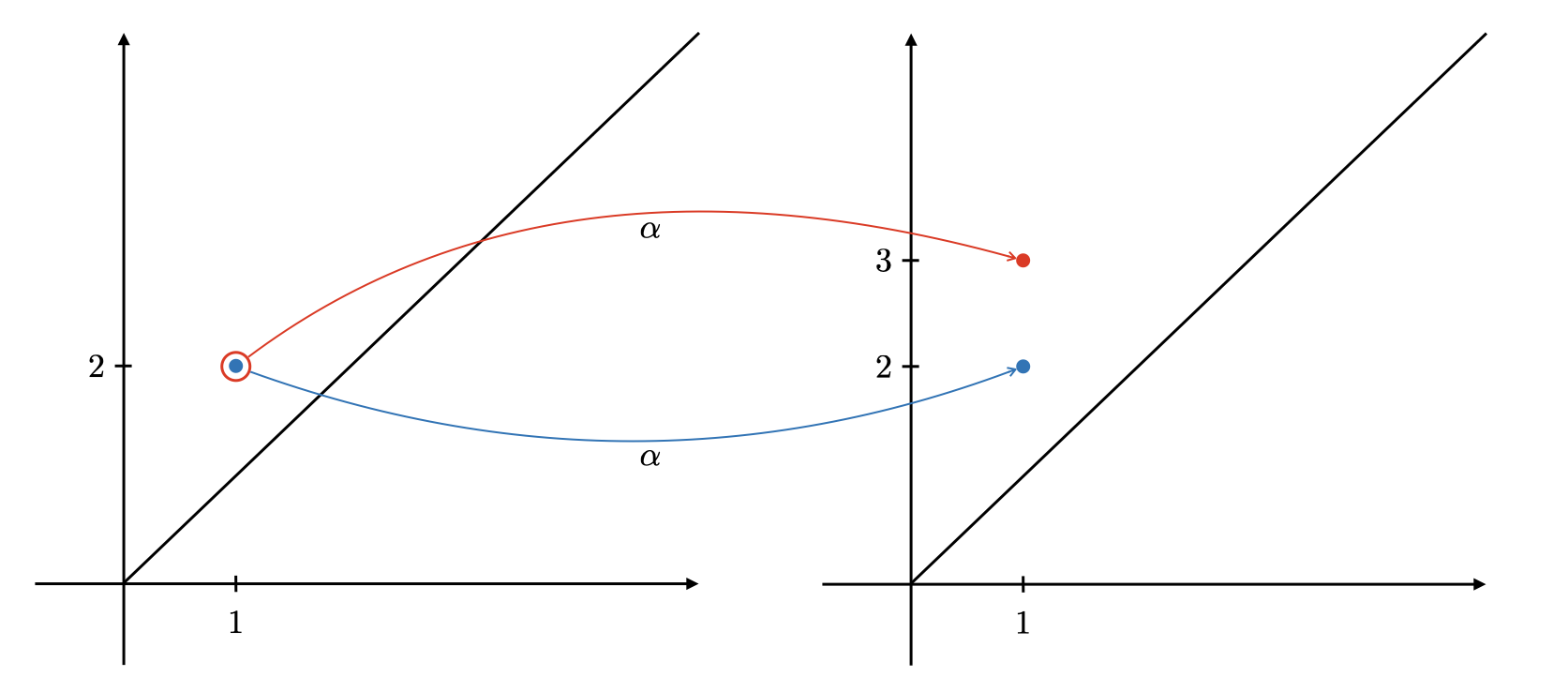}
\caption{The multiplicity of an element is represented by the number of concentric circles. For example $(1,2)$ has multiplicity 2. The arrows represent the matching $\alpha$ on the points with positive multiplicity. }
\label{fig_multiset}
\end{figure}

With a slight abuse of notation, we will identify the multiset $(\mathcal{S},f)$ with the support of $f$, $\{S_i\}_{i\in I}$, where its elements $S_i$ are taken with the multiplicity $m_i=f(S_i)$.  

Since $\mu^\p_{\eps,\eta}$ is non-negative (Corollary~\ref{propmu>=0}) and non-decreasing in $\eps$ (Proposition~\ref{muincreasing}), next definition is well-posed.

\begin{definition}[Persistence diagram]\label{defPD}
Fixed a degree $k\in \Z$, the \myemph{$k$-th persistence diagram} $\dgm_k(\p)$ of the continuous function $\p:X\to\R$ is the multiset $(\bar\Delta^*, f)$, where
$f\colon\bar\Delta^*\to\N\cup\{\infty\}$ is defined by setting, for every $P\in \bar\Delta^*$,
$$
f(P):=
\begin{cases}
\mu^\p(P):=\lim_{\eps\to 0^+} \mu^\p_{\eps,\eps}(P) & \text{ if } P\in \Delta^+; \\
\nu^\p(P):=\lim_{\eps\to 0^+} \left( \beta_k^\p{(u+\eps,\infty)}-\beta_k^\p{(u-\eps,\infty)} \right) &\text{ if } P=(u,\infty); \\
\infty &\text{if } P=\Delta.
\end{cases}
$$
Each element $(u,v)$ of the multiset $\dgm_k(\p)$ with $u<v<\infty$ and with multiplicity greater than 0 is called a \myemph{proper cornerpoint}.
Each element $(u,v)$ of $\dgm_k(\p)$ with $v=\infty$ and with multiplicity greater than 0 is called a \myemph{cornerpoint at infinity} (
or an \myemph{essential cornerpoint}, or an \myemph{improper cornerpoint}).
The point $\Delta$ is called the \myemph{trivial cornerpoint}.
\end{definition}

Please note that, in the presented mathematical framework, each element $P \in \dgm_k(\p)$ is either a point $(u, v)$ with $u<v$ and $v \in \R\cup\{\infty\}$, or $P$ is the diagonal $\Delta$.

We observe that the function $\mu^\p_{\varepsilon,\varepsilon}(P)$ in Definition~\ref{defPD} is defined for every sufficiently small $\varepsilon>0$, and hence the limit
$\lim_{\varepsilon \to 0^+} \mu^\p_{\varepsilon,\varepsilon}(P)$ 
is well defined.

We usually write $\dgm(\p)$ instead of $\dgm_k(\p)$ if a statement about a persistence diagram does not depend on the degree. 

Persistence diagrams provide an alternative and compact way to represent PBNFs. 
For example, Figure~\ref{PD_ex} is the $0$-th persistence diagram of the function in Figure~\ref{figexFunction} whose 0-th PBNF is drawn in Figure~\ref{figexPBNF}. 
There the proper cornerpoints are represented with red dots and the cornerpoint at infinity with a red square. 
In this case, all these cornerpoints have multiplicity one. 
A common interpretation of the points of a persistence diagram is given in terms of birth and death of homological features of the function under consideration.
For example, the function $\p$ in Figure~\ref{figexFunction} has one connected component born at time $1$ and another one born at time $2$, corresponding to two local minima of $\p$. 
The second connected component dies (or merges with the first one) at time $4$. 
This information is summarised in the persistence diagram with a point $(2,4)$, where the $x$-coordinate is in fact the time of birth of the connected component and the $y$-coordinate is its time of death. 
\begin{figure}
\centering
\includegraphics[width=5cm]{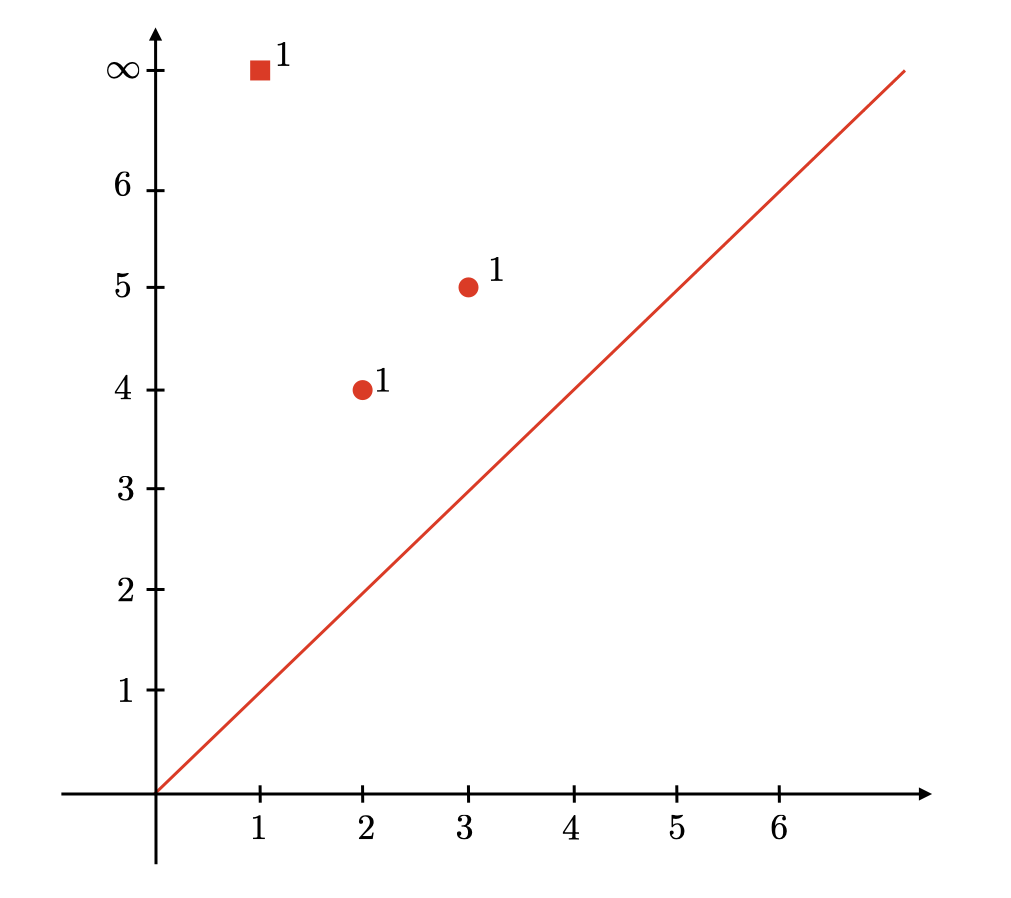}
\caption{Example of 0-th persistence diagram of $\p$ in Figure~\ref{figexFunction}. Proper cornerpoints are depicted as circles, cornerpoints at infinity as squares, the trivial cornerpoint as a red line.}
\label{PD_ex}
\end{figure}

Analogously to the PBNFs, persistence diagrams are invariant under precomposition with homeomorphisms.

\begin{proposition}\label{propinvariancePDs}
$\dgm(\p)=\dgm(\p g^{-1})$ for every homeomorphism $g\colon X\to Y$.
\end{proposition}

\begin{proof}
It immediately follows from Proposition~\ref{propinvariancebeta} and the definition of persistence diagram.
\end{proof}

\begin{exercise}\label{exercise:1DPD}
    Determine the persistence diagrams of the functions:
    \begin{itemize}
        \item $\p$ introduced in Exercise \ref{exercise:1Dcircle};
        \item $\p$ and $\p'$ introduced in Exercise \ref{exercise:eight};
        \item $\p$ introduced in Exercise \ref{exercise:1DcircleDistance}.
    \end{itemize}
\end{exercise}

\begin{exercise}\label{exercize:d_G}
Let $X$ be a compact topological space. Assume that there exists a continuous map $F:C^0(X,\R)\to C^0(X,\R)$ such that $\dgm_0(\varphi)=\dgm_0(F(\varphi))$ for every $\varphi\in C^0(X,\R)$.
Is it true that a homeomorphism $g:X\to X$ exists, such that $F(\varphi)=\varphi g^{-1}$ for every $\varphi\in C^0(X,\R)$? Compare this exercise with Proposition \ref{propinvariancePDs}.
\end{exercise}





\subsection{Cornerpoints in persistence diagrams and discontinuities of PBNFs}\label{subsec:cornerpoints}

In this section, we focus on cornerpoints of persistence diagrams disclosing their relationship with the discontinuities of the PBNFs.\\

The following two propositions claim that the finite coordinates of the proper and essential cornerpoints of a persistence diagram of a function $\p$ can only lie between the minimum and the maximum of $\p$. 

\begin{proposition}\label{propposprcorn}
Let $ P=(u, v)$ be a proper cornerpoint in $\dgm (\p)$. Then $\min \p\le u < v\le \max \p$.
\end{proposition}


\begin{proof}
If $u<\min \p$ and $\eps>0$ is small enough,
$X^\p_{u-\eps}=X^\p_{u+\eps}=\emptyset$ and hence we have 
$\beta_k^\p{(u-\eps,v)}=\dim {i^{\p*}_{u-\eps,v}}(H_k(X^\p_{u-\eps}))=0$ and
$\beta_k^\p{(u+\eps,v)}=\dim {i^{\p*}_{u+\eps,v}}(H_k(X^\p_{u+\eps}))=0$, for every $v>u+\eps$.
Therefore, $\mu^\p_{\eps, \eps}(u,v)=0$ and $(u,v)$ cannot be a cornerpoint.
If $v>\max \p$ and $\eps>0$ is small enough,
$X^\p_{v-\eps}=X^\p_{v+\eps}=X$.
It follows that the map
${i^{\p}_{v-\eps,v+\eps}}\colon X^\p_{v-\eps}=X\to X^\p_{v+\eps}=X$
is the identity.
This implies that the induced map
${i^{\p*}_{v-\eps,v+\eps}}\colon H_k(X^\p_{v-\eps})=H_k(X)\to H_k(X^\p_{v+\eps})=H_k(X)$
is the identity, and hence
\begin{align*}
            \beta_k^\p{(u,v+\eps)}=&\dim {i^{\p*}_{u,v+\eps}}(H_k(X^\p_{u}))\\ =&\dim {i^{\p*}_{v-\eps,v+\eps}}\circ {i^{\p*}_{u,v-\eps}}(H_k(X^\p_{u}))\\ =&
\dim {i^{\p*}_{u,v-\eps}}(H_k(X^\p_{u}))\\
=&\beta_k^\p{(u,v-\eps)} 
\end{align*}
for every $u<v-\eps$.
Therefore, $\mu^\p_{\eps,\eps}(u,v)=0$, and $(u,v)$ cannot be a cornerpoint.
\end{proof}

\begin{proposition}\label{propposcornatinf}
Let $P=(u,\infty)$ be a cornerpoint at infinity in $\dgm (\p)$. Then $\min \p\le u\le \max \p$.
\end{proposition}

\begin{proof}
Analogously to the proof of
Proposition~\ref{propposprcorn},
if $u<\min \p$ and $\eps>0$ is small enough,
$X^\p_{u-\eps}=X^\p_{u+\eps}=\emptyset$ and hence we have 
$\beta_k^\p{(u-\eps,\infty)}=\dim {i^{\p*}_{u-\eps,\max \p}}(H_k(X^\p_{u-\eps}))=0$ and
$\beta_k^\p{(u+\eps,\infty)}=\dim {i^{\p*}_{u+\eps,\max\p}}(H_k(X^\p_{u+\eps}))=0$.
Therefore, $\nu^\p(u,\infty)=0$ and $(u,\infty)$ cannot be a cornerpoint at infinity.
If $u>\max \p$ and $\eps>0$ is small enough,
$X^\p_{u-\eps}=X^\p_{u+\eps}=X$.
It follows that the map
${i^{\p}_{u-\eps,u+\eps}}:X^\p_{u-\eps}=X\to X^\p_{u+\eps}=X$
is the identity.
This implies that the induced map
${i^{\p*}_{u-\eps,u+\eps}}:H_k(X^\p_{u-\eps})=H_k(X)\to H_k(X^\p_{u+\eps})=H_k(X)$
is the identity, and hence
\begin{align*}
            \beta_k^\p{(u-\eps,\infty)}
=&\dim {i^{\p*}_{u-\eps,\max\p}}(H_k(X^\p_{u-\eps}))\\
=&\dim {i^{\p*}_{u+\eps,\max\p}}\circ {i^{\p*}_{u-\eps,u+\eps}}(H_k(X^\p_{u-\eps}))\\
=&\dim {i^{\p*}_{u+\eps,\max\p}}(H_k(X^\p_{u-\eps}))\\
=&\dim {i^{\p*}_{u+\eps,\max\p}}(H_k(X^\p_{u+\eps}))\\
=&\beta_k^\p{(u+\eps,\infty)}. 
\end{align*}
Therefore, $\nu^\p(u,\infty)=0$ and $(u,\infty)$ cannot be a cornerpoint at infinity.
\end{proof}


We now show that the discontinuities of PBNFs propagate downward and to the left from proper cornerpoints, and only downward from essential cornerpoints.

\begin{proposition}\label{proppropadiscfromcp}
If $\bar P=(\bar u,\bar v)$ is a cornerpoint in $\dgm (\p)$, the following statements hold:
\begin{itemize}
  \item[1.] if $\bar P$ is a proper cornerpoint and $\bar u< v< \bar v$, then $\bar u$ is a discontinuity point for $\beta_k^\p{(\cdot,v)}$;
    \item[2.] if $\bar P$ is a proper cornerpoint and $\bar u< u< \bar v$, then $\bar v$ is a discontinuity point for $\beta_k^\p{(u,\cdot)}$;
  \item[3.] if $\bar P$ is a cornerpoint at infinity and $\bar u< v
  $,
  then $\bar u$ is a discontinuity point for $\beta_k^\p{(\cdot,v)}$.
\end{itemize}
\end{proposition}



\begin{proof}
\begin{itemize}

\item[1.] By Proposition~\ref{prop_premult},
since $\bar P=(\bar u,\bar v)$ is a proper cornerpoint in $\dgm (\p)$, we have that
\begin{equation}\label{ruhcnoaivs}
\beta_k^\p{(\bar u+\eps,\bar v-\eps)}
            -\beta_k^\p{(\bar u-\eps,\bar v-\eps)}-\beta_k^\p{(\bar u+\eps,\bar v+\eps)}
            +\beta_k^\p{(\bar u-\eps,\bar v+\eps)}> 0
\end{equation}
for any small enough $\eps>0$.

Since $\beta_k^\p$ is non-decreasing in its first variable (Proposition~\ref{propmonotonicity}),
we have that
\begin{equation}\label{ghufkjovn}  \beta_k^\p{(\bar u+\eps,\bar v+\eps)}
            -\beta_k^\p{(\bar u-\eps,\bar v+\eps)}\ge 0
\end{equation}
for any small enough $\eps>0$.

From \ref{ruhcnoaivs} and \ref{ghufkjovn} it follows that
$\beta_k^\p{(\bar u+\eps,\bar v-\eps)}
            -\beta_k^\p{(\bar u-\eps,\bar v-\eps)}> 0$ for any small enough $\eps>0$.
            Thus, $\bar u$ is a discontinuity point of $\beta_k^\p{(\cdot,\bar v-\eps)}$ for any small enough $\eps>0$.
            Then Statement 1 in this proposition follows from Statement 1 in Proposition~\ref{proppropadisc}.

\item[2.]  By Proposition~\ref{prop_premult} and
since $\bar P=(\bar u,\bar v)$ is a proper cornerpoint in $\dgm (\p)$, we have that
\begin{equation}\label{jkjejkjke}
             \beta_k^\p{(\bar u+\eps,\bar v-\eps)}
            -\beta_k^\p{(\bar u-\eps,\bar v-\eps)}-\beta_k^\p{(\bar u+\eps,\bar v+\eps)}
            +\beta_k^\p{(\bar u-\eps,\bar v+\eps)}> 0
\end{equation}
for any small enough $\eps>0$.

Since $\beta_k^\p$ is non-increasing in its second variable (Proposition~\ref{propmonotonicity}),
we have that
\begin{equation}\label{eopivohxkcn}
\beta_k^\p{(\bar u-\eps,\bar v+\eps)}
            -\beta_k^\p{(\bar u-\eps,\bar v-\eps)}\le 0
\end{equation}            
            for any small enough $\eps>0$.

From \ref{jkjejkjke} and \ref{eopivohxkcn} it follows that
$\beta_k^\p{(\bar u+\eps,\bar v-\eps)}
            -\beta_k^\p{(\bar u+\eps,\bar v+\eps)}> 0$ for any small enough $\eps>0$.
            Thus, $\bar v$ is a discontinuity point of $\beta_k^\p{(\bar u+\eps,\cdot)}$ for any small enough $\eps>0$. 
            Then Statement 2 in this proposition follows from Statement 2 in Proposition~\ref{proppropadisc}.

\item[3.] Since $\bar P=(\bar u,\infty)$ is an essential cornerpoint in $\dgm (\p)$, we have that
$\beta_k^\p{(\bar u+\eps,\max\p)}
            -\beta_k^\p{(\bar u-\eps,\max\p)}>0$
for any $\eps>0$.
Because of Remark~\ref{remmaxphi}, if $v\ge \max\p$ then
            $\beta_k^\p{(\bar u+\eps,v)}
            -\beta_k^\p{(\bar u-\eps,v})>0$
            for any $\eps>0$.
             It follows that for any $v\ge\max\p$,
$\bar u$ is a discontinuity point of $\beta_k^\p{(\cdot,v)}$.
Then, Statement 3 in this proposition follows from Statement 1 in Proposition~\ref{proppropadisc}.
\end{itemize}
\end{proof}

Note that if $P=(\bar u, \bar v)$ is a proper cornerpoint, then we cannot say that $\bar v$, resp. $\bar u$, is a discontinuity point for $\beta^{\varphi}_k(\bar u, \cdot)$, resp. $\beta^{\varphi}_k(\cdot , \bar v)$. 
As an example, we can consider the cornerpoint $(2,4) \in \dgm_0(\varphi)$ in Figure~\ref{PD_ex} when $\varphi$ is the function in Figure~\ref{figexFunction}.
This cornerpoint is a discontinuity point in the variable $v$, but not in the variable $u$.





\begin{corollary}\label{homological_critical}
If $w$ is a finite coordinate of a cornerpoint $P$ in 
$\dgm_k(\p)\setminus \{\Delta\}$, then, for every sufficiently small $\eps>0$,  
the linear map $i^{\p*}_{w-\eps, w+\eps}\colon H_k(X^\p_{w-\eps})\to H_k(X^\p_{w+\eps})$ is not an isomorphism.
\end{corollary}

\begin{proof} 
If $P=(w,v)$ (possibly with $v=\infty$), by Proposition~\ref{proppropadiscfromcp},  $$\dim (\Imm i^{\p*}_{w-\eps,v-\varepsilon})=\beta^\p_k(w-\eps, v-\varepsilon)\neq \beta^\p_k(w+\eps, v-\varepsilon)=\dim (\Imm i^{\p*}_{w+\eps,v-\varepsilon})$$ for every $\eps>0$ with $w+\varepsilon< v-\varepsilon$, i.e., $\eps<\frac{v-w}{2}$. 
Observe that $i^{\p*}_{w-\eps,v-\varepsilon}= i^{\p*}_{w+\eps,v-\varepsilon}i^{\p*}_{w-\eps,w+\eps}$, hence $i^{\p*}_{w-\eps, w+\eps}$ cannot be an isomorphism.  
\end{proof}

The following result completes the statement of Proposition~\ref{proppropadiscfromcp}, showing that each discontinuity of a PBNF originates at a cornerpoint.

\begin{proposition}\label{propalldisccomefromcp}
The following statements hold:
\begin{itemize}
  \item[1.] If $\bar u<\bar v$ is a discontinuity point for $\beta_k^\p{(\cdot,\bar v)}$, then there exists at least one cornerpoint $(\bar u,v')$ (proper or at infinity) with $v'\ge \bar v$.
  \item[2.] If $\bar v>\bar u$ is a discontinuity point for $\beta_k^\p{(\bar u,\cdot)}$, then there exists at least one proper cornerpoint $(u',\bar v)$ with $u'\le \bar u$.
\end{itemize}
\end{proposition}


\begin{proof}
\begin{enumerate}
\item Assume that $\bar u<\bar v$ is a discontinuity point for $\beta_k^\p{(\cdot,\bar v)}$.
Of course, $\bar u<\max\p$. If $\lim_{\eps\to 0^+} \beta_k^\p{(\bar u+\eps,\max\p)}-\beta_k^\p{(\bar u-\eps,\max\p)}>0$, then $(\bar u,\infty)$ is a cornerpoint at infinity and hence the statement  holds. Therefore, we can assume that an $\bar \eps>0$ exists, such that
$\bar u+\bar\varepsilon<\max\p$ and 
$\beta_k^\p{(\bar u+\eps,\max\p)}-\beta_k^\p{(\bar u-\eps,\max\p)}=0$
for any positive $\eps\le\bar\eps$.
Let us now consider an open cover $\mathcal{V}=\{V_P\}_{P\in c}$ of the closed vertical segment
$c$ connecting the points $(\bar u,\bar v)$
and $(\bar u,\max\p)$, where $V_P$
is an open square centred at $P\in c$, with the sides parallel to the axes. 
If $c$ does not contain cornerpoints, we can assume that $\mu^\p_{\eps(P),\eps(P)}(P)=0$ for every $P\in c$, where $\eps(P)$ is half the side length of the open square 
$V_P$, depending on $P$.
We can also assume that $\varepsilon(P)\le \bar\varepsilon$ for any point $P\in c$.
Since $c$ is compact, $\mathcal{V}$ admits a finite subcover
$\{V_{P_1},\ldots,V_{P_n}\}$ (see the example in Figure~\ref{covering_V_i}).

\begin{figure}
\centering
\includegraphics[width=8cm]{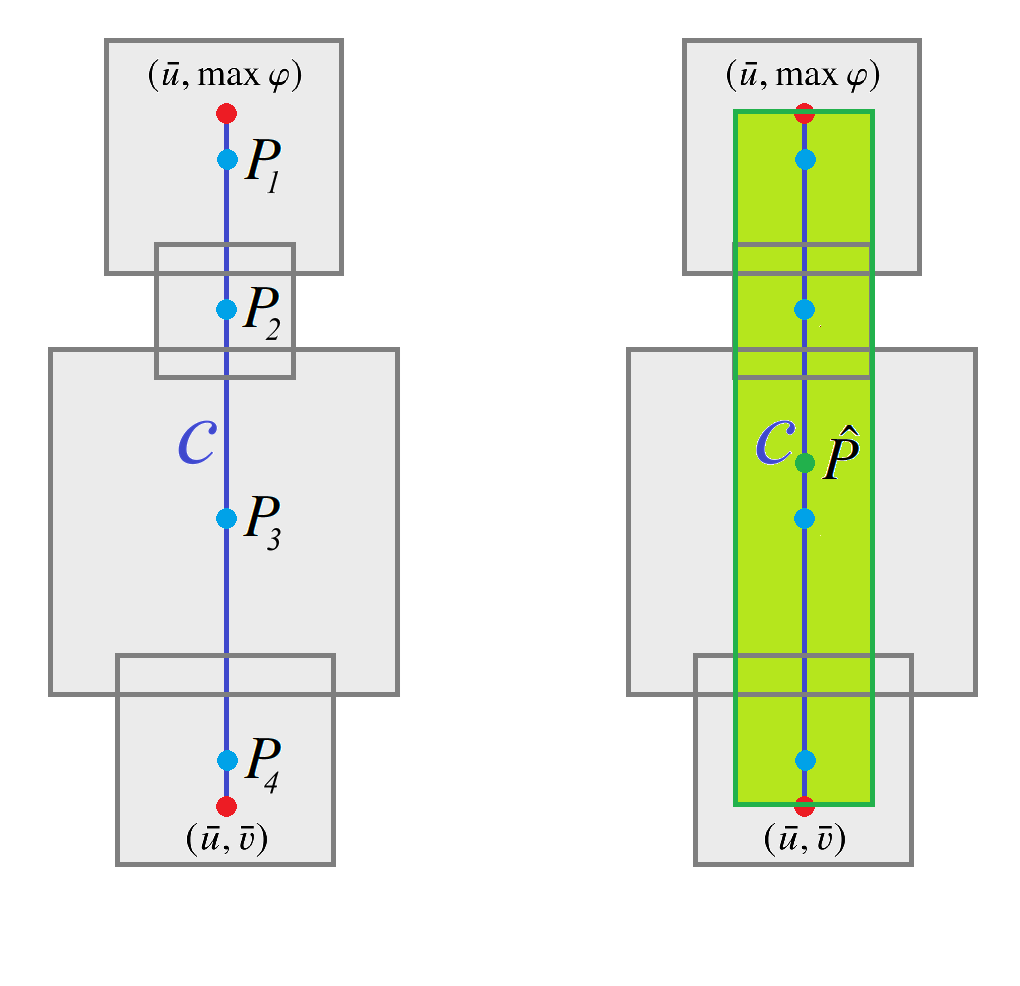}
\caption{On the left, the finite open cover used in the proof of Proposition~\ref{propalldisccomefromcp} is displayed. The points \((\bar u,\bar v)\) and \((\bar u,\max \p)\) are shown in red.
On the right, the rectangle with sides parallel to the axes, barycentre \(\hat P\), width \(2\varepsilon'\), and height \(2\eta'\) is shown in green. The sum with alternating signs of the values taken by \(\beta^\p_k\) at its vertices equals $\sum_{i=1}^n \mu^\p_{\eps(P_i),\eps(P_i)}(P_i)$ and hence vanishes.
}
\label{covering_V_i}
\end{figure}

Let us set $\eps':=\min \eps(P_i)$ and $\eta'$ equal to half the length of $c$.
From Corollary~\ref{propmu>=0} and Proposition~\ref{muincreasing}, it is possible to show (see Exercise~\ref{ex_boxes}) that
$\mu^\p_{\eps',\eta'}(\hat P)=\sum_{i=1}^n \mu^\p_{\eps(P_i),\eps(P_i)}(P_i)=0$, where $\hat P$ is the middle point of $c$.
This implies that
\begin{align*}
            & \beta_k^\p{(\bar u+\eps',\bar v)}
            -\beta_k^\p{(\bar u-\eps',\bar v)}\\
            &=\beta_k^\p{(\bar u+\eps',\max\p)}
            -\beta_k^\p{(\bar u-\eps',\max\p)}=0,
\end{align*}
against the assumption that $\bar u<\bar v$ is a discontinuity point for $\beta_k^\p{(\cdot,\bar v)}$.
Therefore, $c$ must contain at least one cornerpoint.

\item Assume that $\bar v>\bar u$ is a discontinuity point for $\beta_k^\p{(\bar u,\cdot)}$.
Of course, $\bar v\ge\min\p$, and 
we can easily check that  
$$\beta_k^\p{(\min\p-1,\bar v+\eps)}=\beta_k^\p{(\min\p-1,\bar v-\eps)}=0$$ 
for any positive $\eps$ such that $\min\p -1<\bar v-\varepsilon$.
Let us now consider an open cover $\mathcal{V}=\{V_P\}_{P\in c}$ of the closed horizontal segment
$c$ connecting the points $(\min\p -1,\bar v)$ and $(\bar u,\bar v)$, where $V_P$
is an open square centred at $P\in c$, with the sides parallel to the axes. 
If $c$ does not contain cornerpoints, we can assume that $\mu^\p_{\eps(P),\eps(P)}(P)=0$ for every $P\in c$, where $\eps(P)$ is half the side length of the open square 
$V_P$, depending on $P$.
We can also assume that $\min\p-1<\bar v-\varepsilon(P)$ for any point $P\in c$.
Since $c$ is compact, $\mathcal{V}$ admits a finite subcover
$\{V_{P_1},\ldots,V_{P_n}\}$.

Let us set $\eps':=\min \eps(P_i)$ and $\eta'$ equal to half the length of $c$.
From Corollary~\ref{propmu>=0} and Proposition~\ref{muincreasing}, it easily follows that
$$\mu^\p_{\eta',\eps'}(\hat P)=\sum_{i=1}^n \mu^\p_{\eps(P_i),\eps(P_i)}(P_i)=0,$$ where $\hat P$ is the middle point of $c$.
This implies that
\begin{align*}
            &0=\beta_k^\p{(\min\p -1, \bar v-\eps')}
            -\beta_k^\p{(\min\p -1, \bar v+\eps')}\\
            & =\beta_k^\p{(\bar u, \bar v-\eps')} -\beta_k^\p{(\bar u, \bar v+\eps')},
\end{align*}
against the assumption that $\bar v>\bar u$ is a discontinuity point for $\beta_k^\p{(\bar u,\cdot)}$.
Therefore, $c$ must contain at least one cornerpoint.
\end{enumerate}
\end{proof}

\begin{remark}\label{aaohfceolrvjawcol}
    By Proposition~\ref{proppropadiscfromcp}.3 and Proposition~\ref{propalldisccomefromcp}.1, the cornerpoints at infinity of $\dgm_k(\p)$ are all and only the discontinuity points of $\beta^\p_k$ in the first variable. 
\end{remark}

\begin{exercise}\label{ex_boxes}
Consider the $(\eps', \eta')$-box centred at $P'$ and the $(\eps, \eta)$-box centred at $P$ such that the latter is included in the former. 
Show that $\mu_{\eps', \eta'}(P')\ge \mu_{\eps, \eta}(P)$.
\end{exercise}

\subsection{Cardinality of persistence diagrams}\label{subsec:PD_cardinality}

Next results show that, under the hypotheses of this book, the points with positive multiplicity of a persistence diagram form a compact subset of $\bar{\Delta}^\ast$.
First, we show that there are no proper cornerpoints that can be accumulation points.

\begin{proposition}\label{proplocfinite}
For each $\eps>0$, at most a finite number of proper cornerpoints of $\dgm(\p)$ belong to the set $\{(u,v)\in\R^2: u+\eps \le  v\}$.
\end{proposition}

\begin{proof}
By contradiction, let us assume that
the set $\{(u,v)\in\R^2\mid u+\eps \le v\}$ contains an infinite set $\mathcal{S}$ of proper cornerpoints.
Proposition~\ref{propposprcorn}
guarantees that $\mathcal{S}\subseteq\{(u,v)\in\R^2\mid \min \p\le u<v\le \max \p\}$.
Therefore,
$\mathcal{S}\subseteq K:=\{(u,v)\in\R^2\mid \min \p\le u+\eps\le v\le \max \p\}$. Since $K$ is compact with respect to $d$ (check it!), a point $P\in K$ exists, such that $P$ is an accumulation point for $\mathcal{S}$.
Then Proposition~\ref{proppropadiscfromcp} and
Proposition~\ref{propmonotonicity} imply that
the function $\beta_k^{\p}$ takes the value $\infty$ at some point in $\Delta^+$, against Assumption \ref{ass_fingen}.
\end{proof}

The next corollary immediately follows from Proposition~\ref{proplocfinite}.

\begin{corollary} \label{corcountable} 
$\text{Dgm}(\p)\setminus\{\Delta\}$
forms a countable set.
\end{corollary}

\begin{ex}
We can easily find continuous functions $\p:X\to \R$ such that $\text{Dgm}(\p)\setminus\{\Delta\}$ is infinite. For example, consider the height function $\p$ defined on the space $X$ represented in Figure~\ref{fig_infdgm}.
\end{ex}

\begin{figure}
\begin{center}
\includegraphics[width=10cm]{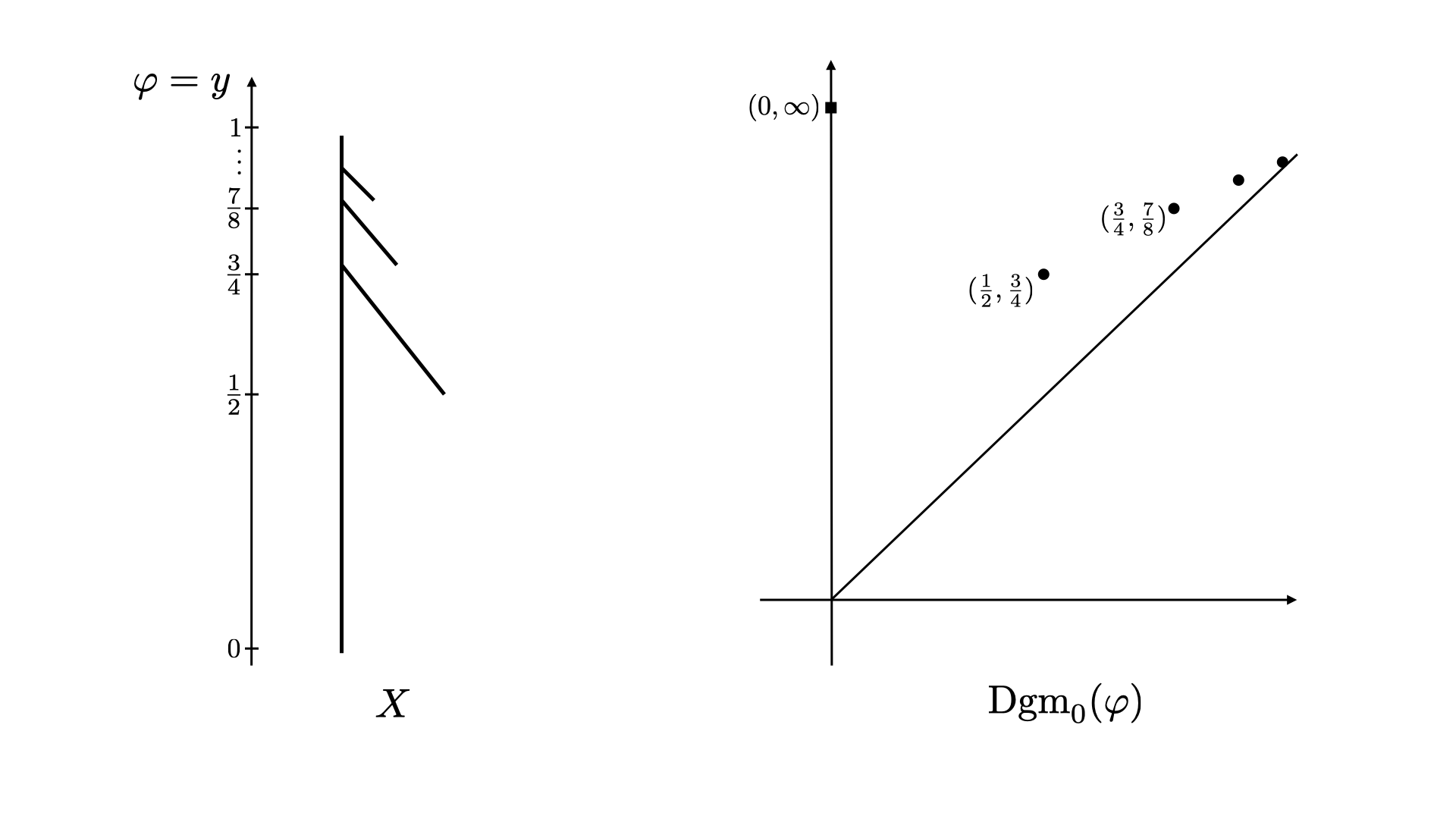}
\caption{A space $X\subseteq\R^2$ with $|\dgm_0(\p)\setminus\{\Delta\}|=\infty$, where $\p$ is the height function. 
The proper cornerpoints of $\dgm_0(\p)$ are the points
$\left(\sum_{i=1}^{k}2^{-i},\sum_{i=1}^{k+1}2^{-i}\right)$ for any natural number $k\ge 1$}.
\label{fig_infdgm}
\end{center}
\end{figure}

\begin{proposition}\label{propfiniteatinf}
$\dgm(\p)$ contains at most a finite number of cornerpoints at infinity.
\end{proposition}

\begin{proof}
By contradiction, let us assume that
$\dgm(\p)$ contains an infinite set $S_\infty$ of cornerpoints at infinity.
Proposition~\ref{propposcornatinf} guarantees that if $(u,\infty)\in S_\infty$, then  $\min \p\le u\le \max \p$.
Therefore, the set $C$ of the abscissas of points in $S_\infty$ admits an accumulation point in the interval $[\min\p,\max\p]$.
Then Proposition~\ref{proppropadiscfromcp} and
Proposition~\ref{propmonotonicity} imply that
the function $\beta_k^{\p}$ takes the value $\infty$ at some point in $\Delta^+$, against Assumption~\ref{ass_fingen}.
\end{proof}


\begin{proposition}\label{propdgmcompact}
The union of the supports of $\mu^\p$ and $\nu^\p$ is a compact set with respect to the topology induced by the metric $d$.
\end{proposition}

\begin{exercise}\label{defdgmcompact}
Prove Proposition~\ref{propdgmcompact}.

[\emph{Hint.} Use Propositions~\ref{proplocfinite} and~\ref{propfiniteatinf} to show that any sequence in a compact set admits a subsequence converging either to a proper cornerpoint, to a cornerpoint at infinity, or to the trivial cornerpoint~$\Delta$.]
\end{exercise}


\section{Representation theorem}\label{sec:rep_thm}

While persistence diagrams constitute the most intuitive and compact representation of persistence features, they do not support certain convenient operations, such as the definition of unique means \cite{TuMiMuHa14}.
PBNFs, on the other hand, provide more suitable representations for these purposes and are better suited for vectorisation in machine learning.
In this section, we prove that persistence diagrams uniquely determine PBNFs and vice versa (Theorem \ref{k-triangle}).
This is done explicitly, showing a formula to go from one representation to the other. 
This correspondence between persistence diagrams and PBNFs holds under the assumption that the function $\beta_k^\p$ is right-continuous in both its variables, to avoid pathological cases.

Example \ref{exPC} shows a case in which right-continuity of PBNFs is not satisfied.

\begin{ex}\label{exPC}
Let $X$ be a closed rectangle of $\R^2$ containing a Warsaw circle (Figure~\ref{figexPC}).
Let also
$\p: X \to \R$ be the Euclidean distance
from the Warsaw circle.
It is easy to see that the dimension of the
persistent homology group $PH_1^\p(u,v)$ is equal to $1$ for $v > u > 0$ and $v$ sufficiently small, whereas it is equal to $0$ when $v>u=0$, showing
that in this case the function $\beta_1^\p$ is not right-continuous in the first variable.
\end{ex}

\begin{figure}[htb!]
\begin{center}
\includegraphics[width=6cm]{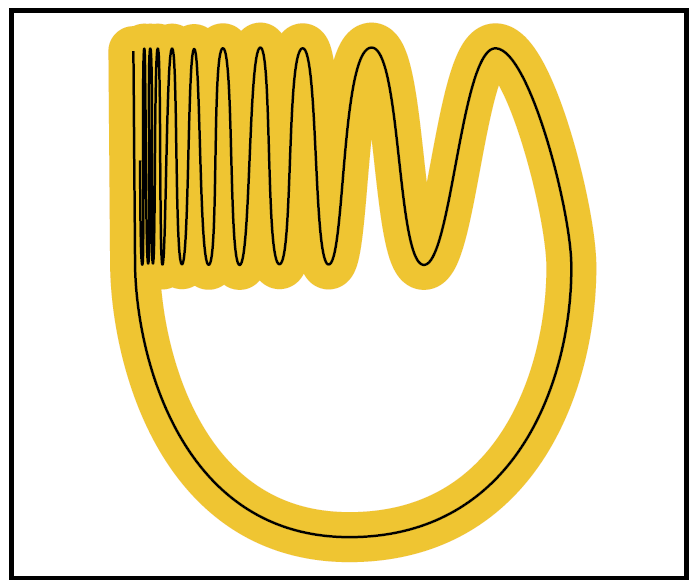}
\caption{
A sublevel set $X^\p_u$, for a small $u > 0$, as considered in Example~\ref{exPC}, corresponds to a small dilation (shaded) of our Warsaw circle. 
}
\label{figexPC}
\end{center}
\end{figure}




From now on, the following will hold.
\begin{assumption}\label{ass_right-cont}
The function $\beta_k^\p$ is right-continuous in both its variables:
\[
\lim_{\eps\to 0^+}\beta_k^\p(u+\eps, v)=\beta_k^\p(u, v)=\lim_{\eps\to 0^+}\beta_k^\p(u, v+\eps).
\]
\end{assumption}

\begin{remark}
Assumption \ref{ass_right-cont} could be avoided by replacing singular homology with Čech homology (cf., e.g., \cite{Sp95}). However, the treatment of this particular homology goes beyond the scope of this book and will not be presented here. 
\end{remark}

The next result is the first part of the Representation Theorem stated below, in the particular case of points in $\Delta^*\setminus \Delta^+$.
We recall that the multiplicity functions $\mu^\p$, $\nu^\p$ are defined in Definition \ref{defPD}.

\begin{lemma}\label{k-triangle_lemma}
If $(\bar u,\infty)\in\Delta^*$, then
$
\beta_k^{\p}(\bar u,\infty)=\sum_{u\leq\bar u}\nu^\p(u,\infty).
$
\end{lemma}

\begin{remark}\label{rem_sums}
Because of Proposition~\ref{propfiniteatinf}, the sum $\sum_{u \leq \bar u} \nu^\varphi(u,\infty)$ in the statement of Lemma \ref{k-triangle_lemma} is in fact taken over a finite number of points. 
Because of Proposition~\ref{proplocfinite}, an analogous observation will hold, in what follows, for $\sum_{\mycom{(u,v)\in\Delta^+}{u\leq\bar u,\,v>\bar v}}\mu^\p(u,v)$ 
in the statement of Theorem \ref{k-triangle}.
\end{remark}

\begin{proof}
We can assume that $\dgm(\varphi)$ contains at least one cornerpoint at infinity whose abscissa is less than or equal to $\bar u$; otherwise, the statement of the lemma is trivial, since Remark~\ref{aaohfceolrvjawcol} 
implies that 
$\beta_k^{\p}(u,\infty)$ is constant in the variable $u$ for $u\le\bar u$, and hence 
$\beta_k^{\p}(\bar u,\infty)=\beta_k^{\p}(\min\p -1,\infty)=0$, while by definition every sum over an empty set vanishes.
Because of Proposition 
\ref{propfiniteatinf},
we can find $u_0<\ldots<u_m$
such that:
\begin{itemize}
\item[(a)] \label{A1} \;$u_0<\min\p$;
by Proposition~\ref{propposcornatinf}, the abscissa of each cornerpoint at infinity of $\dgm(\p)$ is greater than $u_0$;
\item[(b)] \label{A2}\; $u_m> \bar u$, and there is no essential cornerpoint $(u,\infty)$ 
having its abscissa in $]\bar u,u_m]$; 
by Remark~\ref{aaohfceolrvjawcol}, 
no value in $]\bar u,u_m]$ is a discontinuity point of the function $\beta^\p_k(\cdot,\infty)$;
\item[(c)]\; there is no essential cornerpoint 
having its abscissa in $\{u_0,\ldots,u_m\}$; 
by Remark~\ref{aaohfceolrvjawcol}, no value $u_i$ is a discontinuity point of the function $\beta^\p_k(\cdot,\infty)$;
\item[(d)]\; each open interval $]u_{i-1},u_{i}[$ with $1\le i\le m$, contains exactly one abscissa of a cornerpoint $P_i$ at infinity, possibly with multiplicity greater than $1$.
If we set
$\beta_i:=\beta_k^{\p}(u_i,\infty)$
for any $i$ with $0\le i\le m$, we can write
$\nu^\p(p_i)=\beta_{i}-\beta_{i-1}$, where $\beta_{-1}$ is defined to be $0$.
\end{itemize}

By (a), $\beta_{0}=0$.
Therefore,
\begin{align*}
          \sum_{u\leq\bar u}\nu^\p(u,\infty)
          &= \sum_{1\le i\le m}\nu^\p(P_i)\\
          &= \sum_{1\le i\le m}\beta_{i}
          -\beta_{i-1}\\
          &= \sum_{1\le i\le m}\beta_{i}
          -\sum_{1\le i\le m}\beta_{i-1}\\
          &= \sum_{1\le i\le m}\beta_{i}
          -\sum_{0\le i\le m-1}\beta_{i}\\
          &= \beta_{m}-\beta_{0}\\
          &=\beta_m=\beta_k^{\p}(u_m,\infty)=\beta_k^{\p}(\bar u,\infty),
\end{align*}
where the last equality is given by (b) and Assumption~\ref{ass_right-cont}.
\end{proof}

The key role of persistence diagrams is shown in the following Representation Theorem (see \cite{CeDFFeal13},\cite{CSEdHa07}) claiming that persistence diagrams uniquely determine PBNFs (the converse also holds by definition of persistence diagram).

\begin{theorem}\label{k-triangle}
If $(\bar u,\bar v)\in\Delta^*$, then
\begin{equation}\label{EqRT}
\beta_k^{\p}(\bar u,\bar v)=\sum_{\mycom{(u,v)\in\Delta^+}{u\leq\bar u,\,v>\bar v}}\mu^\p(u,v)
+\sum_{u\leq\bar u}\nu^\p(u,\infty).
\end{equation}
\end{theorem}


\begin{figure}
\begin{center}
\includegraphics[width=14cm]{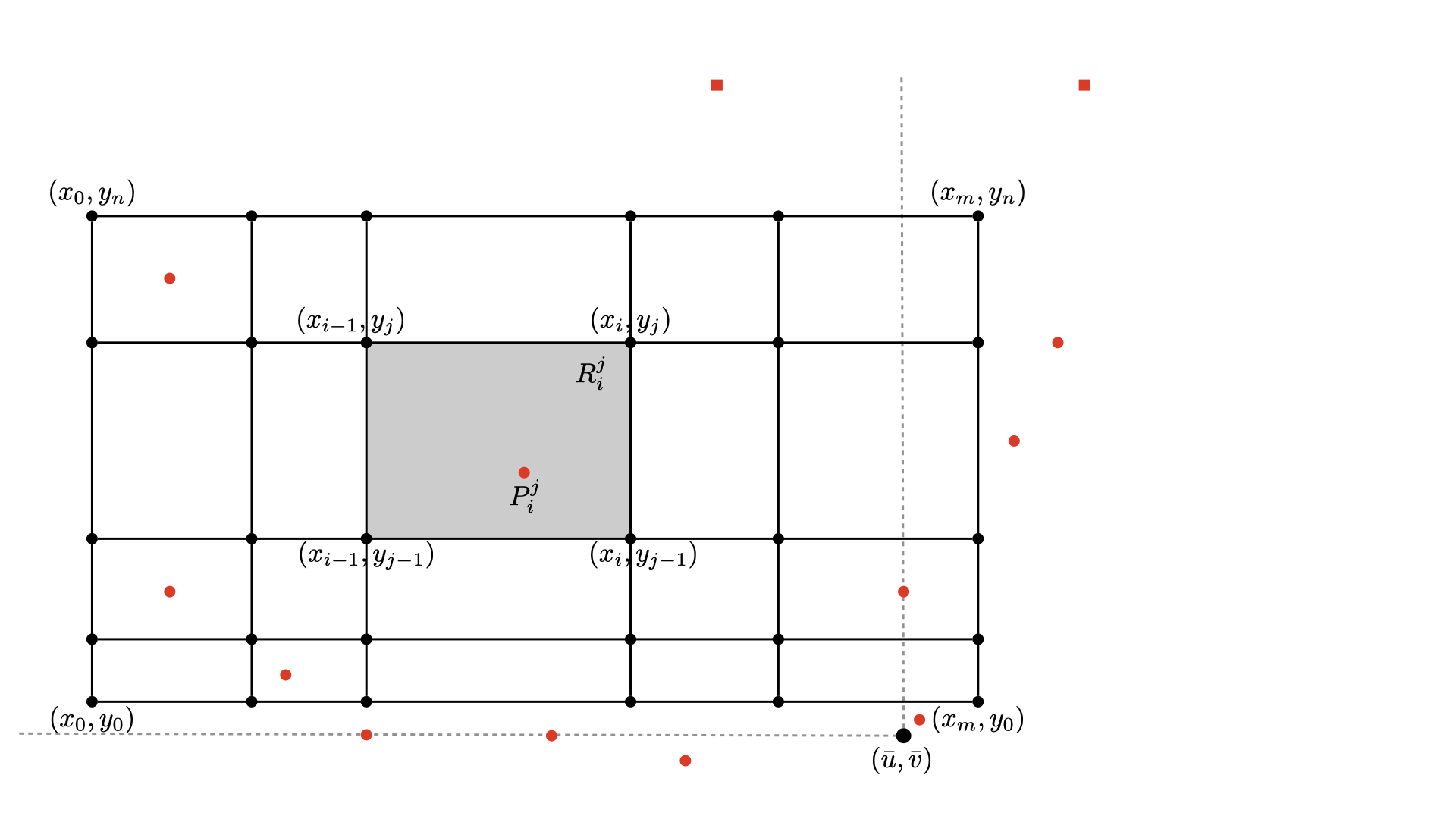}
\caption{The grid used in the proof of Theorem~\ref{k-triangle}. The red points are cornerpoints of $\dgm_k(\p)$.
Proper cornerpoints are depicted as circles, cornerpoints at infinity as squares.}
\label{fig_grid_proof_RT}
\end{center}
\end{figure}


\begin{proof}
The case $(\bar u,\bar v)=(\bar u, \infty)$ has been proven in Lemma \ref{k-triangle_lemma}. So, let us focus on the case $(\bar u,\bar v)\in\Delta^+$.
Proposition~\ref{proplocfinite} and Proposition~\ref{propfiniteatinf} ensure that there exists a finite number of cornerpoints $(u,v)$ such that $u\le\bar u$ and $v>\bar v$.
Such a number can be assumed to be at least one, otherwise the statement of the theorem is trivial because of Proposition~\ref{propalldisccomefromcp} (implying $\beta_k^{\p}(\bar u,\bar v)=\beta_k^{\p}(\bar u,\infty)$) and Lemma~\ref{k-triangle_lemma}.
In this case, both sides of Equation~\ref{EqRT} are equal, since both are zero.

Let $\{u_1, \dots, u_m\}$ be the set of the abscissas of the cornerpoints (proper or essential) $(u,v)$ such that $u\le\bar u$ and $v>\bar v$. Similarly, let $\{v_1, \dots, v_n\}$ be the set of the ordinates of the proper cornerpoints $(u,v)$ such that $u\le\bar u$ and $v>\bar v$.
All coordinates are sorted in increasing order.
Let us choose $x_0, x_1, \dots, x_m$ and $y_0, y_1, \dots, y_n$ in $\R$ such that:
\begin{itemize}
    \item $x_{i-1}<u_i<x_i$ for $i=1, \dots, m$;
    \item $y_{j-1}<v_j<y_j$ for $j=1, \dots, n$;
    \item $\bar u < x_m < \bar v$ and there is no (proper or essential) cornerpoint having its abscissa in $]\bar u,x_m]$;
    \item $\bar v < y_0$.
\end{itemize}
Please refer to Figure~\ref{fig_grid_proof_RT} for a visual representation of the just introduced construction.
For each $1\le i\le m$ and $1\le j\le n$, let us denote by $P_i^j$ the point of coordinates $(u_i,v_j)$ and by $R_i^j$ the open rectangle of vertices $(x_{i},y_{j-1})$, $(x_{i-1},y_{j-1})$, $(x_{i},y_{j})$, and $(x_{i-1},y_{j})$ containing $P_i^j$. Moreover, 
for each $0\le i\le m$ and $0\le j\le n$,
we set $\beta_i^j:=\beta_k^{\p}(x_i,y_j)$.

Proposition~\ref{propalldisccomefromcp} and the above conditions ensure that, for each open rectangle $R_i^j$, there exists no discontinuity point $(u,v)$ for $\beta^\p_k$ such that $u_i<u\le x_i$ and $v_j<v\le y_j$.
So, if we consider $\eps$ sufficiently small such that the point $(u_i+\eps, v_j+\eps)$ is contained in $R_i^j$, then $\beta_k^{\p}(u_i+\eps, v_j+\eps)=\beta_i^j$.
Analogously, we can conclude, for $\eps$ sufficiently small, that:
\begin{itemize}
    \item $\beta_k^{\p}(u_i+\eps, v_j-\eps)=\beta_i^{j-1}$,
    \item $\beta_k^{\p}(u_i-\eps, v_j-\eps)=\beta_{i-1}^{j-1}$,
    \item $\beta_k^{\p}(u_i-\eps, v_j+\eps)=\beta_{i-1}^{j}$.
\end{itemize}
So, we have that:
\begin{align*}\label{RT1}
\mu^\p(P_i^j)= & \lim_{\eps\to 0^+} \mu^\p_{\eps,\eps}(P_i^j)\\
= & \lim_{\eps\to 0^+} \beta_k^{\p}(u_i+\eps, v_j-\eps) - \beta_k^{\p}(u_i-\eps, v_j-\eps)\\
& + \beta_k^{\p}(u_i-\eps, v_j+\eps) - \beta_k^{\p}(u_i+\eps, v_j+\eps)  \\
= & \beta_{i}^{j-1}-\beta_{i-1}^{j-1}+\beta_{i-1}^{j}-\beta_{i}^{j}.
\end{align*}
The above construction also guarantees that:
\begin{itemize}
    \item\label{RTa} $\beta_{0}^{j}=0$, for any index $j=0, \dots n$;
    \item\label{RTb} $\beta_{m}^{n}=\beta_k^\p(\bar u,\infty)=\sum_{u\leq\bar u}\nu^\p(u,\infty)$ by Assumption~\ref{ass_right-cont} (specifically, by the right-continuity of $\beta_k^\p$ in the first variable) and Lemma~\ref{k-triangle_lemma};
    \item\label{RTc} $\beta_{m}^{0}=\beta_k^\p(\bar u,y_0)=\beta_k^\p(\bar u,\bar v)$ by Assumption~\ref{ass_right-cont} (by the right-continuity of $\beta_k^\p$ in the first variable for the first equality and in the second variable for the second equality). 
\end{itemize}

It follows that:
\begin{align*}
          \sum_{\mycom{(u,v)\in\Delta^+}{u\le \bar u,\,v>\bar v}}\mu^\p(u,v)
          &=\sum_{\mycom{(u,v)\in\Delta^+}{u< x_m,\,v>y_0}}\mu^\p(u,v)\\
          &= \sum_{\mycom{1\le i\le m}{1\le j\le n}}\mu^\p(P_i^j)\\
          &= \sum_{\mycom{1\le i\le m}{1\le j\le n}}\beta_{i}^{j-1}
          -\sum_{\mycom{1\le i\le m}{1\le j\le n}}\beta_{i-1}^{j-1}
          +\sum_{\mycom{1\le i\le m}{1\le j\le n}}\beta_{i-1}^{j}
          -\sum_{\mycom{1\le i\le m}{1\le j\le n}}\beta_{i}^{j}\\
          &= \sum_{\mycom{1\le i\le m}{0\le j\le n-1}}\beta_{i}^{j}
          -\sum_{\mycom{0\le i\le m-1}{0\le j\le n-1}}\beta_{i}^{j}
          +\sum_{\mycom{0\le i\le m-1}{1\le j\le n}}\beta_{i}^{j}
          -\sum_{\mycom{1\le i\le m}{1\le j\le n}}\beta_{i}^{j}\\
          &= \sum_{0\le j\le n-1}(\beta_{m}^{j}
          -\beta_{0}^{j})
          +\sum_{1\le j\le n}(\beta_{0}^{j}
          -\beta_{m}^{j})\\
          &= \beta_{m}^{0}-\beta_{0}^{0} + \beta_{0}^{n} -\beta_{m}^{n}\\
        &= \beta_{m}^{0}-\beta_{m}^{n}\\
          &=\beta_k^{\p}(\bar u,\bar v)-\sum_{u\leq\bar u}\nu^\p(u,\infty).
\end{align*}
\end{proof}

Theorem~\ref{k-triangle} states that the value of $\beta^\p_k$ at a point $(\bar u,\bar v)\in\Delta^+$ equals the number of cornerpoints lying strictly above and non-strictly to the left of $(\bar u,\bar v)$.



The following result relates the number of proper cornerpoints of a persistence diagram with positive multiplicity contained in an $(\varepsilon,\varepsilon)$-box with the multiplicity of the $(\varepsilon,\varepsilon)$-box itself.

\begin{corollary}\label{propnumpoints}
Let $\bar P=(\bar u,\bar v)\in\Delta^+$ and $\eps>0$,
with $\bar u+\varepsilon<\bar v-\varepsilon$.
Then $\mu^\p_{\eps,\eps}(\bar P)$ equals the number of proper cornerpoints of $\dgm(\p)$ (counted with their multiplicities) that belong to 
the $(\eps,\eps)$-box centred at $\bar P$, 
i.e.,
\[
B^\lnot_\eps(\bar P)=\{(u,v)\in\Delta^+:
\bar u-\eps<u\le \bar u+\eps, \bar v-\eps<v\le \bar v+\eps\}.
\]
\end{corollary}

\begin{proof}


The statement follows from the definition of
$\mu^\p_{\eps,\eps}(\bar P)$ (Definition~\ref{defmu}) and the Representation Theorem (Theorem~\ref{k-triangle}), since
\begin{align*}
        \mu^\p_{\eps,\eps}(\bar P)
          &=\beta_k^{\p}(\bar u+\eps,\bar v-\eps)
          -\beta_k^{\p}(\bar u-\eps,\bar v-\eps) -\beta_k^{\p}(\bar u+\eps,\bar v+\eps)
          +\beta_k^{\p}(\bar u-\eps,\bar v+\eps)\\
          &=\ \ \ \  \sum_{\mathclap{
          \substack{ (u,v)\in\Delta^+\\
                    u\le \bar u+\eps\\
                    v>\bar v-\eps}
                    }
                    }
                    \mu^\p(u,v)
                    +\sum_{u\leq\bar u+\eps}\nu^\p(u,\infty)
          -\sum_{\mathclap{
          \substack{ (u,v)\in\Delta^+\\
                    u\le \bar u-\eps\\
                    v>\bar v-\eps}
                    }
                    }
                    \mu^\p(u,v)
                    -\sum_{u\leq\bar u-\eps}\nu^\p(u,\infty)\\
          &-\ \ \ \  \sum_{\mathclap{
          \substack{ (u,v)\in\Delta^+\\
                    u\le \bar u+\eps\\
                    v>\bar v+\eps}
                    }
                    }
                    \mu^\p(u,v)
                    -\sum_{u\leq\bar u+\eps}\nu^\p(u,\infty)
          +\sum_{\mathclap{
          \substack{ (u,v)\in\Delta^+\\
                    u\le \bar u-\eps\\
                    v>\bar v+\eps}
                    }
                    }
                    \mu^\p(u,v)
                    +\sum_{u\leq\bar u-\eps}\nu^\p(u,\infty)\\
                    &=\ \ \ \  \sum_{\mathclap{
          \substack{ (u,v)\in\Delta^+\\
                    \bar u-\eps<u\le \bar u+\eps\\
                    v>\bar v-\eps}
                    }
                    }
                    \mu^\p(u,v)
          -\sum_{\mathclap{
          \substack{ (u,v)\in\Delta^+\\
                    \bar u-\eps<u\le \bar u+\eps\\
                    v>\bar v+\eps}
                    }
                    }
                    \mu^\p(u,v)\\
           &= \ \ \ \  \sum_{\mathclap{
          \substack{ (u,v)\in\Delta^+\\
                    \bar u-\eps<u\le \bar u+\eps\\
                    \bar v-\eps<v\le \bar v +\eps}
                    }
                    }
                    \mu^\p(u,v).
\end{align*}
\end{proof}

\begin{corollary}\label{ecrjwedwojeowcw}
Let $\bar P=(\bar u,\infty)\in\Delta^*$ and $\eps>0$.
Then $\beta^\p_k(\bar u+\eps, \infty)-\beta^\p_k(\bar u-\eps, \infty)$ equals the number of improper cornerpoints of $\dgm(\p)$ (counted with their multiplicities) that belong to 
the half-open interval $](\bar u-\eps, \infty), (\bar u+\eps, \infty)]$.
\end{corollary}

\begin{proof}
The statement follows from Lemma~\ref{k-triangle_lemma}, since 
\begin{align*}
\beta^\p_k(\bar u+\eps, \infty)-\beta^\p_k(\bar u-\eps, \infty)&= \sum_{u\leq\bar u+\eps}\nu^\p(u,\infty)-\sum_{u\leq\bar u-\eps}\nu^\p(u,\infty)\\
&=\sum_{\bar u-\eps<u\le \bar u+\eps} \nu^\p(u,\infty).
\end{align*}
\end{proof}

\begin{exercise}
Show that 
$\sum_{(u,\infty)\in \dgm_k(\p)}\nu^\p(u, \infty)=
\dim H_k(X)$.
\end{exercise}




\section{
Stability of persistence diagrams}
\label{sec:stability_thm}

In this section, we aim at proving that persistence diagrams are stable with respect to the $L^\infty$-distance between filtering functions.
To claim such a result, we need to define a notion of distance between persistence diagrams (Section \ref{subsec:matching_distance}).
Then in Section \ref{subsec:stability_thm}, the stability theorem is stated and proven.\\

\subsection{Matching distance}\label{subsec:matching_distance}

Our next goal is to introduce a metric for the set of persistence diagrams.
To do that, we have to see how $\beta_k^\p{(u,v)}$ changes when $\p$ changes.


\begin{proposition}\label{propPBNsandhomeo}
Let $\p,\psi\colon X\to \R$ be two continuous functions with $\|\p-\psi\|_\infty \le \eta$.
Then, for every $(u,v)\in \Delta^*$, we have that
$\beta_k^\p{(u-\eta,v+\eta)}\le \beta_k^\psi{(u,v)}$, where $\beta_k^\p{(u-\eta,v+\eta)}=\beta_k^\p{(u-\eta,\infty)}$ if $(u, v)=(u,\infty)$.
\end{proposition}

\begin{proof}
We recall that $X^\p_\infty=X^\p_{\max\p+\eta}$, for every $\eta\ge 0$.
Let us denote by $S_k^\p(u)$ the set of the singular $k$-chains in $X_u^\p$.
Let us choose an ordered basis $(\alpha_1,\ldots,\alpha_m)$ of  $PH_k^\p{(u-\eta,v+\eta)}\subseteq
H_k(X^\p_{v+\eta})$ and an arbitrary representative $z_i$ of $\alpha_i$, with $z_i\in S_k^{\p}{(u-\eta)}$ for $i=1,\dots,m$.
For each index $i$, we can consider the homology class $\hat\alpha_i\in PH_k^{\psi}{(u,v)}$ that contains $z_i$ (observe that
$z_i\in S_k^{\psi}{(u)}$, because $\|\p-\psi\|_\infty\le \eta$).
We now show that the homology classes $\hat\alpha_i$ are linearly independent in
$PH_k^{\psi}{(u,v)}$. If
$\lambda_1\hat\alpha_1+\ldots+\lambda_m\hat\alpha_m=\mathbf{0}\in PH_k^{\psi}{(u,v)}=\mathrm{Im\ }i^{\psi*}_{u,v}\subseteq H_k(X^{\psi}_{v})$,
we can find a chain $\gamma\in S_{k+1}^{\psi}{(v)}$ such that
$\partial_{k+1} \gamma=\lambda_1 z_1+\ldots+\lambda_m z_m$.
Since $\|\p-\psi\|_\infty\le\eta$,
we have that
$\gamma\in S_{k+1}^{\p}{(v+\eta)}$, and hence
$\lambda_1\alpha_1+\ldots+\lambda_m\alpha_m=\mathbf{0}\in PH_k^{\p}{(u-\eta,v+\eta)}=\mathrm{Im\ }i^{\p*}_{u-v,v+\eta}\subseteq H_k(X^\p_{v+\eta})$.
We know that  $\alpha_1,\ldots,\alpha_m$ are linearly independent
in $PH_k^\p{(u-\eta,v+\eta)}$, and hence $\lambda_1=\ldots=\lambda_m=0$. Therefore,
$\hat\alpha_1,\ldots,\hat\alpha_m$ are linearly independent in $PH_k^\psi{(u,v)}$.
This proves that
$\dim PH_k^{\p}{(u-\eta,v+\eta)}\le \dim PH_k^{\psi}{(u,v)}$.
\end{proof}

Proposition~\ref{propPBNsandhomeo} has the following interesting consequence.

\begin{proposition}\label{proplocconstmult}
Consider $\p\in C^0(X, \R)$.
If $\bar P=(\bar u,\bar v)\in\Delta^+$,
then there is a real number $\bar \eta>0$ such that if $\psi\in C^0(X, \R)$ and $\|\p-\psi\|_\infty \le \eta$ with $0<\eta\le\bar\eta$,
then the closed ball 
$\overline{B_{\eta}}(\bar P)
$ contains exactly $\mu^\p(\bar P)$ proper cornerpoints
(counted with their multiplicities) of the persistence diagram $\dgm_k(\psi)$.
\end{proposition}

\begin{proof}
By Proposition~\ref{propWeps}, a sufficiently small $\eps>0$ exists such that the set
$W_\eps(\bar P):=\{(u,v)\in\R^2:|u-\bar u|<\eps,|v-\bar v|<\eps, u\neq \bar u, v\neq \bar v\}$
is contained in $\Delta^+$
and does not contain any discontinuity point of $\beta^\p_k$.
Proposition~\ref{proppropadiscfromcp} 
implies that $\bar P$ is the only point of $\Delta^+$ that could belong to both $\dgm_k(\p)$ and the open ball 
$B_\eps(\bar P)$.

Let $\bar\eta$ be a real number such that $0 < \bar\eta <\frac{\eps}{2}$. For each real number $\eta$ with $0 < \eta \le \bar\eta$,
let us take a sufficiently small positive real number $\delta<\eta$ with $2\eta+\delta<\eps$, so that
$\bar u + 2\eta + \delta < \bar v -2\eta - \delta$ (we recall that $\bar u +\eps<\bar v-\eps$, since $W_\eps(\bar P)\subseteq \Delta^+$).
We define $\bar P_+^- = (\bar u+\eta+\delta, \bar v-\eta-\delta), \bar P_-^- = (\bar u-\eta-\delta, \bar v-\eta-\delta), \bar P_+^+ = (\bar u+\eta+\delta, \bar v+\eta+\delta), \bar P_-^+ = (\bar u-\eta-\delta, \bar v+\eta+\delta)$ as illustrated in Figure~\ref{figQeps}.


\begin{figure}
\begin{center}
\includegraphics[width=6cm]{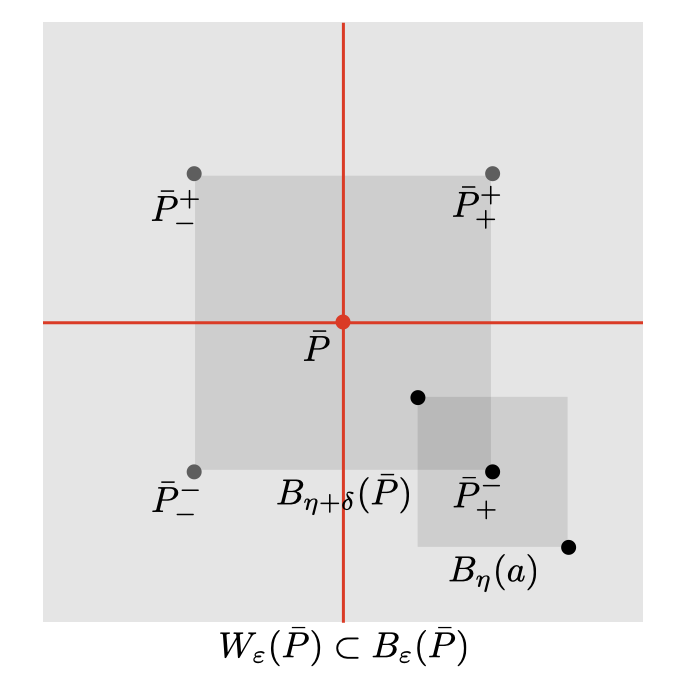}
\caption{The points highlighted and the balls are the ones appearing in the proof of Proposition~\ref{proplocconstmult}. 
The set $W_\eps(\bar P)$ is obtained from the open ball $B_\eps(\bar P)$ by removing the points displayed in red.}
\label{figQeps}
\end{center}
\end{figure}

Let us assume that $\psi\in C^0(X, \R)$ and $\|\p-\psi\|_\infty \le \eta$.
By applying Proposition~\ref{propPBNsandhomeo} twice, we obtain that
\[
\beta^\p_k(\bar u+\delta,\bar v-\delta)\le
\beta^\psi_k(\bar P_+^-)
\le
\beta^\p_k(\bar u+2\eta+\delta,\bar v-2\eta-\delta).
\]
Since $\beta^\p_k$ is constant in each connected component of the set $W_\eps(\bar P)$, in particular,
\[
\beta^\p_k(\bar u+\delta,\bar v-\delta)=
\beta^\p_k(\bar P_+^-)
=
\beta^\p_k(\bar u+2\eta+\delta,\bar v-2\eta-\delta).
\]
This implies that $\beta^\p_k(\bar P_+^-)=\beta^\psi_k(\bar P_+^-)$.
Analogously, we can prove that $\beta^\p_k(\bar P_-^-)=\beta^\psi_k(\bar P_-^-)$,
$\beta^\p_k(\bar P_+^+)=\beta^\psi_k(\bar P_+^+)$, and
$\beta^\p_k(\bar P_-^+)=\beta^\psi_k(\bar P_-^+)$.
Hence, $$\mu^\p(\bar P)=\mu^\p_{\eta+\delta,\eta+\delta}(\bar P)=\mu^\psi_{\eta+\delta,\eta+\delta}(\bar P).$$
From Corollary~\ref{propnumpoints} (applied to $\psi$),
it follows that $\mu^\p(\bar P)$ is equal to the number of cornerpoints in $\dgm_k(\psi)$ that are contained 
in $B^\lnot_{\eta+\delta}(\bar P)$. 
This is true for any sufficiently small
$\delta>0$.
Therefore, $\mu^\p(\bar P)$ is equal to the number of cornerpoints in $\dgm_k(\psi)$
contained in the set 
$\bigcap_{\delta>0} B_{\eta+\delta}^\lnot(\bar P)$, i.e., the closure of the ball $B_{\eta}(\bar P)$.
\end{proof}

The following result extends Proposition~\ref{proplocconstmult} to the points at infinity.

\begin{proposition}\label{proplocconstmultinfty}
If $\bar P=(\bar u,\infty)\in\Delta^*$,
then there is a real number $\bar \eta>0$ such that if $\psi\in C^0(X, \R)$ and $\|\p-\psi\|_\infty \le \eta$ with $0<\eta\le\bar\eta$,
then 
the closed interval $[(\bar u- \eta, \infty),(\bar u+\eta,\infty)]\subset \R\times \{\infty\}$
contains exactly $\nu^\p(\bar P)$ improper cornerpoints
(counted with their multiplicities) of the persistence diagram $\dgm_k(\psi)$.
\end{proposition}

\begin{proof}
By Proposition~\ref{propfiniteatinf}, there exists $\eps>0$ such that the open interval $](\bar u+\eps, \infty), (\bar u+\eps, \infty)[$ contains at most $\bar P$ as a cornerpoint, if $\bar P$ is a cornerpoint, and none otherwise. 
Let $\bar \eta $ be a real number such that $0<\bar \eta<\frac{\eps}{2}$. 
For each real number $\eta$ with $0<\eta\le \bar \eta$, let us take a sufficiently small positive real number $\delta<\eta$ with $2\eta+\delta<\eps$. 
We define $\bar P_+=(\bar u+\eta+\delta, \infty)$ and $\bar P_-=(\bar u-\eta-\delta,\infty)$. 
Let us assume that $\psi\in C^0(X, \R)$ and $\lVert \p-\psi\rVert_\infty\le \eta$. 
By applying Proposition~\ref{propPBNsandhomeo} twice, we obtain 
\[
\beta^\p_k(\bar u+\delta,\infty)\le
\beta^\psi_k(\bar P_+)
\le
\beta^\p_k(\bar u+2\eta+\delta,\infty).
\]
By Remark~\ref{aaohfceolrvjawcol} and the choice of $\eps$, $\beta^\p_k$ is constant in each connected component of $](\bar u-\eps, \infty), (\bar u+\eps, \infty)[\setminus \{\bar P\}$, in particular,
\[
\beta^\p_k(\bar u+\delta,\infty)=
\beta^\p_k(\bar P_+)
=
\beta^\p_k(\bar u+2\eta+\delta,\infty).
\]
This implies $\beta^\p_k(\bar P_+)=\beta^\psi_k(\bar P_+)$. 
Analogously, we can prove that $\beta^\p_k(\bar P_-)=\beta^\psi_k(\bar P_-)$.
Hence, $\nu^\p(\bar P)=\beta^\p_k(\bar P_+)-\beta^\p_k(\bar P_-)=\beta^\psi_k(\bar P_+)-\beta^\psi_k(\bar P_-)$.
From Corollary~\ref{ecrjwedwojeowcw} (applied to $\psi$), it follows that $\nu^\p(\bar P)$ is equal to the number of cornerpoints in $\dgm_k(\psi)$ that are contained in the half-open interval $](\bar u-\eta-\delta, \infty), (\bar u+\eta+\delta,\infty)]$.
This is true for any sufficiently small $\delta>0$. 
Therefore, $\nu^\p(\bar P)$ is equal to the number of cornerpoints in $\dgm_k(\psi)$ contained in the set $\bigcap_{\delta>0}](\bar u-\eta-\delta, \infty), (\bar u+\eta+\delta,\infty)]$, i.e., the closed interval $[(\bar u-\eta, \infty), (\bar u+\eta,\infty)]$.
\end{proof}


The following corollary shows that Propositions~\ref{proplocconstmult} and~\ref{proplocconstmultinfty} can be adapted to open balls.

\begin{corollary}\label{corintorno}
If \(\bar P \in \Delta^+\), then there exists a real number \(\bar \eta > 0\) such that, for every \(\psi \in C^0(X,\mathbb{R})\) with \(\|\varphi - \psi\|_\infty < \eta\), where \(0 < \eta \le \bar \eta\), the open ball \(B_\eta(\bar P)\) with respect to \(d\) contains exactly as many proper cornerpoints (counted with multiplicity) of the persistence diagram \(\dgm_k(\psi)\) as the multiplicity of $\bar P$ in \(\dgm_k(\varphi)\).

Analogously, if \(\bar P =(\bar u, \infty)\), then there exists a real number \(\bar \eta > 0\) such that, for every \(\psi \in C^0(X,\mathbb{R})\) with \(\|\varphi - \psi\|_\infty < \eta\), where \(0 < \eta \le \bar \eta\), the open interval \(](\bar u-\eta, \infty), (\bar u+\eta, \infty)[\) with respect to \(d\) contains exactly as many improper cornerpoints (counted with multiplicity) of the persistence diagram \(\dgm_k(\psi)\) as the multiplicity of $\bar P$ in \(\dgm_k(\varphi)\).
\end{corollary}

\begin{proof}
We apply Propositions~\ref{proplocconstmult}, keeping the notation used in their statements.
If \(\|\varphi - \psi\|_\infty < \eta\), we can find \(\delta\) with \(0 < \delta < \eta\) such that \(\|\varphi - \psi\|_\infty \le \eta - \delta\).
It follows that the closed ball \(\overline{B_{\eta-\delta}}(\bar P)\), with respect to \(d\), contains exactly as many non-trivial cornerpoints (counted with multiplicity) of the persistence diagram \(\dgm_k(\psi)\) as the multiplicity of \(\bar P\) in \(\dgm_k(\varphi)\).
On the other hand, since \(\|\varphi - \psi\|_\infty \le \eta\), the closed ball \(\overline{B_{\eta}}(\bar P)\), with respect to \(d\) has the same property. 
This implies that there are no points of \(\dgm_k(\psi)\) in the set \(\overline{B_{\eta}}(\bar P)\setminus \overline{B_{\eta-\delta}}(\bar P)\), and hence the claim of the corollary follows.

The case of $\bar P=(\bar u, \infty)$ is obtained analogously by applying Proposition~\ref{proplocconstmultinfty}.
\end{proof}

We can now introduce a metric between persistence diagrams.

\begin{definition}[Matching distance]\label{defbottleneckdistance}
For every two $\p, \psi\in C^0(X, \R)$, we define the \emph{matching distance} or \emph{bottleneck distance} as
\[
d_\match (\dgm(\p),\dgm(\psi)):=\inf_{\sigma\in \mathcal{M}(\p,\psi)}\sup_{P\in\dgm(\p)}d(P,\sigma(P))
\]
where $\mathcal{M}(\p,\psi)$ is the set  of all matchings between
$\dgm(\p)$ and $\dgm(\psi)$.
\end{definition}

\begin{proposition}\label{dmatchisadistance}
The function $d_\match$ is a metric on the set of persistence diagrams.
\end{proposition}

\begin{proof}
\begin{itemize}
 \item[1.] We have that \begin{align*}
            d_\match (\dgm(\p),\dgm(\p))
            &=\inf_{\sigma\in \mathcal{M}(\p,\p)}\sup_{P\in\dgm(\p)}d(P,\sigma(P))\\
            &\le \sup_{p\in\dgm(\p)}d(P,\id(P))=0.
\end{align*}
 \item[2.] Since $d_{\mathrm{match}}(\dgm(\p), \dgm(\p')) = 0$, for every $\varepsilon > 0$ we can find a matching $\sigma_\varepsilon \in \mathcal{M}(\p,\p')$ such that $d(P,\sigma_\varepsilon(P)) < \varepsilon$ for every cornerpoint $P \in \dgm(\p)$. 
Because of Proposition~\ref{proplocfinite} and Proposition~\ref{propfiniteatinf}, we know that for any point $P \in \Delta^*$ there exists a sufficiently small open ball $B_\varepsilon(P) \subseteq \Delta^*$ that contains no points of $\dgm(\p) \cup \dgm(\p')$ other than $P$, if any. Therefore, if $\varepsilon$ is sufficiently small, then $\sigma_\varepsilon(P) = P$ for every copy of $P$ in the multiset $\dgm(\p)$.
Hence,
$\mu^{\p}(P) \le \mu^{\p'}(P)
\quad \text{for any } P \in \Delta^+$,
and
$\nu^{\p}(P) \le \nu^{\p'}(P)
\quad \text{for any } P = (u,\infty)$.
 Analogously,
 $\mu^{\p'}(P)\le \mu^\p(P)$
 for any $P\in \Delta^+$ and
 $\nu^{\p'}(P)\le\nu^\p(P)$
 for any $P=(u,\infty)$.
 Therefore, $\dgm(\p)=\dgm(\p')$;
 \item[3.] By setting $Q=\sigma(P)$, we have that \begin{align*}
            d_\match (\dgm(\p),\dgm(\p'))
            &=\inf_{\sigma\in \mathcal{M}(\p,\p')}\sup_{P\in\dgm(\p)}d(\sigma(P),P)\\
            &=\inf_{\sigma^{-1}\in \mathcal{M}(\p',\p)}\sup_{Q\in\dgm(\p')}d(Q,\sigma^{-1}(Q))\\
            &=d_\match (\dgm(\p'),\dgm(\p));
\end{align*}
  \item[4.] If $\sigma\in \mathcal{M}(\p,\p')$ and $\sigma'\in \mathcal{M}(\p',\p'')$, then $d(P,\sigma'(\sigma(P)))\le d(P,\sigma(P))+d(\sigma(P),\sigma'(\sigma(P)))$
  for every $P\in \dgm(\p)$,
  and hence
  \begin{align*}
          \sup_{P\in\dgm(\p)}d(P,\sigma'(\sigma(P)))
          &\le \sup_{P\in\dgm(\p)}d(P,\sigma(P))
  +d(\sigma(P),\sigma'(\sigma(P)))\\
          &\le \sup_{P\in\dgm(\p)}(d(P,\sigma(P))+\\
  & + \sup_{P\in\dgm(\p)}d(\sigma(P),\sigma'(\sigma(P))))\\
          &= \sup_{P\in\dgm(\p)}d(P,\sigma(P))+
  \sup_{Q\in\dgm(\p')}d(Q,\sigma'(Q)).
\end{align*}

It follows that for every $\sigma\in \mathcal{M}(\p,\p')$
  \begin{align*}
          d_\match (\dgm(\p),\dgm(\p''))
          &=\inf_{\sigma''\in \mathcal{M}(\p,\p'')}\sup_{P\in\dgm(\p)}d(P,\sigma''(P))\\
          &=\inf_{\sigma'\in \mathcal{M}(\p',\p'')}\sup_{P\in\dgm(\p)}d(P,\sigma'(\sigma(P)))\\
          &\le \sup_{P\in\dgm(\p)}d(P,\sigma(P))+\\
  & + \inf_{\sigma'\in \mathcal{M}(\p',\p'')}\sup_{Q\in\dgm(\p')}d(Q,\sigma'(Q))\\
   &\le \sup_{P\in\dgm(\p)}d(P,\sigma(P))+\\
& +    d_\match (\dgm(\p'),\dgm(\p'')).
\end{align*}

By taking the infimum for $\sigma$ varying in $\mathcal{M}(\p,\p')$, we obtain the triangle inequality
\begin{align*}
d_\match (\dgm(\p),\dgm(\p''))
\le  &  \; d_\match (\dgm(\p),\dgm(\p'))\\
 & +d_\match (\dgm(\p'),\dgm(\p'')).
\end{align*}
\end{itemize}
\end{proof}

\begin{remark}
The matching distance can also be regarded as a distance on functions, by assigning to each pair $(\varphi,\psi)$ the value $d_\match(\dgm(\varphi), \dgm(\psi))$. 
In this case, it is a pseudo-metric, since two continuous functions may have the same persistence diagram. 
\end{remark}

\begin{exercise}
Find two distinct real-valued functions $\p$ and $\psi$ such that the equality $d_\match(\dgm(\p),\dgm(\psi))=0$ holds.
\end{exercise}

\begin{exercise}\label{denaukhao}
Let
$(M, d_M)$ be a metric space. 
The Hausdorff distance between two non-empty compact subsets $A, B\subseteq M$ 
is defined as $$\delta_H(A, B)=\max\left\{\sup_{a\in A}\inf_{b\in B}d_M(a,b), \allowbreak \sup_{b\in B}\inf_{a\in A}d_M(a,b)\right\}.$$
Consider the persistence diagrams $\dgm(\p_1)$ and $\dgm(\p_2)$ of the functions $\p_1$ and $\p_2$. 
The Hausdorff distance between them is defined as the Hausdorff distance between the union of the supports of $\mu^{\p_1}$ and $\nu^{\p_1}$ and the union of the supports of $\mu^{\p_2}$ and $\nu^{\p_2}$.
Show that $$\delta_H(\dgm_k(\p_1), \dgm_k(\p_2))\le d_{\mathrm{match}}(\dgm(\p_1), \dgm(\p_2)).$$
\end{exercise}


\subsection{Matching distance stability theorem}\label{subsec:stability_thm}

After some preliminary results, in this section, we state and prove the matching distance stability (Theorem \ref{matchingstabilitythm})\cite{CSEdHa07}. 
This result is crucial in Topological Data Analysis as it certifies the robustness of persistence diagrams to perturbations and noise.
In fact, the stability of the persistence diagrams ensures, in particular, that close functions will result in close diagrams.
\\

\begin{proposition}\label{propgoodmatching}
If $\p,\psi\in C^0(X, \R)$ and $\|\p-\psi\|_\infty \le \eps$,
then, for any finite multiset $K\subseteq\dgm(\varphi)\cap\Delta^+$
with $\min_{P\in K}d(P,\Delta)>\eps$, there is an injective multiset map $\sigma:K\to\dgm(\psi)$ such that $\max_{P\in K}d(P,\sigma(P))\le\eps$.
\end{proposition}

\begin{proof}
The statement is trivial if $\varepsilon = 0$, since $\p = \psi$, and we can choose $\sigma$ to be the inclusion map $K \hookrightarrow \dgm(\p)$. Therefore, we may assume that $\varepsilon > 0$.
Let $K = \{P_1, \ldots, P_k\}$, where each $P_j=(u_j,v_j)$ has multiplicity in $K$ equal to
$m_j\le\mu^\p(P_j)$,
and set $m:=\sum_{j=1}^k m_j$.
Let us set
$\p_t:=\frac{\eps-t}{\eps}\p+\frac{t}{\eps}\psi$ for every $t\in[0,\eps]$.
Then, for every $t,t'\in[0,\eps]$,
$\|\p_t-\p_{t'}\|_\infty\le|t-t'|$.
Consider the set $A$ of all values  $\delta\in [0, \eps]$ for which an injective multiset map  $\sigma_\delta: K\to \dgm(\p_\delta)$ exists, such that
$d(P_j,\sigma_\delta(P_j))\le\delta$ for every $P_j\in K$.
In other words, if we think of the variation of $t$ as the flow of time, $A$ is the set of times
$\delta$ for which the cornerpoints in $K$ move less than $\delta$ itself, when $\p$ is changed into $\p_\delta$.

We want to prove that $\sup A=\eps$.
First of all, we observe that $A$ is non-empty, since $0 \in A$ (it suffices to choose $\sigma_0$ equal to the inclusion map $K\xhookrightarrow{} \dgm(\p)$).
Let us set $\bar{\delta} = \sup A$ and show that $\bar{\delta} \in A$. Indeed, let
$(\delta_n)$ be a non-decreasing sequence in $A$, converging to $\bar \delta$. Since $\delta_n\in A$, for each $n$ there exists an injective map  $\sigma_{\delta_n}: K\to \dgm(\p_{\delta_n})$, such that
$\max_j d(P_j,\sigma_{\delta_n}(P_j))\le\delta_n$.
Since $\delta_n\le\eps$,
$\max_j d(P_j,\sigma_{\delta_n}(P_j))\le\eps$ for any $n$.
Thus, $\sigma_{\delta_n}(P_j)\in\overline{B_\eps}(P_j)$
for any $n$,
where $\overline{B_\eps}(P_j)$ is the closed ball 
of radius $\eps$ 
with respect to $d$, centred at $P_j$.
Note that such ball 
does not contain the diagonal $\Delta$ because $\min_{P\in K}d(P, \Delta)>\eps$.
Since the closed balls $\overline{B_\varepsilon}(P_j)$ are compact sets, possibly after extracting a subsequence, we may assume that the sequence $(\sigma_{\delta_n}(P_j))_n$ converges for every index $j\in\{1,\ldots,k\}$.
We set $\bar P_j:=\lim_{n\to\infty}\sigma_{\delta_n}(P_j)$.
We have that $d(P_j, \bar P_j)\le \bar \delta$.
Also, the following property holds:

$(*)$ If $P_{j_1},\ldots,P_{j_r}\in K$ and
$\bar Q=\bar P_{j_1}=\ldots =\bar P_{j_r}$, then the multiplicity of $\bar Q=(\bar u,\bar v)$ in $\dgm(\p_{\bar \delta})$
is not smaller than $r$.

Let us prove $(*)$. We know that $\delta_n\le \eps$ for any index $n$. Let $\eta>0$.
If $n$ is large enough then $|\delta_n-\bar\delta|\le\eta$, and hence $\|\p_{\delta_n}-\p_{\bar\delta}\|_\infty\le \eta$. 
As a consequence, on the one hand, if $\eta$ is sufficiently small, then Proposition~\ref{proplocconstmult} (local constancy of multiplicity) guarantees that, for all sufficiently large $n$, the multiplicity $\overline{m}$ of $\bar Q$ in $\dgm(\p_{\bar\delta})$ equals the number of cornerpoints of $\dgm(\p_{\delta_n})$ that belong to the closed ball $\overline{B_\eta}(\bar Q) := \{(u,v) \in \Delta^+ : \bar u - \eta \le u \le \bar u + \eta, \bar v - \eta \le v \le \bar v + \eta\}$.
On the other hand,
$\overline{B_\eta}(\bar Q)$ contains
at least the $r$ cornerpoints
$\sigma_{\delta_n}(P_{j_1}),\ldots,\sigma_{\delta_n}(P_{j_r})$ of $\dgm(\p_{\delta_n})$,
since $\lim_{n\to\infty}\sigma_{\delta_n}(P_{j_i})=\bar Q$
for $1\le i\le r$.
Therefore, $\overline{m}\ge r$, and $(*)$ is proved.
In particular, $\bar Q$ is a cornerpoint in $\dgm(\p_{\bar \delta})$.
To conclude that $\bar \delta\in A$, it is now sufficient
to consider the multiset map  $\sigma_{\bar \delta}: K\to \dgm(\p_{\bar \delta})$ taking $P_j$ to $\bar P_j$ for every $P_j\in K$. Property $(*)$ guarantees that $\sigma_{\bar\delta}$ is injective  (in the sense of an injective multiset map).
Thus, we have proved that $\sup A \in A$, i.e., $\sup A = \max A$.

We end the proof by showing that
$\max A =\eps$. In fact, if $\bar\delta<\eps$, by using Proposition~\ref{proplocconstmult}
once again, it is not difficult to show that there exists $\eta>0$, with $\bar\delta + \eta < \eps$, and
an injective multiset map $\sigma_{\bar \delta,\bar \delta+\eta}$ from the multiset $\{\bar P_1, \ldots, \bar P_k\}$ to the multiset $\dgm(\p_{\bar \delta+\eta})$
such that
$d(\bar P_j^i,\sigma_{\bar \delta,\bar \delta+\eta}(\bar P_j^i))\le\eta$ for $1\le j\le k$, where $\bar P_j^i$ denotes the $i$-th copy of $\bar P_j$ in the multiset $\sigma_{\bar \delta}(K)$.
Hence $\sigma_{\bar \delta,\bar \delta+\eta}\circ \sigma_{\bar \delta}:K\to \dgm(\p_{\bar \delta+\eta})$ is an injective multiset map and, by
the triangle inequality,
$d(P_j,\sigma_{\bar \delta,\bar \delta+\eta}\circ \sigma_{\bar \delta}(P_j))\le\bar\delta+\eta$
for $1\le j\le k$, implying that
$\bar\delta +\eta\in A$.
This contradicts the fact that
$\bar\delta =\max A$. Therefore,
$\eps=\max A$, and hence $\eps\in A$.
\end{proof}



The following result extends Proposition~\ref{propgoodmatching} to the points at infinity. We omit the proof, as it is analogous to that of Proposition~\ref{propgoodmatching}. It relies on Proposition~\ref{proplocconstmultinfty}.

\begin{proposition}\label{propgoodmatching2}
If $\p,\psi\in C^0(X, \R)$ and $\|\p-\psi\|_\infty\le \eps$, then for any finite multiset of essential cornerpoints $K'\subseteq\dgm(\varphi)$, there exists an injective multiset map $\sigma':K'\to\dgm(\psi)$ such that 
$\max_{P\in K'}d(P,\sigma'(P))\le\eps$.
\end{proposition}

The following result shows that it is possible to match all the points of $\dgm(\p)$ injectively with points of $\dgm(\psi)$, with maximum displacement not exceeding $\|\p - \psi\|_\infty$.

\begin{proposition}\label{propgoodmatching3}
If $\|\p-\psi\|_\infty\le\eps$, then there exists an injective multiset map $\tau:\dgm(\p)\to\dgm(\psi)$ such that $d(P,\tau(P))\le\eps$ for every $P\in \dgm(\p)$.
\end{proposition}

\begin{proof}
Set $K_1 := \{P \in \dgm(\p) : d(P,\Delta) > \varepsilon\}$ and $K_2 := \{P \in \dgm(\p) : d(P,\Delta) \le \varepsilon\}$.
Note that the cardinality of the multiset $K_2$ is always infinite, since $K_2$ contains $\Delta$, counted with  infinite multiplicity.
The cardinality of $K_1$ is finite, according to Propositions~\ref{proplocfinite} and~\ref{propfiniteatinf} on local finiteness, resp. finiteness, of proper and improper cornerpoints. Proposition~\ref{propgoodmatching} and Proposition~\ref{propgoodmatching2} guarantee the existence of an injective multiset map  $\tau_1$ from the multiset $K_1$ to $\dgm(\psi)$,  such that $d(P,\tau(P))\le\eps$ and $\tau(P)\neq \Delta$ for every $P\in K_1$.
We can now consider an injective multiset map  $\tau_2$ from the multiset $K_2$ to the multiset containing just the point $\Delta$ with infinite multiplicity.
The map $\tau$ that coincides with $\tau_1$ on $K_1$ and with $\tau_2$ on $K_2$ is the wanted injective multiset map $\tau$.
\end{proof}

We now recall the following well-known result \cite{KuMo68}.

\begin{theorem}[Cantor-Schröder-Bernstein Theorem]\label{CantorBthm}
Let $A$ and $B$ be sets. If there exist injective maps $f \colon A \to B$ and $g \colon B \to A$, then there exists a bijection $h \colon A \to B$ such that whenever $h(a) = b$, either $f(a) = b$ or $g(b) = a$ (or both).
\end{theorem}



\begin{proof}
We may assume that $A$ and $B$ are disjoint. 
For every $a \in A$ and every $b \in B$, we can construct the two sequences
\[
\cdots \to f^{-1}(g^{-1}(a)) \to g^{-1}(a) \to a \to f(a) \to g(f(a)) \to \cdots,
\]
\[
\cdots \to g^{-1}(f^{-1}(b)) \to f^{-1}(b) \to b \to g(b) \to f(g(b)) \to \cdots.
\]
These sequences are infinite to the right and may be finite or infinite to the left.
In these sequences, the inverses are intended as preimages and, since $f$ and $g$ are injective, each preimage either consists of a single element or is empty.
Any such sequence may terminate on the left or not. If it does terminate, this may occur at an element of $A$ or at an element of $B$. If it does not terminate, then it may be infinite or cyclic.
Let $\mathbb{S}$ denote the set of all maximal sequences obtained by the previous procedure, i.e., sequences that cannot be extended further. 
Note that two distinct sequences in $\mathbb{S}$ cannot have common elements, by maximality and the injectivity of $f$ and $g$. It follows that every $a \in A$ belongs to exactly one sequence in $\mathbb{S}$.
We then define the map $h \colon A \to B$ by setting
$h(a) = f(a)$ if $a$ belongs to a sequence starting from an element of $A$, $h(a) = g^{-1}(a)$ if $a$ belongs to a sequence starting from an element of $B$, and $h(a) = f(a)$
otherwise.
\end{proof}




\begin{theorem}[Matching distance stability theorem]\label{matchingstabilitythm}
Let $\p, \psi\in C^0(X,\R)$. Then
$$d_\match (\dgm(\p),\dgm(\psi))\le \|\p-\psi\|_\infty.$$
\end{theorem}

\begin{proof}
Proposition~\ref{propgoodmatching3} guarantees that there exist an injective multiset map $\tau\colon\dgm(\p)\to\dgm(\psi)$ such that $d(P,\tau(P))\le\eps$ for every $P\in \dgm(\p)$, and an injective multiset map $\tau'\colon\dgm(\psi)\to\dgm(\p)$ such that $d(Q,\tau'(Q))\le\eps$ for every $Q\in \dgm(\psi)$.
Then the claim
follows from Theorem~\ref{CantorBthm}, by setting $A$ equal to a realisation of $\dgm(\p)$ and $B$ equal to a realisation of $\dgm(\psi)$ (see Definition~\ref{Multiset}).
\end{proof}

\begin{corollary}\label{cor_compactness_space_PDs}
If $\Phi\subseteq C^0(X, \R)$ is compact, then so is the set
$\{\dgm(\p)\mid \p\in\Phi\}$ with respect to the topology induced by $d_\match$.
\end{corollary}

\begin{proof}
This follows immediately from Theorem~\ref{matchingstabilitythm}, recalling that a metric space is compact if and only if it is sequentially compact.
\end{proof}

\begin{exercise}\label{exercise_sup=max}
Consider $\p,\psi\in C^0(X, \R)$. 
Prove that for every matching $\sigma\in \mathcal{M}(\p,\psi)$ the equality $\sup_{P\in\dgm(\p)} d(P,\sigma(P))=\max_{P\in\dgm(\p)} d(P,\sigma(P))$ holds.
[\textit{Hint.} 
Note that not every matching $\sigma$ is continuous. The conclusion follows from Proposition~\ref{proplocfinite} and Proposition~\ref{propdgmcompact}.]
\end{exercise}

\begin{exercise}\label{exercise_inf_sup=min_max}
Consider $\p,\psi\in C^0(X, \R)$. 
Assume that the statement of Exercise~\ref{exercise_sup=max} holds. 
Show that $$d_\match (\dgm(\p),\dgm(\psi))
=\min_{\sigma\in B(\p,\psi)}\max_{P\in\dgm(\p)}d(P,\sigma(P)).$$
[\textit{Hint.} 
Consider a sequence $(\sigma_i)_i$ of matchings from $\dgm(\p)$ to $\dgm(\psi)$ such that the sequence
$\left( \max_{P \in \dgm(\p)} d\bigl(P,\sigma_i(P)\bigr) \right)_i$ is non-increasing and converges to 
$d_{\mathrm{match}}\bigl(\dgm(\p),\dgm(\psi)\bigr)$.
Use Proposition~\ref{proplocfinite} and Proposition~\ref{propdgmcompact} to show that there exists a convergent subsequence of $(\sigma_i)_i$ such that its limit $\bar \sigma$ is a matching and $\max_{P \in \dgm(\p)} d\bigl(P,\bar\sigma(P)\bigr)
= d_{\mathrm{match}}\bigl(\dgm(\p),\dgm(\psi)\bigr)$.]
\end{exercise}

\begin{definition}\label{def_optimal_matching}
If $\bar\sigma$ is a matching from $\dgm(\p)$ to $\dgm(\psi)$ and the equality $\max_{P\in\dgm(\p)}d(P,\bar\sigma(P))=d_\match (\dgm(\p),\dgm(\psi))$ holds, then
we say that $\bar\sigma$ is an \emph{optimal matching}.
\end{definition}

Exercise~\ref{exercise_inf_sup=min_max} shows that for each pair $(\dgm(\p),\dgm(\psi))$ of persistence diagrams an optimal matching
$\bar\sigma\colon  \dgm(\p)\to\dgm(\psi)$ exists.


\section{Cornerpoints and critical values of a function on a smooth manifold}\label{sec_critical1}

In this section, we analyze the cornerpoints in the persistence diagram of a function in the differentiable case.
In particular, we show that, for a regular filtering function on a regular manifold, the coordinates of the cornerpoints are related to the critical values of the function.

\smallskip

The following result is classical, and its proof requires knowledge of differential geometry and algebraic topology, which is beyond the scope of this book.
However, it is fundamental for proving Theorem~\ref{thm_coord_cnpts_crit_values}, so we state it in our language and notation, without proof.
The interested reader may refer to~\cite[Theorem 3.1]{milnor}.


\begin{theorem}\label{milnor_critical}
Let $M$ be a smooth and compact manifold, $\p\colon M\to \R$ a smooth function, and $u<v$. 
If $\{x\in M\mid u\le \p(x)\le v\}$ does not contain any critical point for $\p$, then the linear map $i_{u,v}^{\p\ast}\colon H_k(M_u^\p)\to H_k(M_v^\p)$ is an isomorphism, for every $k$.
\end{theorem}

\begin{theorem}\label{thm_coord_cnpts_crit_values}
Let $M$ be a smooth and compact manifold, $\p\colon M\to \R$ a smooth function.
If $w$ is a finite coordinate of a cornerpoint $P\in\dgm_k(\p)\setminus \{\Delta\}$, then $w$ is a critical value for $\p$.
\end{theorem}

\begin{proof}
By Corollary~\ref{homological_critical},
for every sufficiently small $\eps>0$ 
the linear map $$i^{\p*}_{w-\eps, w+\eps}\colon H_k(M^\p_{w-\eps})\to H_k(M^\p_{w+\eps})$$ is not an isomorphism.
Thus, by Theorem~\ref{milnor_critical}, the interval $[w-\eps, w+\eps]$ contains a critical value, $w_\eps$, for $\p$, for every small $\eps$. 
Consider a sequence $(x_n)_n$ in $M$ of critical points corresponding to 
these $w_\eps$.
Since $M$ is compact, we can assume, up to subsequences, that $(x_n)_n$ converges and denote the point of convergence by $\bar x$.
The continuity of $\nabla\p$ and the fact that $x_n$ are all critical imply $0=\lim_{n\to \infty}\nabla \p(x_n)=\nabla \p(\bar x)$.
Thus, $\bar x$ is also a critical point. 
By continuity of $\p$, we have that $\p(\bar x)=w$, concluding the proof.

\end{proof}
\chapter{Biparameter persistent homology}
\label{MPH}

We define persistent homology for functions valued in the real plane.
This is not an immediate generalization of persistent homology for real-valued functions, as different techniques are required to develop the theory in this biparameter setting.
The central result of this chapter allows us to locate persistent features by examining certain critical values of the filtering function.
Furthermore, we define a metric between biparameter persistent Betti numbers functions and show stability with respect to the uniform metric.


\section{Biparameter persistent Betti numbers functions}

Analogously to Sections \ref{sec:PBNF}, \ref{sec:pers_diag}, \ref{sec:stability_thm} of Chapter \ref{ChapterHC}, 
we introduce here the notion of biparameter persistent Betti numbers function, present the definition of biparameter persistence diagrams 
and prove the stability of the information they encode.


The definition of persistent Betti numbers function can be extended to functions taking values in $\R^2$.
Let $\p=(\p_1,\p_2):X\to\R^2$ be a continuous function on a non-empty compact 
space $X$. 
For any $(u_1,u_2)\in \R^2$, the sublevel set $X^\p_{(u_1,u_2)}$ is defined as $\{x\in X\mid\p_1(x)\leq u_1, \p_2(x)\leq u_2\}$, and the collection of sublevel sets, $\{X^\p_{(u_1,u_2)}\}_{(u_1,u_2)\in \R^2}$, is called a \myemph{filtration} (see Figure~\ref{fig:2d_filt}).
\begin{figure}
\centering
\includegraphics[width=\linewidth]{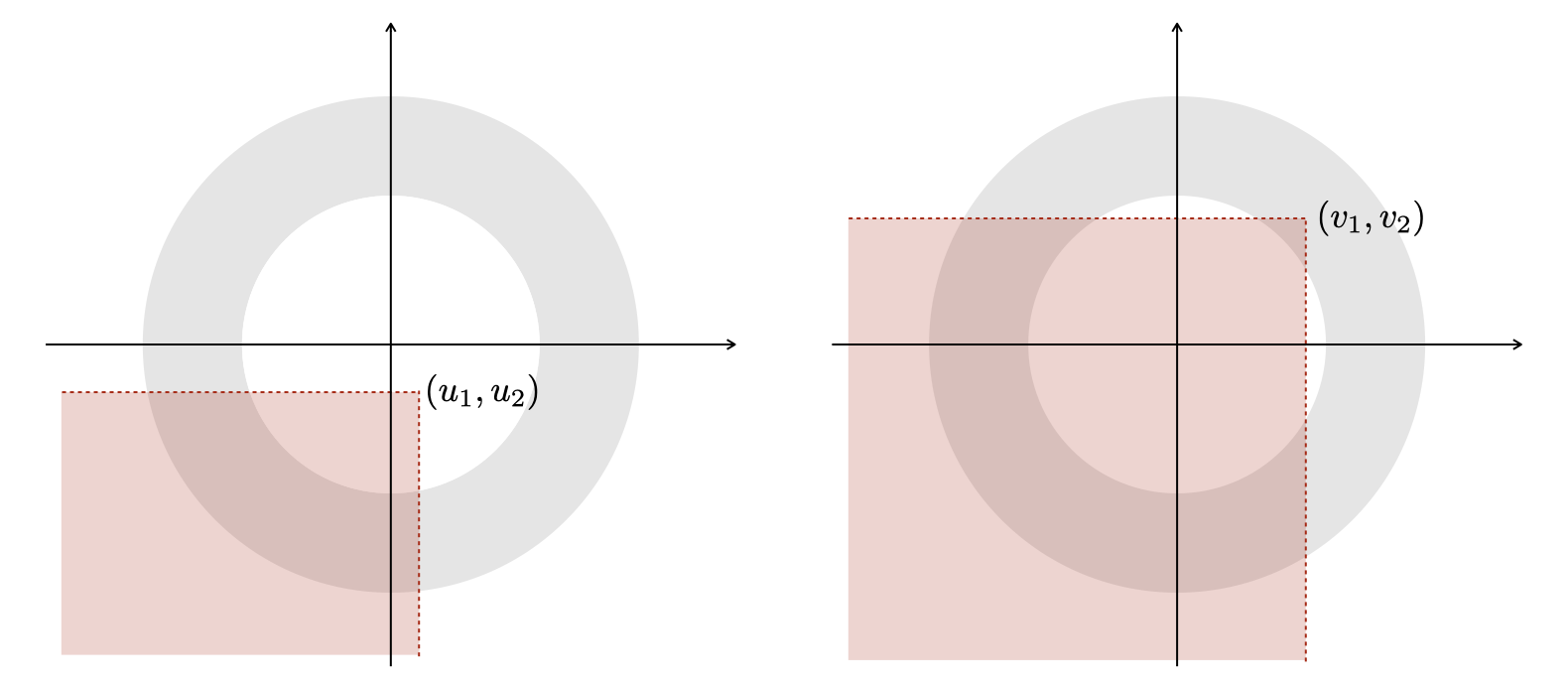}
\caption{Two examples of sublevel set for the inclusion of the annulus into $\R^2$.
$X^\p_{(u_1, u_2)}$ and $X^\p_{(v_1, v_2)}$ are respectively depicted as intersections, on the left and on the right, of the gray annulus with the red box.}
\label{fig:2d_filt}
\end{figure}

The compactness of $X$ ensures that the sequences of spaces in the filtrations $\{X^\p_{(v_1,v_2)}\}_{v_1\in \R}$ and $\{X^\p_{(v_1,v_2)}\}_{v_2\in \R}$ stabilize for a sufficiently large value of $v_1$ and $v_2$, respectively. Explicitly, for any $v_1\geq \max\p_1$ and $v_2\in\R$, $X^\p_{(v_1,v_2)}=X^\p_{(\max\p_1,v_2)}$.
Similarly, for any $v_2\geq \max\p_2$ and $v_1\in\R$, $X^\p_{(v_1,v_2)}=X^\p_{(v_1,\max\p_2)}$. 
Moreover, if $v_i\geq \max\p_i$ for both $i=1,2$, $X^\p_{(v_1,v_2)}=X^\p_{(\max\p_1,\max\p_2)}=X$.
In the following, the symbols $X^\p_{(\infty,v_2)}$, $X^\p_{(v_1,\infty)}$, and $X^\p_{(\infty,\infty)}$ denote $X^\p_{(\max\p_1,v_2)}$, $X^\p_{(v_1,\max\p_2)}$, and $X^\p_{(\max\p_1,\max\p_2)}$, respectively.

Consider $k\in \Z$.  
If $u_1 \leq v_1$ and $u_2\leq v_2$, the inclusion $$i^{\p}_{(u_1,u_2),(v_1,v_2)}\colon X^\p_{(u_1,u_2)}\hookrightarrow X^\p_{(v_1,v_2)}$$ induces a linear map 
$i^{\p*}_{(u_1,u_2),(v_1,v_2)}\colon H_k(X^\p_{(u_1,u_2)})\to H_k(X^\p_{(v_1,v_2)})$.

Sometimes, when $u_1 \le u_2$ and $v_1 \le v_2$, we will write $(u_1, u_2) \preceq (v_1, v_2)$. 
Similarly, when $u_1 < u_2$ and $v_1 < v_2$, we will write $(u_1,u_2) \prec (v_1,v_2)$.

\begin{definition}[Biparameter persistent homology group]\label{MPHG}
Let $k\in \Z$.
The vector subspace $\mathrm{Im\ }i^{\p*}_{(u_1,u_2),(v_1,v_2)}\subseteq H_k(X^\p_{(v_1, v_2)})$, with $(u_1,u_2),(v_1,v_2) \in (\R\cup \{\infty\})^2$ such that the inequalities $u_1 \le v_1$ and $u_2 \le v_2$ hold, 
is denoted by the symbol $PH_k^\p{((u_1,u_2),(v_1,v_2))}$ and called
the \myemph{$k$-th biparameter persistent homology group} of $\p$ at $((u_1,u_2),(v_1,v_2))$.
\end{definition}

Analogously to Chapter \ref{ChapterHC}, from now on the following will hold.
\begin{assumption}\label{BI_ass_fingen}
For every pair $(u_1,u_2)\in\R^2$, the vector space $H_k(X^\p_{(u_1,u_2)})$ is finitely generated.
\end{assumption}

Similarly to the monoparametric case, we adopt the symbols $\Delta^+$, $\Delta^*$ to denote the sets $$\{((u_1,u_2),(v_1,v_2))\in \R^2\times\R^2 \mid u_1<v_1, u_2<v_2\}$$ and $$\{((u_1,u_2),(v_1,v_2))\in \R^2\times(\R\cup \{\infty\})^2 \mid u_1<v_1, u_2<v_2\},$$ respectively.

\begin{definition}[Biparameter persistent Betti numbers function]\label{MPBNF}
The \myemph{biparameter persistent Betti numbers function} of $\p=(\p_1,\p_2)\colon X\to\R^2$ in degree $k$, briefly PBNF, is the function $\beta_k^\p$$\colon\Delta^*\to\mathbb{N}$ defined by
\begin{displaymath}
\beta_k^\p((u_1,u_2),(v_1,v_2)):=\dim PH_k^\p{((u_1,u_2),(v_1,v_2))}.
\end{displaymath}
\end{definition}

Consider the annulus, $X$, in Figure~\ref{fig:2d_filt}, which is a topological space in $\R^2$, and the filtering function $\p\colon X\to \R^2$ given by the inclusion. 
The 0-th homology groups of the sublevel sets depicted in the figure, $H_0(X^\p_{(u_1,u_2)})$ and $H_0(X^\p_{(v_1,v_2)})$, are respectively isomorphic to $\Z_2$ and $\Z_2^2$.
Moreover, the linear map $i^\p_{(u_1,u_2),(v_1,v_2)}$, induced by the inclusion between the 0-th homology groups (which are indeed vector spaces), maps 
each homology class of a point \( p\in X^\p_{(u_1,u_2)} \) (i.e., its connected component) to the homology class of $p$ in \( X^\p_{(v_1,v_2)} \).
Thus, $\beta_0^\p((u_1,u_2),(v_1,v_2))=1$. 

\begin{exercise}
Consider the inclusion $\p\colon X\to \R^2$ of the annulus $X$ in Figure~\ref{fig:2d_filt}. 
Compute $H_i(X^\p_{(u_1,u_2)})$ and $\beta^\p_k((u_1,u_2),(v_1,v_2))$, for every $((u_1,u_2)$, $(v_1,v_2))\in \Delta^*$, and for $i=0,1$.
\end{exercise}

\begin{exercise}\label{xca_2D>1D}
Find a compact space $X$ and two continuous functions $\p=(\p_1,\p_2),\psi=(\psi_1,\psi_2)\colon X\to\R^2$ such that, for any $k\in \Z$,
$\beta_k^{\p_1}\equiv \beta_k^{\psi_1}$
and $\beta_k^{\p_2}\equiv \beta_k^{\psi_2}$,
but there exists a value $k$ for which $\beta_k^{\p}\not\equiv \beta_k^{\psi}$.
\end{exercise}

\subsection{Biparameter persistence diagrams}\label{sec:BI_pers_diag}

We now show how to study biparameter persistent Betti numbers via a reduction to the monoparametric setting, referring the reader to Figure~\ref{foliation} for a pictorial representation.

\begin{figure}
    \centering
    \includegraphics[width=\linewidth]{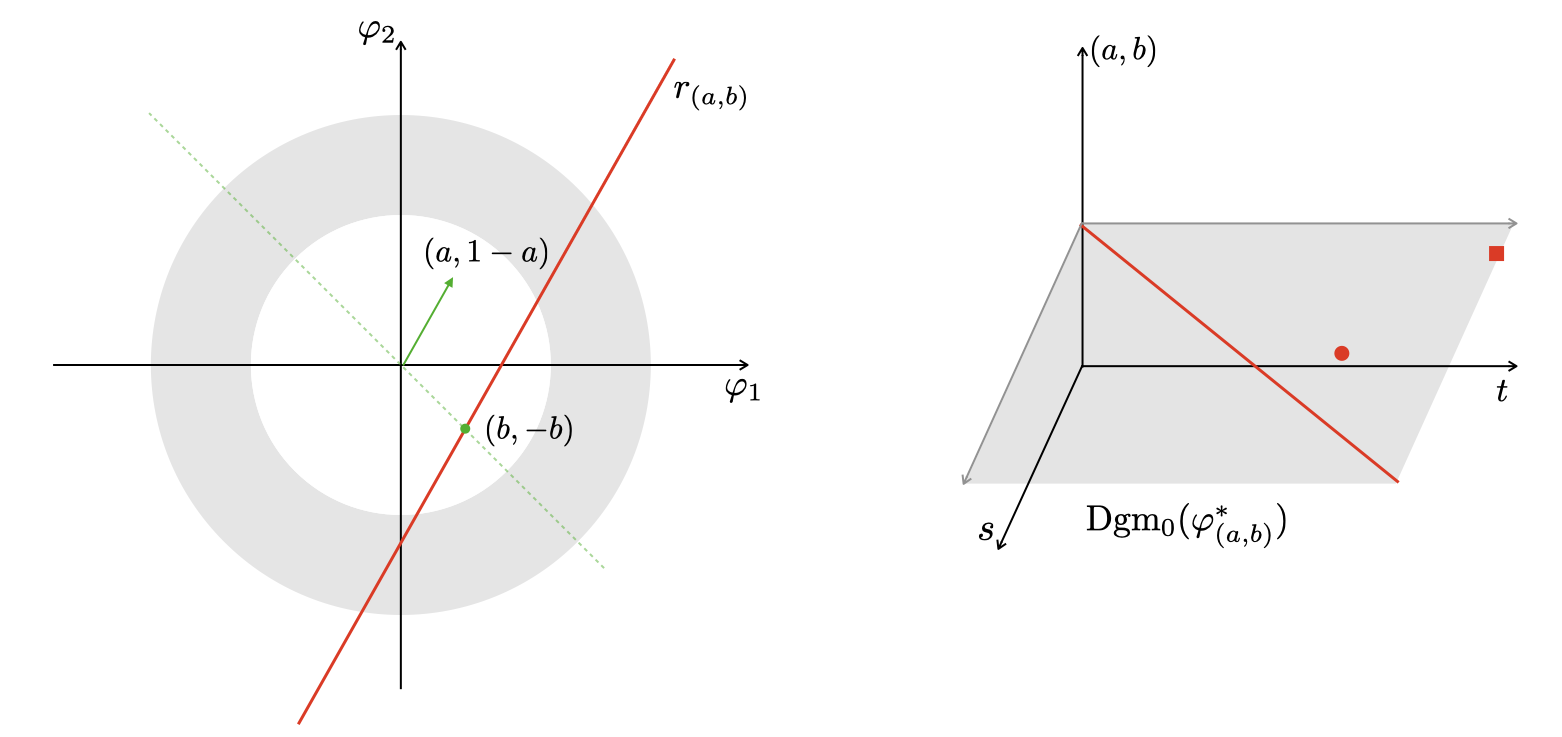}
    \caption{Correspondence between an admissible line $r_{(a,b)}$
and a persistence diagram in the case of the annulus $X$ depicted in Figure~\ref{fig:2d_filt}. Left: a monoparametric filtration is constructed by sweeping the line $r_{(a,b)}$. The green vector $(a,1-a)$ and the green point $(b,-b)$ are used to parameterize this line as $r_{(a,b)}: t\cdot (a,1-a) + (b,-b)$. Right: the persistence diagram of the monoparametric filtration on the left can be found on a planar section of the domain of the biparameter persistent Betti numbers function $\beta_k^\p$.}
    \label{foliation}
\end{figure}

In accordance with Chapter \ref{ChapterHC}, from now on the following assumption will hold.

\begin{assumption}
\label{ass_rcont}
The function $\beta^\p_k$ is right-continuous in both pairs of entries: for every 
$((u_1,u_2),(v_1,v_2))\in \Delta^*$,
\begin{itemize}
\item $\lim_{\substack {\eps_1\to 0^+\\ \eps_2\to 0^+}}\beta^\p_k((u_1+\eps_1, u_2+\eps_2), (v_1, v_2))= \beta^\p_k((u_1,u_2), (v_1,v_2))$,
\item 
$\lim_{\substack {\eps_1\to 0^+\\ \eps_2\to 0^+}}\beta^\p_k((u_1, u_2), (v_1+\eps_1, v_2+\eps_2))= \beta^\p_k((u_1,u_2), (v_1,v_2))$, 
\end{itemize}
where the first condition is equivalent to: for every $\eta>0$ there exist $\eps_1, \eps_2>0$ such that if $0< u'_1-u_1\le \eps_1$ and $0< u'_2-u_2\le \eps_2$ then $\lvert \beta^\p_k((u'_1,u'_2), (v_1,v_2)) - \beta^\p_k((u_1,u_2), (v_1,v_2)) \rvert\le \eta$. Since $\beta^\p_k$ takes values in $\mathbb{N}$, the above condition is equivalent to: for sufficiently small $\varepsilon_1,\varepsilon_2$, $$\beta^\p_k((u'_1,u'_2), (v_1,v_2)) = \beta^\p_k((u_1,u_2), (v_1,v_2)).$$ An analogous statement holds for the right-continuity condition on the second pair.
\end{assumption}

Let us consider the set \myemph{$\Lambda^+$} of all lines of $\R^2$ that have positive slope. This set can be parameterized by the set $\mathcal {P}(\Lambda^+):=\,]0,1[\,\times\R$, by taking each line $r\in\Lambda^+$ to the unique pair $(a,b)$ with $0<a<1$ and $b\in \R$ such that $(a,1-a)$ is a direction vector of $r$ and $(b,-b)\in r$.
The line $r$ will be denoted by $r_{(a,b)}$.
The set $\Lambda^+$ is referred to as the
\myemph{set of admissible lines}.
Each point $(u_1,u_2)$ of $r_{(a,b)}$ can be expressed as $t\cdot (a,1-a)+(b,-b)$  for some $t\in \R$.
So, one can define the filtration $\{X^{a,b}_t\}_{t}$ where $X^{a,b}_t:=X^\p_{(u_1,u_2)}$ is the set of points of $X$ whose image by $\p$ is ``under and on the left of $(u_1,u_2)$'' while $(u_1,u_2)$ moves along the line $r_{(a,b)}$.
As a consequence, each admissible line $r_{(a,b)}$ defines a filtration $\{X^{a,b}_t\}_t$ of $X$ and a persistence diagram associated with this filtration.

It is interesting to observe that the filtration $\{X^{a,b}_t\}_t$ can also be defined as the sublevel sets filtration induced by a suitable real-valued function.
In fact, we have that $X^{a,b}_t=\{x\in X\mid \p_{(a,b)}(x)\leq t\}$ where $\p_{(a,b)}:X\to\R$ is defined by setting $\p_{(a,b)}(x):=\max\left\{\frac{\p_1(x)-b}{a},\frac{\p_2(x)+b}{1-a}\right\}$.
Thus, the following definition is well-posed.

\begin{definition}[Biparameter persistence diagram]\label{defBPD}
Fixed an integer $k$, the family of $k$-th persistence diagrams $\{\dgm_k(\p_{(a,b)})\}_{(a,b)}$ associated with the lines $r_{(a,b)}$ varying $(a,b)$ is called
\myemph{$k$-th biparameter persistence diagram of $\p$}.
\end{definition}

The persistent Betti numbers function $\beta_k^\p$ can be completely recovered by considering all and only the persistent Betti numbers functions $\beta^{\p_{(a,b)}}_k$ associated with the admissible lines $r_{(a,b)}$.
Explicitly, $\beta_k^\p((u_1,u_2),(v_1,v_2))=\beta^{\p_{(a,b)}}_k(s,t)$ where $r_{(a,b)}$ is the admissible line passing through the points $(u_1,u_2)$ and $(v_1,v_2)$, and $s,t \in \R$ satisfy $(u_1,u_2)=s\cdot (a,1-a)+(b,-b)$ and $(v_1,v_2)=t\cdot (a,1-a)+(b,-b)$.
Note that $u_1+u_2=s$ and $v_1+v_2=t$.

\begin{exercise}\label{xca_no_pathol_cases}
Under the assumptions above, prove that for every pair $(a,b)\in {\mathcal {P}}(\Lambda^+)$: 
\begin{itemize}
    \item for every $u\in\R$, the vector space $H_k(X^{\p_{(a,b)}}_{u})$ satisfies Assumption~\ref{ass_fingen},
    \item the function $\beta_k^{\p_{(a,b)}}$ satisfies Assumption~\ref{ass_right-cont}.
\end{itemize}
\end{exercise}

Analogously to the monoparametric case, PBNFs are non-decreasing in the first pair of variables and non-increasing in the second pair, as the following result states. 
We omit the proof, as it 
is completely analogous to that
of Proposition~\ref{propmonotonicity}.
\begin{proposition}\label{PBNF_non}
If $u_1\le u'_1< v_1\le v'_1$ and $u_2\le u'_2< v_2\le v'_2$,
\begin{enumerate}
\item $\beta^\p_k((u_1, u_2), (v_1,v_2))\le \beta^\p_k((u'_1, u'_2), (v_1,v_2))$,
\item $\beta^\p_k((u_1, u_2), (v_1,v_2))\ge \beta^\p_k((u_1, u_2), (v'_1,v'_2))$. 
\end{enumerate}
\end{proposition}

Next result states that, even though Assumption~\ref{ass_rcont} for $\beta^\p_k$ seems stronger, it is equivalent to requiring Assumption~\ref{ass_right-cont} for $\beta^{\p_{(a,b)}}_k$, for every $(a,b)$.
\begin{proposition}\label{jhfgsrelifshejz}
The following are equivalent:
\begin{enumerate}
\item $\lim_{\substack {\eps_1\to 0^+\\ \eps_2\to 0^+}}\beta^\p_k((u_1+\eps_1, u_2+\eps_2), (v_1, v_2))= \beta^\p_k((u_1,u_2), (v_1,v_2))$, for every $((u_1, u_2), (v_1,v_2))\in \Delta^*$, 
\item $\lim_{\eps\to 0^+}\beta^{\p_{(a,b)}}_k(s+\eps, t)=\beta^{\p_{(a,b)}}_k(s, t)$, for every $(a,b)\in \mathcal{P}(\Lambda^+)$ and every $s,t$ with $s<t$.
\end{enumerate}
Symmetrically, the following are equivalent:
\begin{enumerate}
\item[3.] $\lim_{\substack {\eps_1\to 0^+\\ \eps_2\to 0^+}}\beta^\p_k((u_1, u_2), (v_1+\eps_1, v_2+\eps_2))= \beta^\p_k((u_1,u_2), (v_1,v_2))$, for every $((u_1, u_2), (v_1,v_2))\in \Delta^*$, 
\item[4.] $\lim_{\eps\to 0^+}\beta^{\p_{(a,b)}}_k(s, t+\eps)=\beta^{\p_{(a,b)}}_k(s, t)$, for every $(a,b)\in \mathcal{P}(\Lambda^+)$ and every $s,t$ with $s<t$.
\end{enumerate}
\end{proposition}

\begin{proof}
The implication from Statement 1 to Statement 2 is the content of Exercise~\ref{xca_no_pathol_cases}. 
So, we just show that Statement 2 implies Statement 1.
Moreover, the equivalence between Statements 3 and 4 is analogous.
Assume by contradiction that there is a strictly decreasing 
sequence $(P_i)_i=(u_1^i, u_2^i)_i$ such that $\lim_{i\to \infty}P_i=P=(u_1,u_2)$ and $\lim_{i\to \infty}\beta^\p_k(P_i, Q)\neq\beta^\p_k(P,Q)$, for some $Q=(v_1,v_2)$. 
We can consider the line, $r_{(a_i, b_i)}$, passing through $P_i$ and $Q$, for every $i$, giving rise to the sequences $(s_i)_i$, $(a_i)_i$ and $(b_i)_i$ such that $(u_1^i, u_2^i) =(a_i,1-a_i) s_i+(b_i,-b_i)$ and $(v_1,v_2)  =(a_i,1-a_i)t+(b_i,-b_i)$.
Since $s_i=u_1^i+u_2^i$, $(s_i)_i$ is strictly decreasing and $\lim_{i\to \infty} (u_1^i+u_2^i)=u_1+u_2$, which we denote by $s$. 
Furthermore, since $t=v_1+v_2$, we have $a_i=\frac{v_1-u_1^i}{v_1+v_2-u^i_1-u_2^i}$. 
So, the limit $\lim_{i\to \infty}a_i=\frac{v_1-u_1}{v_1+v_2-u_1-u_2}$ exists and it is denoted by $a$. 
As a consequence, the limit $\lim_{i\to \infty}b_i=u_1-(u_1+u_2)\frac{v_1-u_1}{v_1+v_2-u_1-u_2}$ also exists and it is denoted by $b$. 
Let us now consider the orthogonal projection of the sequence $(P_i)_i$ on the line $r_{(a,b)}$ and denote it by $(\hat P_{i})_i=(\hat u_1^i, \hat u_2^i)_i$. 
Note that, since $(P_i)_i$ is strictly decreasing
and $r_{(a,b)}$ has positive slope, the sequence $(\hat P_i)_i$ is also strictly descreasing. 
It is possible to choose a subsequence of indices $(i_j)_j$ such that 
\[
P\prec \cdots \prec \hat P_{i_{j+3}}\prec P_{i_{j+2}}\prec \hat P_{i_{j+1}}\prec P_{i_j}\prec \cdots
\]
The existence of this sequence can be shown with the following procedure, which is also illustrated in Figure~\ref{fig_sequence}.
Take $P_{i_1}=P_1$. 
The next element $\hat P_{i_2}$ can be chosen as the biggest $\hat P_i$ smaller than $P_{i_1}$. 
This point exists because the downset of $P_{i_1}$ contains all $\hat P_{i_
j}$ with $j\ge l$, for a certain $l$. 
This is because such downset contains $P$  and the sequence $(\hat P_i)_i$ converges to $P$ and it is strictly decreasing.
We can iteratively proceed in the same way to construct the entire sequence.
Since this sequence is strictly decreasing, Proposition~\ref{PBNF_non} guarantees that $\beta^\p_k(\hat P_{i_j}, Q)\le \beta^\p_k(P_{i_{j+1}}, Q)\le \beta^\p_k(\hat P_{i_{j+2}}, Q)$. 
Consequently, $\lim_{j\to \infty}\beta^\p_k(\hat P_{i_j}, Q)=\lim_{j\to \infty}\beta^\p_k(P_{i_j}, Q)$, which, by assumption, is different from $\beta^\p_k(P,Q)$.
This leads to a contradiction because $\beta^\p_k(\hat P_{i_j}, Q)=\beta_k^{\p_{(a,b)}}(\hat s_i, t)$ and $\lim_{i\to \infty}\beta_k^{\p_{(a,b)}}(\hat s_i, t)\neq \beta^\p_k(P,Q)=\beta_k^{\p_{(a,b)}}(s, t)$ in contrast with Statement 2.
\end{proof}

\begin{figure}
    \centering
    \includegraphics[width=0.7\linewidth]{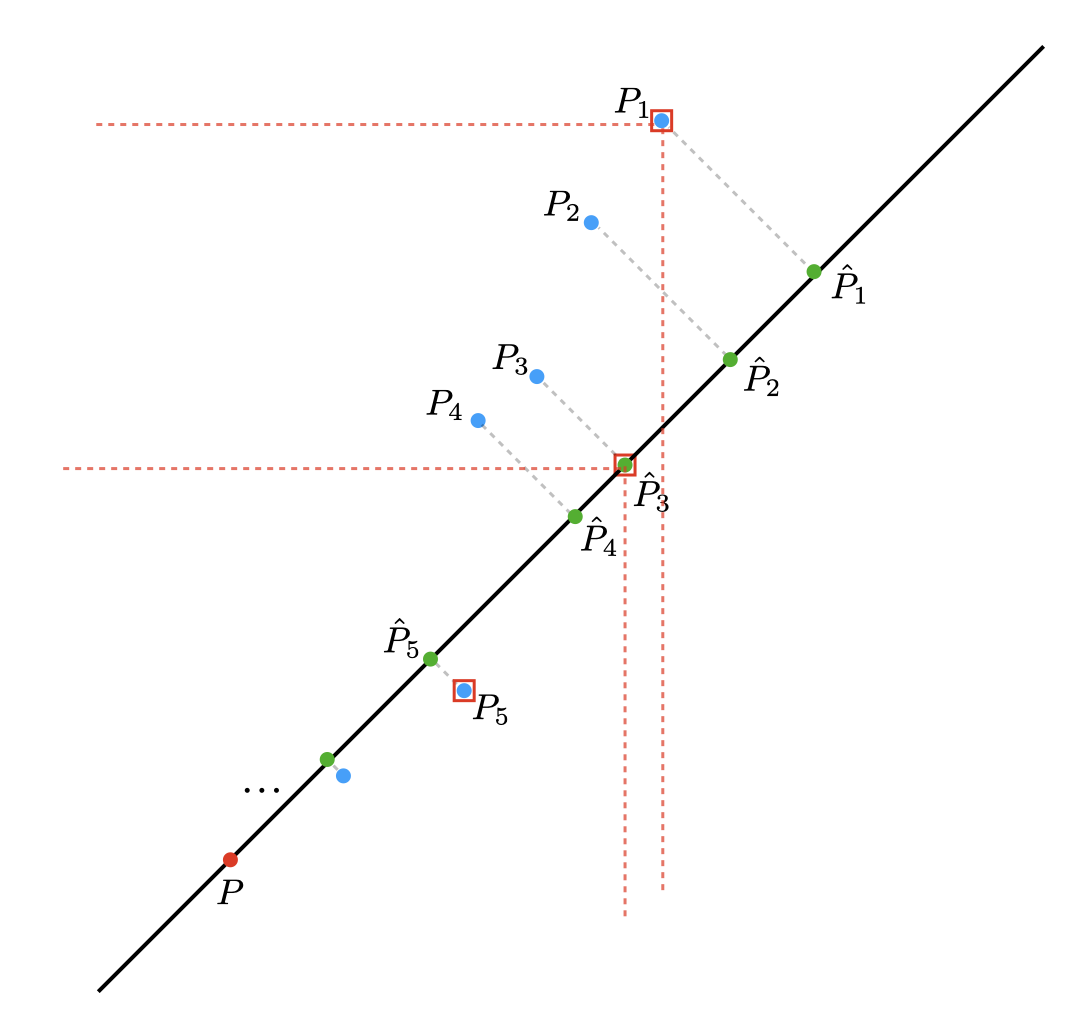}
    \caption{The first elements of the sequence constructed in the proof of Proposition~\ref{jhfgsrelifshejz} are $P_1\succ \hat P_3\succ P_5\succ \cdots$.}
    \label{fig_sequence}
\end{figure}

\subsection{Stability of biparameter persistence diagrams}\label{2DMD}

In order to define a stable distance between the PBNFs of $\p$ and $\psi$ we need to normalize the functions $\p_{(a,b)}$ and $\psi_{(a,b)}$.
So, we set, for every $(a,b)\in {\mathcal {P}}(\Lambda^+)$, $\p_{(a,b)}^*:=\min\{a,1-a\}\cdot \p_{(a,b)}$ and $\psi_{(a,b)}^*:=\min\{a,1-a\}\cdot \psi_{(a,b)}$, respectively.
These normalized functions determine the \myemph{normalized persistence diagrams} $\dgm\left(\p_{(a,b)}^*\right)$, $\dgm\left(\psi_{(a,b)}^*\right)$.
This is equivalent to considering the diagrams $\dgm\left(\p_{(a,b)}\right)$, $\dgm\left(\psi_{(a,b)}\right)$ associated with the admissible line $r_{(a,b)}$ and normalize them by multiplying coordinatewise their points by $\min\{a,1-a\}$.  


We start from the following results providing an alternative, yet equivalent, formulation of the  $L^{\infty}$-distance between $\p$ and $\psi$.

\begin{lemma}\label{lemmafab}
For every $(a,b)\in {\mathcal {P}}(\Lambda^+)$
, $\left\|\p_{(a,b)}^*- \psi_{(a,b)}^*\right\|_\infty\le \|\p-\psi\|_\infty$.
\end{lemma}
\begin{proof}\label{prooflemmafab}
By recalling that for any $s,t,s',t'\in\R$
\[
\left|\max\left\{s,t\right\}-\max\left\{s',t'\right\}\right|\le
\max\left\{|s-s'|,|t-t'|\right\},
\]
for every $(a,b)\in {\mathcal {P}}(\Lambda^+)$ and every $x\in X$, we have
\[
\begin{array}{lll}
\left|\p_{(a,b)}^*(x)-\psi_{(a,b)}^*(x)\right| & = & \min\{a,1-a\}\cdot\left|\p_{(a,b)}(x)-\psi_{(a,b)}(x)\right|  \\
&\le & \min\{a,1-a\}\cdot\max\left\{\left|\frac{\p_1(x)-{\psi}_1(x)}{a}\right|, \left|\frac{\p_2(x)-{\psi}_2(x)}{1-a}\right|\right\}  \\
&\le & \max\left\{\left|\p_1(x)-{\psi}_1(x)\right|, \left|\p_2(x)-{\psi}_2(x)\right|\right\}.
\end{array}
\]
\end{proof}

\begin{proposition}\label{statement}
Let $\p,\psi:X\to\R^2$ be two continuous functions.
Then
\[
\|\p-\psi\|_\infty=
\sup_{(a,b)\in {\mathcal {P}}(\Lambda^+)}\left\|\p^*_{(a,b)}-\psi^*_{(a,b)}\right\|_\infty=
\sup_{b\in \R}\left\|\p^*_{(1/2,b)}-\psi^*_{(1/2,b)}\right\|_\infty.
\]
\end{proposition}

\begin{proof}
By Lemma~\ref{lemmafab}, we know that if $(a,b)\in {\mathcal {P}}(\Lambda^+)$ then
$\left\|\p_{(a,b)}^*- \psi_{(a,b)}^*\right\|_\infty\le \|\p-\psi\|_\infty$. Therefore,
we have that
\[
\|\p-\psi\|_\infty\ge
\sup_{(a,b)\in {\mathcal {P}}(\Lambda^+)}\left\|\p^*_{(a,b)}-\psi^*_{(a,b)}\right\|_\infty\ge
\sup_{b\in \R}\left\|\p^*_{(1/2,b)}-\psi^*_{(1/2,b)}\right\|_\infty.
\]
Let us take a point $\bar x\in X$, whose existence is ensured by the compactness of $X$, such that $\|\p-\psi\|_\infty=\|\p(\bar x)-\psi(\bar x)\|_\infty$.
We can assume that $\|\p(\bar x)-\psi(\bar x)\|_\infty=|\p_1(\bar x)-\psi_1(\bar x)|$.
If $a=1/2$, then $\min\{a,1-a\}=a=1-a$, so that
$\p^*_{(a,b)}(\bar x)=\max\{\p_1(\bar x)-b,\p_2(\bar x)+b\}$ and
$\psi^*_{(a,b)}(\bar x)=\max\{\psi_1(\bar x)-b,\psi_2(\bar x)+b\}$.

If we also assume that {$b<\min\{\min \p_1,-\max \p_2,\min \psi_1,-\max \psi_2\}$}, then
$\p^*_{(a,b)}(\bar x)=\p_1(\bar x)-b$ and $\psi^*_{(a,b)}(\bar x)=\psi_1(\bar x)-b$.
It follows that
$$
\sup_{b}\left\|\p^*_{(1/2,b)}-\psi^*_{(1/2,b)}\right\|_\infty\ge
|\p_1(\bar x)-\psi_1(\bar x)|=\|\p-\psi\|_\infty.
$$
\end{proof}


\begin{definition}[Biparameter matching distance]\label{defbimatchingdistance}
For any two continuous functions $\p,\psi:X\to\R^2$ we define the \myemph{biparameter matching distance} \myemph{$D_{\mathrm{match}}(\p,\psi)$} as
$$
D_{\mathrm{match}}(\p,\psi):=\sup_{(a,b)\in{\mathcal {P}}(\Lambda^+)}d_\match\left(\dgm\left(\p_{(a,b)}^*\right),\dgm\left(\psi_{(a,b)}^*\right)\right)
.$$
\end{definition}

\begin{remark}
The biparameter matching distance can be seen as an extended metric between PBNFs (or biparameter persistence diagrams), or as pseudo-metric between functions.
This is due to the fact that it is defined as a supremum of metrics, respectively, pseudo-metrics.
Unlike the monoparametric case, the chosen notation for the biparameter matching distance favors the interpretation of metric between functions for simplicity of notation. 
\end{remark}

By applying Lemma~\ref{lemmafab} and Theorem~\ref{matchingstabilitythm} about the stability of the matching distance in the monoparametric case, the next result~\cite{CeDFFeal13} immediately follows.

\begin{theorem}[Biparameter matching distance stability theorem]\label{biparam_matchingstabilitythm}
Let $\p,\psi:X\to\R^2$ be two continuous functions.
Then
$$D_{\mathrm{match}}(\p,\psi)\le \|\p-\psi\|_\infty.$$
\end{theorem}

\begin{remark}\label{remNormStability}
The introduction of normalized persistence diagrams in the definition of $D_{\mathrm{match}}$ is crucial to obtain a stable pseudo-metric.
The statement of Lemma~\ref{lemmafab} indeed implies that $d_\match\left(\dgm\left(\p_{(a,b)}^*\right),\dgm\left(\psi_{(a,b)}^*\right)\right)$ is less than or equal to $\|\p-\psi\|_\infty$, while this is not true for the bottleneck distance $d_\match\left(\dgm\left(\p_{(a,b)}\right),\dgm\left(\psi_{(a,b)}\right)\right)$.
\end{remark}

\begin{exercise}
Exhibit two continuous functions $\p,\psi:X\to\R^2$ and parameters $(a,b)\in {\mathcal {P}}(\Lambda^+)$ such that
$$d_\match\left(\dgm\left(\p_{(a,b)}\right),\dgm\left(\psi_{(a,b)}\right)\right) > \|\p-\psi\|_\infty.$$
\end{exercise}

For details on the computation of the matching distance, we refer the interested reader to the papers \cite{BiCeFrGi11,KeLeOu19,CeFr20,KeNi20,BaBrHaLaMaSt22,BjKe23}.

\section{The extended Pareto grid}\label{preparing}
The purpose of this section is to establish a link between the position of cornerpoints of $\dgm\left(\p_{(a,b)}^*\right)$ for a function $\p$ and the intersections of the admissible line $r_{(a,b)}$ with a certain subset of the plane $\R^2$, called the \textit{extended Pareto grid} of $\p$~\cite{CeEtFr19,EtFrQuTo23,FrGaQuTo25}.
The link that we establish in this section is the analogous of the one discussed in Section~\ref{sec_critical1} describing the relation between cornerpoints and critical values in the monoparameter persistent homology setting.

We will assume that $M$ is a compact $m$-manifold, with $m\ge 2$, and the smooth 
filtering function $\p=(\p_1,\p_2)\colon M\to\R^2$ is sufficiently regular, in the sense described below. 
If not differently stated, we will also assume that a degree $k$ has been fixed for the computation of persistence diagrams.
\medskip


\begin{definition}\label{Jacobi}
\begin{enumerate}
    \item The \myemph{Jacobi set}, $\mathbb{J}(\p)$, is the set of all points $x\in M$ at which the gradients of $\p_1$ and $\p_2$ are linearly dependent, 
    namely $\nabla \p_1(x)=\lambda\nabla \p_2(x)$ or $\nabla \p_2(x)=\lambda\nabla \p_1(x)$ for some $\lambda\in\R$.
    \item If $\lambda\leq 0$, the point $x\in M$ is said to be a \myemph{Pareto critical point} for $\p$. 
    The set of all {Pareto critical points} of $\p$ is denoted by $\mathbb{J}_P(\p)$ and is a subset of the Jacobi set $\mathbb{J}(\p)$. 
    Note that $\mathbb{J}_P(\p)$ contains both the critical points of $\p_1$ and the critical points of $\p_2$.
\end{enumerate}
\end{definition}

\begin{assumption}\label{ass_wan}
\begin{enumerate}
    \item No point $x\in M$ exists such that both $\nabla \p_1(x)$ and $\nabla \p_2(x)$ vanish.
	\item 
    $\mathbb{J}(\p)$ is a finite union of smooth compact 1-manifolds 
    each one diffeomorphic to a circle.
	\item
    The set of \myemph{cusp points}, $\mathbb{J}_C(\p)$, of $\p$ is defined as the subset of $\mathbb{J}(\p)$ consisting of those points at which the restriction of $\p$ to $\mathbb{J}(\p)$ fails to be an immersion, i.e., the points at which the differential $d\p$ is not injective. 
In other words, $\mathbb{J}_C(\p)$ is the subset of $\mathbb{J}(\p)$ at which both $\nabla \p_1$ and $\nabla \p_2$ are orthogonal to $\mathbb{J}(\p)$.
The set $\mathbb{J}_P(\p)\setminus \mathbb{J}_C(\p)$ consists of
 finitely many connected components, each one diffeomorphic to an interval.
  Each component can meet critical points for $\p_1,\p_2$ only at its endpoints.
    \item 
    With respect to any parameterization of each component of $\mathbb{J}_P(\p)\setminus \mathbb{J}_C(\p)$, one of $\p_1$ and $\p_2$ is strictly increasing and the other is strictly decreasing.
\end{enumerate}
\end{assumption}


In~\cite{Wa75} (see also \cite{EdHa04}), it was proved that the properties of Assumption~\ref{ass_wan} are generic in the set of smooth maps from $M$ to $\R^2$. This means that these properties hold for a dense open subset of the space of smooth functions from $M$ to $\R^2$.

Assumption~\ref{ass_wan}.3 implies that the connected components of $\mathbb{J}_P(\p)\setminus\mathbb{J}_C(\p)$ are open, or closed, or semi-open arcs in $M$. 
They are referred to as \myemph{critical arcs} of $\p$ (see Figure~\ref{torusw} for an example). 
If an endpoint $x$ of a critical arc belongs to that critical arc and hence is not a cusp point, then it is a critical point for either $\p_1$ or $\p_2$.
Furthermore, Assumption~\ref{ass_wan}.3 implies that both the set of critical points of $\p_1$ and the set of critical points of $\p_2$ are finite.

Let $x_1,\ldots,x_h$ be the critical points of $\p_1$ and $y_1,\ldots,y_k$ the critical points of $\p_2$ (Assumption~\ref{ass_wan}.1 guarantees that $\{x_1,\ldots,x_h\}\cap \{y_1,\ldots,y_k\}=\emptyset$). Consider the following closed half-lines: for each critical point $x_i$ of $\p_1$ (respectively each critical point $y_j$ of $\p_2$), the half-line $\{(x^1,x^2)\in\R^2 \mid x^1=\p_1(x_i), x^2\ge \p_2(x_i)\}$ (respectively the half-line $\{(x^1,x^2)\in\R^2 \mid x^1\ge \p_1(y_j), x^2=\p_2(y_j)\}$).
\begin{definition}[Extended Pareto grid]
The \myemph{extended Pareto grid} \myemph{$\Gamma(\p)$} is the union of $\p(\mathbb{J}_P(\p))$ with these closed half-lines.
The closures of the images of critical arcs of $\p$ are called \myemph{proper contours} of $\p$ associated with those critical arcs of $\p$, while the closed half-lines are called \myemph{improper contours} of $\p$ associated with the corresponding critical points of $\p_1$ and $\p_2$~\cite{CeEtFr19}. 
\end{definition}
We observe that every contour is a closed set and the number of contours of $\p$ is finite because of Assumption~\ref{ass_wan}.3.

The definition of proper and improper contours makes it possible to define the multiplicity of the points of the extended Pareto grid.

\begin{definition}[Multiplicity of a point of the extended Pareto grid]
Let \myemph{$\mathcal{S}(\p)$} be the set of all points of $\Gamma(\p)$ that belong to more than one (proper or improper) contour.
If $\mathcal{S}(\p)$ consists of isolated points, we say that the \myemph{multiplicity of $A\in \Gamma(\p)$} is the greatest 
$k$ such that
for every $\varepsilon>0$
a line $r_{(a,b)}$ with $(a,b)\in{\mathcal {P}}(\Lambda^+)$ exists, verifying these two properties: $r_{(a,b)}$ does not intersect $\mathcal{S}(\p)$ and the cardinality of $r_{(a,b)}\cap\Gamma(\p)\cap B(A,\varepsilon)$ is $k$, where $B(A,\varepsilon)$ is the open ball of centre $A$ and radius $\varepsilon$ with respect to the Euclidean distance.
\end{definition}
In other words, the multiplicity of $A\in \Gamma(\p)$ is the maximum $k$ such that we can find a line with positive slope that does not intersect $\mathcal{S}(\p)$ and contains $k$ points of the extended Pareto grid that are arbitrarily close to $A$.

Under the assumption that $\mathcal{S}(\mathbf{\p})$ consists of isolated points, let $\mathcal{D}(\mathbf{\p})$ be the set of all points $A \in \Gamma(\mathbf{\p})$ that have multiplicity strictly greater than $1$.
We observe that $\mathcal{D}(\mathbf{\p}) \subseteq \mathcal{S}(\mathbf{\p})$. 
Furthermore, the image under $\mathbf{\p}$ of every critical point of either $\p_1$ or $\p_2$, as well as of every cusp point, lies in $\mathcal{S}(\mathbf{\p})$, since each such image belongs to at least two contours.

\begin{definition}[Contour-arc]
Each connected component of $\Gamma(\p)\setminus\mathcal{D}(\p)$ is called a \myemph{contour-arc} of $\p$. 
\end{definition}


A visual intuition is given by Figure~\ref{figgrid3}, showing the extended Pareto grid of the function $\p$ taking each point $x$ of the torus in Figure~\ref{torusw} to the pair $\p(x)=(x^1(x),x^2(x))$. 
The images of the critical arcs are in red, the vertical half-lines with abscissas equal to critical values of $\p_1$ are in magenta, and the horizontal half-lines with ordinates equal to critical values of $\p_2$ are in orange. 
The extended Pareto grid $\Gamma(\p)$ contains the red, magenta and orange points. 
The highlighted red points are endpoints of contours. 
A blue admissible line $r_{(a,b)}$ that does not intersect $\mathcal{S}(\p)$ is also represented. 
The black point $A$ belongs to $\mathcal{D}(\p)$, since we can find a line with positive slope which does not intersect $\mathcal{S}(\p)$ and contains exactly two points of the extended Pareto grid that are arbitrarily close to $A$ (see the  green points in the figure). 
The circled point is an example of point of $\mathcal{S}(\p)$ whose multiplicity is 1, and hence that does not belong to $\mathcal{D}(\p)$.

From now on, we assume that the following properties hold.

\begin{assumption}\label{defnassumptions}
\begin{enumerate}
    \item \label{defnassumptionsinjectivity} For every two distinct critical arcs $\alpha_1, \alpha_2\subseteq \mathbb{J}_P(\p)$, the restriction of $\p$ to $\alpha_1\cup \alpha_2$ is injective up to a finite number of points.  
    \item \label{defnassumptionscontourarcs}
Every contour-arc $\gamma$ of $\p$ is associated with a pair $(d(\gamma),s(\gamma))\in\Z\times\{-1,1\}$
such that at each point $(u_1,u_2)$ of $\gamma$ the following statements hold for every small enough $\varepsilon>0$, where $i^k_*\colon H_k(M^\p_{(u_1-\varepsilon,u_2-\varepsilon)})\to H_k(M^\p_{(u_1+\varepsilon,u_2+\varepsilon)})$ is the linear map induced by the inclusion $M^\p_{(u_1-\varepsilon,u_2-\varepsilon)}\hookrightarrow M^\p_{(u_1+\varepsilon,u_2+\varepsilon)}$:
\begin{itemize}
  \item If $k\neq d(\gamma)$, $i^k_*$ is an isomorphism;
  \item If $k= d(\gamma)$ and $s(\gamma)=1$, $i^k_*$ is injective and $$\textrm{rank} \left(H_k(M^\p_{(u_1+\varepsilon,u_2+\varepsilon)})\right)=
      \textrm{rank} \left(H_k(M^\p_{(u_1-\varepsilon,u_2-\varepsilon)})\right)+1;$$
  \item If $k= d(\gamma)$ and $s(\gamma)=-1$, $i^k_*$ is surjective and $$\textrm{rank} \left(H_k(M^\p_{(u_1+\varepsilon,u_2+\varepsilon)})\right)=
      \textrm{rank} \left(H_k(M^\p_{(u_1-\varepsilon,u_2-\varepsilon)})\right)-1.$$
\end{itemize}

\end{enumerate}
\end{assumption}


\begin{remark}\label{remcontour1}
Assumption~\ref{defnassumptions}.\ref{defnassumptionsinjectivity} prevents different critical arcs to have the exact same image, hence, giving rise to the same contour.  
Moreover, it implies that $\mathcal{S}(\p)$ is finite, and hence $\mathcal{D}(\p)\subseteq \mathcal{S}(\p)$ is also finite. 
\end{remark}

\begin{remark}\label{remcontour2}
It is not difficult to prove that in Assumption\ref{defnassumptions}.\ref{defnassumptionscontourarcs} the homology groups $H_k(M^\p_{(u_1-\varepsilon,u_2-\varepsilon)})$ and $H_k(M^\p_{(u_1+\varepsilon,u_2+\varepsilon)})$ can be replaced by the homology groups $H_k(M^\p_{(u_1-a\varepsilon,u_2-(1-a)\varepsilon)})$ and $H_k(M^\p_{(u_1+a\varepsilon,u_2+(1-a)\varepsilon)})$
for any fixed $a\in\,]0,1[$ without changing the property. 
In plain words, Assumption~\ref{defnassumptions}.\ref{defnassumptionscontourarcs} guarantees that the passage across a contour-arc $\gamma$ along any direction $(a,1-a)$ just creates ($s(\gamma)=1$) or destroys ($s(\gamma)=-1$) exactly one homological class in degree $d(\gamma)$, without producing any homological change in the other degrees. 
\end{remark}

Figure~\ref{figgrid4} shows the contour-arcs and the set $\mathcal{D}(\p)$ (in white) for the function taking each point $x$ of the torus in Figure~\ref{torusw} to the pair $\p(x)=(x^1(x),x^2(x))$. 
Each of the two magenta contour-arcs corresponds to the birth of a homology class in degree $0$ (i.e. $(d(\gamma),s(\gamma))=(0,1)$). 
Each of the ten black contour-arcs corresponds to the birth of a homology class in degree $1$ (i.e. $(d(\gamma),s(\gamma))=(1,1)$). 
Each of the two blue contour-arcs corresponds to the birth of a homology class in degree $2$ (i.e. $(d(\gamma),s(\gamma))=(2,1)$). 
Each of the two red contour-arcs corresponds to the death of a homology class in degree $0$ (i.e. $(d(\gamma),s(\gamma))=(0,-1)$). 
Each of the four green contour-arcs corresponds to the death of a homology class in degree $1$ (i.e. $(d(\gamma),s(\gamma))=(1,-1)$). 
Note that the homological event associated with the points of a contour of $\p$ can change along the considered contour. 
This justifies the choice of stating Assumption~\ref{defnassumptions}.\ref{defnassumptionscontourarcs} for contour-arcs instead of contour.

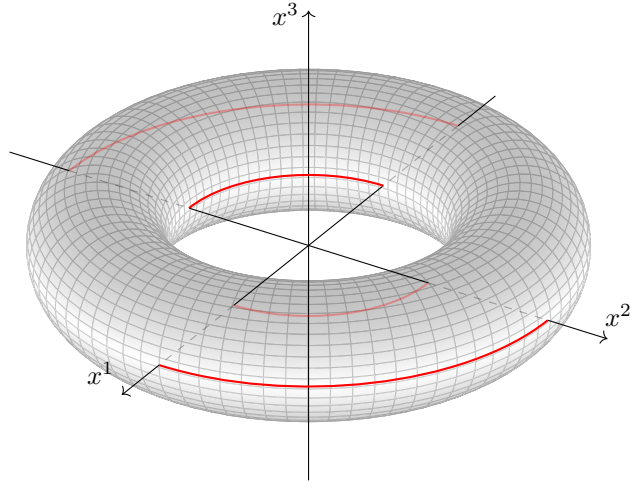
\begin{figure}
\begin{center}
\begin{tikzpicture}
    \begin{axis}[
        axis equal image,
        hide axis,
        z buffer = sort,
        view = {122}{30},
        scale = 1.5
    ]
        \addplot3[
            surf,
            shader = faceted interp,
            samples = 40,
            samples y = 80,
            domain = 0:2*pi,
            domain y = 0:2*pi,
            colormap name = mycolormap,
            thin,
        ](
            {(3+sin(deg(\x)))*cos(deg(\y))},
            {(3+sin(deg(\x)))*sin(deg(\y))},
            {cos(deg(\x))}
        );
            \node (A) at (0,0,-4) {};
            \node (D) at (0,0,4) {};
            \draw [black, ->] (A) -- (D) node[pos=0.95,above left] {$x^3$};
            \draw [black] (-5,0,0) -- (-4,0,0);
            \draw [black, dashed, opacity=0.3] (-4,0,0) -- (-2,0,0);
            \draw [black] (-2,0,0) -- (2,0,0);
            \draw [black, dashed, opacity=0.3] (2,0,0)-- (4,0,0);
            \draw[black, ->] (4,0,0) -- (5,0,0)
            node[pos=0.95, above left] {$x^1$};
            
            \draw [black] (0,-5,0) -- (0,-4,0);
            \draw [black, dashed, opacity=0.3] (0,-4,0) -- (0,-2,0);
            \draw [black] (0,-2,0) -- (0,2,0);
            \draw [black, dashed, opacity=0.3] (0,2,0)-- (0,4,0);
            \draw[black, ->] (0,4,0) -- (0,5,0)
            node[pos=0.80, above right] {$x^2$};

            \draw[red, thick] (4,0,0) arc(0:90:4);
            \draw[red, opacity=0.3, thick] (-4,0,0) arc(180:270:4);
            \draw[red, opacity=0.3, thick] (2,0,0) arc(0:90:2);
            \draw[red, thick] (-2,0,0) arc(180:270:2);
    \end{axis}
\end{tikzpicture}
\end{center}\caption{The torus endowed with the filtering function $\p(x)=(x^1(x), x^2(x))$. 
The critical arcs of $\p$ are displayed in red (the lighter shadow of red reflects the fact that those arcs are on a part of the surface that is not visible from this perspective).}
\label{torusw}
\end{figure}

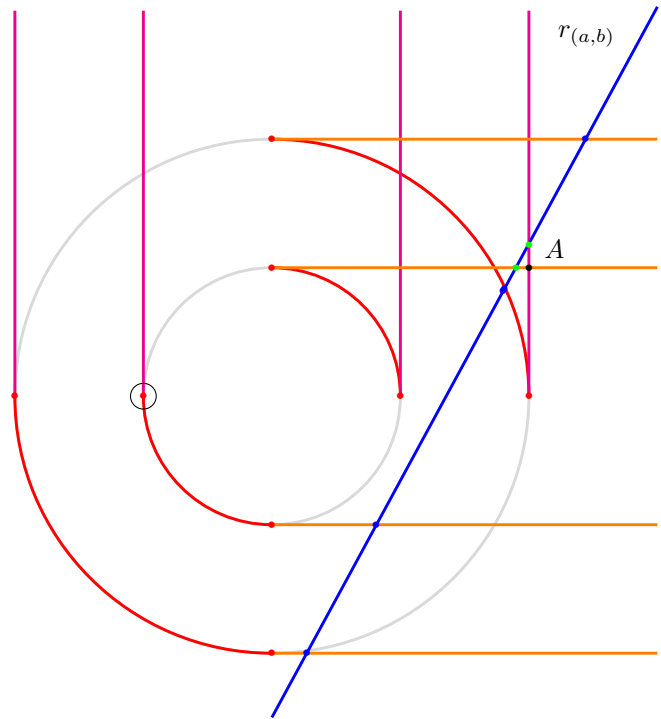
\begin{figure}
\begin{center}

\begin{tikzpicture}[scale=1.7]
\draw [gray!30,line width=1.1] (0,0) circle (1);
\draw [gray!30,line width=1.1] (0,0) circle (2);
\draw [red,line width=1.1] (1,0) arc (0:90:1);
\draw [red,line width=1.1] (2,0) arc (0:90:2);
\draw [red,line width=1.1] (-1,0) arc (180:270:1);
\draw [red,line width=1.1] (-2,0) arc (180:270:2);

\foreach \x in {-2,-1,1,2}
{
	\draw [magenta,line width=1.1] (\x,0) -- (\x,3);
	\draw [orange,line width=1.1] (0,\x) -- (3,\x);

	\draw [red] (\x,0) node {{\tiny $\bullet$}};
	\draw [red] (0,\x) node {{\tiny $\bullet$}};
}

\draw [blue,line width=1.1] (0,-2.5) -- (3,3.03);
\draw [black] (2.45,2.8) node {$r_{(a,b)}$};
\draw [black] (2.2,1.15) node {$A$};

\draw [blue] (0.27,-2) node {{\tiny $\bullet$}};
\draw [blue] (0.81,-1) node {{\tiny $\bullet$}};
\draw [blue] (1.8,0.82) node {{\tiny $\bullet$}};
\draw [blue] (1.81,0.835) node {{\tiny $\bullet$}};
\draw [green] (1.9,1) node {{\tiny $\bullet$}};
\draw [green] (2,1.18) node {{\tiny $\bullet$}};
\draw [black] (2,1) node {{\tiny $\bullet$}};
\draw [blue] (2.44,2) node {{\tiny $\bullet$}};
\draw (-1,-0) circle (0.1cm);
\end{tikzpicture}
\end{center}\caption{The extended Pareto grid for the torus in Figure~\ref{torusw} 
endowed with the filtering function $\p(x):=(x^1(x),x^2(x))$.}\label{figgrid3}
\end{figure}


\begin{figure}
\begin{center}
\begin{tikzpicture}[scale=1.3]
\draw [magenta,line width=1.1] (1,0) arc (0:90:1);
\draw [black,line width=1.1] (2,0) arc (0:90:2);
\draw [black,line width=1.1] (-1,0) arc (180:270:1);
\draw [magenta,line width=1.1] (-2,0) arc (180:270:2);

\draw [red,line width=1.1] (1,0) -- (1,1);
\draw [red,line width=1.1] (0,1) -- (1,1);

\draw [black,line width=1.1] (1,1) -- (3,1);
\draw [black,line width=1.1] (1,1) -- (1,3);

\draw [black,line width=1.1] (-1,0) -- (-1,3);
\draw [magenta,line width=1.1] (-2,0) -- (-2,3);
\draw [black,line width=1.1] (0,-1) -- (3,-1);
\draw [magenta,line width=1.1] (0,-2) -- (3,-2);
\draw [green,line width=1.1] (0,2) -- (2,2);
\draw [green,line width=1.1] (2,0) -- (2,2);
\draw [blue,line width=1.1] (2,2) -- (2,3);
\draw [blue,line width=1.1] (2,2) -- (3,2);

\draw [line width=0.1,black,fill=white] (1,1.73) circle (0.04);
\draw [line width=0.1,black,fill=white] (1.73,1) circle (0.04);

\draw [line width=0.1,black,fill=white] (1,0) circle (0.04);
\draw [line width=0.1,black,fill=white] (2,0) circle (0.04);
\draw [line width=0.1,black,fill=white] (0,1) circle (0.04);
\draw [line width=0.1,black,fill=white] (0,2) circle (0.04);

\foreach \x in {1,2}
{
	\foreach \y in {1,2}
	{
		\draw [line width=0.1,black,fill=white] (\x,\y) circle (0.04);
	}
}

\end{tikzpicture}
\end{center}\caption{The connected components obtained by deleting the double points (in white) from $\Gamma(\p)$ are the contour-arcs for the torus in Figure~\ref{torusw} endowed with the filtering function $\p(x)=(x^1(x),x^2(x))$. 
In this example $\Gamma(\p)$ contains 20 contour-arcs.
}\label{figgrid4}
\end{figure}
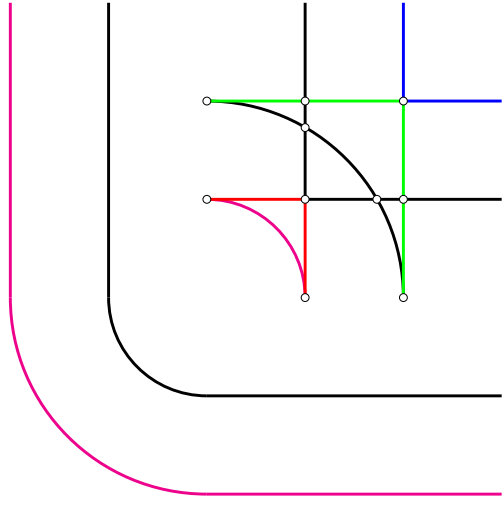

\subsection{The Position Theorem}
The key theorem of this section relates the finite coordinates of the non-trivial cornerpoints of the persistence diagram of $\p^*_{(a,b)}$ to the intersection of the line $r_{(a,b)}$ with the extended Pareto grid of the function $\p$. 
This is fundamental, not only to locate the cornerpoints, but also to keep track of them while the line $r_{(a,b)}$ continuously changes.

We recall that, for every $x \in M$,
\[\p^*_{(a,b)}(x)=\min\{a,1-a\}\cdot \max\left\{\frac{\p_1(x)-b}{a},\frac{\p_2(x)+b}{1-a}\right\}.
\]


We can now state the first result in this section.

\begin{lemma}\label{mainlemma}
If $w$ is a finite coordinate of $P\in \dgm\left(\p_{(a,b)}^*\right)\setminus \{\Delta\}$ , then there exists a point $\bar x\in M$ such that exactly one of the following properties holds:
\begin{enumerate}
  \item $\nabla\p_1(\bar x)=\mathbf{0}$ and
  $w=\frac{\min\{a,1-a\}}{a}\cdot (\varphi_1(\bar x)-b)>\frac{\min\{a,1-a\}}{1-a}\cdot (\varphi_2(\bar x)+b);$
  \item $\nabla\p_2(\bar x)=\mathbf{0}$ and
  $w=\frac{\min\{a,1-a\}}{1-a}\cdot (\varphi_2(\bar x)+b)>\frac{\min\{a,1-a\}}{a}\cdot (\varphi_1(\bar x)-b);$
  \item
$\bar x\in \mathbb{J}_P(M)$ and $w=\frac{\min\{a,1-a\}}{a}\cdot (\varphi_1(\bar x)-b)=\frac{\min\{a,1-a\}}{1-a}\cdot (\varphi_2(\bar x)+b).$
\end{enumerate}
\end{lemma}

\begin{proof} 
Let us set
\[\hat\p_1(x)=\frac{\min\{a,1-a\}}{a}\cdot (\varphi_1(x)-b), \ \text{and} \ 
\hat\p_2(x)=\frac{\min\{a,1-a\}}{1-a}\cdot (\varphi_2(x)+b).\] Moreover, we
 choose a 
real number $c>0$ 
 such that $$\min_{x\in M}\hat\p_1(x),\min_{x\in M}\hat\p_2(x)>-c.$$
We now adopt the approach proposed in \cite{CeFr15}.
Consider the sequence $(F_n)_{n\in \mathbb{N}\setminus\{0\}}$, 
of real-valued functions where
$$F_n(x)=\left((\hat\p_1(x)+c)^n+(\hat\p_2(x)+c)^n\right)^{\frac{1}{n}}-c$$ for any $x$ in $M$. Such a sequence converges uniformly to the function $\p_{(a,b)}^*(x)=
\max\left\{\hat\p_1(x),\hat\p_2(x)\right\}$.
Indeed, recalling that if $\alpha,\beta\ge 0$ then $(\alpha^n+\beta^n)^{\frac{1}{n}}\ge \max\{\alpha,\beta\}$, 
for every $x\in M$ and for every index $n$, we have
that {\setlength\arraycolsep{2pt}
\begin{eqnarray*}
&&|\p_{(a,b)}^*(x)-F_n(x)|\\
&=&\left|\max\left\{\hat\p_1(x),\hat\p_2(x)\right\}-\left((\hat\p_1(x)+c)^n+(\hat\p_2(x)+c)^n\right)^{\frac{1}{n}}+c\right|\\
&=&\left|\max\left\{\hat\p_1(x)+c,\hat\p_2(x)+c\right\}-\left((\hat\p_1(x)+c)^n+(\hat\p_2(x)+c)^n\right)^{\frac{1}{n}}\right|\\
&=&\left((\hat\p_1(x)+c)^n+(\hat\p_2(x)+c)^n\right)^{\frac{1}{n}}-\max\left\{\hat\p_1(x)+c,\hat\p_2(x)+c\right\}\\
&\le&\left(2(\max\left\{\hat\p_1(x)+c,\hat\p_2(x)+c\right\})^n\right)^{\frac{1}{n}}-\max\left\{\hat\p_1(x)+c,\hat\p_2(x)+c\right\}\\
&=&\left(2^{\frac{1}{n}}-1\right)\cdot\max\left\{\hat\p_1(x)+c,\hat\p_2(x)+c\right\}\\
&\le& \left(2^{\frac{1}{n}}-1\right)\cdot\max\left\{\|\hat\p_1+c\|_\infty,\|\hat\p_2+c\|_\infty\right\}.
\end{eqnarray*}}
By the Matching Stability Theorem
\ref{matchingstabilitythm} it follows that it is possible to find a sequence $(P_n)_n$ converging to $P$ such that $P_n\in \dgm\left(F_n\right)\setminus \{\Delta\}$.
Let $w_n$ be the coordinate of $P_n$ converging to $w$. 
Since
$P_n$ is a cornerpoint of $\dgm\left(F_n\right)\setminus \{\Delta\}$, it follows from Theorem~\ref{thm_coord_cnpts_crit_values}
that $w_n$ is a critical value of the function
$F_n$.
Then there exists also a point $x_n\in M$ such that
$\nabla F_n(x_n)=\mathbf{0}$ and $w_n=F_n(x_n)$. 
This means that (with respect to local
coordinates $r^1,\dots,r^m$ of the manifold $M$), 
for any index $j$
the following equalities hold:
{\setlength\arraycolsep{2pt}
\begin{eqnarray*}
0&=&\frac{\partial F_n}{\partial
r^j}(x_n)\\
&=&\left((\hat\p_1(x_n)+c)^n+(\hat\p_2(x_n)+c)^n\right)^{\frac{1-n}{n}}\\
&\cdot&
\left((\hat\p_1(x_n)+c)^{n-1}\cdot\frac{\partial
\hat\p_1}{\partial r^j}(x_n)+
(\hat\p_2(x_n)+c)^{n-1}\cdot\frac{\partial
\hat\p_2}{\partial r^j}(x_n)\right).
\end{eqnarray*}}

Hence, recalling that $\hat\p_1(x_n)+c$ and $\hat\p_2(x_n)+c$ are positive, $\frac{\partial F_n}{\partial
r^j}(x_n)=0$ if and only if
{\setlength\arraycolsep{2pt}
\begin{eqnarray*}
(\hat\p_1(x_n)+c)^{n-1}\cdot\frac{\partial
\hat\p_1}{\partial r^j}(x_n)&+&
(\hat\p_2(x_n)+c)^{n-1}\cdot\frac{\partial
\hat\p_2}{\partial r^j}(x_n)=0.
\end{eqnarray*}}

Therefore, by setting
$$\boldsymbol{v}_n=(v_n^1,v_n^2)=\left((\hat\p_1(x_n)+c)^{n-1},(\hat\p_2(x_n)+c)^{n-1}\right),$$
we can write
$$\begin{pmatrix}
\frac{\partial\hat\p_1}{\partial r^1}(x_n) & \frac{\partial\hat\p_2}{\partial r^1}(x_n) \\
\vdots & \vdots \\
\frac{\partial\hat\p_1}{\partial r^m}(x_n) & \frac{\partial\hat\p_2}{\partial r^m}(x_n)
\end{pmatrix}
\begin{pmatrix}
v_n^1 \\
v_n^2
\end{pmatrix}=
\begin{pmatrix}
0 \\
\vdots\\
0
\end{pmatrix}.$$

By the compactness of $M$, we can assume
(possibly by extracting a subsequence) that $(x_n)_n$ converges to a
point $\bar x$.
Let us define $\boldsymbol{u}_n=\frac{\boldsymbol{v}_n}{\left\|\boldsymbol{v}_n\right\|_{\infty}}$. 
Again by compactness (possibly by considering a subsequence), the sequence $(\boldsymbol{u}_n)_n$
converges to a vector $\boldsymbol{\bar u}=(\bar u^1,\bar u^2)$, where
$\bar{u}^1=\lim_{n\to\infty}\frac{v_n^1}{\left\|\boldsymbol{v}_n\right\|_{\infty}}$,
$\bar{u}^2=\lim_{n\to\infty}\frac{v_n^2}{\left\|\boldsymbol{v}_n\right\|_{\infty}}$, and where $\left\|\boldsymbol{\bar u}\right\|_{\infty}=1$, because $\left\|\boldsymbol{u}_n\right\|_{\infty}=1$.
Since
$$\ \ \ \ \begin{pmatrix}
\frac{\partial\hat\p_1}{\partial r^1}(x_n) & \frac{\partial\hat\p_2}{\partial r^1}(x_n) \\
\vdots & \vdots \\
\frac{\partial\hat\p_1}{\partial r^m}(x_n) & \frac{\partial\hat\p_2}{\partial r^m}(x_n)
\end{pmatrix}
\begin{pmatrix}
u_n^1 \\
u_n^2
\end{pmatrix}=
\begin{pmatrix}
0 \\
\vdots\\
0
\end{pmatrix}$$
for any index $n$, it follows that
$$(*)\ \ \begin{pmatrix}
\frac{\partial\hat\p_1}{\partial r^1}(\bar x) & \frac{\partial\hat\p_2}{\partial r^1}(\bar x) \\
\vdots & \vdots \\
\frac{\partial\hat\p_1}{\partial r^m}(\bar x) & \frac{\partial\hat\p_2}{\partial r^m}(\bar x)
\end{pmatrix}
\begin{pmatrix}
\bar u^1 \\
\bar u^2
\end{pmatrix}=
\begin{pmatrix}
0 \\
\vdots\\
0
\end{pmatrix}.$$
We observe that $u_n^1,u_n^2\ge 0$ for every index $n$, and hence $\bar u^1,\bar u^2\ge 0$.
By $(*)$, it follows that
\[
\bar u^1\cdot\frac{\min\{a,1-a\}}{a}\cdot \frac{\partial\p_1}{\partial r^j}(\bar x) + \bar u^2\cdot\frac{\min\{a,1-a\}}{1-a}\cdot \frac{\partial\p_2}{\partial r^j}(\bar x)=0.
\]
Hence,
$\frac{\bar u^1}{a}\nabla\varphi_{1}(\bar x)+
\frac{\bar u^2}{1-a}\nabla\varphi_{2}(\bar x)=\mathbf{0}$.
By recalling
that (a) $\bar u^1,\bar u^2\geq 0$, (b) $\boldsymbol{\bar u}$ is a non-zero vector, and (c) $a,1-a>0$,
it immediately follows that a $\lambda\le 0$ exists with $\nabla\p_1(\bar x)=\lambda\nabla\p_2(\bar x)$ or $\nabla\p_2(\bar x)=\lambda\nabla\p_1(\bar x)$.
Therefore, $\bar x\in \mathbb{J}_P(\p)$, and hence $\p(\bar x)\in \Gamma(\p)$.

There are three possibilities: $\hat\p_1(\bar x)>\hat\p_2(\bar x)$, $\hat\p_1(\bar x)<\hat\p_2(\bar x)$, and $\hat\p_1(\bar x)=\hat\p_2(\bar x)$.

If $\hat\p_1(\bar x)>\hat\p_2(\bar x)$, then $\p_{(a,b)}^*(x)=\hat\p_1(x)=\frac{\min\{a,1-a\}}{a}\cdot (\p_1(x)-b)$ in an open ball $B$ centred at $\bar x$. Because of the continuity of $\hat\p_1$ and $\hat\p_2$, we can assume that the radius 
$\delta$
of the ball $B$ is so small that 
$0 < \frac{\hat\p_2(x)+c}{\hat\p_1(x)+c}\le\delta<1$
in $B$. 
Since the functions $\frac{\partial\hat\p_1}{\partial r^j}(x)$, $\frac{\partial\hat\p_2}{\partial r^j}(x)$ are continuous, we can also assume that 
$\left|\frac{\partial\hat\p_1}{\partial r^j}(x)\right|,\left|\frac{\partial\hat\p_2}{\partial r^j}(x)\right|\le \mu$ in $B$.

On the one hand, we can prove that $\frac{\partial F_n}{\partial r^j}$ uniformly converges to
$\frac{\partial\hat\p_1}{\partial r^j}$ in $B$, for any index $j$. Indeed, it turns out that, for $x\in B$,

\begin{eqnarray*}
\frac{\partial F_n}{\partial r^j}(x)
&=&\frac{\partial\ }{\partial r^j}\Big(\big((\hat\p_1(x)+c)^n+(\hat\p_2(x)+c)^n\big)^{\frac{1}{n}}-c\Big)\\
&=&n\left( (\hat\p_1(x)+c)^{n-1} \frac{\partial\hat\p_1}{\partial r^j}(x) + (\hat\p_2(x)+c)^{n-1}\frac{\partial\hat\p_2}{\partial r^j}(x)\right)\\
&\cdot&\frac{1}{n} \big((\hat\p_1(x)+c)^n+(\hat\p_2(x)+c)^n\big)^{\frac{1-n}{n}}\\
&=& (\hat\p_1(x)+c)^{n-1}\left(
\frac{\partial\hat\p_1}{\partial r^j}(x)+
\left(\frac{\hat\p_2(x)+c}{\hat\p_1(x)+c}\right)^{n-1}\frac{\partial\hat\p_2}{\partial r^j}(x)
\right)\\
&\cdot&(\hat\p_1(x)+c)^{1-n}\left(1+
\left(\frac{\hat\p_2(x)+c}{\hat\p_1(x)+c} \right)^n
\right)^{\frac{1-n}{n}}\\
&=&\left(
\frac{\partial\hat\p_1}{\partial r^j}(x)+
\left(\frac{\hat\p_2(x)+c}{\hat\p_1(x)+c} \right)^{n-1}
\frac{\partial\hat\p_2}{\partial r^j}(x)\right)\cdot\left(1+\left(\frac{\hat\p_2(x)+c}{\hat\p_1(x)+c} \right)^{n}\right)^{\frac{1-n}{n}}.
\end{eqnarray*}

Since $0\le\frac{\hat\p_2(x)+c}{\hat\p_1(x)+c}\le\delta<1$ and
$\left|\frac{\partial\hat\p_i}{\partial r^j}(x)\right|\le \mu$ in $B$ for $i\in\{1,2\}$,
it follows that the function
$\left(
\frac{\partial\hat\p_1}{\partial r^j}(x)+
\left(\frac{\hat\p_2(x)+c}{\hat\p_1(x)+c} \right)^{n-1}
\frac{\partial\hat\p_2}{\partial r^j}(x)\right)$
is upper bounded by the function
$\frac{\partial\hat\p_1}{\partial r^j}(x)+
\delta^{n-1}\mu$, and lower bounded by the function
$\frac{\partial\hat\p_1}{\partial r^j}(x)-
\delta^{n-1}\mu$. 
The uniform convergence of these two functions to $\frac{\partial\hat\p_1}{\partial r^j}(x)$ 
as $n\to\infty$ 
implies that the sequence $\frac{\partial\hat\p_1}{\partial r^j}(x)+
\left(\frac{\hat\p_2(x)+c}{\hat\p_1(x)+c} \right)^{n-1}
\frac{\partial\hat\p_2}{\partial r^j}(x)$ also uniformly converges to the function 
$\frac{\partial\hat\p_1}{\partial r^j}(x)$.
Moreover, 
the function
$\left(1+\left(\frac{\hat\p_2(x)+c}{\hat\p_1(x)+c} \right)^{n}\right)^{\frac{1-n}{n}}$
is upper bounded by the constant function
$\left(1+\delta^{n}\right)^{\frac{1-n}{n}}$, and lower bounded by the constant function $1$.
The uniform convergence of the constant function $\left(1+\delta^{n}\right)^{\frac{1-n}{n}}$ to the constant function $1$ 
as $n\to\infty$ 
implies that the sequence $\left(1+\left(\frac{\hat\p_2(x)+c}{\hat\p_1(x)+c} \right)^{n}\right)^{\frac{1-n}{n}}$ also uniformly converges to the constant function 
$1$.
It follows that 
the function 
$\left(
\frac{\partial\hat\p_1}{\partial r^j}(x)+
\left(\frac{\hat\p_2(x)+c}{\hat\p_1(x)+c} \right)^{n-1}
\frac{\partial\hat\p_2}{\partial r^j}(x)\right)\cdot\left(1+\left(\frac{\hat\p_2(x)+c}{\hat\p_1(x)+c} \right)^{n}\right)^{\frac{1-n}{n}}$
(i.e., $\frac{\partial F_n}{\partial r^j}(x)$)
uniformly converges to $\frac{\partial\hat\p_1}{\partial r^j}(x)$ in $B$ as $n\to\infty$.

On the other hand, we already know that
the sequence $(F_n)_n$ converges uniformly to the function $\p_{(a,b)}^*$, which equals $\hat\p_1$ in a neighborhood of $\bar x$.
Since $\nabla F_n(x_n)=\mathbf{0}$, $\bar x=\lim_{n\to\infty}x_n$ and $w_n=F_n(x_n)$, 
from the uniform convergence  of $\frac{\partial F_n}{\partial r^j}(x)$  to $\frac{\partial\hat\p_1}{\partial r^j}(x)$ in $B$ as $n\to\infty$,  
it follows that $\nabla\hat \p_1(\bar x)=\mathbf{0}$ and $\hat\p_1(\bar x)=\lim_{n\to\infty}w_n$.
We also know that $P=\lim_{n\to\infty}P_n$, and hence
$w=\lim_{n\to\infty}w_n$. 
This implies that $w=\hat\p_1(\bar x)$.
As a consequence, $w=\hat\p_1(\bar x)=\frac{\min\{a,1-a\}}{a}\cdot (\varphi_1(\bar x)-b)>\hat\p_2(\bar x)=\frac{\min\{a,1-a\}}{1-a}\cdot (\varphi_2(\bar x)+b)$.

Analogously, if $\hat\p_2(\bar x)>\hat\p_1(\bar x)$, then $\p_{(a,b)}^*(x)=\hat\p_2(x)=\frac{\min\{a,1-a\}}{1-a}\cdot (\p_2(x)+b)$ in  
a small open ball $B$ centred at $\bar x$. 
On the one hand, it is possible to show as above that $\frac{\partial F_n}{\partial r^j}$ uniformly converges to
$\frac{\partial\hat\p_2}{\partial r^j}$ in such a neighborhood, for any index $j$.
On the other hand, we already know that
the sequence $(F_n)_n$ converges uniformly to the function $\p_{(a,b)}^*$, which equals $\hat\p_2$ in $B$.
Since $\nabla F_n(x_n)=\mathbf{0}$, $\bar x=\lim_{n\to\infty}x_n$ and $w_n=F_n(x_n)$, 
from the uniform convergence  of $\frac{\partial F_n}{\partial r^j}(x)$  to $\frac{\partial\hat\p_2}{\partial r^j}(x)$ in $B$, 
it follows that $\nabla\hat \p_2(\bar x)=\mathbf{0}$ and $\hat\p_2(\bar x)=\lim_{n\to\infty}w_n$.
We also know that $P=\lim_{n\to\infty}P_n$, and hence
$w=\lim_{n\to\infty}w_n$.
This implies that $w=\hat\p_2(\bar x)$.
As a consequence, $w=\hat\p_2(\bar x)=\frac{\min\{a,1-a\}}{1-a}\cdot (\varphi_2(\bar x)+b)>\hat\p_1(\bar x)=\frac{\min\{a,1-a\}}{a}\cdot (\varphi_1(\bar x)-b)$.

Let us finally consider the case $\hat\p_1(\bar x)=\hat\p_2(\bar x)$
(i.e., $\frac{\min\{a,1-a\}}{a}\cdot (\varphi_1(\bar x)-b)=\frac{\min\{a,1-a\}}{1-a}\cdot (\varphi_2(\bar x)+b)$). 
We already know that a $\lambda\le 0$ exists, such that $\nabla\p_1(\bar x)=\lambda\nabla\p_2(\bar x)$ or $\nabla\p_2(\bar x)=\lambda\nabla\p_1(\bar x)$.
Since the sequence $(F_n)_n$ converges uniformly to the function $\p_{(a,b)}^*$,
the equalities $\lim_{n\to\infty}w_n
=\lim_{n\to\infty}F_n(x_n)=\p_{(a,b)}^*(\bar x)=
\hat\p_1(\bar x)=\hat\p_2(\bar x)$ hold.
We also know that $P=\lim_{n\to\infty}P_n$, and hence
$w=\lim_{n\to\infty}w_n$. 
This implies that $w=\frac{\min\{a,1-a\}}{a}\cdot (\varphi_1(\bar x)-b)=\frac{\min\{a,1-a\}}{1-a}\cdot (\varphi_2(\bar x)+b)$.
\end{proof}


With the extended Pareto grid at hand, we can state and prove the following result, which gives a necessary condition for $P$ to be a point of $\dgm\left(\p_{(a,b)}^*\right)$.

\begin{theorem}[Position Theorem]\label{GT*}
Let $(a,b)\in {\mathcal {P}}(\Lambda^+)$, $P\in\dgm\left(\p_{(a,b)}^*\right)\setminus \{\Delta\}$. 
Then, for each finite coordinate $w$ of $P$, a point $A=(x^1_A, x^2_A)\in r_{(a,b)}\cap\Gamma(\p)$
exists such that $w=\frac{\min\{a,1-a\}}{a}\cdot (x^1_A-b)=\frac{\min\{a,1-a\}}{1-a}\cdot (x^2_A+b)$.
\end{theorem}

\begin{proof}
Let us apply Lemma~\ref{mainlemma}.
Assume that property 1 of Lemma~\ref{mainlemma} holds, and hence that
$\nabla\p_1(\bar x)=\mathbf{0}$ and
  $w=\frac{\min\{a,1-a\}}{a}\cdot (\varphi_1(\bar x)-b)>\frac{\min\{a,1-a\}}{1-a}\cdot (\varphi_2(\bar x)+b)$. 
  Recall that the admissible line $r_{(a,b)}$ is parameterized by $t$ and has equation $(u(t),v(t))=t\cdot(a,1-a)+(b,-b)$. 
  Looking for the point $A=(u(\bar t),v(\bar t))\in r_{(a,b)}$ whose abscissa is $\p_1(\bar x)$, we find that $u(\bar t)=a\bar t+b=\p_1(\bar x)$, i.e., $\bar t=\frac{\p_1(\bar x)-b}{a}$.
It follows that $v(\bar t)=(1-a)\frac{\p_1(\bar x)-b}{a}-b>\p_2(\bar x)$,
since $\frac{\varphi_1(\bar x)-b}{a}>\frac{\varphi_2(\bar x)+b}{1-a}$.
This means that at $(u(\bar t),v(\bar t))$ the line $r_{(a,b)}$ meets the vertical (open) half-line $r: x=\p_1(\bar x),y>\p_2(\bar x)$, which is part of the extended Pareto grid (recall that, by property 1, $\bar x$ is a critical point for $\p_1$).
We also observe that $\frac{\bar x^1_A-b}{a}=\bar t= \frac{\bar x^2_A+b}{1-a}$.
Therefore, $w=\frac{\min\{a,1-a\}}{a}\cdot (\varphi_1(\bar x)-b)=\frac{\min\{a,1-a\}}{a}\cdot (x^1_A-b)=\frac{\min\{a,1-a\}}{1-a}\cdot (x^2_A+b)$ with $A=(x^1_A, x^2_A)\in r_{(a,b)}\cap\Gamma(\p)$.

We skip the case in which property 2 of Lemma~\ref{mainlemma} holds, because it is completely analogous to the one just considered.

To conclude the proof, assume now that property 3 holds.
We know that
$w=\frac{\min\{a,1-a\}}{a}\cdot (\varphi_1(\bar x)-b)=\frac{\min\{a,1-a\}}{1-a}\cdot (\varphi_2(\bar x)+b)$ and
the point $(\p_1(\bar x),\p_2(\bar x))$ belongs to $\Gamma(\p)$, because $\bar x\in\mathbb{J}_P(\p)$. 
Given that the admissible line $r_{(a,b)}$ is parameterized by $t$ and has equation $(u(t),v(t))=t\cdot(a,1-a)+(b,-b)$, by taking $\hat t=\frac{w}{\min\{a,1-a\}}$ we have that $(u(\hat t),v(\hat t))=(\p_1(\bar x),\p_2(\bar x))$, and hence this point belongs to $r_{(a,b)}$.
By setting $A=(x^1_A, x^2_A)=(\p_1(\bar x),\p_2(\bar x))\in r_{(a,b)}\cap \Gamma(\p)$ we get $w=\frac{\min\{a,1-a\}}{a}\cdot (x^1_A-b)=\frac{\min\{a,1-a\}}{1-a}\cdot (x^2_A+b)$.
This yields the claim.
\end{proof}

It is important to note that the reverse of the Position Theorem~\ref{GT*} may not hold: a point of intersection between the extended Pareto grid of a function $\p$ and a line $r_{(a,b)}$ may not correspond to any coordinate of a persistence diagram $\dgm\left(\p^*_{(a,b)}\right)$.
However, the Position Theorem suggests a way to find the possible positions for the cornerpoints of $\dgm\left(\p_{(a,b)}^*\right)\setminus \{\Delta\}$.
It consists in drawing the extended Pareto grid $\Gamma(\p)$ and considering its intersections $(x^1_1,x^2_1),\ldots, (x^1_l,x^2_l)$ with the admissible line $r_{(a,b)}$. 
For each 
cornerpoint in
$\dgm\left(\p_{(a,b)}^*\right)\setminus \{\Delta\}$, both its coordinates belong to the finite set
\begin{equation}\label{coord}
\left\{
\frac{\min\{a,1-a\}}{a}\cdot (x^1_i-b)=\frac{\min\{a,1-a\}}{1-a}\cdot (x^2_i+b)\right\}_{1\le i\le l}\cup \{\infty\}.
\end{equation}
Note that when $b<0$ and $|b|$ is sufficiently large, the admissible line $r_{(a,b)}$ can intersect $\Gamma(\p)$ only at the vertical half-lines (see line $r_{(a,b')}$ in Figure~\ref{figgrid5}). 
In this case, $\p^*_{(a,b)}=\frac{\min\{a,1-a\}}{a}\cdot (\p_1-b)$, and the values $x^1_1,\dots,x^1_l$ in (\ref{coord}) are the critical values of $\p_1$.
Similarly, when $b>0$ and $|b|$ is large enough, $r_{(a,b)}$ intersects $\Gamma(\p)$ only at the horizontal half-lines (see line $r_{(a,b'')}$ in Figure~\ref{figgrid5}). 
Then $\p^*_{(a,b)}=\frac{\min\{a,1-a\}}{1-a}\cdot (\p_2+b)$, and the values $x^2_1,\dots,x^2_l$ in (\ref{coord}) are the critical values of $\p_2$.

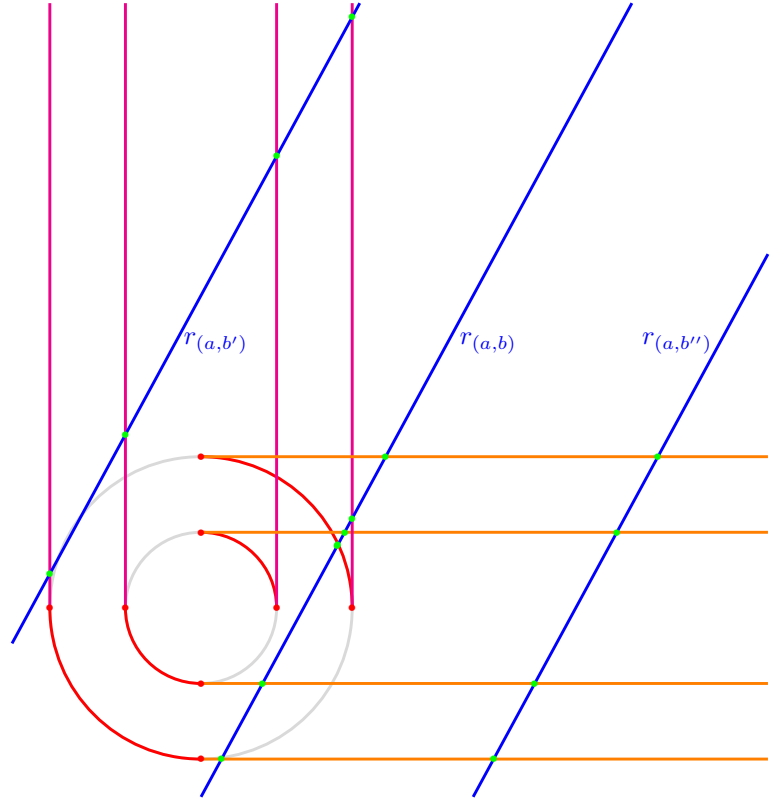
\begin{figure}
\begin{center}
\begin{tikzpicture}[scale=1.0]
\draw [gray!30,line width=1.1] (0,0) circle (1);
\draw [gray!30,line width=1.1] (0,0) circle (2);
\draw [red,line width=1.1] (1,0) arc (0:90:1);
\draw [red,line width=1.1] (2,0) arc (0:90:2);
\draw [red,line width=1.1] (-1,0) arc (180:270:1);
\draw [red,line width=1.1] (-2,0) arc (180:270:2);

\foreach \x in {-2,-1,1,2}
{
	\draw [magenta,line width=1.1] (\x,0) -- (\x,8);
	\draw [orange,line width=1.1] (0,\x) -- (7.5,\x);

	\draw [red] (\x,0) node {{\tiny $\bullet$}};
	\draw [red] (0,\x) node {{\tiny $\bullet$}};
}

\draw [blue,line width=1.1] (-2.5,-0.47) -- (2.1,8);
\draw [blue] (0.2,3.5) node {$r_{(a,b')}$};
\draw [blue,line width=1.1] (0,-2.5) -- (5.7,8);
\draw [blue] (3.8,3.5) node {$r_{(a,b)}$};
\draw [blue,line width=1.1] (3.6,-2.5) -- (7.5,4.68);
\draw [blue] (6.3,3.5) node {$r_{(a,b'')}$};

\draw [green] (-2,0.45) node {{\tiny $\bullet$}};
\draw [green] (-1,2.29) node {{\tiny $\bullet$}};
\draw [green] (1,5.98) node {{\tiny $\bullet$}};
\draw [green] (2,7.82) node {{\tiny $\bullet$}};

\draw [green] (0.27,-2) node {{\tiny $\bullet$}};
\draw [green] (0.81,-1) node {{\tiny $\bullet$}};
\draw [green] (1.8,0.82) node {{\tiny $\bullet$}};
\draw [green] (1.81,0.835) node {{\tiny $\bullet$}};
\draw [green] (1.9,1) node {{\tiny $\bullet$}};
\draw [green] (2,1.18) node {{\tiny $\bullet$}};
\draw [green] (2.44,2) node {{\tiny $\bullet$}};

\draw [green] (3.87,-2) node {{\tiny $\bullet$}};
\draw [green] (4.41,-1) node {{\tiny $\bullet$}};
\draw [green] (5.5,1) node {{\tiny $\bullet$}};
\draw [green] (6.04,2) node {{\tiny $\bullet$}};
\end{tikzpicture}
\end{center}\caption{When $b<0$ and $|b|$ is large enough, the line $r_{(a,b)}$ intersects only the vertical half-lines
in the extended Pareto grid. When $b>0$ and $|b|$ is large enough, the line $r_{(a,b)}$ intersects only the horizontal half-lines
in the extended Pareto grid.}
\label{figgrid5}
\end{figure}

\section{Pairing of contour-arcs}
\label{pairing}

\begin{definition}\label{defn_sing_pair}
A pair $(a,b)\in \, ]0,1[\times\R$ is a \myemph{singular pair} for $\p:M\to\R^2$ if and only if the set  $\dgm\left(\p_{(a,b)}^*\right)\setminus \{\Delta\}$ contains at least one (proper or improper) cornerpoint having multiplicity strictly greater than 1. 
A pair $(a,b)$ that is not singular is called a \myemph{regular pair}. The sets of all singular pairs and all regular pairs for $\p$ will be denoted by the symbols
\myemph{$\mathrm{Sing}(\p)$} and \myemph{$\mathrm{Reg}(\p)$}, respectively.
\end{definition}

\begin{definition}
Two contour-arcs $\gamma_1,\gamma_2$ for $\p$ are called \myemph{paired} with respect to $r_{(a,b)}$ if $r_{(a,b)}$ intersects both $\gamma_1$ and $\gamma_2$ in two respective points $A=(x^1_A,x^2_A)$, $B=(x^1_B,x^2_B)$, and $\frac{\min\{a,1-a\}}{a}\cdot \left(x^1_A-b,x^1_B-b\right)\in \dgm\left(\p_{(a,b)}^*\right)$ (or, equivalently, $\frac{\min\{a,1-a\}}{1-a}\cdot \left(x^2_A+b,x^2_B+b\right)\in \dgm\left(\p_{(a,b)}^*\right)$).
Note that this implies $A\neq B$.
\end{definition}

In plain words, two contour-arcs are paired with respect to $r_{(a,b)}$ if one of them is associated with the birth of a homological class in the filtration given by $\p_{(a,b)}^*$ and the other is associated with the death of the same homological class in the same filtration.
We underline that each contour-arc can be paired to different contour-arcs with respect to different admissible lines.

\begin{exercise}
If the contour-arcs $\gamma_1,\gamma_2$ are paired with respect to $r_{(a,b)}$, and $r_{(a',b')}$ is an admissible line intersecting both $\gamma_1$ and $\gamma_2$ at two respective points $C=(x^1_C,x^2_C)$, $D=(x^1_D,x^2_D)$, then $\gamma_1,\gamma_2$ are paired with respect to $r_{(a',b')}$ as well. 

[\emph{Hint:} Use the Position Theorem~\ref{GT*} and the Matching Distance Stability Theorem~\ref{matchingstabilitythm}.]
\end{exercise}

\begin{exercise}
Find an example of a function $\p\colon M\to \R^2$ such that a contour-arc is not paired with any other with respect to every line $r_{(a,b)}$.
\end{exercise}

\subsection{Localization of singular pairs}
\label{locsing}

The Position Theorem allows us to deduce where singular pairs can be in ${\mathcal {P}}(\Lambda^+)$.

\begin{proposition}\label{propldp}
Let $\left(\bar a,\bar b\right)\in \mathrm{Sing}(\p)$. If $\dgm\left(\p_{\left(\bar a,\bar b\right)}^*\right)$ contains a proper cornerpoint of multiplicity greater than 1, then $r_{\left(\bar a,\bar b\right)}$ contains at least two points of $\mathcal{D}(\p)$.
If $\dgm\left(\p_{\left(\bar a,\bar b\right)}^*\right)$ contains an improper cornerpoint of multiplicity greater than 1, then $r_{\left(\bar a,\bar b\right)}$ contains at least one point of $\mathcal{D}(\p)$.
\end{proposition}

\begin{proof}
Let us first assume that $\dgm\left(\p_{\left(\bar a,\bar b\right)}^*\right)$ contains a proper cornerpoint of multiplicity $n>1$.
By Assumption~\ref{defnassumptions}, we can find a line $r_{(a',b')}$ that is arbitrarily close to $r_{\left(\bar a,\bar b\right)}$ and such that $(a', b')$ is not in $\mathcal{S}(\p)$.
Because of the Matching Distance Stability Theorem~\ref{matchingstabilitythm}, $\dgm\left(\p_{(a',b')}^*\right)$ must contain $n$ proper cornerpoints arbitrarily close to each other. 
Therefore, the Position Theorem~\ref{GT*} and the definition of $\mathcal{D}(\p)$ imply that $r_{\left(\bar a,\bar b\right)}$ contains two double points of $\Gamma(\p)$.

Let us now assume that $\dgm\left(\p_{\left(\bar a,\bar b\right)}^*\right)$ contains an improper cornerpoint of multiplicity $n$.
Also in this case let us consider a line $r_{(a',b')}$ that is arbitrarily close to $r_{\left(\bar a,\bar b\right)}$ and such that $(a', b')$ is not in $\mathcal{S}(\p)$.
Because of the Matching Distance Stability Theorem~\ref{matchingstabilitythm}, $\dgm\left(\p_{(a',b')}^*\right)$ must contain $n$ improper points arbitrarily close to each other.
Therefore, the Position Theorem~\ref{GT*} and the definition of $\mathcal{D}(\p)$ imply that $r_{\left(\bar a,\bar b\right)}$ contains at least one double point of $\Gamma(\p)$.
\end{proof}

Figure~\ref{figdp} illustrates the statement in Proposition~\ref{propldp} in the case of a proper double cornerpoint of $\dgm\left(\p_{\left(\bar a,\bar b\right)}^*\right)$.

\begin{figure}
\begin{center}
\begin{tikzpicture}[scale=1.0]
\draw [line width=1.2] (0,0) -- (6,5);
\draw (5.8,4.3) node {$r_{\left(\bar a,\bar b\right)}$};
\draw [red,line width=1.2] (2,0.28) arc (30:70:3);
\draw [red,line width=1.2] (1.52,0) arc (0:40:3);
\draw [line width=0.1,black,fill=white] (1.31,1.09) circle (0.04);

\draw [red,line width=1.2] (5.5,3.20) arc (30:70:3);
\draw [red,line width=1.2] (5.02,2.92) arc (0:40:3);
\draw [line width=0.1,black,fill=white] (4.81,4.01) circle (0.04);
\end{tikzpicture}
\end{center}\caption{A line $r_{\left(\bar a,\bar b\right)}$ associated with a singular pair $\left(\bar a,\bar b\right)\in{\mathcal {P}}(\Lambda^+)$, in case $\dgm\left(\p_{\left(\bar a,\bar b\right)}^*\right)$ contains a proper double cornerpoint. Parts of four contours (split in eight proper contour-arcs) are displayed in red.}
\label{figdp}
\end{figure}
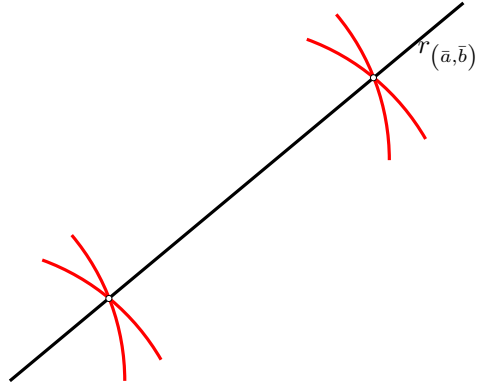

We conclude this subsection by stating some useful results.

\begin{corollary}\label{corfiniteness}
$\dgm\left(\p_{(a,b)}^*\right)\setminus \{\Delta\}$ is finite for any $(a,b)\in{\mathcal {P}}(\Lambda^+)$.
\end{corollary}

\begin{proof}
Assume by contradiction that $\dgm\left(\p_{(a,b)}^*\right)\setminus \{\Delta\}$ is infinite. 
Assumption~\ref{BI_ass_fingen} and Proposition~\ref{jhfgsrelifshejz} imply that the multiplicity of every cornerpoint is well-defined and finite. 
Thus, there exists an infinite sequence of cornerpoints, each with finite multiplicity. 
By the Position Theorem~\ref{GT*}, there exists a sequence of points in $\Gamma(\p)\cap r_{(a,b)}$ giving the finite coordinates of cornerpoints in the sequence above. 
Since the contours have negative slope (Assumption~\ref{ass_wan}.4), each of them can intersect the line $r_{(a,b)}$ at most once.
This leads to a contradiction because the contours are finitely many, by Assumption~\ref{ass_wan}.2.
\end{proof}

\begin{corollary}\label{corsp}
The following statements hold:
\begin{enumerate}
\item The set of all pairs $(a,b)\in\mathrm{Sing}(\p)$ such that $\dgm\left(\p_{(a,b)}^*\right)$ contains a proper cornerpoint of multiplicity greater than 1 is finite.
\item The set of all pairs $(a,b)\in\mathrm{Sing}(\p)$ such that $\dgm\left(\p_{(a,b)}^*\right)$ contains an improper cornerpoint of multiplicity greater than 1 is a finite union of open segments with endpoints $(0,b_1)$ and $(1, b_2)$, for some $b_1$ and $b_2$.
\end{enumerate}
\end{corollary}

\begin{proof}
The first statement follows from Proposition~\ref{propldp} and Assumption~\ref{defnassumptions}, by recalling the inclusion $\mathcal{D}(\p)\subseteq\mathcal{S}(\p)$. 

As for the second statement, let us assume that the set $\dgm\left(\p_{\left(\bar a,\bar b\right)}^*\right)$ contains an improper cornerpoint $P$ of multiplicity greater than 1. 
Because of Proposition~\ref{propldp}, $r_{\left(\bar a,\bar b\right)}$ contains a point $A\in\mathcal{D}(\p)$ such that $P=(x^1_A-\bar b, \infty)$ or $P=(x^2_A+\bar b, \infty)$.
It is easy to check that the sheaf of positive slope lines passing through $A$ corresponds to a segment $S_A$ in ${\mathcal {P}}(\Lambda^+)$ whose closure has $(0,b_1)$ and $(0,b_2)$ as endpoints, for some $b_1$ and $b_2$.
Observe that rotating the line $r_{(\bar a,\bar b)}$ around $A$, the only intersection of the rotated line with $\Gamma(\p)$ with abscissa close enough to $x^1_A$ is $A$ itself.
By the Matching Distance Stability Theorem~\ref{biparam_matchingstabilitythm},  the cornerpoint $(x^i_A-b, \infty)$ in $\dgm\left(\p_{(a,b)}^*\right)$ has multiplicity greater than 1, for every $(a,b)\in S_A$. 
We can then conclude by observing that $\mathcal{D}(\p)\subseteq \mathcal{S}(\p)$ is finite by Assumption~\ref{defnassumptions}.
\end{proof}

The results proved in this subsection are illustrated by the following example.

\begin{ex}\label{excor}
Let us consider the union $M$ of two disjoint spherical surfaces in $\R^3$, having the extended Pareto grid represented in Figure~\ref{figcor} with respect to the filtering function $\p$ that takes each point $x\in M$ to the pair $\p(x)=(x^1(x),x^2(x))$.
Let us consider the admissible line $r_{\left(\bar a,\bar b\right)}$ containing the double points $A$ and $B$.
It is easy to check that $\left(\bar a,\bar b\right)\in \mathrm{Sing}(\p)$ and that a proper cornerpoint of multiplicity greater than 1 is found in degree 1.
Moreover, if the line $r_{(a,b)}$ contains the double point $C$, then $(a,b)\in \mathrm{Sing}(\p)$ and an improper cornerpoint of multiplicity greater than 1 is found in degree $0$. 
Furthermore, if the line $r_{(a,b)}$ contains the double point $D$, then $(a,b)\in \mathrm{Sing}(\p)$ and an improper cornerpoint of multiplicity greater than 1 is found in degree $2$.
\end{ex}

\begin{figure}
\begin{center}
\begin{tikzpicture}[scale=0.5]

\draw [red,line width=1.1] (3,1.25) arc (0:90:4);
\draw [red,line width=1.1] (-5,1.25) arc (180:270:4);

\draw [red,line width=1.1] (-1,-2.75) -- (12,-2.75);
\draw [red,line width=1.1] (-5,1.25) -- (-5,13);

\draw [red,line width=1.1] (-1,5.25) -- (12,5.25);
\draw [red,line width=1.1] (3,1.25) -- (3,13);

\draw [red,line width=1.1] (4,0) arc (0:90:4);
\draw [red,line width=1.1] (-4,0) arc (180:270:4);

\draw [red,line width=1.1] (0,-4) -- (12,-4);
\draw [red,line width=1.1] (-4,0) -- (-4,13);

\draw [red,line width=1.1] (0,4) -- (12,4);
\draw [red,line width=1.1] (4,0) -- (4,13);

\draw [blue,line width=1.1] (-1.55,-4) -- (8.2,13);

\draw [black] (2.55,3.1) node {{\tiny $\bullet$}};
\draw [black] (3,4) node {{\tiny $\bullet$}};
\draw [black] (-3.55,-1.8) node {{\tiny $\bullet$}};
\draw [black] (4,5.25) node {{\tiny $\bullet$}};

\draw (9,12.5) node {$r_{\left(\bar a,\bar b\right)}$};

\draw (2.4,2.5) node {$A$};

\draw (3.4,3.6) node {$B$};

\draw (-4,-2.3) node {$C$};

\draw (4.5,5.75) node {$D$};

\end{tikzpicture}
\end{center}\caption{
The extended Pareto grid of the manifold $M$ described in Example~\ref{excor}, with respect to the filtering function $\p$ that takes each point $x\in M$ to the pair $\p(x)=(x^1(x),x^2(x))$. The blue line corresponds to a singular pair of ${\mathcal {P}}(\Lambda^+)$ such that the associated persistence diagram in degree 1 has a proper cornerpoint of multiplicity greater than 1.
}\label{figcor}
\end{figure}
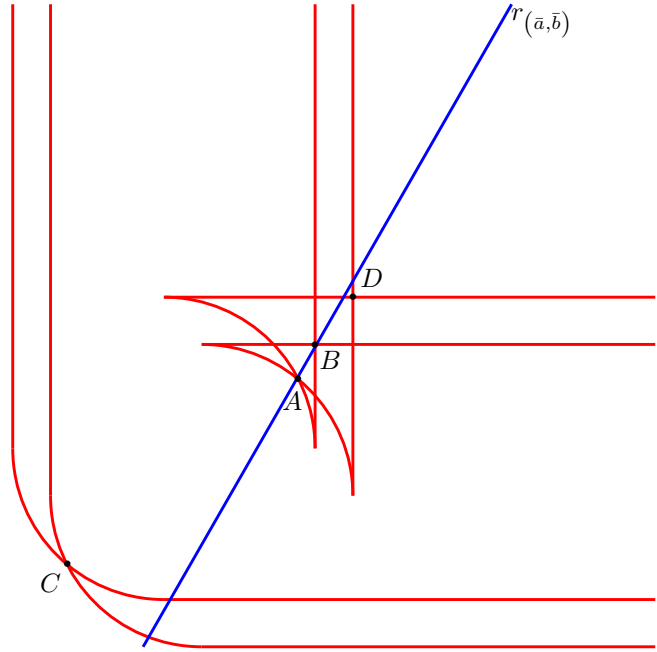

\subsection{Creation and destruction of cornerpoints}
\label{C&D}

In this subsection we show how to detect where in $\Delta$ a proper cornerpoint of $\dgm\left(\p_{(a,b)}^*\right)$ can be created or destroyed, when $(a,b)$ varies. 
This is described in Proposition~\ref{propDelta}, which is a consequence of the Position Theorem~\ref{GT*}.

\begin{definition}
Let $\gamma_1,\gamma_2$ be two paired contour-arcs (with respect to any line $r_{(a,b)}$ that intersects both of them).
If $\gamma_1$ and $\gamma_2$ have a common endpoint $A=(x^1_A,x^2_A)$, it is called an \myemph{annihilation crossing} for $\p$ associated with the contour-arcs $\gamma_1,\gamma_2$.
\end{definition}

By definition, each annihilation crossing for $\p$ belongs to the set $\mathcal{D}(\p)$.
The set of all annihilation crossings for $\p$ is denoted by $\mathcal{A}(\p)$.

\begin{definition}
Let $(a,b)\in{\mathcal {P}}(\Lambda^+)$ and $(\bar u,\bar u)\in\Delta$.
If for every $\delta>0$ a pair $(a',b')\in{\mathcal {P}}(\Lambda^+)$ and a proper cornerpoint $(u',v')\in \dgm\left(\p^*_{(a',b')}\right)\setminus \{\Delta\}$ exist with
$|a-a'|,|b-b'|< \delta$ and $|\bar u-u'|,|\bar u-v'|<\delta$, then $(\bar u,\bar u)$ is called an \myemph{annihilation point} at $(a,b)$ for $\p$.
\end{definition}

In plain words, the annihilation points at $(a,b)$ are the locations on the diagonal $\Delta$ at which the proper cornerpoints of the persistence diagram $\dgm\left(\p_{(a,b)}^*\right)$ can appear and disappear when we slightly change $(a,b)$.

\begin{proposition}\label{propDelta}
A point $A=(x^1_A, x^2_A)\in\R^2$ is an annihilation crossing for $\p$ if and only if, for every
$r_{(a,b)}$ containing $A$, the point $(\bar u,\bar u)$ with $\bar u=\frac{\min\{a,1-a\}}{a}\cdot (x^1_A-b)=\frac{\min\{a,1-a\}}{1-a}\cdot (x^2_A+b)$ is an annihilation point at $(a,b)$ for $\p$.
\end{proposition}

\begin{exercise}
By applying the Position Theorem~\ref{GT*} and the Matching Distance Stability Theorem~\ref{matchingstabilitythm} prove Proposition~\ref{propDelta}.
\end{exercise}


This corollary immediately follows.

\begin{corollary}\label{corDelta}
Let $(a(t),b(t))$ be a continuous curve in ${\mathcal {P}}(\Lambda^+)$ such that the value 
$\inf\{\max\{\lvert u-\bar u\rvert, \lvert v-\bar u\rvert\}\mid (u,v)\in \dgm(\p^*_{(a(t), b(t))})\setminus\{\Delta\}\}$
tends to $0$ for $t\to \bar t$.
Then an annihilation crossing $A=(x^1_A,x^2_A)\in \mathcal{A}(\p)$ exists with $\bar u=\frac{\min\{a(\bar t),1-a(\bar t)\}}{a(\bar t)}\cdot (x^1_A-b(\bar t))=\frac{\min\{a(\bar t),1-a(\bar t)\}}{1-a(\bar t)}\cdot (x^2_A+b(\bar t))$.
\end{corollary}

The previous result shows that points of $\dgm\left(\p^*_{(a(t),b(t))}\right)$ can be created or destroyed only when the line $r_{(a(t),b(t))}$ reaches an annihilation crossing $A=(x^1_A,x^2_A)\in \mathcal{A}(\p)$ (see Figure~\ref{figdestruction}).
In this case, the creation or destruction occurs at the annihilation point $(\bar u,\bar u)$ at
$\left(a,b\right)$ with $\bar u=\frac{\min\{a(\bar t),1-a(\bar t)\}}{a(\bar t)}\cdot (x^1_A-b(\bar t))$.

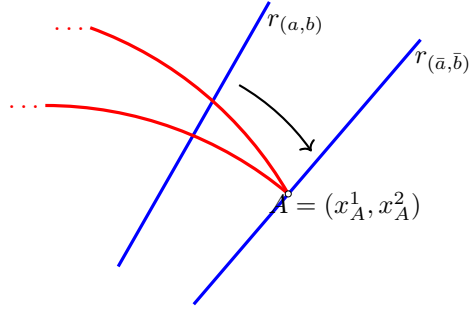
\begin{figure}
\begin{center}
\begin{tikzpicture}[scale=1.0]
\draw [blue,line width=1.2] (2,0.5) -- (4,4);
\draw [blue,line width=1.2] (3,0) -- (6,3.5);
\draw (4.35,3.7) node {$r_{(a,b)}$};
\draw (6.3,3.2) node {$r_{(\bar{a},\bar{b})}$};

\draw [red,line width=1.2] (4.25,1.46) arc (30:70:5);
\draw [red,line width=1.2] (4.25,1.46) arc (50:90:5);

\draw [line width=0.1,black,fill=white] (4.25,1.46) circle (0.04);
\draw (5,1.3) node {$A=(x^1_A,x^2_A)$};

\draw [red] (0.75,2.62) node {$\ldots$};
\draw [red] (1.35,3.65) node {$\ldots$};

\draw [->,line width=0.8] (3.6,2.9) arc (60:35:3);
\end{tikzpicture}
\end{center}\caption{When $(a,b)$ moves towards $(\bar{a},\bar{b})$ and, correspondingly, the line $r_{(a,b)}$ moves and meets an annihilation crossing $A\in \mathcal{A}(\p)$, a proper cornerpoint of $\dgm\left(\p_{(a,b)}^*\right)$ reaches the diagonal $\Delta$ at an annihilation point $(\bar u,\bar u)$ and disappears. By reversing the movement of $r_{(a,b)}$ we get the birth of a proper cornerpoint of $\dgm\left(\p_{(a,b)}^*\right)$. Parts of two contour-arcs of $\p$ are displayed in red.}
\label{figdestruction}
\end{figure}


\section{Monodromy in biparameter persistent homology}
\label{Monodromy}

We conclude this chapter by illustrating the phenomenon of monodromy in biparameter persistent homology.
We warn the reader that this section requires some notions from algebraic topology, such as homotopy and the fundamental group.
None of the subsequent sections depend on this material.

Monodromy occurs whenever a loop in the parameter space $\mathcal{P}(\Lambda^+)$ induces a non-trivial permutation of the cornerpoints of the persistence diagram  \cite{CeEtFr19}. 
This section is devoted to explain the meaning of this and its implications in the transport of matchings between persistence diagrams.

Given a connected open subset $U\subset \mathcal{P}(\Lambda^+)$ and a constant $c>0$, consider all the functions $\p\colon M\to \R^2$, satisfying the properties of the previous sections, such that $U \subseteq \mathrm{Reg}(\p)$ and, 
for any two 
distinct cornerpoints $P, Q$ in $\dgm(\p^*_{(a,b)})$, $d(P, Q)>2c$. 
Let us denote the set of these functions by $\mathcal{F}_{U, c}.$
In the rest of the section we will only consider $\p\in \mathcal{F}_{U,c}$.
By Corollary~\ref{corsp}, the distance between $U$ and $\mathrm{Sing}(\p)$ is positive.
The conditions above guarantee that the persistence diagrams $\dgm(\varphi^*_{(a,b)})$ have finitely many proper and improper cornerpoints, that each of them has multiplicity $1$, and that their distance from the diagonal $\Delta$ is greater than $2c$.

A continuous function $\gamma\colon [0,1]\to A$, where $A$ is a topological space, is called a \myemph{path}.
Its inverse path is denoted by $\pi^{-1}$, and, for any two paths $\pi$ and $\pi'$, $\pi\ast \pi'$ denotes their concatenation.

\begin{definition}
A path $P\colon [0,1]\to \bar \Delta^*$ is said to be \myemph{induced by} the path $\pi\colon [0,1]\to \mathcal{P}(\Lambda^+)$ if 
$P(t)\in \dgm(\p^*_{\pi(t)})$,
for every $t\in [0,1]$.
\end{definition}

Before proceeding, we need the following result \cite{EtFrQuTo23}.

\begin{proposition}\label{ghjywuikdc}
For every $0<a,a'<1$ and every $b,b'\in \R$,
\[
\left\lVert \p^*_{(a,b)}-\p^*_{(a',b')}\right\rVert_{\infty}\le 4\lvert a-a'\rvert (\lVert \p\rVert_\infty+\lvert b\rvert)+ 3\lvert b-b'\rvert.
\]
\end{proposition}
\begin{proof}
We split the proof into three cases: $a,a'\le \frac{1}{2}$, $a,a'\ge \frac{1}{2}$, and $a\le \frac{1}{2}$, $a'\ge \frac{1}{2}$ (the case where $a$ and $a'$ are inverted is symmetrical). 

If $a,a'\le \frac{1}{2}$, then $\min\left\{a,1-a\right\}=a$ and $\min\left\{a',1-a'\right\}=a'$. 
Therefore, recalling that $\lvert \max \left\{\alpha,\beta\right\}-\max\{\gamma,\delta\}\rvert\le \max\{\lvert\alpha-\gamma\rvert,\lvert\beta-\delta\rvert\}$ and observing that $(1-a)(1-a')\ge \frac{1}{4}$,
\begingroup
\allowdisplaybreaks
\begin{align*}
& \left\lVert \p^*_{(a,b)}-\p^*_{(a',b')}\right\rVert_\infty = 
\left\lVert a\max \left\{\frac{\p_1-b}{a},  \frac{\p_2+b}{1-a}\right\}-
a' \max \left\{\frac{\p_1-b'}{a'},\frac{\p_2+b'}{1-a'}\right\}
\right\rVert_\infty \nonumber\\
&= \sup_{x\in M} \left\lvert
\max \left\{\p_1(x)-b,  \frac{a}{1-a}(\p_2(x)+b)\right\}  - \max \left\{\p_1(x)-b', \frac{a'}{1-a'}(\p_2(x)+b')\right\} \right\rvert \nonumber\\
&\le \sup_{x\in M} 
\max \left\{\lvert b-b'\rvert,  \left\lvert\frac{a}{1-a}(\p_2(x)+b)-\frac{a'}{1-a'}(\p_2(x)+b')\right\rvert\right\} \nonumber\\
&=\sup_{x\in M}
\max \left\{\lvert b-b'\rvert,  \frac{\left\lvert a\p_2(x)+ab -aa'b-a'\p_2(x)-a'b'+aa'b'\right\rvert}{(1-a)(1-a')}\right\}
\nonumber\\
&\le \max \left\{\lvert b-b'\rvert,  \frac{\lvert a-a'\rvert\lVert \p_2\rVert_\infty+\lvert ab-a'b'\rvert + \lvert b-b'\rvert aa'}{(1-a)(1-a')}\right\} \nonumber\\
&\le
\max \left\{\lvert b-b'\rvert ,  4\lvert a-a'\rvert \lVert \p_2\rVert_\infty+4\lvert ab-a'b'\rvert+ 4\lvert b-b'\rvert aa'\right\} \nonumber\\
&\le
\max \{\lvert b-b'\rvert ,  4\lvert a-a'\rvert \lVert \p_2\rVert_\infty+4\lvert ab-a'b'\rvert+ \lvert b-b'\rvert\} \nonumber\\
&=
4\lvert a-a'\rvert \lVert \p_2\rVert_\infty+4\lvert ab-a'b'\rvert + \lvert b-b'\rvert \nonumber\\
&\le
4\lvert a-a'\rvert\lVert \p_2\rVert_\infty+4\lvert a-a'\rvert \lvert b\rvert +4\lvert b-b'\rvert a'+ \lvert b-b'\rvert \nonumber\\
&\le
4\lvert a-a'\rvert\lVert \p_2\rVert_\infty+4\lvert a-a'\rvert\lvert b\rvert+ 3\lvert b-b'\rvert \nonumber\\
&\le 
4\lvert a-a'\rvert(\lVert \p_2\rVert_\infty+\lvert b\rvert)+ 3\lvert b-b'\rvert\\
&\le 4\lvert a-a'\rvert(\lVert \p\rVert_\infty+\lvert b\rvert)+ 3\lvert b-b'\rvert. \nonumber
\end{align*}
\endgroup

By observing that, if $\p=(\p_1,\p_2)$ and $h=(\p_2,\p_1)$, then $\p^*_{(a,b)}=h^*_{(1-a,-b)}$,
we obtain the proof for $a,a'\ge \frac{1}{2}$.

If $a \le \frac{1}{2}$ and $a' \ge \frac{1}{2}$ and considering $\left(\frac{1}{2},b\right), \left(\frac{1}{2},b'\right)$, we have that
\begin{align*}
     \left\lVert \p^*_{(a,b)} - \p^*_{(a',b')} \right\rVert_\infty & \le  \left\lVert \p^*_{(a,b)} - \p^*_{(\frac{1}{2},b)} \right\rVert_\infty + \left\lVert \p^*_{(\frac{1}{2},b)} - \p^*_{(\frac{1}{2},b')}\right\rVert_\infty + \\ &+ \left\lVert \p^*_{(\frac{1}{2},b')} - \p^*_{(a',b')} \right\rVert_\infty \\ 
    & \le 4\left\lvert a-\frac{1}{2} \right\rvert \left(\lVert \p\rVert_\infty+\lvert b\rvert\right)+ 3\lvert b-b\rvert + 4\left\lvert \frac{1}{2}- \frac{1}{2}\right\rvert (\lVert \p\rVert_\infty+\lvert b\rvert)+ \\
    & + 3\lvert b-b'\rvert  + 4\left\lvert \frac{1}{2}-a'\right\rvert (\lVert \p\rVert_\infty+\lvert b\rvert)+ 3\lvert b'-b'\rvert \\
    & = 4 \left(\left\lvert a-\frac{1}{2}\right\rvert + \left\lvert \frac{1}{2}-a'\right\rvert\right) (\lVert \p\rVert_\infty+\lvert b\rvert) + 3\lvert b-b'\rvert  \\ & = 4\lvert a-a'\rvert (\lVert \p\rVert_\infty+\lvert b\rvert)+ 3\lvert b-b'\rvert.
\end{align*}
\end{proof}

Next result shows that for any path with values in $\mathrm{Reg}(\p)$, an induced path exists and, given its initial point, it is unique. 

\begin{proposition}\label{sdsuhfrlszdchj}
Let $\varphi \in \mathcal{F}_{U,c}$ and let $\pi\colon [0,1]\to U$ be a path.
For every $X$ in $\dgm(\p^*_{\pi(0)})$, a unique path $P\colon [0,1]\to \bar \Delta^*$ induced by $\pi$ exists such that $P(0)=X$.
Moreover, if $X=\Delta$, then $P([0,1])=\{\Delta\}$, otherwise $P([0,1])\cap \{\Delta\}=\varnothing$.
\end{proposition}

\begin{proof}

If $X=\Delta$, the existence of the path is proved by setting $P(t):=\Delta$ for every $t\in[0,1]$.
The uniqueness of this path follows from the assumption $\p \in\mathcal{F}_{U,c}$.
Indeed, if $P'$ is another path induced by $\pi$ with $P'(0)=\Delta$, it must map the connected set $[0,1]$ into a connected set $\mathcal{C} \subseteq \bar{\Delta}^*$. Since $P'(0)=\Delta$, $\mathcal{C}$ must contain the point $\Delta$. Moreover, because $P'(t)\in \dgm(\p^*_{\pi(t)})$ for every $t\in [0,1]$ and every non-trivial cornerpoint of $\dgm(\p^*_{\pi(t)})$ has distance greater than $2c$ from $\Delta$, it follows that $\mathcal{C}=\{\Delta\}$.
Indeed, if this were not the case, we could decompose $\mathcal{C}$ into the two non-empty relatively open sets $\{\Delta\}=\{Y\in\mathcal{C} : d(Y,\Delta)<2c\}$ and $\{Y\in\mathcal{C} : d(Y,\Delta)>2c\}$, contradicting the connectedness of $\mathcal{C}$.  Therefore, $P'=P$.

If \(X \neq \Delta\), the existence and uniqueness of the path can be proved as follows.
Let \(\Theta\) be the set of all values \(\theta \in [0,1]\) such that there exists exactly one path \(P_\theta \colon [0,\theta] \to \Delta^*\) satisfying \(P_\theta(0) = X\) and \(P_\theta(t) \in \dgm\bigl(\varphi_{\pi(t)}^*\bigr) \setminus \{\Delta\}\) for every \(t \in [0,\theta]\).
Clearly, \(0 \in \Theta\).
Moreover, if \(\theta_1, \theta_2 \in \Theta\) and \(\theta_1 \le \theta_2\), then \(P_{\theta_1}\) and \(P_{\theta_2}\) coincide on \([0,\theta_1]\).
Set \(\bar\theta := \sup \Theta\).
Our aim is to show that \(\bar\theta = 1\), which will conclude the proof.

First, under our assumptions, \(\dgm\bigl(\varphi_{\pi(\bar\theta)}^*\bigr)\) is finite (Corollary~\ref{corfiniteness}), and each of its points has multiplicity \(1\), since \(\varphi \in \mathcal{F}_{U,c}\).
Therefore, we can consider the minimum distance \(\varepsilon\) between two distinct points in \(\dgm\bigl(\varphi_{\pi(\bar\theta)}^*\bigr)\).
Second, Proposition~\ref{ghjywuikdc} guarantees that 
\(\lim_{\theta \to \bar\theta} \|\varphi_{\pi(\theta)}^* - \varphi_{\pi(\bar\theta)}^*\|_\infty = 0\),
and Corollary~\ref{corintorno} 
implies that for every proper (respectively improper) cornerpoint 
\(Y \in \dgm\bigl(\varphi_{\pi(\bar\theta)}^*\bigr)\) there exists an open ball 
\(B_{\eta}(Y)\) (respectively open interval of length 2$\eta$ centred at $Y$) containing exactly one cornerpoint of 
\(\dgm\bigl(\varphi_{\pi(\theta)}^*\bigr)\) whenever 
\(\|\varphi_{\pi(\bar\theta)}^* - \varphi_{\pi(\theta)}^*\|_\infty < \eta\le\bar\eta\). 
We can assume that \(\bar\eta\le\frac{\varepsilon}{4}\).
Note that all such balls \(B_{\eta}(Y)\) are pairwise disjoint.
From now on, we just state the proof for $Y$ proper cornerpoint not to continue the dichotomy, but the proof is analogous if $Y$ is improper.

By the stability of the matching distance (Theorem~\ref{matchingstabilitythm}), it follows that there exists exactly one point \(\bar Y \in \dgm\bigl(\varphi_{\pi(\bar\theta)}^*\bigr)\) such that \(P_\theta(\theta)\) belongs to \(B_{\eta}(\bar Y)\) for all \(\theta < \bar\theta\) sufficiently close to \(\bar\theta\).
Indeed, if this were not the case, then there would exist two values \(\tilde\theta_1, \tilde\theta_2 < \bar\theta\), arbitrarily close to \(\bar\theta\), such that
\(P_{\tilde\theta_1}(\tilde\theta_1) \in B_{\frac{\varepsilon}{4}}(Y_1)\) and
\(P_{\tilde\theta_2}(\tilde\theta_2) \in B_{\frac{\varepsilon}{4}}(Y_2)\),
with \(Y_1, Y_2 \in \dgm\bigl(\varphi_{\pi(\bar\theta)}^*\bigr)\) and \(Y_1 \neq Y_2\).
Moreover, we may assume that \(\tilde\theta_1 < \tilde\theta_2\) and that the image \(P_{\tilde\theta_2}([\tilde\theta_1,\tilde\theta_2])\) does not intersect any other open ball \(B_{\frac{\varepsilon}{4}}(Y)\), with \(Y \in \dgm\bigl(\varphi_{\pi(\bar\theta)}^*\bigr)\), except for \(B_{\frac{\varepsilon}{4}}(Y_1)\) and \(B_{\frac{\varepsilon}{4}}(Y_2)\).
Since the paths \(P_\theta\) are compatible for \(\theta < \bar\theta\), the map \(\theta \mapsto P_\theta(\theta)\) is continuous on \([0,\tilde\theta_2]\). Therefore, this would imply the existence of points in \(P_{\tilde\theta_2}([0,\tilde\theta_2])\) whose distance from the set \(\dgm\bigl(\varphi_{\pi(\bar\theta)}^*\bigr)\) is greater than \(\frac{\varepsilon}{4}\), contradicting the stability of the matching distance.

We can now extend the paths \(P_\theta\) to a unique path \(P_{\theta'} \colon [0,\theta'] \to \Delta^*\), where \(\theta' = \bar\theta\) if \(\bar\theta = 1\), and \(\bar\theta < \theta' < 1\) otherwise.
This is done by defining \(P_{\theta'}(\theta)\) as the unique point in 
\(B_{\eta}(\bar Y) \cap \dgm\bigl(\varphi_{\pi(\theta)}^*\bigr)\)
for \(\theta\) sufficiently close to \(\theta'\).

By construction, \(P_{\theta'}\) is continuous and unique, satisfies \(P_{\theta'}(0) = X\), and
\(P_{\theta'}(\theta) \in \dgm\bigl(\varphi_{\pi(\theta)}^*\bigr) \setminus \{\Delta\}\) for every \(\theta \in [0,\theta']\), hence \(\theta' \in \Theta\).
Since \(\bar\theta = \sup \Theta\), it follows that \(\theta' = \bar\theta\), and therefore \(\bar\theta = 1\).

\end{proof}

By Proposition~\ref{sdsuhfrlszdchj}, the following definition is well posed. 

\begin{definition}
Given a path $\pi\colon [0,1]\to \mathcal{P}(\Lambda^+)$ and a cornerpoint $X$ in the persistence diagram $\dgm(\p^*_{\pi(0)})$, the path $P$, given as in Proposition~\ref{sdsuhfrlszdchj}, is said to \myemph{transport} $X$ to $X'$ in $\dgm(\p^*_{\pi(1)})$. 
In this case, we write $T_\pi^\p(X)=X'$. 
\end{definition}

The function $T_\pi^\p$ is a bijection between $\dgm(\p^*_{\pi(0)})$ and $\dgm(\p^*_{\pi(1)})$, whose inverse is $T_{\pi^{-1}}^\p$. 
Furthermore, it is immediate to check that $T^\p_{\pi\ast\pi'}=T^\p_{\pi'}T^\p_{\pi}$, for any two paths $\pi, \pi'\colon [0,1]\to U$ with $\pi(1)=\pi'(0)$.

Our next goal is to describe under what circumstances different paths may induce the same transport. 
With this in mind, we show that the transport $T^\p_\pi$ continuously depends on the path $\pi$, and a preliminary result to that. 


Proposition~\ref{ghjywuikdc} allows us to prove the following result, implying that the transport along a path in $U$ is continuous with respect to changes in the path.

\begin{proposition}\label{transcontinuity}
Let $\p$ be in $\mathcal{F}_{U,c}$.
Let $\bar \pi=\left(\bar \pi^a,\bar \pi^b\right)\colon[0,1]\to U\subset ]0,1[\times \R$ be a path.
If $\pi\colon[0,1]\to U$ is a path such that $\pi(0)=\bar\pi(0)$ and 
$\left\|\bar\pi-\pi\right\|_\infty\le \frac{c}{C}$ with $C=4(\lVert \p\rVert_\infty+\lVert \bar \pi^b\rVert_\infty)+3$, then the inequality $d(T^\p_{\bar\pi}(X),T^\p_{\pi}(X))\le C\left\|\bar\pi-\pi\right\|_\infty$ holds for every $X\in \dgm\left(\p_{\bar \pi(0)}^*\right)$.
\end{proposition}


\begin{proof}
If $X=\Delta$, then both $T^\p_\pi(X)$ and $T^\p_{\bar \pi}(X)$ equal $\Delta$, by Proposition~\ref{sdsuhfrlszdchj}, so the inequality holds.
We can now suppose that $X\neq \Delta$.
By Proposition~\ref{sdsuhfrlszdchj}, there is a unique path $\bar P\colon [0,1]\to\Delta^*$ induced by $\bar \pi$ such that $\bar P(0)=X$ and $\bar P(1)=T_{\bar\pi}^\p(X)$, and a unique path $P\colon[0,1]\to\Delta^*$ induced by $\pi$ such that $P(0)=X$ and $P(1)=T_{\pi}^\p(X)$.
Let us set
\[
\theta:= \max\left\{\tau\in [0,1]:d(T_{\bar\pi_\tau}^\p(X),T_{\pi_\tau}^\p(X))\le
C\|\bar\pi-\pi\|_{\infty} \right\},
\]
where the paths $\bar \pi_\tau,\pi_\tau\colon[0,1]\to U$ are defined by setting $\bar\pi_\tau(t):=\bar\pi(\tau t)$ and
$\pi_\tau(t):=\pi(\tau t)$ for every $t\in [0,1]$. 
This maximum exists because $T^\p_{\gamma_\tau}(X)$ depends continuously on $\tau$ by Proposition~\ref{sdsuhfrlszdchj}, for any $\gamma$.
In plain words, $T_{\bar\pi_\tau}^\p(X)$ and $T_{\pi_\tau}^\p(X)$ represent, respectively, the transport of $X$ along $\bar\pi$ and $\pi$ with respect to $\p$ until time $\tau$ instead of the usual final time $1$.
We observe that $T_{\bar\pi_0}^\p(X)=T_{\pi_0}^\p(X)=X$.
Moreover, for every $\tau\in [0,1]$,
\begin{equation}\label{edoiescrfhnhkerjn}
\left\|\bar\pi_\tau-\pi_\tau\right\|_\infty\le\left\|\bar\pi-\pi\right\|_\infty.
\end{equation}

If $\theta<1$, then on the one hand we can find a $\theta_+\in\ ]\theta,1]$ arbitrarily close to $\theta$ such that
$d(T_{\bar\pi_{\theta_+}}^\p(X),T_{\pi_{\theta_+}}^\p(X))> C\left\|\bar\pi-\pi\right\|_\infty$ and
$d(T_{\bar\pi_{\theta_+}}^\p(X),T_{\pi_{\theta_+}}^\p(X))$ is arbitrarily close to $C\left\|\bar\pi-\pi\right\|_\infty
$.
We recall that $T_{\bar\pi_{\theta_+}}^\p(X)\in \dgm\left(\p^*_{\bar\pi_{\theta_+}(1)}\right)\setminus\{\Delta\}$ and
$T_{\pi_{\theta_+}}^\p(X)\in \dgm\left(\p^*_{\pi_{\theta_+}(1)}\right)\setminus\{\Delta\}$.
On the other hand, for every $t\in[0,1]$ the inequalities 
\begin{equation}\label{eiwjoejwocw}
|\bar \pi^a(t)-\pi^a(t)|,|\bar \pi^b(t)-\pi^b(t)|\le \left\|\bar\pi-\pi\right\|_\infty\le \frac{c}{C}
\end{equation}
hold.
By applying Proposition~\ref{ghjywuikdc} together with Equation~~\ref{edoiescrfhnhkerjn} and~
\ref{eiwjoejwocw}, the stability of the matching distance (Theorem~\ref{matchingstabilitythm}), and the hypothesis $\left\|\bar\pi-\pi\right\|_\infty\le c/
C$
we obtain the inequalities
\begin{equation}\label{sdcwejcweilwc}
d_B\left(\dgm\left(\p_{\bar\pi_{ \theta_+}(1)}^*\right),\dgm\left(\p_{\pi_{\theta_+}(1)}^*\right)\right)
\le C\|\bar\pi-\pi\|_{\infty}\le c,
\end{equation}
so that a cornerpoint $Y_{\theta_+}\in \dgm\left(\p^*_{\pi_{\theta_+}(1)}\right)$ exists such that
$d\left(Y_{\theta_+},T_{\bar\pi_{\theta_+}}^\p(X)\right)\le C\|\bar\pi-\pi\|_{\infty}
\le c$.
Since $\p\in \mathcal{F}_{U,c}$, we have that $d\left(T_{\bar\pi_{\theta_+}}^\p(X),\Delta\right)> 2c$.
Therefore,
$d\left(Y_{\theta_+},\Delta\right)\ge d\left(T_{\bar\pi_{\theta_+}}^\p(X),\Delta\right)-d\left(Y_{\theta_+},T_{\bar\pi_{\theta_+}}^\p(X)\right)>2c-c=c$.
It follows that $Y_{\theta_+}\neq\Delta$ and
\[
d\left(Y_{\theta_+},T_{\bar\pi_{\theta_+}}^\p(X)\right)
\le C \|\bar\pi-\pi\|_{\infty}
<
d(T_{\pi_{\theta_+}}^\p(X),T_{\bar\pi_{\theta_+}}^\p(X)),
\]
so that $Y_{\theta_+}\neq T_{\pi_{\theta_+}}^\p(X)$.
Analogously to Equation~\ref{sdcwejcweilwc}, we have
\[
d_B\left(\dgm\left(\p_{\bar\pi_{ \theta}(1)}^*\right),\dgm\left(\p_{\pi_{\theta}(1)}^*\right)\right)
\le C\|\bar\pi-\pi\|_{\infty}\le c.
\]
Since $\theta_+$ is arbitrarily close to $\theta$, Theorem~\ref{matchingstabilitythm} implies that
a cornerpoint $Z\in \dgm\left(\p^*_{\pi_{\theta}(1)}\right)$ exists such that
the inequality $d\left(Z,T_{\bar\pi_{\theta}}^\p(X)\right)\le C \|\bar\pi-\pi\|_{\infty}$ holds,
where $Z$ is the limit of the previously considered points $Y_{\theta_+}$. 
Furthermore, we have 
$d(T_{\pi_{\theta}}^\p(X),T_{\bar\pi_{\theta}}^\p(X))\le C\|\bar\pi-\pi\|_{\infty}$.
If $Z\neq T_{\pi_{\theta}}^\p(X)$, then $\dgm\left(\p^*_{\pi_{\theta}(1)}\right)$ contains at least two points, $Z$ and $T_{\pi_{\theta}}^\p(X)$, that have a distance less than or equal to $C\|\bar\pi-\pi\|_{\infty}$ from $T_{\bar\pi_{\theta}}^\p(X)$, and hence these two points have a distance less than or equal to $2C\|\bar\pi-\pi\|_{\infty}\le 2c$ from each other. 
If $Z= T_{\pi_{\theta}}^\p(X)$, then this point is double in $\dgm\left(\p^*_{\pi_{\theta}(1)}\right)$,
because of the construction of $Z$ and the inequality $Y_{\theta_+}\neq T_{\pi_{\theta_+}}^\p(X)$.
Both cases contradict the assumption that $\p\in \mathcal{F}_{U,c}$.

Therefore, if $\left\|\bar\pi-\pi\right\|_\infty\le \frac{c}{C}$, then $\theta=1$, and hence, $d(T_{\bar\pi}^\p(X),T_{\pi}^\p(X))=
d(T_{\bar\pi_1}^\p(X),T_{\pi_1}^\p(X))
\le C \left\|\bar\pi-\pi\right\|_\infty$.
\end{proof}


A consequence of Proposition~\ref{transcontinuity} is that the transport
along a path in $U$ is continuous with respect to changes in the path. 
From this and the discreteness of cornerpoints in a persistence diagram, we obtain the following.

\begin{proposition}\label{homotopy}
If two paths $\pi,\pi'$ in $U$ are homotopic to each other relatively to their common extrema, then $T^\p_{\pi}(X)= T^\p_{\pi'}(X)$, for every $X$ in $\dgm(\p^*_{\pi(0)})$.
\end{proposition}

\begin{proof}
Suppose, by contradiction, that there exists $X$ in $\dgm(\p^*_{\pi(0)})$ with $Y=T^\p_{\pi}(X)\neq T^\p_{\pi'}(X)=Y'$. 
Because $T^\p_{\pi'}$ is a bijection, there must be $X'\neq X$ in $\dgm(\p^*_{\pi(0)})$ such that $T^\p_{\pi'}(X')=Y$.
We note that the concatenation $\pi'^{-1}\ast \pi$ is a path in $U$ and, because $\pi$ and $\pi'$ are homotopic, it is homotopic to the constant path on $\pi(0)$, $c_{\pi(0)}$ (check it!). 
By Proposition~\ref{sdsuhfrlszdchj}, each path $\pi'^{-1}\ast \pi$ and $c_{\pi(0)}$ induce a unique path on persistence diagrams. 
Thus, because the constant path $c_{\pi(0)}$ induces the identical transport, $T^\p_\pi$
depends continuously on the path $\pi$(Proposition~\ref{transcontinuity}) and the non-trivial cornerpoints of persistence diagrams form discrete sets, then $T^\p_{\pi'^{-1}\ast \pi}$ is also the identical transport, contradicting $X\neq X'$.
\end{proof}



We recall that being homotopic for two paths is an equivalence relation and that the fundamental group of a topological space $U$ at a point $x\in U$ is the group of such equivalence classes of closed paths (loops) in $U$ (see~\cite[Ch. 1]{Ha02}). 

\begin{corollary}\label{loops}
The map $T^\p$ taking each equivalence class $[\pi]$ of homotopic loops at $(a,b)$ to the bijection $T^\p_{\pi}$ is a well-defined homomorphism from the fundamental group of $U$ at $(a,b)\in U$ to the group of permutations of $\dgm\left(\p_{(a,b)}^*\right)$.
\end{corollary}

\begin{proof}
This follows from Proposition~\ref{homotopy}
and the fact that $T_\pi^\p$ gives a bijection between $\dgm(\p^*_{(a,b)})$ and itself because, in this case, $(a,b)=\pi(0)=\pi(1)$.
\end{proof}

The following is a consequence of Corollary~\ref{loops}.

\begin{corollary}\label{generators}
If the set $\{[\pi_j]\}_{j\in J}$ of homotopy classes of loops based at a point $(a,b)\in U$ is a set of generators for the fundamental group of $U$ at $(a,b)$, then the persistent monodromy group of $\p$ with respect to $U$ is generated by the permutations $T^\p_{\pi_j}$.
\end{corollary}

\begin{definition}
The image of the group homomorphism $T^\p$ is called the \emph{persistent monodromy group} of the filtering function $\p$ with respect to $U$.
\end{definition}

The following is an example of a function $\p$ with non-trivial persistent monodromy group with respect to $U\subseteq \mathrm{Reg}(\p)$.

\begin{ex}\label{ex_monodromy}
Consider the filtering function $\p=(\p_1,\p_2)\colon\R^2\to \R^2$
with
$\p_1(x,y)=x$, and
$$\p_2(x,y) = \left\{
\begin{array}{ll}
-x & \mbox{ if } y=0\\
-x+1 & \mbox{ if } y=1\\
-2x & \mbox{ if } y=2\\
-2x+\frac{5}{4} & \mbox{ if } y=3\\
\end{array}\right.,$$
$\p_2(x,y)$ then being extended linearly for every $x$ on the segments respectively joining $(x,0)$ with $(x,1)$, $(x,1)$ with $(x,2)$, and $(x,2)$ to $(x,3)$. 
On the half-lines $\{(x,y)\in \R^2\mid y<0\}$ and $\{(x,y)\in \R^2\mid y>3\}$, $\p_2$ is then taken with constant slope $-1$ in the variable $y$. 
The graph of $\p_2$ is shown in Figure~\ref{figmon1}.

\begin{figure}
\begin{center}
\includegraphics[width=0.7\textwidth]{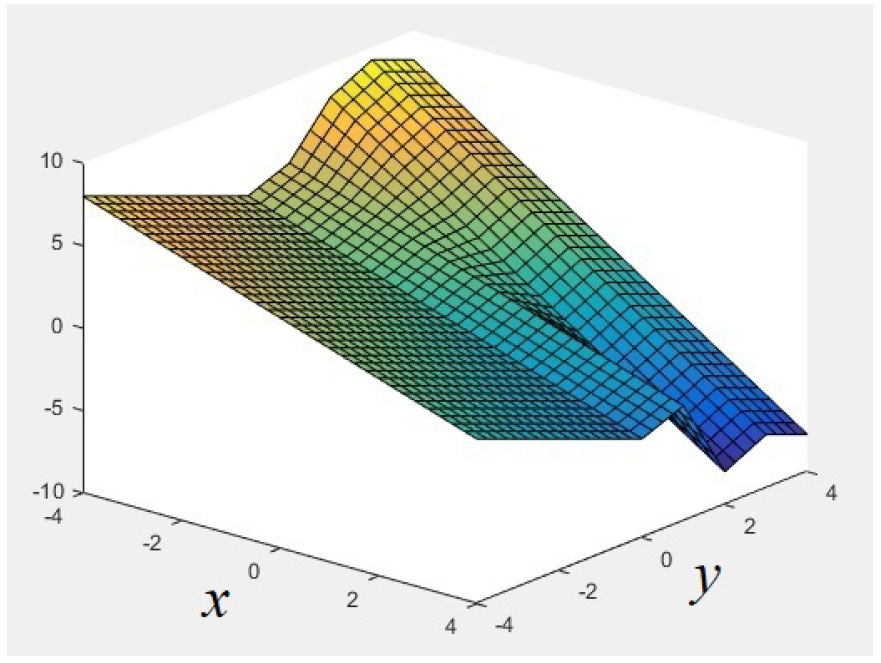}
\end{center}
\caption{The graph of $\p_2$ in Example~\ref{ex_monodromy}.}
\label{figmon1}
\end{figure}

The persistence diagram in degree $0$ of the function $\p_{\left(1/4,0\right)}^*$ contains a double point, so that $\left(1/4,0\right)$ is a singular pair for $\p$. 
Considering a loop $\gamma$ around the point $(1/4,0)$ in the parameter space, we can follow the different trajectories of two points of the persistence diagram $\dgm\left(\p_{(a,b)}^*\right)$. 
Figure~\ref{fig_inversion} shows that each of this trajectories ends at the starting point of the other. 
It is possible to adapt this example and get a smooth filtering function defined on a smooth closed manifold, revealing a similar phenomenon of monodromy, and falling in the setting considered in this chapter. 

\begin{figure}
\begin{center}
\includegraphics[width=\textwidth]{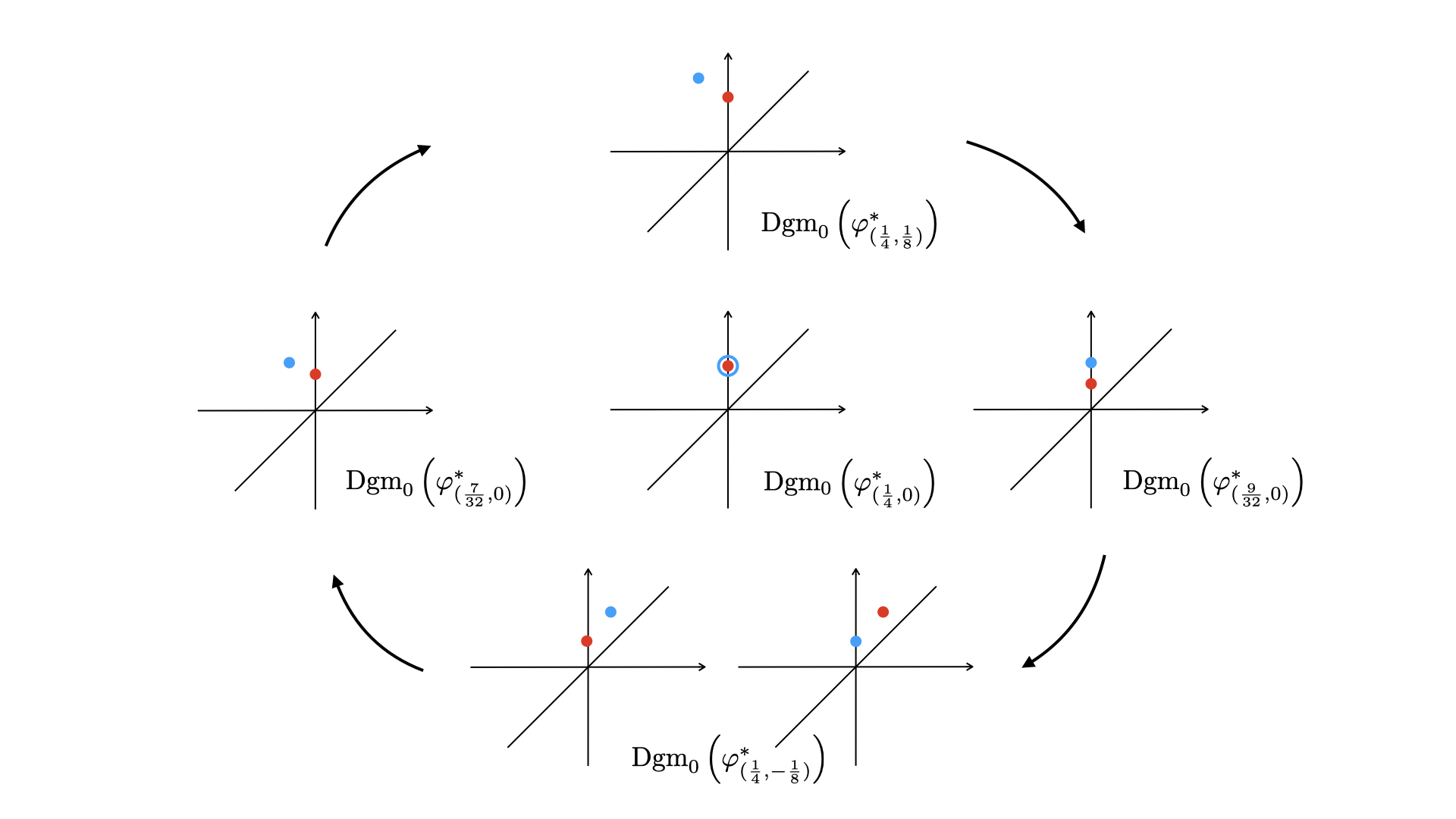}
\end{center}
\caption{Schematic of the evolution of the persistence diagram $\dgm(\p_{\gamma(t)})$ in degree $0$ for
Example~\ref{ex_monodromy} as $t$ goes from $0$ to $1$. The red and blue dots are proper cornerpoints.
The persistence diagrams are not scaled, but the respective positions of cornerpoints are
preserved.
Similarly, the path $\gamma$ in $\mathcal{P}(\Lambda^+)$ does not actually follow a geometric circle,
but is a simple closed curve.}
\label{fig_inversion}
\end{figure}
\end{ex}

\chapter{Natural pseudo-distance and Group Equivariant Non-Expansive Operators}
\chaptermark{Natural pseudo-distance and GENEOs}
\label{NPD&GENEOs}

Persistent homology theory is invariant under the action given by the precomposition of a group of domain homeomorphisms. 
In this chapter, we present a way to introduce a sensitivity to group action into the study of the topology of objects through group equivariant non-expansive operators. 
We present the natural pseudo-distance, which measures the distance between functions by taking into account the action of a given group. 
We then relate these two notions with persistence theory, and study how their metric aspects interlace.


\section{Perception pairs and topological groups}
\label{TG}

In this section, we introduce perception pairs. 
These model the data as a collection of measurements, encoded as a set of continuous functions, and a group of chosen transformations to which we would like our data to be sensitive. 
Preliminarily, we show some
results on the metric structure of the group.

Consider a set $\Phi$ of bounded functions from a set $X$ to $\mathbb{R}^k$ with $k \in \mathbb{N}$. 
Since $X$ is the domain of all functions in $\Phi$, we say that $X$ is the domain of $\Phi$, and write $\mathrm{dom}(\Phi)=X$.
The functions in $\Phi$ are called
\myemph{admissible functions} or
\myemph{admissible measurements}
or \myemph{admissible signals}, and represent the data that can be produced by the measurement tools or observers in which we are interested.


The set $X$ inherits the structure of a topological space from the (extended) pseudo-metric $D_X$ which distinguishes points only if they are seen as different by some measurement in $\Phi$:
\begin{equation}\label{defDX}
D_X(x_1,x_2)=\sup_{\varphi \in \Phi}\lVert
\varphi (x_1) - \varphi (x_2)\rVert_\infty
\end{equation}
for every $x_1,x_2\in X$.
We observe that every function $\varphi\in\Phi$ is non-expansive with respect to $D_X$ and 
the uniform metric on $\mathbb{R}^k$.
Thus, this choice of pseudo-metric on $X$ guarantees that all functions in $\Phi$ are continuous, satisfying the need of modeling stability of data.

In the following, we will denote by $\mathrm{Aut}_\Phi(X)$ the set of all bijections $g \colon X\to X$ such that
$\varphi g,\varphi g^{-1}\in\Phi$ for every $\varphi\in\Phi$. 
This set is a group with respect to the composition of functions. 
Every element in $\mathrm{Aut}_\Phi(X)$ is not just a bijection, but also an isometry:

\begin{proposition}
\label{propgisometry}
Every $g\in \mathrm{Aut}_\Phi(X)$ is an isometry 
with respect to $D_X$.
\end{proposition}

\begin{proof}
By definition of $\mathrm{Aut}_\Phi(X)$, 
the map $R_g \colon \Phi\to\Phi$ taking each function $\varphi$ to $\varphi g$ is surjective, since $\varphi=R_g\left(R_{g^{-1}}(\varphi)\right)$.
Hence $R_g(\Phi)=\Phi$. Therefore, 
for every $x, x'\in X$,
\begin{align*}
D_X(g(x),g(x')) &= \sup_{\varphi \in \Phi}\left\lVert \varphi(g(x))-\varphi(g(x'))\right\rVert_\infty
\\&= \sup_{\varphi \in R_g(\Phi)}\lVert\varphi(x)-\varphi(x')\rVert_\infty
\\&= \sup_{\varphi\in \Phi}\lVert\varphi(x)-\varphi(x')\rVert_\infty= D_X(x,x'). 
\end{align*}
Since $g$ is bijective, it follows that $g$ is an isometry with respect to $D_X$.
\end{proof}
\begin{definition}
If $G$ is a subgroup of $\mathrm{Aut}_\Phi(X)$, we say that $(\Phi,G)$ is a \myemph{perception pair}.
\end{definition}
The group $G$ is an extended pseudo-metric space with respect to $D_G$ defined, for every $g_1,g_2\in G$, by setting
\begin{equation} \label{eq:digi}
D_G(g_1,g_2):=\sup_{\varphi \in \Phi}\lVert \varphi g_1 - \varphi g_2\rVert_\infty.
\end{equation}
If $X$ is a bounded pseudo-metric space with respect to $D_X$, then $D_G$ is a bounded pseudo-metric. Indeed, 
\begin{align*}
          D_G(g_1,g_2)
          &= \sup_{\varphi\in\Phi} \lVert \varphi g_1 - \varphi g_2\rVert_\infty
          \\&= \sup_{\varphi\in\Phi} \sup_{x\in X}\lVert\p(g_1(x))-\p(g_2(x))\rVert_\infty
          \\&= \sup_{x\in X}\sup_{\varphi\in\Phi} \lVert\p(g_1(x))-\p(g_2(x))\rVert_\infty
          \\&= \sup_{x\in X}D_X(g_1(x),g_2(x)).
\end{align*}
It also follows that $D_G$ coincides with the usual uniform metric on $G$.

\begin{exercise}
Prove that $D_X$ and $D_G$ are (extended) pseudo-metrics.
\end{exercise}
In addition to the action of $G$ on $X$, the actions of $G$ on $\Phi$ and on $G$ by composition are also isometric, as shown by the following proposition, whose simple proof is left to the reader.

\begin{proposition}\label{propDphidU}
The following statements hold:
\begin{enumerate}
  \item $\lVert \varphi_1 g - \varphi_2 g\rVert_\infty =\lVert \varphi_1 - \varphi_2\rVert_\infty$ for every $\varphi_1,\varphi_2\in\Phi$ and $g\in G$;
  \item $D_G(gg_1 ,gg_2 )=D_G(g_1 g,g_2 g)=D_G(g_1,g_2)$
  for every $g_1,g_2,g\in G$.
\end{enumerate}
\end{proposition}

\begin{exercise}
Prove Proposition~\ref{propDphidU}.
\end{exercise}

We recall that a \myemph{topological group} is a group $(\hat G, \star)$ endowed with a topology such that 
composition and inversion
are continuous operations in $\hat G$, 
when the product topology is considered on $\hat G\times\hat G$.
We also recall that a \myemph{(right) action} of a group $\hat G$ on a set $\hat \Phi$ is a map $\rho: \hat \Phi\times \hat G\to \hat \Phi$ such that the following two properties hold:
\begin{enumerate}
  \item $\rho(\varphi,e)=\varphi$ for any $\varphi\in \hat \Phi$ (where $e$ is the identity element of $\hat G$);
  \item $\rho(\rho(\varphi,g_1),g_2)=\rho(\varphi, g_1\star g_2)$ for any $\varphi\in \hat \Phi$ and any $g_1,g_2\in \hat G$.
\end{enumerate}
We say that $\rho$ is a \myemph{continuous (right) action} if $\hat G$ is a topological group, $\hat \Phi$ is a topological space, and the map 
$\rho$
is continuous for the product topology on
$\hat \Phi\times \hat G$. 
In our setting $G$ acts on $\Phi$ by right composition.

\begin{exercise}
Find a group and a topology on it such that the group is not a topological group with respect to such a topology.
\end{exercise}



\begin{theorem}\label{thmGtopgroup}
If $(\Phi,G)$ is a perception pair, then $(G,D_G)$ is a topological group, and the composition on the right by elements of $G$ is a continuous (right) action of $G$ on $\Phi$.
\end{theorem}

\begin{proof}
First we prove that the binary composition in $G$ is continuous.
Consider the 1-product metric on $G \times G$, which induces the product topology.
Consider $(g_1,g_2),(g'_1,g'_2)\in G \times G$. Using Proposition \ref{propDphidU}, we have that
\begin{align*}
D_G(g_1  g_2,g'_1 g'_2)&=D_G (g_1,g'_1 g'_2 g_2^{-1})
\\& \le D_G (g_1,g'_1) + D_G (g'_1,g'_1 g'_2 g_2^{-1})
\\& = D_G (g_1,g'_1) +D_G (\mathrm{id}, g'_2 g_2^{-1})
\\& = D_G (g_1,g'_1) +D_G (g_2, g'_2).
\end{align*}
It follows that the binary operation on $G$ is non-expansive, 
and hence continuous.

Observe now that the inverse operation is a bijection from $G$
to itself. 
Consider $h_1,h_2 \in G$.
By Proposition \ref{propDphidU}, 
\begin{align*}
D_G(h_1^{-1},h_2^{-1})  &= D_G(h_1^{-1}h_2,h_2^{-1}h_2)
\\& = D_G(h_1^{-1}h_2,\mathrm{id})
\\& = D_G(h_1^{-1}h_2,h_1^{-1}h_1)
\\& = D_G(h_2,h_1)
\\& = D_G(h_1,h_2).
\end{align*}
Thus,  the function mapping the elements of $G$ to their respective inverses is an isometry, and hence continuous.

The continuity of the action 
of $G$ on $\Phi$ by (right) composition is proved by the following inequalities:
\begin{align*}
\lVert \varphi f - \psi g \rVert_\infty & \le \lVert \varphi f - \varphi g \rVert_\infty + \lVert \varphi g - \psi g \rVert_\infty
\\& = \lVert \varphi f - \varphi g \rVert_\infty + \lVert \varphi - \psi \rVert_\infty
\\& \le D_G(f,g)+ \lVert \varphi - \psi\rVert_\infty.
\end{align*}

\end{proof}




\section{Natural pseudo-distance
}
\label{NPWRTAG}

In this section, we introduce a pseudo-metric between functions that takes into account how a certain given group of isometries acts on the functions space\cite{DoFr04,DoFr07,DoFr09,BeFrGiQu19}. 
This will be considered as the ground truth distance in this chapter.

Consider a perception pair $(\Phi,G)$ with $X=\mathrm{dom}(\Phi)$.

\begin{definition}
The \myemph{natural pseudo-distance} associated with the group $G$ is the function $d_G\colon\Phi\times \Phi\to \mathbb{R}$ defined by setting
\[
d_G(\varphi_1,\varphi_2):=\inf_{g\in G}\lVert \p_1-\p_2 g\rVert_\infty.
\]
\end{definition}

\begin{proposition}
\label{propPM}
The function $d_G$ is a pseudo-metric.
\end{proposition}
\begin{proof}
\begin{enumerate}
  \item $0\le d_G(\varphi,\varphi)=\inf_{g\in G}\|\p-\p  g\|_\infty\le \|\p-\p  \mathrm{id} \|_\infty=\|\p-\p\|_\infty=0$, and hence $d_G(\varphi,\varphi)=0$ for every $\varphi\in\Phi$.
  \item By Proposition~\ref{propDphidU}.1, for every $\p_1,\p_2\in\Phi$,
  \begin{align*}
  d_G(\p_1,\p_2) & =\inf_{g\in G}\|\p_1-\p_2  g\|_\infty\\
  & =\inf_{g\in G}\|\p_1  g^{-1}-\p_2\|_\infty\\
  & =\inf_{g^{-1}\in G}\|\p_2-\p_1  g^{-1}\|_\infty\\
  & =\inf_{g\in G}\|\p_2-\p_1  g\|_\infty\\
  & = d_G(\p_2,\p_1).
  \end{align*}
  \item Again by Proposition~\ref{propDphidU}, 
  for every $\p_1,\p_2, \p_3\in\Phi$ and every $f\in G$
  \begin{align*}
            d_G(\p_1,\p_2)&=\inf_{g\in G}\|\p_1-\p_2  g\|_\infty\\
            &= \inf_{g\in G}\|\p_1-\p_2  g  f\|_\infty\\
            &\le \inf_{g\in G}\left(\|\p_1-\p_3  f\|_\infty+\|\p_3  f-\p_2  g  f\|_\infty\right)\\
            &= \inf_{g\in G}\left(\|\p_1-\p_3  f\|_\infty+\|\p_3-\p_2  g\|_\infty\right)\\
            &= \|\p_1-\p_3  f\|_\infty+\inf_{g\in G}\|\p_3-\p_2  g\|_\infty\\
            &= \|\p_1-\p_3  f\|_\infty+d_G(\p_3,\p_2).
\end{align*}
Since the previous inequality holds for any $f\in G$, it follows that for every $\p_1,\p_2,\p_3\in\Phi$:
\begin{align*}
            d_G(\p_1,\p_2)&=\inf_{f\in G}d_G(\p_1,\p_2)\\
            &\le \inf_{f\in G}\left(\|\p_1-\p_3  f\|_\infty+d_G(\p_3,\p_2)\right)\\
            &= d_G(\p_1,\p_3) + d_G(\p_3,\p_2).
\end{align*}
\end{enumerate}
\end{proof}

If $G$ is the trivial group $\{\mathrm{id}_X\}$, then $d_G(\varphi_1,\varphi_2)=\lVert \varphi_1 - \varphi_2 \rVert_\infty$ for any $\varphi_1,\varphi_2 \in \Phi$. Moreover, if $G_1$ and $G_2$ are subgroups of $\mathrm{Aut}_\Phi(X)$ and $G_1\subseteq G_2$, then
\[
d_{\mathrm{Aut}_\Phi(X)}(\p_1,\p_2)\le d_{G_2}(\p_1,\p_2)\le d_{G_1}(\p_1,\p_2)\le \lVert \p_1 - \p_2 \rVert_\infty
\]
for every $\p_1,\p_2\in \Phi$.

We say that a pseudo-metric $\hat{d}$ on $\Phi$ is strongly \myemph{$G$-invariant} if it is invariant under the action of $G$ with respect to each variable, i.e., if
\[
\hat{d}(\varphi_1, \varphi_2)=\hat{d}(\varphi_1  g, \varphi_2)=\hat{d}(\varphi_1, \varphi_2 g)=\hat{d}(\varphi_1  g, \varphi_2  g)
\]
for every $\varphi_1,\varphi_2 \in \Phi$ and every $g \in G$.
The key property of the natural pseudo-distance $d_G$ is that it is strongly $G$-invariant.
This property is of great use when we wish to compare data ``up to transformations in $G$''.
\begin{exercise}
Prove that the natural pseudo-distance $d_G$ is strongly $G$-invariant.
\end{exercise}

From now on, we assume that the topological space $(X, D_X)$ is compact and that Assumptions~\ref{ass_fingen} and~\ref{ass_right-cont} from Chapter~\ref{ChapterHC} hold for every filtering function under consideration.

Let us restrict our attention to continuous functions with values in $\R$. 
In this case, the natural pseudo-distance is lower-bounded by the matching distance between the persistence diagrams of the functions. 

\begin{proposition}\label{corstabilityofPDswrtNPD}
If $\p,\psi\in\Phi$, then
\[
d_\match (\dgm(\p),\dgm(\psi))\le d_G(\p,\psi).
\]
\end{proposition}

\begin{proof}
From Proposition~\ref{propinvariancePDs} and
Theorem~\ref{matchingstabilitythm}
it follows that
\begin{align*}
         d_\match (\dgm(\p),\dgm(\psi))
          &=d_\match (\dgm(\p),\dgm(\psi g))\\
          &\le \|\p-\psi g\|_\infty
\end{align*}
for every $g\in G$. This implies the desired inequality.
\end{proof}
Notice that if $G=\{\text{id}_X\}$ in Proposition~\ref{corstabilityofPDswrtNPD}, we obtain Theorem~\ref{matchingstabilitythm}.

\begin{exercise}\label{exercise_natpD>matchd}
Let $S^1$ be the set $\{(x,y)\in\mathbb{R}^2 : x^2 + y^2 = 1\}$. 
Consider the perception pair $(\Phi, G)$, where $\Phi$ is the set of functions from $S^1$ to $[0,1]$ that are $1$-Lipschitz with respect to the metric on $S^1$ induced by the Euclidean distance on $\mathbb{R}^2$, and $G$ is the group of rotations of $S^1$.
Find two functions $\varphi,\psi \in \Phi$ such that
\[
0 = d_\match(\dgm(\varphi), \dgm(\psi)) < d_G(\varphi,\psi).
\]
\end{exercise}

Another relation between persistence and the natural pseudo-distance is the following.

\begin{proposition}\label{corPBNsanddH}
Consider 
two functions $\p,\psi\in\Phi$.
If $(u,v)\in \Delta^*$ 
and $\eta>0$ such that 
$\beta_k^\p{(u-\eta,v+\eta)}> \beta_k^\psi{(u,v)}$,
then $d_G(\p,\psi)\ge\eta$.
\end{proposition}
\begin{proof}
Proposition~\ref{propinvariancebeta} guarantees that
$\beta_k^{\psi g}{(u,v)}=\beta_k^\psi{(u,v)}$ for every $(u,v)\in \Delta^*$ and every $g\in G$.
Because of Proposition~\ref{propPBNsandhomeo},
the inequality
$\beta_k^\p{(u-\eta,v+\eta)}> \beta_k^\psi{(u,v)}=
\beta_k^{\psi\circ g}{(u,v)}$
implies that
$\|\p-\psi g\|_\infty > \eta$
for every $g\in G$.
It follows that
$d_G(\p,\psi):=\inf_{g\in G}\|\p-\psi g\|_\infty\ge \eta$.
\end{proof}
The next example shows how Proposition~\ref{corPBNsanddH} can be applied.

\begin{ex}
\label{exapplycorollary}
Let us consider the two embeddings of $S^1$ into $\mathbb{R}^2$ represented in Figure \ref{duecurve}. The $y$-coordinate defines two filtering functions $\varphi$ and $\psi$ on $S^1$ (through the identifications given by the embeddings). 
In Figure \ref{duecurve} a homeomorphism $g_\varepsilon:S^1\to S^1$ is displayed, such that $\|\p-\psi g_\varepsilon\|_\infty\le \varepsilon$ (we set $g_\varepsilon(D_\varepsilon)=H_\varepsilon$, $g_\varepsilon(C)=G$ and $g_\varepsilon(E_\varepsilon)=F_\varepsilon$; the first red arc is taken to the second red arc).
The points outside the red arc joining $D_\varepsilon$ to $E_\varepsilon$ are mapped by $g_\varepsilon$ to points having the same $y$-coordinate.
Since $\eps$ can be arbitrarily small, this shows that 
$d_{\Homeo(S^1)}(\p,\psi)\le \frac{y_A-y_B}{2}$.
In Figure~\ref{fig_two_PDs} the persistent Betti numbers functions $\beta_0^{\p}$, $\beta_0^{\psi}$ are displayed. 
The inequality $d_{\Homeo(S^1)}(\p,\psi)\ge\frac{y_A-y_B}{2}$ follows by taking an arbitrarily small $\delta$ with $0<\delta< \frac{y_A-y_B}{2}$ and applying Proposition~\ref{corPBNsanddH} to the point $(u,v)=(\frac{y_A+y_B}{2}-\delta,\frac{y_A+y_B}{2}+\delta)$, after setting $\eta= \frac{y_A-y_B}{2}-2\delta$.
Therefore, 
$d_{\Homeo(S^1)}(\p,\psi)=\frac{y_A-y_B}{2}$. 
\end{ex} 

\begin{figure}
\begin{center}
\includegraphics[width=9cm]{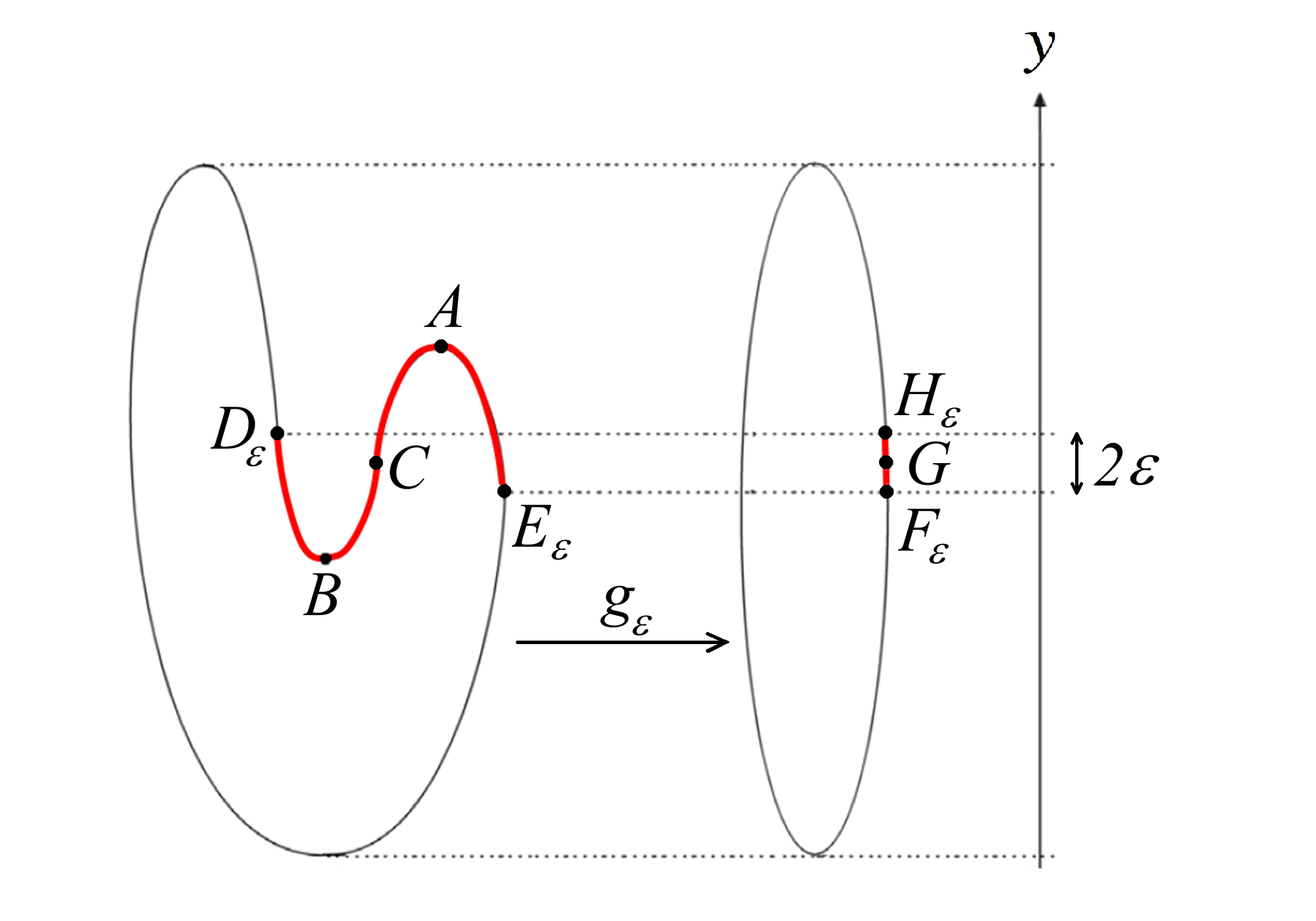}
\caption{The two functions $\p,\psi\colon S^1\to \mathbb{R}$ considered in Example~\ref{exapplycorollary}. They are obtained by taking the $y$-coordinate function corresponding to two different embeddings of $S^1$ into $\mathbb{R}^2$.
The homeomorphism $g_\varepsilon$ is also shown in the figure.
}
\label{duecurve}
\end{center}
\end{figure}

\begin{figure}
\begin{center}
\includegraphics[width=12cm]{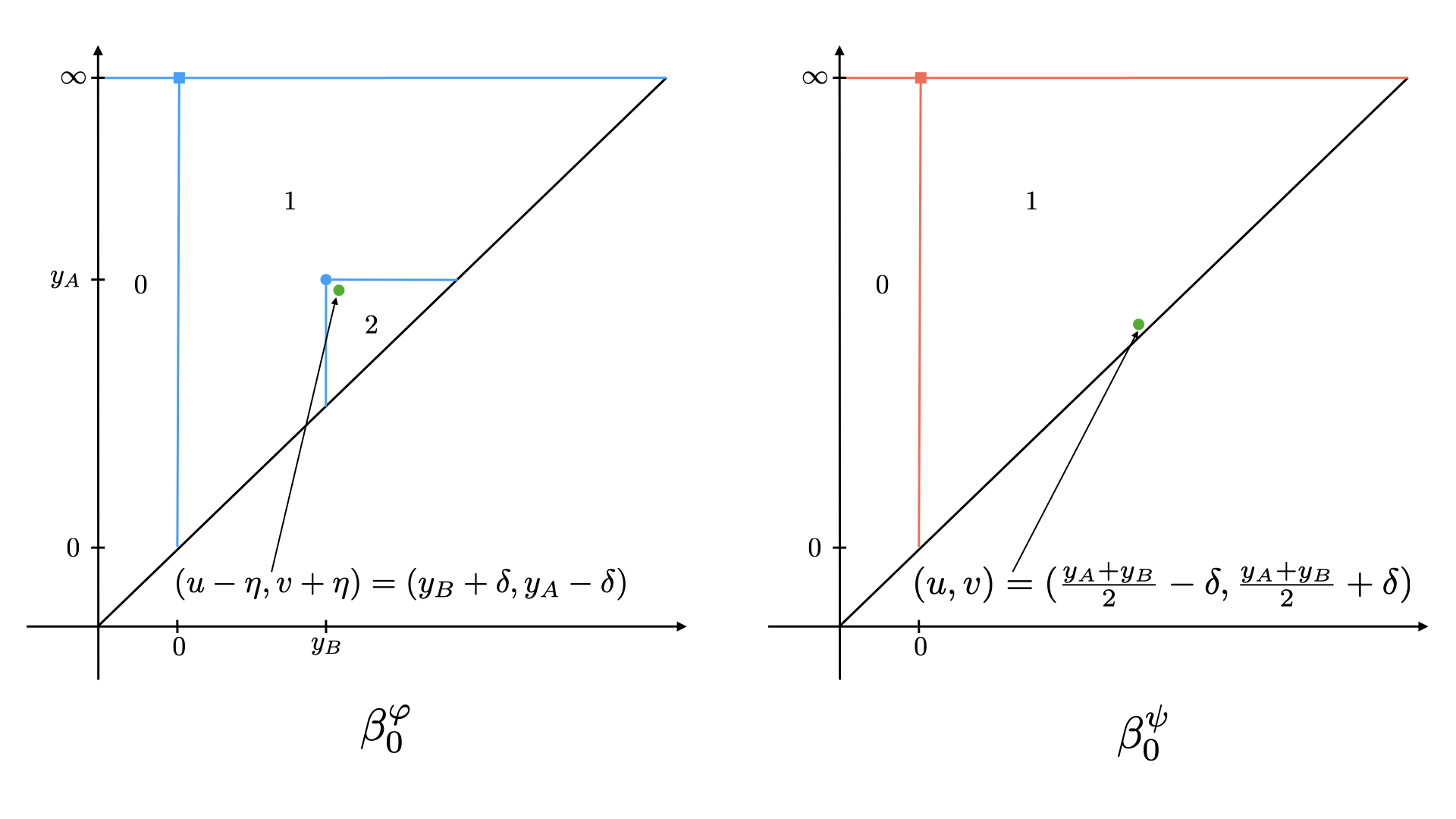}
\caption{The persistent Betti numbers functions in degree $0$ of the filtering functions $\p$, $\psi$ described in Example~\ref{exapplycorollary}. The inequality $d_{\Homeo(S^1)}(\p,\psi)\ge\frac{y_A-y_B}{2}$ follows by applying Proposition~\ref{corPBNsanddH} to the green points.}
\label{fig_two_PDs}
\end{center}
\end{figure}

\section{Group Equivariant Non-Expansive Operators}

\label{section_GENEOs}

Operators that map filtering functions to their persistence diagrams in any degree are powerful tools for data analysis. However, they represent only a particular instance of more general and flexible mathematical constructions. As stated in Proposition~\ref{propinvariancePDs}, the computation of persistence diagrams is invariant under the action of any homeomorphism of the domain of the filtering function. This property may be undesirable in several applications. For instance, Figure~\ref{lettere} shows grayscale images corresponding to some filtering functions defined on the real plane that cannot be distinguished by the computation of their persistence diagrams, in any degree.

This observation naturally leads to the problem of reducing the invariance of persistence diagrams in order to obtain stronger shape discrimination capabilities. One possible approach is provided by the notion of group equivariant non-expansive operators (GENEOs), which combines persistence theory with the theory of group actions. 
In this section we introduce this mathematical tool and discuss some of its most relevant properties.

\begin{figure}
\begin{center}
\includegraphics[width=12cm]{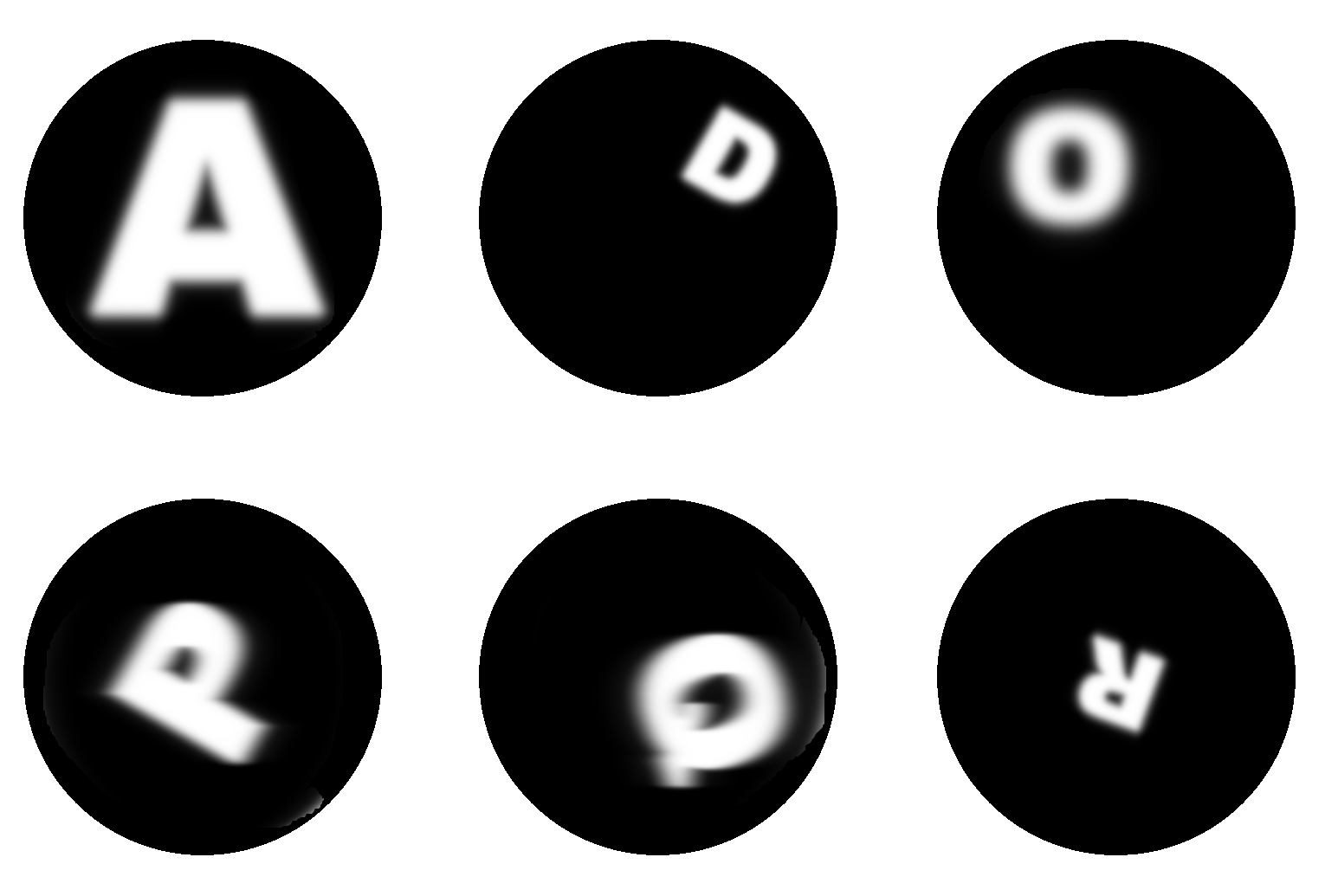}
\caption{The filtering functions associated with these grayscale images have identical persistence diagrams in every degree (white=1, black=0).}
\label{lettere}
\end{center}
\end{figure}


\begin{definition}[Group Equivariant Operator]
Consider two perception pairs $(\Phi,G)$, $(\Psi,H)$.
A pair of functions $(F,T):(\Phi,G)\to(\Psi,H)$ is said to be a \myemph{Group Equivariant Operator} (GEO) from $(\Phi,G)$ to $(\Psi,H)$ if
\begin{itemize}
    \item $T \colon G \to H$ is a group homomorphism;
    \item $F(\varphi g)=F(\varphi) T(g)$ for every $\varphi\in\Phi,g\in G$ (i.e., $F \colon \Phi \to \Psi$ is $T$-equivariant).
\end{itemize}
\end{definition}
We observe that the functions in $\Phi$ and the functions in $\Psi$ are defined on spaces that are generally different from each other ($\mathrm{dom}(\Phi) \ne \mathrm{dom}(\Psi)$), and the groups of invariance can be different as well. 

\begin{definition}
Consider two perception pairs $(\Phi,G)$ and $(\Psi,H)$. If $(F,T)$ is a GEO from $(\Phi,G)$ to $(\Psi,H)$ and $F$ is non-expansive (i.e., $\lVert F(\varphi_1) - F(\varphi_2)\rVert_\infty \le \lVert\varphi_1 - \varphi_2\rVert_\infty$ for every $\varphi_1,\varphi_2\in\Phi$), then $(F,T)$ is called a \myemph{Group Equivariant Non-Expansive Operator} (GENEO) \cite{FrJa16,BeFrGiQu19}.
\end{definition}

\begin{ex}
\label{basicex}
Firstly, we consider $S^2=\{(x,y,z)\in\R^3:x^2+y^2+z^2=1\}$ and $S^1=\{(x,y,z)\in\R^3:x^2+y^2=1\mbox{ and }z=0\}$. We set $\Phi$ as the space of functions  $\varphi\colon S^2 \to [0,1]$ such that $\lvert \varphi(a) - \varphi(b)\rvert \le \lVert a - b\rVert_\infty$ for any $a,b \in S^2$, and $G$ be the group of all rotations of $S^2$ around the $z$-axis. Let $\Psi$ be the set containing all functions from $S^1$ to $[0,1]$ such that $\lvert \varphi(a') - \varphi(b')\rvert \le \lVert a' - b'\rVert_\infty$ for any $a',b' \in S^1$, and $H$ be the group of all rotations of $S^1$.
We observe that $(\Phi,G)$ and $(\Psi,H)$ are two perception pairs. Let us consider the map
$F:\Phi\to\Psi$ that assigns to each function $\varphi\in \Phi$ the function \emph{average on
meridians} $\psi\in \Psi$, defined by
\[
\psi(\theta):=\frac{1}{\pi}\int_{0}^{\pi} \varphi(\theta,\alpha)\, d\alpha,
\]
where, with a slight abuse of notation, the variable of $\varphi$ is expressed in terms of its
azimuthal and polar angles $\theta$ and $\alpha$, while the variable of $\psi$ is denoted by
$\theta$. We also consider the homomorphism $T$ that maps a rotation of $S^{2}$ by $\theta$
radians about the positively oriented $z$-axis to the counterclockwise rotation of $S^{1}$ by
$\theta$ radians.
One can easily verify that $(F,T)$ is a GENEO from $(\Phi,G)$ to $(\Psi,H)$.
\end{ex}

\begin{ex}
The operator mapping each continuous function to its persistence diagram can be regarded as a GENEO.
This can be shown in several ways, depending on the continuous functions considered and on the chosen functional representation of persistence diagrams.
For example, let us consider the space $\tilde{\Phi} \subseteq \Phi$ of functions $\varphi \in \Phi$ such that $\dgm_k(\varphi)\setminus \{\Delta\}$ is finite, contains at least one point at infinity, and has no points of multiplicity greater than $1$. 
Observe that $(\tilde\Phi,G)$ is still a perception pair and that every persistence diagram $\mathcal{D}$ of a function in $\tilde\Phi$ can be identified with the function $\psi_\mathcal{D}\colon\bar\Delta^*\to\R$ that takes each
$P\in\bar\Delta^*$ to $d(P,\mathcal{D})=\min_{Q\in \mathcal{D}}d(P, Q)$.
The set of such continuous functions $\psi_\mathcal{D}$
is denoted by $\Psi$. 
Endowing it with the uniform metric is equivalent to endow the set of persistence diagrams with the Hausdorff distance, $\delta_H$,
associated with $d$.
This is seen by writing an equivalent formulation of the Hausdorff distance
between two compact non-empty sets $A,B\subseteq \bar\Delta^*$: 
$\delta_H(A, B)=\sup_{x\in X}\lvert d(x, A)-d(x,B)\rvert$.
Consider now the perception pair $(\Psi, H)$ where $H$ only contains the identity from $\bar \Delta^*$ to itself, and the trivial group homomorphism $T\colon G\to H$.
Observing that $\delta_H(\dgm_k(\varphi_1),\dgm_k(\varphi_2))\le \|\varphi_1-\varphi_2\|_\infty$ (see Exercise~\ref{denaukhao}), we can conclude that the pair $(F, T)$ where $F$ sends $\p\in \tilde\Phi$ to $\psi_{\dgm_k(\p)}\in \Psi$ is a GENEO.



\end{ex}

\begin{ex}
Let us assume that $(\Phi,G)$ is a perception pair and that the functions in $\Phi$ take values in $\R^2$.
In Chapter~\ref{MPH}, we have seen that the persistent homology of $\p\in \Phi$ is equivalent to the persistent homology of all the filtrations $\p_{(a,b)}^*\colon X\to \R$, for $(a,b)\in ]0,1[\times \R$,
under suitable assumptions.
The pair $(F_{a,b}, T)$, where $F_{a,b}(\varphi) := \varphi^*_{(a,b)}$ and $T$ is the identity from $G$ to $G$,
is a GENEO, for every $(a,b)$. 
In particular, $F_{a,b}$ is non-expansive by Lemma~\ref{lemmafab}.
\end{ex}
The following results state how GENEOs relate to the natural pseudo-distances. 

\begin{proposition}
\label{thmContraction}
If $(F,T)$ is a GENEO from $(\Phi, G)$ to $(\Psi, H)$, then the map $F \colon (\Phi,d_G) \to (\Psi, d_H)$ is non-expansive.
\end{proposition}

\begin{proof}%
Since $F$ is non-expansive, it follows that
\begin{align*}
d_H(F(\varphi_1), F(\varphi_2)) & =\inf_{h \in H}\lVert(F(\varphi_1) - F(\varphi_2) h\rVert_\infty\nonumber\\
& \leq \inf_{g \in G} \lVert( F(\varphi_1) -  F(\varphi_2) T(g)\rVert_\infty\\
& = \inf_{g \in G} \lVert F(\varphi_1) - F(\varphi_2 g)\rVert_\infty\nonumber\\
& \leq \inf_{g \in G} \lVert \varphi_1 - \varphi_2 g\rVert_\infty = d_G(\varphi_1, \varphi_2).\nonumber
\end{align*}
\end{proof}

In other words, GENEOs are not only non-expansive with respect to the uniform metric, but also with respect to the natural pseudo-distances.

\section{Compactness and convexity of the space of GENEOs} 

In what follows we show that if the function spaces are compact and convex, then the space of GENEOs is also compact and convex, whenever the homomorphism between the groups is fixed.
The compactness result guarantees that the space of GENEOs can be approximated by a finite set. 
The convexity implies that new GENEOs can be obtained by convex combinations of given GENEOs.

Given two perception pairs $(\Phi,G)$ and $(\Psi,H)$ and a group homomorphism $T \colon G \to H$, if $\mathcal{F}^{\mathrm{all}}_T$ denotes the set of GENEOs between $(\Phi,G)$ and $(\Psi,H)$ associated with the homomorphism $T: G \to H$, then the following theorem holds:
\begin{theorem}\label{t17}
If $\Phi$ and $\Psi$ are  compact, then $\mathcal{F}^{\mathrm{all}}_T$ is compact with respect to the metric
\begin{equation}
D_\mathrm{GENEO}\left((F_1,T),(F_2,T)\right):=\max_{\varphi\in \Phi}\lVert F_1(\varphi) - F_2(\varphi)\rVert_\infty.
\end{equation}
\end{theorem}

\begin{proof} We observe that $D_{\GENEO}$ coincides with the usual uniform metric on metric spaces.
Therefore it suffices to prove that $\mathcal{F}^{\mathrm{all}}_T$ is sequentially compact.
To do this, let us consider a sequence $\left(\left(F_i,T\right)\right)_i$ in $\mathcal{F}^{\mathrm{all}}_T$.
Given that $\Phi$ is a compact (and hence separable) metric space,
we can find a countable and dense subset $\Phi^*=\{\varphi_j \}_{j\in \mathbb{N}}$ of $\Phi$.
We can extract a subsequence $\left(\left(F'_{i},T\right)\right)_i$ from $\left(\left(F_i,T\right)\right)_i$, such that for every fixed index $j$ the sequence $\left(F'_{i}(\varphi_j)\right)_i$ uniformly converges to a function in $\Psi$ with respect to $\lVert \cdot \rVert_\infty$. 
We will show it by means of a classical diagonal argument.
Since $\Psi$ is compact, the sequence $(F_i(\varphi_1))_i$ admits a subsequence $\left(F_i^{(1)}(\varphi_1)\right)_i$ that uniformly converges in $\Psi$. Again, since $\Psi$ is compact, the sequence $\left(F_i^{(1)}(\varphi_2)\right)_i$ admits a subsequence $\left(F_i^{(2)}(\varphi_2)\right)_i$ that converges in $\Psi$. Recursively, we can build a family $\left\{\left(F_i^{(k)}\right)_{i}\right\}_{k \in \mathbb{N}}$of subsequences of $\left(F_i\right)_i$ such that, for every $k\in \mathbb{N}$, $\left(F_i^{(k+1)}\right)_{i}$ is a subsequence of $\left(F_i^{(k)}\right)_{i}$ and $\left(F_i^{(k+1)}(\varphi_{k+1})\right)_i$ converges in $\Psi$.

Now we set $\left(F_i',T\right)=\left(F_i^{(i)},T\right)$ for every index $i$.
By definition of $F'_i$, $\lim_{i \to \infty} F_i'(\varphi_k)$ exists in $\Psi$ for every $k\in\mathbb{N}$.
Hence, we define $\overline{F}$ on $\Phi^{*}$ by setting
$\overline{F}(\varphi_j):= \lim_{i \to \infty}F'_{i}(\varphi_j) $ for each $\varphi_j \in \Phi^{*}$.

We extend $\overline{F}$ to $\Phi$ as follows. Since $\Phi^{*}$ is dense in $\Phi$, for every $\varphi \in \Phi$ we can choose a sequence $(\varphi_{j_r})_r$ in $\Phi^{*}$, converging to $\varphi$, and set $\overline{F}(\varphi):= \lim_{r \to \infty} \overline{F}(\varphi_{j_r})$. 
We claim that such a limit exists in $\Psi$ and does not depend on the choice of the sequence. 
To prove that the previous limit exists, observe that $\overline{F}: \Phi^{*} \rightarrow\ \Psi$ is non-expansive:
\begin{align*}
\left\lVert \overline{F}(\varphi_{j_{1}}) - \overline{F}(\varphi_{j_{2}})\right\rVert_\infty & = \left\lVert \lim_{i \to \infty}F'_{i}(\varphi_{j_{1}})-\lim_{i \to \infty}F'_{i}(\varphi_{j_{2}})\right\rVert_\infty 
\\& = \lim_{i \to \infty} \left\rVert F'_{i}(\varphi_{j_{1}}) - F'_{i}(\varphi_{j_{2}})\right\rVert_\infty
\\ & \le  \lim_{i \to \infty} \left \lVert \varphi_{j_{1}} - \varphi_{j_{2}}\right\rVert_\infty\\
& = \left\lVert \varphi_{j_{1}} - \varphi_{j_{2}}\right\rVert_\infty,
\end{align*}
for every $\varphi_{j_{1}},\varphi_{j_{2}} \in \Phi^*$, because every $F'_{i}$ is non-expansive.
Since $\left(\varphi_{j_r}\right)_r$ is a Cauchy sequence, 
it follows that $\left(\overline{F}(\varphi_{j_r})\right)_r$ is a Cauchy sequence with respect to $\lVert \cdot \rVert_\infty$. The  compactness of $\Psi$ implies that $\Psi$ is complete, and hence $\left(\overline{F}\left(\varphi_{j_r}\right)\right)_r$ converges in $\Psi$.
If $\left(\varphi_{k_r}\right)_r$ is another sequence in $\Phi^{*}$ converging to $\varphi \in \Phi$, the non-expansiveness of $\overline{F}$ on $\Phi^*$ implies that 
$\left\lVert\overline{F}(\varphi_{j_r}) - \overline{F}(\varphi_{k_r})\right\rVert\le \left\lVert\varphi_{j_r} -\varphi_{k_r}\right\rVert_\infty$ for every index $r \in \mathbb{N}$.
Since both $\left(\varphi_{j_r}\right)_r$ and $\left(\varphi_{k_r}\right)_r$ converge to $\varphi$, then it follows that
$\lim_{r \to \infty} \overline{F}(\varphi_{j_r})=\lim_{r \to \infty} \overline{F}(\varphi_{k_r})$. Therefore the definition of $\overline{F}(\varphi)$ does not depend on the choice of the sequence.

Next, we show that $\overline{F} \in \mathcal{F}^{\mathrm{all}}_T$. 
For every $\varphi, \varphi'\in\Phi$ we can consider two sequences $\left(\varphi_{j_r}\right)_r$, $\left(\varphi_{k_r}\right)_r$ in $\Phi^{*}$, converging to $\varphi$ and $\varphi'$, respectively. 
Due to the fact that the operator $\overline F$ is non-expansive on $\Phi^*$, we have that
\begin{align*}
 \left\lVert\overline{F}(\varphi) - \overline{F}(\varphi')\right\rVert_\infty & = \left\lVert\lim_{r \to \infty}\overline{F}(\varphi_{j_r}) - \lim_{r \to \infty}\overline{F}(\varphi_{k_r})\right\rVert_\infty \\
 & =  \lim_{r \to \infty}  \left\lVert \overline F(\varphi_{j_r}) - \overline{F}(\varphi_{k_r})\right\rVert_\infty \\
  & \le \lim_{r \to \infty} \left\lVert\varphi_{j_r}-\varphi_{k_r}\right\rVert_\infty \\
 & = \left\lVert\varphi - \varphi'\right\rVert_\infty.
\end{align*}
Therefore, $\overline{F}\colon\Phi\to\Psi$ is non-expansive.

The next step is to show that the sequence $\left(F'_{i}\right)_i$ converges to $\overline{F}$ with respect to the uniform metric.
Consider an arbitrarily small $\varepsilon>0$. Since $\Phi$ is compact and $\Phi^{*}$ is dense in $\Phi$, we can find a finite subset $\left\{ \varphi_{j_1},\dots, \varphi_{j_n} \right\}$ of $\Phi^{*}$  such that for every $\varphi \in \Phi$, there exists an index $r \in \{1, \dots, n \}$, for which $\left\lVert \varphi - \varphi_{j_r} \right\rVert_\infty < \varepsilon$.
Since the sequence  $\left(F'_{i}\right)_i$ converges pointwise to $\overline{F}$ on the set $\Phi^{*}$, an index $\overline\imath$ exists, such that $\left\lVert\overline{F}(\varphi_{j_r}) - F'_{i}(\varphi_{j_r})\right\rVert_\infty < \varepsilon$ for every $i \ge \overline\imath$ and every $r \in \{ 1, \dots , n \}$.
The following inequalities hold for every index $i\ge \overline\imath$ because of the non-expansiveness of $\overline{F}$ and $F'_{i}$:
\begin{align*}
    &\left\lVert\overline{F}(\varphi) - F'_{i}(\varphi)\right\rVert_\infty\\
    & \le\left\lVert\overline{F}(\varphi) - \overline{F}(\varphi_{j_r})\right\rVert_\infty +\left\lVert\overline{F}(\varphi_{j_r}) - F'_{i}(\varphi_{j_r})\right\rVert_\infty +\left\lVert F'_{i}(\varphi_{j_r}) - F'_{i}(\varphi)\right\rVert_\infty \\
    & \le \left\lVert \varphi - \varphi_{j_r}\right\rVert_\infty + \left\lVert\overline{F}(\varphi_{j_r}) - F'_{i}(\varphi_{j_r})\right\rVert_\infty + \left\lVert \varphi_{j_r} - \varphi\right\rVert_\infty < 3 \varepsilon.
\end{align*}

We observe that $\bar\imath$ does not depend on $\varphi$, but only on $\varepsilon$ and on the set $\left\{ \varphi_{j_1},\dots, \varphi_{j_n} \right\}$. It follows that  $\left\lVert\overline{F}(\varphi) - F'_{i}(\varphi)\right\rVert_\infty < 3\varepsilon$ for every $\varphi \in \Phi$ and every $i \ge \overline\imath$.
Hence, $\sup_{\varphi \in \Phi}\left\lVert\overline{F}(\varphi)- F'_{i}(\varphi)\right\rVert_\infty\le 3\varepsilon$ for every $i \ge \overline\imath$.

This proves that the sequence $\left(F'_i\right)_i$ converges to $\overline{F}$ with respect to the uniform metric.
It remains to show that $\overline{F}$ is $T$-equivariant. 
Consider $\varphi \in \Phi$ and a sequence $(\varphi_{j_r})_r$ in $\Phi^*$ converging to $\varphi$, and a $g \in G$.
Since 
$G$ continuously acts on $\Phi$,
$H$ continuously acts on $\Psi$ (see Theorem \ref{thmGtopgroup}),
$\overline F$ is continuous, each $F'_{i}$ is $T$-equivariant, and the sequence $\left(F'_{i}\right)_i$ uniformly converges to $\overline{F}$, we have that 
\begin{align*}
    \overline{F}(\varphi g) & = \overline F((\lim_{r \to \infty} \varphi_{j_r})  g)
    \\ & = \overline F(\lim_{r \to \infty} \varphi_{j_r}  g)
    \\ & = \lim_{r \to \infty}\overline{F}(\varphi_{j_r}   g)
    \\ & = \lim_{r \to \infty}\lim_{i \to \infty}F'_{i}(\varphi_{j_r}   g)\\
    & = \lim_{r \to \infty} \lim_{i \to \infty}  \left(F'_{i}(\varphi_{j_r})   T(g)\right)\\
    & = \left( \lim_{r \to \infty}\lim_{i \to \infty}  F'_{i}(\varphi_{j_r})\right)   T(g)\\
    & = \overline{F}(\varphi)   T(g).
\end{align*}
Therefore, the sequence $\left(\left(F'_{i},T\right)\right)_i$ converges to $\left(\overline{F},T\right)$ with respect to $D_{\GENEO}$.

\end{proof}
\begin{exercise}
Prove that the statement of Theorem~\ref{t17} does not hold for group equivariant operators, by giving a counterexample where $\Phi$, $\Psi$, $X$, $Y$, $G$ and $H$ are compact, but the space of GEOs is not compact with respect to the metric $D_\mathrm{GEO}\left(F_1,F_2\right):=\sup_{\varphi\in \Phi}\lVert F_1(\varphi) - F_2(\varphi)\rVert_\infty$. 
\end{exercise}

Let us assume that $(F_1,T), (F_2,T), \dots , (F_n,T) \in \mathcal{F}^{\mathrm{all}}_T$. 
If  $(a_1, a_2, \dots , a_n) \in \mathbb{R}^n$ with $\sum_{i = 1}^n|a_i|\leq 1$, then we set 
\[
F_{\Sigma}(\varphi) := \sum_{i = 1}^n a_i F_i(\varphi).
\]

\begin{proposition}
\label{propconvex}
If $F_{\Sigma}(\Phi)\subseteq \Psi$, then 
$(F_{\Sigma},T)$ is a GENEO from $(\Phi, G)$ to $(\Psi, H)$.
\end{proposition}

\begin{proof}
First, we prove that $F_{\Sigma}$ is $T$-equivariant. 
Since $F_i$ is $T$-equivariant for every $i$, we have that:
\begin{align*}
F_{\Sigma}(\varphi g)
 & = \sum_{i=1}^n a_i F_i (\varphi g)
 \\& = \sum_{i=1}^n a_i (F_i (\varphi) T(g))
 \\& = \left(\sum_{i=1}^n a_i F_i (\varphi)\right) T(g)
 \\& = F_{\Sigma}(\varphi) T(g).
\end{align*}
Since $F_i$ is non-expansive for every $i$, $F_{\Sigma}$ is non-expansive:
\begin{align*}
\left \lVert F_{\Sigma}(\varphi_1)-F_{\Sigma}(\varphi_2)\right \rVert_\infty & = \left\lVert\sum_{i=1}^n a_i F_i(\varphi_1) - \sum_{i=1}^n a_i F_i(\varphi_2) \right\rVert_{\infty}
\\& =  \left\lVert\sum_{i=1}^n a_i (F_i(\varphi_1) - F_i(\varphi_2))\right\rVert_{\infty} 
\\& \le  \sum_{i=1}^n |a_i| \left\lVert F_i(\varphi_1) - F_i(\varphi_2)\right\rVert_{\infty} 
\\& \le  \sum_{i=1}^n \lvert a_i \rvert \left\lVert\varphi_1 - \varphi_2 \right\rVert_{\infty}\\
&\le \rVert\varphi_1-\varphi_2\rVert_\infty.
\end{align*}
Therefore $F_{\Sigma}$ is a GENEO.
\end{proof}

\begin{theorem}
\label{thmconvex}
If $\Psi$ is convex, then the set of GENEOs from $(\Phi, G)$ to $(\Psi, H)$ with respect to $T$ is convex.
\end{theorem}

\begin{proof}
It is sufficient to apply Proposition~\ref{propconvex} for $n=2$, by setting $a_1=t$, $a_2=1-t$ for $0\le t\le 1$, and observing that the convexity of $\Psi$ implies $F_{\Sigma}(\Phi)\subseteq \Psi$.
\end{proof}

\section{Pseudo-metrics induced by persistent homology}

In this section, we introduce pseudo-metrics in the space of measurements based on a given collections of GENEOs operating on them. 
One of these pseudo-metrics is based on the matching distance between the persistence diagrams of the measurement functions after the GENEOs are applied.
We show that GENEOs combined with the matching distance allow us to infer the natural pseudo-distance, which is our ground truth.

Consider the perception pairs $(\Phi, G)$ and $(\Psi, H)$, where $\Phi, \Psi$ are two sets of bounded real-valued functions, with $\mathrm{dom}(\Phi)=X$ and $\mathrm{dom}(\Psi)=Y$, and a homomorphism $T\colon G \to H$.
Let us consider a subset $\mathcal{F} \ne \emptyset$ of all GENEOs between $(\Phi, G)$ and $(\Psi, H)$ associated with $T$, $\mathcal{F}^{\mathrm{all}}_T$. 
To compare data under the action of $\mathcal{F}$, one could define
the pseudo-metric 
$D_{{\mathcal{F}},\Phi}$ by setting, for $\varphi_1,\varphi_2 \in \Phi$:
\[
D_{\mathcal{F},\Phi}(\varphi_1,\varphi_2):=\sup_{F\in \mathcal{F}}\|F(\varphi_1)-F(\varphi_2)\|_\infty.
\]
Persistent homology allows us to replace $D_{\mathcal{F},\Phi}$ with a pseudo-metric $\mathcal{D}^{\mathcal{F},k}_{\mathrm{match}}$ computationally more efficient, stable and, above all, strongly invariant with respect to the action of $G$.
For any fixed degree $k$, we define the pseudo-metric $\mathcal{D}^{\mathcal{F},k}_{\mathrm{match}}$ on $\Phi$ as
\begin{equation*}\label{eq:dfkmatch}
\mathcal{D}^{\mathcal{F},k}_{\mathrm{match}}(\varphi_1, \varphi_2):= \sup_{F \in \mathcal{F}} d_{\mathrm{match}}(\dgm_k(F(\varphi_1)),\dgm_k(F(\varphi_2)))
\end{equation*}
for every $\varphi_1,\varphi_2 \in \Phi$. 

\begin{proposition}
\label{stronglyinv}
  $\mathcal{D}^{\mathcal{F},k}_{\mathrm{match}}$ is a strongly $G$-invariant pseudo-metric on $\Phi$.
\end{proposition}

\begin{proof}
The Matching Distance Stability Theorem~\ref{matchingstabilitythm} and the non-expansiveness of every $F \in \mathcal{F}$ imply that
\begin{align*}
   d_{\mathrm{match}} (\dgm_k(F(\varphi_1)),\dgm_k(F(\varphi_2)))& \le \lVert F(\varphi_1)-F(\varphi_2)\rVert_\infty\\
   & \le \lVert\varphi_1-\varphi_2\rVert_\infty.
\end{align*}
Therefore $\mathcal{D}^{\mathcal{F},k}_{\mathrm{match}}$ always assumes finite values. 
Moreover, for every $\varphi_1,\varphi_2 \in \Phi$ and every $g \in G$
\begin{align*}
  \mathcal{D}^{\mathcal{F},k}_{\mathrm{match}}(\varphi_1,\varphi_2 g)
  & = \sup_{F \in \mathcal{F}} d_{\mathrm{match}}(\dgm_k(F(\varphi_1)),\dgm_k(F(\varphi_2 g)))\\
  & = \sup_{F \in \mathcal{F}} d_{\mathrm{match}}(\dgm_k(F(\varphi_1)),\dgm_k(F(\varphi_2) T( g)))\\
  & = \sup_{F \in \mathcal{F}} d_{\mathrm{match}}(\dgm_k(F(\varphi_1)),\dgm_k(F(\varphi_2))\\
  & =\mathcal{D}^{\mathcal{F},k}_{\mathrm{match}}(\varphi_1,\varphi_2)\\
\end{align*}
because every $F$ is $T$-equivariant and persistent homology is invariant under the action of the homeomorphisms (see Proposition~\ref{propinvariancePDs}). 
Since the function $\mathcal{D}^{\mathcal{F},k}_{\mathrm{match}}$ is symmetric, this is sufficient to guarantee that $\mathcal{D}^{\mathcal{F},k}_{\mathrm{match}}$ is strongly $G$-invariant.
\end{proof}

The pseudo-metric $\mathcal{D}^{\mathcal{F},k}_{\mathrm{match}}$ is stable with respect to both the pseudo-metric $d_G$ and the uniform metric.



\begin{theorem}
\label{t15}
If $\mathcal{F}$ is a non-empty subset of $\mathcal{F}^{\mathrm{all}}_T$, then for any $\varphi_1, \varphi_2 \in \Phi$
\begin{equation*}\label{eq:stability}
\mathcal{D}^{\mathcal{F},k}_{\mathrm{match}}(\varphi_1,\varphi_2) \le d_G(\varphi_1,\varphi_2) \le \lVert \varphi_1 - \varphi_2 \rVert_\infty.
\end{equation*}
\end{theorem}

\begin{proof}
For every $F \in \mathcal{F}$, every $g \in G$ and every $\varphi_1, \varphi_2 \in \Phi$, we have that
\begin{align*}
     d_{\mathrm{match}}(\dgm_k(F(\varphi_1)),\dgm_k(F(\varphi_2)))& =d_{\mathrm{match}}(\dgm_k(F(\varphi_1)),\dgm_k(F(\varphi_2)  T( g)))\\
    & =d_{\mathrm{match}}(\dgm_k(F(\varphi_1)),\dgm_k(F(\varphi_2  g)))\\
    & \le \lVert F(\varphi_1) - F(\varphi_2  g)\rVert_\infty\\
    & \le \lVert \varphi_1 - \varphi_2 g\rVert_\infty.
\end{align*}
The first equality follows from the invariance of persistent homology under the action of $\Homeo(X)$ (Proposition~\ref{propinvariancePDs}), and the second equality follows from the fact that $F$ is $T$-equivariant. 
The first inequality is the Matching Distance Stability Theorem~\ref{matchingstabilitythm}, while the second inequality is the non-expansiveness of $F$.
It follows that, if $\mathcal{F}\subseteq \mathcal{F}^{\mathrm{all}}_T$, then for every $g \in G$ and every $\varphi_1, \varphi_2 \in \Phi$
\begin{equation*}
\mathcal{D}^{\mathcal{F},k}_{\mathrm{match}}(\varphi_1,\varphi_2)\le \lVert \varphi_1 - \varphi_2  g\rVert_\infty, 
\end{equation*}
hence $\mathcal{D}^{\mathcal{F},k}_{\mathrm{match}} \le d_G$.
Moreover, $d_G(\varphi_1,\varphi_2) \le \lVert \varphi_1 - \varphi_2 \rVert_\infty$, by considering $g=\mathrm{id}$ in the definition of $d_G$.
\end{proof}

We observe that the pseudo-metrics $d_G$ and $\mathcal{D}^{\mathcal{F},k}_{\mathrm{match}}$ have different nature. 
While both depend on group actions, they are also different in what they measure.
The natural pseudo-distance $d_G$ is based on a variational approach involving the set of all elements in $G$, and $\mathcal{D}^{\mathcal{F},k}_{\mathrm{match}}$ measures the difference between functions after the information is being condensed into their persistence diagrams. 
Thus, the next result may appear unexpected.

\begin{theorem}
\label{thrEqualityDFall_dmatch}
Let 
$\mathcal{F}^{\mathrm{all}}_{\mathrm{id}}$ be the space of all GENEOs from $(\Phi,G)$ to itself with respect to the identity $\mathrm{id} \colon G \to G$, where every function in $\Phi$ is non-negative, and $\Phi$ contains every constant function $c$ for which a function $\varphi \in \Phi $ exists such that $0\le c \le \|\varphi\|_{\infty}$. 
If the $H_k(X)$ is non-zero, then $\mathcal{D}^{\mathcal{F}^{\mathrm{all}}_\mathrm{id},k}_{\mathrm{match}}=d_G$.
\end{theorem}

\begin{proof}
For every $\varphi' \in \Phi$ let us consider the operator $F_{\varphi'} \colon \Phi \to \Phi$ defined by setting $F_{\varphi'}(\varphi)$ equal to the constant function taking everywhere the value $d_G(\varphi, \varphi')$ for every $\varphi \in \Phi$ (i.e., $F_{\varphi'}(\varphi)(x)=d_G(\varphi,\varphi')$ for any $x \in X$). 
Our assumptions guarantee that such a constant function belongs to $\Phi$.
We observe that $F_{\varphi'}$ is $\mathrm{id}$-equivariant, because the strong invariance of the natural pseudo-distance $d_G$ with respect to the group $G$ implies that if $\varphi \in \Phi$ and $g \in G$, then for every $x\in X$ \[F_{\varphi'}(\varphi   g)(x) = d_G(\varphi  g,\varphi') = d_G(\varphi,\varphi') = F_{\varphi'}(\varphi)(g(x)) = (F_{\varphi'}(\varphi)  g)(x).\]
Moreover, $F_{\varphi'}$ is non-expansive on $\Phi$, because for every $\varphi_1, \varphi_2 \in \Phi$:
\[
\lVert F_{\varphi'}(\varphi_1) - F_{\varphi'}(\varphi_2)\rVert_\infty  = |d_G(\varphi_1, \varphi') - d_G(\varphi_2,\varphi')| \le d_G(\varphi_1,\varphi_2) \le \lVert \varphi_1 - \varphi_2 \rVert_\infty.
\]
Therefore, $F_{\varphi'}$ is a GENEO. For every $\varphi_1,\varphi_2,\varphi' \in \Phi$ we have that
\begin{equation*}d_{\mathrm{match}}(\dgm_k(F_{\varphi'}(\varphi_1)),\dgm_k(F_{\varphi'}(\varphi_2)))=|d_G(\varphi_1, \varphi') - d_G(\varphi_2,\varphi')|.
\end{equation*}
Indeed, $\dgm_k(F_{\varphi'}(\varphi_1))\setminus \{\Delta\}$ contains at most the point $(d_G(\varphi_1,\varphi'),\infty)$, while $\dgm_k(F_{\varphi'}(\varphi_2))\setminus \{\Delta\}$ contains at most the point $(d_G(\varphi_2,\varphi'),\infty)$. Both points have the same multiplicity, which equals 
the non-zero dimension of $H_k(X)$.

Setting $\varphi'=\varphi_2$, we have that
\begin{equation*}
d_{\mathrm{match}}(\dgm_k(F_{\varphi_2}(\varphi_1)),\dgm_k(F_{\varphi_2}(\varphi_2)))=d_G(\varphi_1,\varphi_2).
\end{equation*}
As a consequence, we have that \begin{equation}
\mathcal{D}^{\mathcal{F}_\mathrm{id}^{\mathrm{all}},k}_{\mathrm{match}}(\varphi_1,\varphi_2)\ge d_G(\varphi_1,\varphi_2).
\end{equation}
By applying Theorem~\ref{t15}, we get
\begin{equation*}
\mathcal{D}^{\mathcal{F}_\mathrm{id}^{\mathrm{all}},k}_{\mathrm{match}}(\varphi_1,\varphi_2)= d_G(\varphi_1,\varphi_2)
\end{equation*}
for every $\varphi_1,\varphi_2$.
\end{proof}
\begin{remark}
The reader might consider replacing the condition “$\Phi$ contains each constant function $c$ for which a function $\varphi \in \Phi $ exists such that $0\le c \le \|\varphi\|_{\infty}$" in Theorem~\ref{thrEqualityDFall_dmatch} with the simpler “$\Phi$ is a space of functions from $X$ to $\R$ containing each constant function $0 \le c \le \mathrm{diam}(\Phi)$". However, we note that in this case, if $\Phi$ contains a function that takes a negative value at some point, it implies that $\Phi$ is unbounded.
\end{remark}
\begin{remark}
We observe that if $\Phi$ is (totally) bounded, the assumption that every function in $\Phi$ is non-negative is not restrictive. 
Indeed, we can obtain it by adding a suitable constant value to every admissible function. 
Moreover, if $X$ is non-empty, then there always exists an index $k$ such that $H_k(X)$ is not zero.
\end{remark}

Next we show how $\mathcal{D}^{\mathcal{F},k}_{\mathrm{match}}$ can be approximated arbitrarily well with a finite subset of operators.

\begin{proposition}
\label{propapprox}
Let $\Phi$ be a compact space and $\mathcal{F}$ be a non-empty subset of $\mathcal{F}^{\mathrm{all}}_T$. For every $\varepsilon>0$, a finite subset $\mathcal{F}'$ of $\mathcal{F}$ exists, such that
\begin{equation*}
\left \lvert \mathcal{D}^{\mathcal{F'},k}_{\mathrm{match}}(\varphi_1, \varphi_2) - \mathcal{D}^{\mathcal{F},k}_{\mathrm{match}}(\varphi_1, \varphi_2)\right \rvert\le \varepsilon
\end{equation*}
for every $\varphi_1,\varphi_2 \in \Phi$.
\end{proposition}

\begin{proof}
Let us consider the closure $\overline{\mathcal{F}}$ of $\mathcal{F}$ in $\mathcal{F}^{\mathrm{all}}_T$. Let us also consider the covering $\mathcal{U}$ of $\overline{\mathcal{F}}$ obtained by taking all the open balls of radius $\frac{\varepsilon}{2}$ centred at points of $\mathcal{F}$, with respect to $D_{\GENEO}$. 
Theorem~\ref{t17} guarantees that $\mathcal{F}^{\mathrm{all}}_T$ is compact, hence $\overline{\mathcal{F}}$ is also compact. 
Therefore we can extract a finite covering $\{B_1, \dots, B_m \}$ of $\overline{\mathcal{F}}$ from $\mathcal{U}$. We can set $\mathcal{F}'$ equal to the set of centers of the balls $B_1, \dots, B_m$.

Now, for every $F\in\mathcal{F}$, a $F'\in\mathcal{F}'$ exists such that
$D_{\GENEO}(F,F')< \frac{\varepsilon}{2}$.
The definition of $D_{\GENEO}$ implies that $\lVert F(\varphi) - F'(\varphi)\rVert_\infty<\frac{\varepsilon}{2}$ for every $\varphi\in \Phi$. From the Matching Distance Stability Theorem~\ref{matchingstabilitythm} it follows that
\begin{equation*}
d_{\mathrm{match}}(\dgm_k(F(\varphi_1)),\dgm_k(F'(\varphi_1))) < \frac{\varepsilon}{2}
\end{equation*}
and
\begin{equation*}
d_{\mathrm{match}}(\dgm_k(F(\varphi_2)),\dgm_k(F'(\varphi_2)))<\frac{\varepsilon}{2}
\end{equation*}
for every $\varphi_1,\varphi_2 \in \Phi$.

By using the reverse triangle inequality and the previous inequalities we can show that 
$d_{\mathrm{match}}(\dgm_k(F(\varphi_2)), \dgm_k(F'(\varphi_1)))$ has a distance less than $\frac{\varepsilon}{2}$ from both the values 
$$d_{\mathrm{match}}(\dgm_k(F(\varphi_1)), \dgm_k(F(\varphi_2)))$$
and 
$$d_{\mathrm{match}}(\dgm_k(F'(\varphi_1)), \dgm_k(F'(\varphi_2))).$$ It follows that
\[
\lvert d_{\mathrm{match}}(\dgm_k(F(\varphi_1)),\dgm_k(F(\varphi_2)))-d_{\mathrm{match}}(\dgm_k(F'(\varphi_1)),\dgm_k(F'(\varphi_2)))\rvert
 < \varepsilon.
\]
Hence,
$\mathcal{D}^{\mathcal{F},k}_{\mathrm{match}}(\varphi_1,\varphi_2)\le \mathcal{D}^{\mathcal{F}',k}_{\mathrm{match}}(\varphi_1,\varphi_2)+ \varepsilon$. 
Since $\mathcal{F}'\subseteq \mathcal{F}$, we also have 
$\mathcal{D}^{\mathcal{F}',k}_{\mathrm{match}}(\varphi_1,\varphi_2)\le \mathcal{D}^{\mathcal{F},k}_{\mathrm{match}}(\varphi_1,\varphi_2)$, 
and the proof follows.
\end{proof}

Proposition~\ref{propapprox} states that the approximation of $\mathcal{D}^{\mathcal{F},k}_{\mathrm{match}}(\varphi_1, \varphi_2)$ can be reduced to the computation of $\mathcal{D}^{\mathcal{F}',k}_{\mathrm{match}}(\varphi_1, \varphi_2)$, i.e. the maximum of a finite set of bottleneck distances between persistence diagrams. Combining Theorem~\ref{thrEqualityDFall_dmatch} and Proposition~\ref{propapprox}, we obtain the following approximation result:
\begin{corollary}
Let $\Phi$ be a compact space. For every $\varepsilon>0$, a finite subset $\mathcal{F}'$ of $\mathcal{F}^\mathrm{all}_\mathrm{id}$ exists such that
\begin{equation*}
\left \lvert d_G(\varphi_1, \varphi_2) - \mathcal{D}^{\mathcal{F}',k}_{\mathrm{match}}(\varphi_1, \varphi_2)\right\rvert\le \varepsilon
\end{equation*}
for every $\varphi_1,\varphi_2 \in \Phi$.
\end{corollary}

\begin{exercise}
Assume that $(F_1,T_1)\colon (\Phi_1,G_1) \to (\Phi_2,G_2)$ and $F_2\colon (\Phi_2,G_2) \to (\Phi_3,G_3)$ are two GENEOs.
Prove that $(F_2 F_1,T_2 T_1)$ is a GENEO.
\end{exercise}

\begin{exercise}
Assume that $(F_1,T)$ and $(F_2,T)$ are GENEOs from $(\Phi,G)$ to $(\Psi,H)$. Prove that $(\max \{F_1,F_2\},T)$ is a GENEO, provided that the inclusion $\max \{F_1,F_2\}(\Phi)\subseteq \Psi$ holds. 
\end{exercise}

\begin{exercise}
Assume that $(F_1,T)$ and $(F_2,T)$ are GENEOs from $(\Phi,G)$ to $(\Psi,H)$. 
Prove that $(\min \{F_1,F_2\}, T)$ is a GENEO, provided that the inclusion $\min \{F_1,F_2\}(\Phi)\subseteq \Psi$ holds.
\end{exercise}

\begin{exercise}
Assume that $(F_1, T), \dots, (F_n, T)$ are GENEOs from $(\Phi,G)$ to $(\Psi,H)$ and $L$ is a 1-Lipschitz map from $\mathbb{R}^n$ to $\mathbb{R}$, where $\mathbb{R}^n$ is endowed with the uniform norm. Consider the map $L^*(F_1,\dots, F_n):\Phi\to C^0(X,\R)$ defined as 
$$L^*(F_1,\dots, F_n)(\varphi):=[L(F_1(\varphi),\dots, F_n(\varphi))], $$
where $[L(F_1(\varphi),\dots, F_n(\varphi))](x):=L(F_1(\varphi)(x),\dots, F_n(\varphi)(x)).$ 

\noindent Prove that $L^*(F_1,\dots, F_n)$ is a GENEO from $(\Phi,G)$ to $(\Psi,H)$, provided that $L^*(F_1,\dots, F_n)(\Phi) \subseteq \Psi$.
\end{exercise}

\begin{exercise}
Assume that $(F_1, T_1)$ is a GENEO from $(\Phi_1,G_1)$ to $(\Psi_1,H_1)$, and $(F_2, T_2)$ is a GENEO from $(\Phi_2,G_2)$ to $(\Psi_2,H_2)$. 
Prove that $(F_1\times F_2, T_1\times T_2)$ is GENEO from $(\Phi_1\times\Phi_2,G_1\times G_2)$ to $(\Psi_1\times \Psi_2,H_1\times H_2)$.
\end{exercise}


\chapter{Further readings}
\label{FR}

The basic ideas underlying what is now known as Topological Data Analysis began to be studied in the early 1990s, and the literature on their subsequent development is now vast. This book does not aim to provide an account of these developments; rather, we include a selection of references that may be useful to readers interested in deepening their understanding of Topological Data Analysis in its various aspects. The list of books below is by no means exhaustive; it is intended simply as a guide to further reading on some of the topics discussed in this text.

\begin{itemize}
\item H. Edelsbrunner, J. Harer, \textit{Computational Topology: An Introduction} (American Mathematical Society, 2009)
\item S. Y. Oudot, \textit{Persistence Theory: From Quiver Representations to Data Analysis} (American Mathematical Society, 2015)
\item J. Tierny, \textit{Topological Data Analysis for Scientific Visualization} (Springer, 2017)
\item F. Chazal, J.-D. Boissonnat, M. Yvinec, \textit{Geometric and Topological Inference} (Cambridge University Press, 2018)
\item R. Rabadan, A. Blumberg, \textit{Topological Data Analysis for Genomics and Evolution} (Cambridge, 2019)
\item G. Carlsson, M. Vejdemo-Johansson, \textit{Topological Data Analysis with Applications} (Cambridge, 2021)
\item T. Dey, Y. Wang, \textit{Computational Topology for Data Analysis} (Cambridge University Press, 2022)
\item H. Schenck, \textit{Algebraic Foundations for Applied Topology and Data Analysis} (Springer, 2022)
\item A. Clark, \textit{The Shape of Data in Chemistry: An Introduction to Graphs and Topological Data Analysis} (Wiley, 2022)
\item P. Joharinad, J. Jost, \textit{Mathematical Principles of Topological and Geometric Data Analysis} (Springer, 2023)
\item M. Olejniczak, \textit{Topological Data Analysis for Quantum Chemistry Practices} (Wiley, 2023)
\end{itemize}
\bibliographystyle{plain}
\bibliography{Arxiv/refsLNTDA}

@article{BaBrHaLaMaSt22,
title = "Computing the Matching Distance of 2-Parameter Persistence Modules from Critical Values",
author = "Asilata Bapat and Robyn Brooks and Celia Hacker and Claudia Landi and Mahler, \{Barbara I.\} and Stephenson, \{Elizabeth R.\}",
year = "2022",
month = oct,
day = "23",
doi = "10.48550/arXiv.2210.12868",
language = "English",
journal = "arXiv e-prints",
publisher = "Cornell University",
}

@article{CeFr20,
 ISSN = {02549409, 19917139},
 URL = {https://www.jstor.org/stable/48581606},
 author = {Andrea Cerri and Patrizio Frosini},
 journal = {Journal of Computational Mathematics},
 number = {2},
 pages = {pp. 291--309},
 publisher = {Institute of Computational Mathematics and Scientific/Engineering Computing},
 title = {A NEW APPROXIMATION ALGORITHM FOR THE MATCHING DISTANCE IN MULTIDIMENSIONAL PERSISTENCE},
 urldate = {2026-03-29},
 volume = {38},
 year = {2020}
}

@article{BjKe23, 
title={Asymptotic improvements on the exact matching distance for $2$-parameter persistence}, 
volume={14}, 
url={https://jocg.org/index.php/jocg/article/view/3341}, 
DOI={10.20382/jocg.v14i1a12}, 
abstractNote={&amp;lt;p&amp;gt;In the field of topological data analysis, persistence modules are used to express geometrical features of data sets. The matching distance $d_\mathcal{M}$ measures the difference between $2$-parameter persistence modules by taking the maximum bottleneck distance between $1$-parameter slices of the modules. The previous fastest algorithm to compute $d_\mathcal{M}$ exactly runs in $O(n^{8+\omega})$, where $n$ is the number of generators and relations of the modules and $\omega$ is the matrix multiplication constant. We improve significantly on this by describing an algorithm with expected running time $O(n^5 \log^3 n)$ and using $O(n^2)$ space. We first solve the decision problem $d_\mathcal{M}\leq \lambda$ for a constant $\lambda$ in $O(n^5\log n)$ by traversing a line arrangement in the dual plane, where each point represents a slice. Then we lift the line arrangement to a plane arrangement in $\mathbb{R}^3$ whose vertices represent possible values for $d_\mathcal{M}$, and use a randomized incremental method to search through the vertices and find $d_\mathcal{M}$. The expected running time of this algorithm is $O((n^4+T(n))\log^2 n)$, where $T(n)$ is an upper bound for the complexity of deciding if $d_\mathcal{M}\leq \lambda$. Moreover, we show how to compute the matching distance using only linear space, to the price of a much worse time complexity.&amp;lt;/p&amp;gt;}, 
number={1}, 
journal={Journal of Computational Geometry}, 
author={Bjerkevik, Havard Bakke and Kerber, Michael}, 
year={2023}, 
month={Dec.}, 
pages={309–342} 
}

@inproceedings{KeLeOu19,
  TITLE = {{Exact computation of the matching distance on 2-parameter persistence modules}},
  AUTHOR = {Kerber, Michael and Lesnick, Michael and Oudot, Steve},
  URL = {https://inria.hal.science/hal-01966666},
  BOOKTITLE = {{SoCG 2019 - International Symposium on Computational Geometry}},
  ADDRESS = {Portland, Oregon, United States},
  YEAR = {2019},
  MONTH = Jun,
  HAL_ID = {hal-01966666},
  HAL_VERSION = {v1},
}

@InProceedings{KeNi20,
  author =	{Kerber, Michael and Nigmetov, Arnur},
  title =	{{Efficient Approximation of the Matching Distance for 2-Parameter Persistence}},
  booktitle =	{36th International Symposium on Computational Geometry (SoCG 2020)},
  pages =	{53:1--53:16},
  series =	{Leibniz International Proceedings in Informatics (LIPIcs)},
  ISBN =	{978-3-95977-143-6},
  ISSN =	{1868-8969},
  year =	{2020},
  volume =	{164},
  editor =	{Cabello, Sergio and Chen, Danny Z.},
  publisher =	{Schloss Dagstuhl -- Leibniz-Zentrum f{\"u}r Informatik},
  address =	{Dagstuhl, Germany},
  URL =		{https://drops.dagstuhl.de/entities/document/10.4230/LIPIcs.SoCG.2020.53},
  URN =		{urn:nbn:de:0030-drops-122116},
  doi =		{10.4230/LIPIcs.SoCG.2020.53}
}

@article{FrGaQuTo25,
author = {Frosini, Patrizio and Garc\'{\i}a, Eloy M\'{o}sig and Quercioli, Nicola and Tombari, Francesca},
title = {Matching {D}istance via the {E}xtended {P}areto {G}rid},
journal = {SIAM Journal on Applied Algebra and Geometry},
volume = {9},
number = {3},
pages = {554-576},
year = {2025},
doi = {10.1137/24M1680398},
URL = {https://doi.org/10.1137/24M1680398},
eprint = {https://doi.org/10.1137/24M1680398}
}

@article{EtFrQuTo23,
author={Ethier, Marc
and Frosini, Patrizio
and Quercioli, Nicola
and Tombari, Francesca},
title={Geometry of the matching distance for 2D filtering functions},
journal={Journal of Applied and Computational Topology},
year={2023},
month={Dec},
day={01},
volume={7},
number={4},
pages={815-830},
issn={2367-1734},
doi={10.1007/s41468-023-00128-7},
url={https://doi.org/10.1007/s41468-023-00128-7}
}

@book{Sp95,
    AUTHOR = {Spanier, Edwin H.},
     TITLE = {Algebraic topology},
      NOTE = {Corrected reprint of the 1966 original},
 PUBLISHER = {Springer-Verlag, New York},
      YEAR = {[1995]},
     PAGES = {xvi+528},
      ISBN = {0-387-94426-5},
   MRCLASS = {55-01},
  MRNUMBER = {1325242},
}

@article{milnor,
    author = {Milnor, John},
    title = {Morse theory},
    journal = {Princeton University Press, NJ},
    year = {1963},
}

@book{milnor1997topology,
  title={Topology from the differentiable viewpoint},
  author={Milnor, John Willard and Weaver, David W},
  volume={21},
  year={1997},
  publisher={Princeton university press}
}

@incollection{tu2011manifolds,
  title={Manifolds},
  author={Tu, Loring W},
  booktitle={An Introduction to Manifolds},
  pages={47--83},
  year={2011},
  publisher={Springer}
}

@book{munkres2018analysis,
  title={Analysis on manifolds},
  author={Munkres, James R},
  year={2018},
  publisher={CRC Press}
}

@article {CeFr15,
    AUTHOR = {Cerri, Andrea and Frosini, Patrizio},
     TITLE = {Necessary conditions for discontinuities of multidimensional
              persistent {B}etti numbers},
   JOURNAL = {Math. Methods Appl. Sci.},
  FJOURNAL = {Mathematical Methods in the Applied Sciences},
    VOLUME = {38},
      YEAR = {2015},
    NUMBER = {4},
     PAGES = {617--629},
      ISSN = {0170-4214},
   MRCLASS = {55N35 (68U05)},
  MRNUMBER = {3310145},
MRREVIEWER = {Barbara Di Fabio},
       DOI = {10.1002/mma.3093},
       URL = {https://doi-org.ezproxy.unibo.it/10.1002/mma.3093},
}

@incollection {EdHa04,
    AUTHOR = {Edelsbrunner, Herbert and Harer, John},
     TITLE = {Jacobi sets of multiple {M}orse functions},
 BOOKTITLE = {Foundations of computational mathematics: {M}inneapolis, 2002},
    SERIES = {London Math. Soc. Lecture Note Ser.},
    VOLUME = {312},
     PAGES = {37--57},
 PUBLISHER = {Cambridge Univ. Press, Cambridge},
      YEAR = {2004},
   MRCLASS = {58E05 (57R45 68U05)},
  MRNUMBER = {2189626},
MRREVIEWER = {Janko Latschev},
}

@article{Wa75,
title = {Morse theory for two functions},
journal = {Topology},
volume = {14},
number = {3},
pages = {217--228},
year = {1975},
issn = {0040-9383},
doi = {https://doi.org/10.1016/0040-9383(75)90002-6},
url = {https://www.sciencedirect.com/science/article/pii/0040938375900026},
author = {Wan, Y.H.}
}

@book {KuMo68,
    AUTHOR = {Kuratowski, K. and Mostowski, A.},
     TITLE = {Set theory},
      NOTE = {Translated from the Polish by M. Maczy\'{n}ski},
 PUBLISHER = {PWN---Polish Scientific Publishers, Warsaw; North-Holland
              Publishing Co., Amsterdam},
      YEAR = {1968},
     PAGES = {xi+417},
   MRCLASS = {04.00},
  MRNUMBER = {0229526},
}

@Article{BeFrGiQu19,
author={Bergomi, Mattia G.
and Frosini, Patrizio
and Giorgi, Daniela
and Quercioli, Nicola},
title={Towards a topological--geometrical theory of group equivariant non-expansive operators for data analysis and machine learning},
journal={Nature Machine Intelligence},
year={2019},
month={Sep},
day={01},
volume={1},
number={9},
pages={423-433},
issn={2522-5839},
doi={10.1038/s42256-019-0087-3},
url={https://doi.org/10.1038/s42256-019-0087-3}
}

@article{BiCeFrGi11,
title = {A new algorithm for computing the 2-dimensional matching distance between size functions},
journal = {Pattern Recognition Letters},
volume = {32},
number = {14},
pages = {1735-1746},
year = {2011},
issn = {0167-8655},
doi = {https://doi.org/10.1016/j.patrec.2011.07.014},
url = {https://www.sciencedirect.com/science/article/pii/S0167865511002273},
author = {Silvia Biasotti and Andrea Cerri and Patrizio Frosini and Daniela Giorgi}
}

@article {CaZo09,
    AUTHOR = {Carlsson, Gunnar and Zomorodian, Afra},
     TITLE = {The theory of multidimensional persistence},
   JOURNAL = {Discrete Comput. Geom.},
  FJOURNAL = {Discrete \& Computational Geometry. An International Journal
              of Mathematics and Computer Science},
    VOLUME = {42},
      YEAR = {2009},
    NUMBER = {1},
     PAGES = {71--93},
      ISSN = {0179-5376},
     CODEN = {DCGEER},
   MRCLASS = {52C35 (68U05)},
  MRNUMBER = {2506738},
       DOI = {10.1007/s00454-009-9176-0},
       URL = {http://dx.doi.org/10.1007/s00454-009-9176-0},
}

@article {CeDFFeal13,
    AUTHOR = {Cerri, Andrea and Di Fabio, Barbara and Ferri, Massimo and
              Frosini, Patrizio and Landi, Claudia},
     TITLE = {Betti numbers in multidimensional persistent homology are
              stable functions},
   JOURNAL = {Math. Methods Appl. Sci.},
  FJOURNAL = {Mathematical Methods in the Applied Sciences},
    VOLUME = {36},
      YEAR = {2013},
    NUMBER = {12},
     PAGES = {1543--1557},
      ISSN = {0170-4214},
   MRCLASS = {55N35},
  MRNUMBER = {3083259},
MRREVIEWER = {Clara L{\"o}h},
       DOI = {10.1002/mma.2704},
       URL = {http://dx.doi.org/10.1002/mma.2704},
}

@Article{CeEtFr19,
author={Cerri, Andrea
and Ethier, Marc
and Frosini, Patrizio},
title={On the geometrical properties of the coherent matching distance in 2D persistent homology},
journal={Journal of Applied and Computational Topology},
year={2019},
month={Dec},
day={01},
volume={3},
number={4},
pages={381-422},
issn={2367-1734},
doi={10.1007/s41468-019-00041-y},
url={https://doi.org/10.1007/s41468-019-00041-y}
}

@article {CSEdHa07,
    AUTHOR = {Cohen-Steiner, David and Edelsbrunner, Herbert and Harer,
              John},
     TITLE = {Stability of persistence diagrams},
   JOURNAL = {Discrete Comput. Geom.},
  FJOURNAL = {Discrete \& Computational Geometry. An International Journal
              of Mathematics and Computer Science},
    VOLUME = {37},
      YEAR = {2007},
    NUMBER = {1},
     PAGES = {103--120},
      ISSN = {0179-5376},
     CODEN = {DCGEER},
   MRCLASS = {68U05 (55N05)},
  MRNUMBER = {2279866},
       DOI = {10.1007/s00454-006-1276-5},
       URL = {http://dx.doi.org/10.1007/s00454-006-1276-5},
}

@article {DoFr04,
    AUTHOR = {Donatini, Pietro and Frosini, Patrizio},
     TITLE = {Natural pseudodistances between closed manifolds},
   JOURNAL = {Forum Math.},
  FJOURNAL = {Forum Mathematicum},
    VOLUME = {16},
      YEAR = {2004},
    NUMBER = {5},
     PAGES = {695--715},
      ISSN = {0933-7741},
   MRCLASS = {58C05 (57R99)},
  MRNUMBER = {2096683},
MRREVIEWER = {Janko Latschev},
       DOI = {10.1515/form.2004.032},
       URL = {https://doi-org.ezproxy.unibo.it/10.1515/form.2004.032},
}

@article {DoFr07,
    AUTHOR = {Donatini, Pietro and Frosini, Patrizio},
     TITLE = {Natural pseudodistances between closed surfaces},
   JOURNAL = {J. Eur. Math. Soc. (JEMS)},
  FJOURNAL = {Journal of the European Mathematical Society (JEMS)},
    VOLUME = {9},
      YEAR = {2007},
    NUMBER = {2},
     PAGES = {331--353},
      ISSN = {1435-9855},
   MRCLASS = {53C23 (58E20)},
  MRNUMBER = {2293959},
MRREVIEWER = {Mario Bonk},
}

@article {DoFr09,
    AUTHOR = {Donatini, Pietro and Frosini, Patrizio},
     TITLE = {Natural pseudo-distances between closed curves},
   JOURNAL = {Forum Math.},
  FJOURNAL = {Forum Mathematicum},
    VOLUME = {21},
      YEAR = {2009},
    NUMBER = {6},
     PAGES = {981--999},
      ISSN = {0933-7741},
   MRCLASS = {58C05 (53A04 58E35)},
  MRNUMBER = {2574144},
MRREVIEWER = {Janko Latschev},
       DOI = {10.1515/FORUM.2009.049},
       URL = {https://doi-org.ezproxy.unibo.it/10.1515/FORUM.2009.049},
}

@book{CT_Edels,
  author    = {Herbert Edelsbrunner and
               John Harer},
  title     = {Computational Topology - an Introduction},
  publisher = {American Mathematical Society},
  year      = {2010},
  url       = {http://www.ams.org/bookstore-getitem/item=MBK-69},
  isbn      = {978-0-8218-4925-5},
  bibsource = {dblp computer science bibliography, https://dblp.org}
}

@article {FrJa16,
    AUTHOR = {Frosini, Patrizio and Jab{\l}o{\'n}ski, Grzegorz},
     TITLE = {Combining persistent homology and invariance groups for shape
              comparison},
   JOURNAL = {Discrete Comput. Geom.},
  FJOURNAL = {Discrete \& Computational Geometry. An International Journal
              of Mathematics and Computer Science},
    VOLUME = {55},
      YEAR = {2016},
    NUMBER = {2},
     PAGES = {373--409},
      ISSN = {0179-5376},
   MRCLASS = {55N35 (47H09 54H15 57S10 65D18 68U05)},
  MRNUMBER = {3458602},
       DOI = {10.1007/s00454-016-9761-y},
       URL = {http://dx.doi.org/10.1007/s00454-016-9761-y},
}

@book {Ha02,
    AUTHOR = {Hatcher, Allen},
     TITLE = {Algebraic topology},
 PUBLISHER = {Cambridge University Press, Cambridge},
      YEAR = {2002},
     PAGES = {xii+544},
      ISBN = {0-521-79160-X; 0-521-79540-0},
   MRCLASS = {55-01 (55-00)},
  MRNUMBER = {1867354},
MRREVIEWER = {Donald W. Kahn},
}

@article{TuMiMuHa14,
author={Turner, Katharine
and Mileyko, Yuriy
and Mukherjee, Sayan
and Harer, John},
title={Fr{\'e}chet Means for Distributions of Persistence Diagrams},
journal={Discrete {\&} Computational Geometry},
year={2014},
month={Jul},
day={01},
volume={52},
number={1},
pages={44-70},
issn={1432-0444},
doi={10.1007/s00454-014-9604-7},
url={https://doi.org/10.1007/s00454-014-9604-7}
}

@book{Greub,
  title={Linear algebra},
  author={Greub, Werner H},
  volume={23},
  year={2012},
  publisher={Springer Science \& Business Media}
}

@book{Munkres,
  title={Elements of algebraic topology},
  author={Munkres, James R},
  year={2018},
  publisher={CRC press}
}

\backmatter
\printindex
\clearpage
\thispagestyle{empty}
\null\vfill
\begin{center}
  This draft manuscript is currently under review for publication as a book.\\
  We would be grateful to readers who report any errors, inaccuracies, or misprints they may encounter to the following email address:\\
\href{patrizio.frosini@unipi.it}{patrizio.frosini@unipi.it}.
\end{center}
\end{document}